\documentclass[a4paper,11pt,reqno]{amsart}
\usepackage{amsmath, amscd, amssymb, hyperref,enumitem}
\usepackage{tabularx}
\usepackage{tabulary}
\usepackage{color}
\usepackage{xcolor}

\usepackage{ulem}

\usepackage{euscript}
\usepackage{epsfig}
\usepackage{tikz-cd}

\usepackage[alphabetic,initials]{amsrefs}

\begin{document}

\theoremstyle{plain} \numberwithin{equation}{section}

\def\attn{{\color{red} Do we need this ???}}    
\def\unf{{\color{red} Unfinished !!!}}    

\newtheorem{thm}{Theorem}[section]
\newtheorem{prop}[thm]{Proposition}
\newtheorem{cor}[thm]{Corollary}
\newtheorem{conj}[thm]{Conjecture}
\newtheorem{lemma}[thm]{Lemma}
\newtheorem{lemdef}[thm]{Lemma--Definition}
\newtheorem{ques}[thm]{Question}
\newtheorem{claim}[thm]{Claim}

\theoremstyle{definition} 

\newtheorem{defn}[thm]{Definition}
\newtheorem{ex}[thm]{Example}
\newtheorem{examples}[thm]{Examples}
\newtheorem{notn}[thm]{Notation} 
\newtheorem{note}[thm]{Note}
\newtheorem{question}[thm]{Question}
\newtheorem{obs}[thm]{Observation}
\newtheorem{rmk}[thm]{Remark}
\newtheorem{terminology}[thm]{Terminology}
\newtheorem{review}[thm]{}
\newtheorem{Empty}[thm]{}
\newcommand{\bi}{\begin{itemize}}  
\newcommand{\ei}{\end{itemize}}
\newcommand{\bp}{\begin{proof}}
\newcommand{\ep}{\end{proof}}

\def\Trick#1{\begin{Empty}\bf#1\end{Empty}}

\def\AA{\mathbb{A}}
\def\CC{\mathbb{C}}
\def\FF{\mathbb{F}}
\def\GG{\mathbb{G}}
\def\GGG{\mathbb{G}}
\def\PP{\mathbb{P}}
\def\cPP{\check{\mathbb{P}}}
\def\SS{\mathbb{S}}
\def\WW{\mathbb{W}}
\def\DD{\mathbb{D}}
\def\arrow#1{\mathop{\ra}\limits^{#1}}

\def\QQ{\mathbb{Q}}
\def\RR{\mathbb{R}}
\def\VVV{\mathbb{V}}
\def\ZZ{\mathbb{Z}}
\def\NN{\mathbb{N}}
\def\VV{\mathbb{V}}

\def\m{\mathfrak{m}}

\def\arr{\mathop{\longrightarrow}}

\def\ov{\overline}

\def\bG{\mathbb G}
\def\bP{\mathbb P}
\def\al{\alpha}
\def\be{\beta}
\def\de{\delta}
\def\eps{\epsilon}
\def\ga{\gamma}
\def\io{\iota}
\def\ka{\kappa}
\def\la{\lambda}
\def\na{\nabla}
\def\om{\omega}
\def\si{\sigma}
\def\th{\theta}
\def\ups{\upsilon}
\def\ve{\varepsilon}
\def\vp{\varpi}
\def\vt{\vartheta}
\def\ze{\zeta}

\def\De{\Delta}
\def\Ga{\Gamma}
\def\La{\Lambda}
\def\Om{\Omega}

\def\cA{\mathcal{A}}
\def\cB{\mathcal{B}}
\def\cC{\mathcal{C}}
\def\cD{\mathcal{D}}
\def\cE{\mathcal{E}}
\def\cF{\mathcal{F}}
\def\cG{\mathcal{G}}
\def\cH{\mathcal{H}}
\def\cI{\mathcal{I}}
\def\cJ{\mathcal{J}}
\def\cK{\mathcal{K}}
\def\cL{\mathcal{L}}
\def\LL{\tilde{L}}
\def\oL{\overline{L}}
\def\cM{\mathcal{M}}
\def\cN{\mathcal{N}}
\def\cO{\mathcal{O}}
\def\cP{\mathcal{P}}
\def\cQ{\mathcal{Q}}
\def\cT{\mathcal{T}}
\def\cTT{\tilde{\mathcal{T}}}
\def\coT{\overline{\mathcal{T}}}
\def\cS{\mathcal{S}}
\def\tcS{\tilde{\mathcal{S}}}
\def\cU{\mathcal{U}}
\def\cHom{\mathcal{H}\text{\it om}}

\def\ocM{\overline{\mathcal{M}}}

\def\tE{\tilde E}
\def\tN{\tilde N}
\def\tQ{\tilde Q}
\def\tX{\tilde X}
\def\tY{\tilde Y}
\def\tS{\tilde S}
\def\tl{\tilde l}

\def\bA{\overline{A}}
\def\bB{\overline{B}}
\def\bcC{\overline{\mathcal{C}}}
\def\bE{\overline{E}}
\def\bF{\overline{F}}
\def\bL{\overline{L}}
\def\bM{\overline{M}}
\def\bZ{\overline{Z}}
\def\bbf{\overline{f}}
\def\bg{\overline{g}}

\def\c{\text{c}}
\def\ch{\text{ch}}
\def\codim{\text{codim}}
\def\dim{\text{\rm dim}}

\def\pr{\text{pr}}

\def\Bl{\text{\rm Bl}}
\def\Eff{\text{\rm Eff}}
\def\Exc{\text{\rm Exc}}
\def\Ext{\text{\rm Ext}}
\def\Sym{\text{\rm Sym}}
\def\Aut{\text{\rm Aut}}
\def\Hom{\text{\rm Hom}}
\def\Coker{\text{\rm Coker}}
\def\Ker{\text{\rm Ker}}
\def\Tor{\text{\rm Tor}}
\def\Cone{\text{\rm Cone}}
\def\h{\text{h}}
\def\inf{\infty}
\def\p{\mathfrak{p}}
\def\HH{\text{H}}
\def\Id{\text{Id}}
\def\LM{\overline{\text{LM}}}
\def\M{\overline{\text{\rm M}}}
\def\N{\overline{\text{\rm N}}}
\def\NE{\overline{\text{NE}}}
\def\PGL{\text{\rm PGL}}
\def\GL{\text{\rm GL}}
\def\SL{\text{\rm SL}}
\def\Pic{\text{\rm Pic}}
\def\Proj{\text{\rm Proj}} 
\def\Spec{\text{Spec}}
\def\Supp{\text{\rm Supp}}
\def\X{\text{\rm X}}

\def\cV{\mathcal{V}}

\def\td{\text{td}}
\def\wt{\text{weight}}
\def\ba{{\bold a}}
\def\bb{{\bold b}}
\def\bbe{{\bold e}}
\def\bbu{{\bold u}}

\def\dra{\dashrightarrow}
\def\hra{\hookrightarrow}
\def\lra{\leftrightarrow}
\def\ra{\rightarrow}

\newcommand{\sslash}{\mathbin{/\mkern-6mu/}}

%%%%%%%%%%%%%%%%%%%%%%%%%%%%%%%%%%%%%%%%%%%%%%%%%%%%%%%%%%%%%%%%%%%%%%%%%
%%%%%%%%%%%%%%%%%%%%%%%%%%%%%%%%%%%%%%%%%%%%%%%%%%%%%%%%%%%%%%%%%%%%%%%%%

\title{Derived category of moduli of pointed curves - II}

\author{Ana-Maria Castravet and Jenia Tevelev}

\dedicatory{In memory of Yuri Manin}

\address{Universit\'e Paris-Saclay, UVSQ, CNRS, Laboratoire de Math\'ematiques de Versailles, 78000, Versailles, France}
\email{ana-maria.castravet@uvsq.fr }

\address{Department of Mathematics and Statistics, University of Massachusetts Amherst, 710 North Pleasant Street, Amherst, MA 01003} % and Laboratory of Algebraic Geometry and its Applications, HSE, Moscow, Russia}
\email{tevelev@math.umass.edu}



%\date{\today}
%\thanks{Mathematics Subject Classification: Primary } 
%\thanks{Keywords: }
\begin{abstract}
We prove a conjecture of Manin and Orlov: the moduli space $\ocM_{0,n}$ of stable rational curves with $n$ marked points admits a full exceptional collection which is equivariant under the action of the symmetric group $S_n$ permuting the marked points. In particular, its $K_0$-group with integral coefficients is a permutation $S_n$-lattice.

\end{abstract}

\maketitle

%%%%%%%%%%%%%%%%%%%%%%%%%%%%%%%%%%%%%%%%%%%%%%%%%%%%%%%%%%%%%%%%
%%%%%%%%%%%%%%%%%%%%%%%%%%%%%%%%%%%%%%%%%%%%%%%%%%%%%%%%%%%%%%%%

\section{Introduction}

This is the final paper in the sequence devoted to the derived category of moduli spaces of stable rational curves with marked points.
We will prove the following theorem conjectured by Manin and Orlov:

\begin{thm}\label{sfzfgafgafga}\label{MainTh}
The Grothendieck--Knudsen moduli space $\ocM_{0,n}$ of stable rational curves with $n$ marked points has a full exceptional collection equivariant under the action of the symmetric group $S_n$ permuting the marked points. In particular, the $K_0$-group of $\ocM_{0,n}$ with integer coefficients is a permutation $S_n$-lattice.
\end{thm}

A sequence of objects $E_1,\ldots,E_r$ in the bounded derived category $D^b(X)$ of a smooth projective variety $X$ over $\CC$ 
is called an {\it exceptional collection} if 
$R\Hom(E_i,E_i)=\CC$ for all $i$ and $R\Hom(E_i,E_j)=0$ for all $i>j$.
An exceptional collection is called {\it full} if $D^b(X)$ is the smallest full triangulated subcategory of $D^b(X)$ containing all $E_i$. 
If a full exceptional collection on $X$ exists, then $K_0(X)$  is freely generated by the classes $[E_1],\ldots,[E_r]$.
%of the objects in the collection. 

%\smallskip 

%In general, there are many examples of full exceptional collections (see for example \cite{BondalOrlov}, \cite{Huybrechts}, \cite{KuznetsovICM}). 

\medskip
\noindent
{\bf Motivation and discussion.} 
The existence of a full exceptional collection on $\ocM_{0,n}$ is a  corollary of 
Kapranov's %blow-up 
description \cite{Kapranov} of $\ocM_{0,n}$ as an iterated blow-up of $\PP^{n-3}$, 
Orlov's theorem on semi-orthogonal decompositions (s.o.d.) of 
blow-ups with smooth centers \cite{Orlov blow-up} and existence of a full exceptional collection
$\cO, \ldots, \cO(k)$  in $D^b(\PP^k)$
\cite{Beilinson}.
However, Kapranov's model is not equivariant and therefore objects in these collections are not permuted by $S_n$. 
Before asking the reader to embark on the long journey, we would like to offer some motivation for the study of of equivariant exceptional collections  and discuss potential directions for future research. 

\medskip
\noindent
{\bf Symmetric $\ocM_{0,n}$.} 
The Deligne--Mumford stack $[\ocM_{0,n}/S_n]$ is a compactification of the moduli space of $n$ unordered points in $\bP^1$.
It frequently appears in moduli theory, for example in the study of the ample cone of~$\ocM_g$, see \cite{GKM}.
We will show that our exceptional collection is $S_n$-equivariant in a  strong sense (see Definition~\ref{sDGASGASFGARH}:)
the $S_n$-orbits of the exceptional objects in Theorem~\ref{MainTh} induce 
a s.o.d.~
of the derived category of $[\ocM_{0,n}/S_n]$ into derived categories of representations of subgroups of $S_n$,
%and more generally, $\M_\ba/\Gamma$, 
see Lemma~\ref{asfvqg}.

\medskip
\noindent
{\bf Keel--Hassett program.} 
From a modern perspective,
$[\ocM_{0,n}/S_n]$ is the $r=1$ case of the Koll\'ar--Shepherd-Barron--Alexeev
compactification of the moduli space of smooth hypersurfaces $D\subset\bP^r$
of degree $n\ge r+2$, viewed as log canonical pairs $(\bP^r,D)$, see \cite{Hac}.
The Keel--Hassett program is the analysis of wall crossing
between moduli spaces of log canonical pairs $(\bP^r,\epsilon D)$ with
$\frac{r+1}{n}\le \epsilon\le1$.
Little is known about the cohomology or derived categories of these spaces but
the hope is that the Keel--Hassett program breaks them into pieces corresponding
to smaller~$\epsilon$ and to the centers of  blow-downs or flips.
For $r=2$, the first model 
($\epsilon=\frac{3}{n}$) was described by Hacking  \cite{Hac}.
For $r=1$, the first model (with $\epsilon=\frac{2}{n}$) is the quotient stack $[\X_n/S_n]$, where $\X_n$ is the symmetric GIT quotient $(\PP^1)^n\sslash\PGL_2$. Via~symplectic reduction, $\X_n$ can be identified with the well-studied moduli space of equilateral polygons in $\Bbb R^3$ \cite{Kly,KaMi,HK}.

The Keel--Hassett  program for $\ocM_{0,n}$,  completed by
Hassett \cite{Ha},  gives a presentation of $\ocM_{0,n}$ as an iterated
$S_n$-equivariant blow-up of $\X_n$.
This presentation, analyzed in detail in Section~\ref{reduction section}, yields an equivariant s.o.d.\ of
$D^b(\ocM_{0,n})$ into derived categories of simple moduli spaces $\M_{p,q}$,
parametrizing nodal curves with $p$ ``heavy'' and $q$ ``light'' points (see Notation \ref{precise weights}). 

Starting in Section~\ref{zxFbFFafsg}, the heavy/light moduli spaces $\M_{p,q}$, equipped
with a natural $S_p\times S_q$-action, become the main
focus of the paper.
We hope that these ``building blocks'' of $\ocM_{0,n}$, and their higher-dimensional
analogues, will have further applications in moduli theory.

\medskip
\noindent
{\bf $S_n$-character of $\HH^*(\ocM_{0,n},\QQ)$.} 
%The $S_n$-character of $\HH^*(\ocM_{0,n},\QQ)$ appears in the  theory of modular operads \cite{GetzlerKapranov,Getzler}.
%, as physical theories provide a representation of these operads in the category of dg vector spaces.
Getzler gave recursive formulas for the $S_n$-character using mixed Hodge theory \cite{Getzler}.
Our approach was inspired by the work of 
Bergstrom and Minabe \cite{BergstromMinabe}, who gave another  algorithm, 
using Hassett's moduli spaces of weighted stable curves \cite{Ha}.
They computed the length of the $S_n$-module 
$\HH^*(\ocM_{0,n},\QQ)$, improving, for genus~$0$, results of Faber and Pandharipande \cite{FP}. %on the length of the $S_n$-module $\HH^*(\ocM_{g,n},\QQ)$. 
%New restrictions on the irreducible representations appearing in the decomposition of the $S_n$-module $\HH^*(\ocM_{g,n},\QQ)$ appear in \cite{Tosteson}, using work of Sam and Snowden on  FS$^{op}$-modules \cite{SS}. Recent work on the $S_n$-module given by the top weight cohomology of $\cM_{g,n}$ appears in  \cite{CFGP}. 
Other work on the $S_n$-representations given by (pieces of) the cohomology of special cases of Hassett spaces, 
such as symmetric GIT quotients $\X_n$, appeared in \cite{HK}. 
All mentioned recursive formulas 
%computing the equivariant Poincar\'e-Serre polynomials of $\ocM_{0,n}$
are, however, not ``effective'' in the sense that the sums involve $\pm$ signs. The fact that 
the $S_n$-module $\HH^*(\ocM_{0,n},\QQ)$ is a permutation representation 
%(i.e., it has a basis that is permuted, as a set, by the action of $S_n$) 
is new and allows for a straightforward decomposition of this module into a sum
of irreducible $S_n$-representations. 

\medskip
\noindent
{\bf Modular vector bundles on $\M_{p,q}$.} A full exceptional collection in $D^b(X)$ allows us to
identify $D^b(X)$ with the bounded derived category of modules over the
endomorphism algebra of this collection \cite{Bon}
(equipped with its natural dg or
$A_\infty$-structure if the collection is not strong, see e.g.~\cite{Bod}). However,
computations of this algebra become hopelessly convoluted (and thus useless) unless
the exceptional objects can be described explicitly.

Exceptional collections on $\M_{p,q}$ contain  vector bundles $F_{l,E}$ of rank $l+1$
%in the exceptional collections on $\M_{p,q}$ are
%introduced  in later sections. They  are 
indexed by an integer $l\ge0$ and a subset $E\subseteq\{1,\ldots,p+q\}$ such that $l+e$ is even,
where $e=|E|$.
%They have the following properties:
%$F_{l,E}$ has rank $l+1$. 
For example, $F_{0,\emptyset}=\cO$.
At~a~point $[C]$ of $\M_{p,q}$ which corresponds to an irreducible curve $C\cong \bP^1$,
% with $n=p+q$ marked points,
the fiber of $F_{l,E}$ is equal to 
$$F_{l,E}|_{[C]}=H^0\bigl(C,\omega_C^{\otimes\frac{e-l}{2}}(E)\bigr),$$
where we identify $E$ with a subset of marked points.
The group $S_p\times S_q$ %, which acts naturally on $\M_{p,q}$ by permuting heavy and light points separately,
acts on the set of  bundles $\{F_{l,E}\}$ via its action on the set of subsets $\{E\}$.

The equivariant structure on objects in our exceptional collections is  either canonical (e.g.~the ideal sheaf of an equivariant Weil divisor) or obtained from a canonical structure via functors between equivariant derived categories (equivariant pushforward, equivariant pullback, etc.). However,  care is necessary when verifying the order of objects in these collections.

\medskip
\noindent
{\bf Symmetric GIT quotient.}
The following theorem is our model case for all subsequent calculations.
Recall that an exceptional collection $E_1,\ldots, E_r$ in  $D^b(X)$ is called {\it strong} if in addition $\Hom(E_i, E_j[k])=0$ for all $i<j$, $k\ne0$.

\begin{thm}\label{symm}
Let $p=2r+1$.
Then $\M_p:=\M_{p,0}$ is a symmetric GIT quotient $\X_p=(\PP^1)^p\sslash\PGL_2$.
The vector bundles $\{F_{l,E}\}$ form a full strong $S_p$-equivariant exceptional collection in $D^b(\M_p)$
under the following condition:
$$l+\min(e,p-e)\leq r-1\quad\hbox{\rm and}\quad l+e\quad\hbox{\rm is even.}$$
The $S_p$-orbits on the set $\{F_{l,E}\}$ are indexed by $l$ and $e$. In the exceptional collection, the objects from each orbit appear as a consecutive block $F_{l,e}$; these blocks are ordered by increasing $e$, and for fixed $e$ their order is arbitrary.
\end{thm}

By Lemma~\ref{asfvqg}, Theorem~\ref{symm} gives a semi-orthogonal decomposition of the bounded derived category of the Deligne--Mumford stack $[\M_p/S_p]$   with 
blocks given by the representation categories of the groups $S_e\times S_{p-e}$ for $l+\min(e,p-e)\leq r-1$ and $l+e$ even. The study of this derived category
may shed new light on the classical invariant theory of binary forms \cite{Dolgachev}.

\begin{examples}
For the first few values of $p=2r+1$, the space $\M_p$ and the reduction map
$\ocM_{0,p}\to\M_p$ can be described as follows. 
%We write $F_{l,e}$ for the $S_p$-orbit of the objects $F_{l,E}$ with $|E|=e$. 
\begin{itemize}
    \item[(p=3)]
$\ocM_{0,3}=\M_3$ is a point. The collection contains $1$ object, $F_{0,0}\cong\cO$.
    
    \item[(p=5)]
The reduction map $\ocM_{0,5}\to\M_5$ is an isomorphism (although the universal families are different).
The space $\M_5$ is a del Pezzo surface of degree~$5$.
The collection contains $7$ vector bundles: 
$F_{0,0}\cong\cO$, five line bundles $F_{0,4}$, which 
can be identified with $\pi_i^*(\cO(1))$ for each conic bundle $\pi_i:\,\M_{5}\to \bP^1$, 
and a rank $2$ vector bundle $F_{1,5}$.\break 
%, which can be identified with the log tangent bundle of $\ocM_{0,5}$.
The map given by the global sections  of $F_{1,5}$
%$H^0(\ocM_{0,5}, F_{1,5})=\CC^5$ 
provides a well-known embedding of $\M_5$ %(the del Pezzo surface of degree $5$) 
into the Grassmannian $G(2,5)$.
    
    \item[(p=7)]
The reduction map  $\ocM_{0,7}\to\M_7$ is the blow-up of $35$ planes $\bP^2\subset \M_7$
intersecting transversally in $70$ points.
The exceptional collection on $\M_7$ contains $38$ vector bundles:
$F_{0,0}, F_{2,0}, F_{1,1}, F_{0,2}, F_{0,6}, F_{1,7}$.
%As $p$ grows, the reduction morphism  $\ocM_{0,p}\to\M_p$ factors into more and more steps.
\end{itemize}
\end{examples}

%This process is analyzed in detail in Section \ref{reduction section}.

%The exceptional collection on $\M_9$ contains $187$ objects: 
%$$F_{0,0}\ F_{2,0}\ F_{1,1}\ F_{0,2}\ F_{0,6}\ F_{1,7}\ F_{0,8}\ F_{2,8}\ F_{1,9}\ F_{3,9}.$$
%The number of objects in the exceptional collections of $\M_{2r+1}$ is:
%\begin{equation}\label{formula_p_odd}
%r+(r-1)\cdot{2r+1 \choose 1}+(r-2)\cdot{2r+1 \choose 2}+\ldots+1\cdot {2r+1 \choose r-1}.
%\end{equation}

\medskip
\noindent
{\bf Arithmetic aspects.}
Permutation actions of $\Aut(X)$ on $K_0(X)$ seem rare; see Example~\ref{asgashawrb}
for a variety with an involution such that the corresponding $S_2$–action on $K_0(X)$
is not a permutation representation. A great insight of Yuri Manin was that $\ocM_{0,n}$ is
“defined over $\mathbb F_1$”. Thus one expects that $K_0(\ocM_{0,n})$, as an $S_n$–representation,
is also “defined over $\mathbb F_1$” and hence must be a permutation representation. Indeed,
one imagines that a vector space over $\mathbb F_1$ is  a finite set (whose elements become basis vectors after tensoring with an honest field or ring), a general linear group is a symmetric group, and thus a representation of $\Gamma$ is just a permutation representation.

In Theorem~\ref{MainHassettTh} we show that more general Hassett spaces carry exceptional
collections equivariant under their automorphism groups. Another standard example of a
variety “defined over $\mathbb F_1$” is a smooth projective toric variety~$X$, so it is
natural to ask whether $D^b(X)$ has a full $\Gamma$–equivariant exceptional collection,
and thus $K_0(X)$ is a permutation $\Gamma$–module, for every finite subgroup
$\Gamma\subset\Aut(X)$ normalizing the torus action. A similar question in the arithmetic
setting was raised by Merkurjev and Panin \cite{MP}, and we hope that our techniques  will be useful in addressing~it.
We also expect our results to extend to the arithmetic setting. The analogue
of Theorem~1.2 should hold for every form $\mathcal X$ of $\M_p$ defined over a non-algebraically
closed field $k$, or for the total space of a smooth morphism $\mathcal X\to B$ with
fibers isomorphic to $\M_p$. In both cases, we expect a $k$–linear (resp.\ $B$–linear) semi-orthogonal
decomposition of $D^b(\mathcal X)$ with blocks indexed by $l$ and $e$ as in Theorem~1.2.
For $p=5$, this result appears in \cite{ABer}.

\medskip
\noindent
{\bf Connection to quantum cohomology.}
The study of $D^b(\ocM_{0,n})$ was initiated in the
work of Manin and Smirnov \cite{ManinSmirnov1} %(see also \cite{Smirnov Thesis,ManinSmirnov2}) 
and of Ballard, Favero and Katzarkov \cite{BFK}. Part of the motivation in \cite{ManinSmirnov1} was to study 
the relationship between derived categories and quantum cohomology. We expect that our equivariant exceptional collections, via the Dubrovin/Gamma conjecture \cites{Dubrovin, GGI}, manifest themselves in the clustering behavior of the eigenvalues of the quantum multiplication by the Euler field. For concreteness, consider again the setup of Theorem~\ref{symm} rather than Theorem~\ref{sfzfgafgafga}. While the Dubrovin/Gamma conjecture predicts that the quantum spectrum  
consists of isolated eigenvalues for a general point $\tau\in\HH^*(\M_{p},\CC)$
(the base of the big quantum cohomology), we expect that the eigenvalues coalesce into groups indexed by $(l,e)$ as in Theorem~\ref{symm} for a general point in $\HH^*(\M_{p},\CC)^{S_p}$. 
Concretely, the quantum spectrum at the anticanonical point of the small quantum cohomology of $\M_p$ (with $p=2r+1$) should be equal to $\{2q(-1)^{s-1}\cos{2\pi t\over q}\}$
with $q=2s+1$, $s=1\ldots,r$, $t=1,\ldots,s$.
When $p=5$, this agrees with the calculation of Bayer and Manin \cite{BaMa}: the quantum spectrum consists of $\frac{5+5\sqrt{5}}{2}$ (which corresponds to $F_{0,0}=\cO$ in the notation of Theorem~\ref{symm}),
$\frac{5-5\sqrt{5}}{2}$
(which corresponds to $F_{1,5}$),
and the eigenvalue $-3$ with multiplicity $5$
(which corresponds to line bundles $F_{0,4}$.)

%\smallskip

%Even ignoring the $S_n$ action, 
%there has been a lot of work on cohomology %the Chow ring and the Poincar\'e polynomial 
%of $\ocM_{0,n}$.  
%Keel gave a presentation of the Chow ring and recursive formulas for the Betti numbers in %\cite{Keel}, with some further work by 
%Fulton and MacPherson \cite{FM} and Manin \cite{Manin}. 
%The Chow rings of Hassett spaces of weighted stable rational curves has been calculated in \cite{Ceyhan}.

\medskip
\noindent
{\bf Strategy of the proof of Theorem  \ref{sfzfgafgafga}.} 
In the rest of the introduction, we present our main results and explain the structure of the proof.
We follow a strategy outlined in \cite{CT_part_I} and 
inspired by the work of \cite{BergstromMinabe}. 

The first step, accomplished in Section~\ref{reduction section}, is to construct an equivariant s.o.d.\ of 
$D^b(\ocM_{0,n})$ from the equivariant s.o.d.'s of  $D^b(\M_{p,q})$ (for various $p$ and~$q$), where 
$\M_{p,q}$ is the Hassett space parametrizing $p$ ``heavy'' and $q$ ``light'' points on a nodal rational curve (see Notation \ref{precise weights}):

\begin{defn}\label{sgSgSgSgSgSRg}
$\M_{p,q}$ is defined by the following stability conditions:
\begin{itemize}
\item
When $p=2r+1$ is odd, at most $r$ of the heavy points may coincide, and moreover they may coincide with all the light points.
\item
When $p=2r$ is even, at most $(r-1)$ of the heavy points may coincide with each other and with all the light points.
Furthermore, if $q>0$ then $r$ heavy points may coincide, and they may further coincide with at most $\lfloor \frac{q-1}{2}\rfloor$ light points.
\end{itemize}
\end{defn}

The group $S_p\times S_q$ acts on $\M_{p,q}$ by permuting heavy and light points
separately. The bulk of the 
paper is devoted to the equivariant description of $D^b(\M_{p,q})$ in
%. The result for $\M_{p,q}$ is a combination of 
Theorems ~\ref{symm}, \ref{dfkfvkfvjkfvjk}, \ref{asfffdvzsfvsfb},
\ref{asdvzsfvsfb}, and \ref{even 0}.

When $p$ is odd, the space $\M_p$ is isomorphic to the symmetric GIT quotient
$(\PP^1)^p\sslash\PGL_2$, and we prove Theorem~\ref{symm} using 
the theory of windows into derived categories of GIT quotients
\cites{Teleman,DHL,BFK}, which we briefly review in
Section~\ref{windows section}. %, where we also study the derived category of the stack of $\bP^1$-bundles with $n$ sections. 
This section serves
as a model for the rest of the paper: the results are easy to state and
the proofs are quite transparent.\break
If $p$ is odd and $q>0$ then we have a morphism $\M_{p,q}\to \M_p$, which is
an iterated ($q$ times) universal $\bP^1$-bundle of $\M_p$.
By applying Orlov's theorem on the derived category of a projective bundle, we 
immediately construct an equivariant exceptional collection in this case
(see Theorem~\ref{symm+}.)
 
\medskip
\noindent
{\bf Torsion sheaves $\cT_{l,E}$.}
The bulk of the paper, starting in Section \ref{inequalities section}, is focussed on the difficult case of even $p$.
Some of the vector bundles $F_{l,E}$ in this case have to be replaced with new objects, torsion sheaves $\cT_{l,E}$ on $\M_{p,q}$:

\begin{notn}\label{T}
Let $P$ (resp.,~$Q$) be the set of heavy (resp., light) points.
For every subset $E\subseteq P\cup Q$, we denote by $E_p$ (resp., ~$E_q$)
its intersection with $P$ (resp., ~$Q$) and their cardinalities by $e_p$ and $e_q$.
For $p=2r\geq 4$, $q\geq 1$, $R\subseteq P$, $|R|=r$, 
let $i_R: Z_R\hra \M_{p,q}$ be the locus in $\M_{p,q}$ where the points from $R$ come together.  
Let $\pi_R: \cU_R\ra Z_R$ be the universal family over  $\M_{p,q}$ restricted to $Z_R$
and let 
$\si_u$ be the section of $\pi_R$ that corresponds 
%to the marked point~$u$. 
%If $u$ denotes the marking corresponding 
to the combined points of $R$.
% and $R':=P\setminus R$, then 
%$Z_R$ is isomorphic to the Hassett space with points marked by $u$, $R'$ and $Q$.
%Let $\pi_R: \cU_R\ra Z_R$ be the universal family over of $\M_{p,q}$ restricted to $Z_R$. 
For $l\geq 0$, $E\subseteq P\cup Q$ with $e=|E|$ such that $e+l$ is even and 
$E_p=R$, consider the following torsion sheaf on $\M_{p,q}$: 
$$\cT_{l,E}={i_R}_*\si_u^*\bigl(\omega_{\pi_R}^{\frac{e-l}{2}}(E)\bigr),$$
where we identify $E$ with a subset of sections of the universal family. 
%where $i_R: Z_R\hra \M_{p,q}$ is the inclusion map.
%5 and $\si_u$ is the section of the universal family $\pi_R$ that corresponds to the marked point~$u$.
%$\cV_{l,E}$ is the line bundle 
%$$\cV_{l,E}=\si_u^*\left(\omega_{\pi_R}^{\frac{e-l}{2}}(E)\right)=\frac{e-l}{2}\psi_u+\sum_{j\in E}\de_{ju}=\frac{e_q-r-l}{2}\psi_u+\sum_{j\in E_q}\de_{ju}.$$
%Note that when $q$ is odd, we have  
%$$\cV_{l,E}=-\left(\frac{r+l}{2}\right)\psi_u-\frac{1}{2}\sum_{j\in E_q}\psi_j.$$
\end{notn}

\begin{thm}\label{asfffdvzsfvsfb}\label{p,q case}
Let $p=2r\geq 4$, $q=2s+1\ge1$. 
A full $S_p\times S_q$ equivariant exceptional collection on $\M_{p,q}$
consists of the following objects parametrized by integers $l\geq0$ and subsets $E\subset P\cup Q$ such that $l+e$ is even:
\smallskip

(group $1A$)
The vector bundles $F_{l,E}$ for $l+\min(e_p, p+1-e_p)\le r-1$;

(group $2$)
The torsion sheaves  $\cT_{l,E}$ for 
$e_p=r$,  $l+\min(e_q, q-e_q)\le s-1$.

\noindent
We denote the $S_p\times S_q$ orbits of $\cT_{l,E}$ and $F_{l,E}$ by $\cT_{l,r,e_q}$ and $F_{l,e_p,e_q}$, respectively.
In the exceptional collection, the objects from each orbit appear as a consecutive block. We first order blocks by $e_q$ in the increasing order. For fixed $e_q$, 
the blocks $\cT_{l,r,e_q}$ come first, ordered by 
decreasing $l$. Then we put the blocks $F_{l,e_p,e_q}$, ordered  by increasing $e_p$, and, for fixed $e_p$, in an arbitrary order in $l$.
\end{thm}

\begin{rmk}\label{ajhsbvajhsbvakjhsbvks}
There are other ways to reorder this exceptional collection if we do not insist that objects from the same orbit appear consecutively. Namely, we can arrange these objects in blocks indexed by a subset $E_q$. %The order is as follows: 
Blocks are ordered by increasing $e_q$ and arbitrarily if $e_q$ is the same (but the set $E_q$ is different).
Within  each block with the same $E_q$ we can put the sheaves  $\{\cT_{l,E}\}$ first,
in arbitrary order if $E_p\neq E'_p$ and in order of decreasing $l$ when $E_p=E'_p$.
After the sheaves we put  the bundles $\{F_{l,E}\}$
in order of increasing  $e_p$ and, for a given $e_p$, arbitrarily. 

Put differently, the exceptional collection has many trivial $R\Hom$’s in the forward direction, which can be useful for describing its quiver.
\end{rmk}

%\begin{rmk}
%Note that we have an isomorphism of space $\M_{p,1}\cong\M_{p+1}$ but the collections from Theorem  \ref{p,q case} and Theorem ~\ref{symm} are \emph{not} the same. Furthermore,  

\begin{rmk}
Theorem~\ref{asfffdvzsfvsfb} and the next Theorem~\ref{asdvzsfvsfb}
remain true if $r=1$, i.e., when there are only $p=2$ heavy points.
With small modifications, the proofs are essentially the same.
For brevity, we do not include the case $p=2$ here, since it was already settled in \cite{CT_part_Ib}. 
In this case, the collection in Theorem~\ref{asfffdvzsfvsfb} contains the vector bundles $F_{0,E}$ with $E\subseteq Q$ arbitrary and the objects $\cT_{l,E}$ in group~2 (which are line bundles, since $Z_R$ is the whole moduli space). The dual collection (see Remark~\ref{dual collection} for the definition) was constructed in \cite{CT_part_Ib}*{Theorem~1.6, see also Remark~3.7}.
\end{rmk}

\medskip
\noindent
{\bf Subcategories $\cA$, $\cB$, and complexes $\tilde\cT_{l,E}$.}
For the case when the number of heavy and light points are both even, we need yet another type of objects. 

\begin{notn}\label{TT}
For inductive purposes
we always assume that $q=2s+1$ is odd and work on $\M_{p,q+1}$.
In particular, $|Q|=q+1$ in this case.
The space $\M_{p,q+1}$
is a  resolution of singularities of the asymmetric GIT quotient $\X_{p,q+1}$ (see Notation~\ref{precise weights})
with $\frac{1}{2}{p\choose r}{q+1\choose s+1}$  exceptional divisors $\delta=\bP^{r+s-1}\times\bP^{r+s-1}$.

Let 
$t=r+s-1$ and 
let $\cA\subset D^b(\M_{p,q+1})$ be the triangulated subcategory generated by the torsion sheaves 
$\cO_\delta(-a,-b)$
where either 
%$1\le a,b\le r+s-1$ or\break $a=0$, $1\le b\le {r+s-1\over2}$ or $b=0$, $1\le a\le {r+s-1\over 2}$.

$$
1\le a,b\le t
\quad \hbox{\rm or}\quad
a=0,\ 1\le b\le t/2
\quad \hbox{\rm or}\quad
b=0,\ 1\le a\le t/2.
$$

Let $\cB={}^\perp\cA=\{T\in D^b(\M_{p,q+1})\ |\ \Hom(T,A)=0\ \hbox{\rm for every}\ A\in\cA$\}.  
\end{notn}

We prove in Section~\ref{exceptional p,q+1} that
$\cA$ is an admissible $(S_p\times S_{q+1})$ equivariant subcategory and that 
$\cB$ is an  $(S_p\times S_{q+1})$ equivariant non-commutative resolution of singularities
of the GIT quotient $\X_{p,q+1}$ in the sense of \cite{KuLu}.
While $\X_{p,q+1}$ has many small resolutions related by flops obtained by contracting boundary divisors
$\delta$ onto one of the factors $\PP^{r+s-1}$, none of them
is  $(S_p\times S_{q+1})$ equivariant unlike the category $\cB$,
which in some sense is the ``minimal'' {\it equivariant} resolution. 
% Perhaps the biggest technical issue contributing to the complexity of this case
% is that $\cB$ is a non-commutative {\em strongly crepant} resolution only if $r+s$ is odd.
The vector bundles $F_{l,E}$ belong to $\cB$ (Proposition~ \ref{perpendicular general}) but 
the torsion objects have to be projected onto $\cB$:

\begin{defn}\label{sDGASRHADHA}
We define the objects in $\cB\subset D^b(\M_{p,q+1})$ by
$\cTT_{l,E}=(\cT_{l,E})_{\cB}$,
where $T\to T_{\cB}$ is a canonical functorial projection
and the torsion sheaf $\cT_{l,E}$ is defined as in Notation~\ref{T}
%Let $p=2r\geq 4$, $q=2s+1\geq 1$, $|P|=p$, $|Q|=q+1$, $R=\{1,\ldots, r\}\subseteq P$. 
for $E\subseteq P\cup Q$, $l\geq0$, $e+l$ even.
%we let as before $$\cT_{l,E}={i_R}_*\big(\frac{e_q-r-l}{2}\psi_u+\sum_{j\in E_q}\de_{ju}\big)$$
%where  $i_R: Z_R\hra \M_{p,q}$ is the inclusion map, $Z_R$ is the locus where the points in $R$ coincide, with $u$ the corresponding marking. 
\end{defn}

\begin{thm}\label{asdvzsfvsfb}\label{p,q+1 case}
Let $p=2r\geq 4$, $q=2s+1\ge1$.
The space $\M_{p,q+1}$ has a full $S_p\times S_{q+1}$ equivariant  exceptional  collection of 
torsion sheaves $\cO_\delta(-a,-b)$ %in subcategory $\cA$ 
(see Notation \ref{TT})
in decreasing order of $a+b$
followed by the following objects 
parametrized by integers $l\geq0$ and subsets $E\subset P\cup Q$ such that $l+e$ is even:
\smallskip

(group $1A$) The vector bundles $F_{l,E}$ for $l+\min(e_p, p+1-e_p)\le r-1$;

(group $2B$) The complexes  $\cTT_{l,E}$ for  $e_p=r$, 
$l+\min(e_q, q+2-e_q)\le s$.
  
  \smallskip
 \noindent
When combining group $2B$ with $1A$, the order is the same as in  
Theorem~\ref{asfffdvzsfvsfb} or Remark~\ref{ajhsbvajhsbvakjhsbvks}.
\end{thm}

The same collection as in Theorem  \ref{p,q+1 case} works for the case of $s=-1$. We state it as a separate theorem because it is of independent interest.

\begin{thm}\label{even 0}
If $p=2r\ge4$, then $\M_p$
is the Kirwan resolution of singularities of the 
symmetric GIT quotient $\X_p=(\PP^1)^p\sslash\PGL_2$.
It admits a full $S_p$-equivariant exceptional collection consisting of the 
torsion sheaves
$\cO_{\bP^{r-2}\times\bP^{r-2}}(-a,-b)$, where 
$$
1\le a,b\le r-2
\quad \hbox{\rm or}\quad
a=0,\ 1\le b\le r/2-1
\quad \hbox{\rm or}\quad
b=0,\ 1\le a\le r/2-1,
$$
arranged in decreasing order of $a+b$,
followed by

\noindent
(group $1A$) 
the vector bundles $F_{l,E}$ with $l+e$ even and 
$l+\min(e, p+1-e)\le r-1$.

\smallskip

The order of blocks $F_{l,e}$ is first by increasing $e$, and, for a given $e$, arbitrary.

\end{thm}

\begin{ex}\label{M4} For
$\M_4\cong\ocM_{0,4}\cong\bP^1$,  $\cA$ is empty. The collection contains $2$ objects, 
$F_{0,0}\cong\cO$ and $F_{0,4}\cong\cO(1)$. In fact $F_{0,P}$ is always the pull-back of the GIT
polarization from the symmetric GIT quotient $\X_p$ (see Corollary~\ref{nbvcnbvcnvbc}).

The space $\X_6$ is the Segre cubic threefold in $\bP^4$. %with $10$ singularities.
The space $\M_6\cong\ocM_{0,6}$ 
is the blow-up of $10$ singularities of $\X_6$ with exceptional divisors $\bP^1\times\bP^1$.
%(this is the last case when $\M_p\cong\ocM_{0,p}$).
The category $\cA$ contains torsion sheaves $\cO_{\bP^1\times\bP^1}(-1,-1)$.
The category $\cB$ is a weakly crepant resolution of $\X_6$ with the full exceptional collection
of $24$ vector bundles of type 
$F_{0,0},  F_{2,0}, F_{1,1}, F_{0,2}, F_{0,6}$.
\end{ex}

\begin{rmk}\label{bjbhjf,ds,jbahs,bj}\label{SgSRgSRhSRh}\label{1B stuff}
In the setup of Theorems~\ref{p,q case}, \ref{p,q+1 case} and ~\ref{even 0}, 
we will also show that there is another full equivariant exceptional collection, obtained by 
swapping vector bundles from group 1A to vector bundles in
\smallskip

(group $1B$) Vector bundles $F_{l,E}$ for $l+\min(e_p+1, p-e_p)\le r-1$.

\smallskip
\noindent
The order is the same. The groups $1B$ and $1A$ are related by the operation
%taking the complement 
$E\mapsto E^c$. 
From this perspective, 
in Theorem  \ref{p,q+1 case},
we may also consider

\smallskip

(group $2A$) Complexes $\{\cTT_{l,E}\}$ for $e_p=r$, $l+\min(e_q+1, q+1-e_q)\le s$ 

\smallskip
\noindent
and combine them with vector bundles from groups $1A$ or $1B$.
The same proof as in Theorem  \ref{p,q+1 case}
shows that this collection is exceptional and of the expected length, but we did not attempt to prove fullness. 
\end{rmk}

%\noindent {\bf Connection with  \cite{CT_part_I}}. 
\noindent {\bf Cuspidal block}. 
In the first paper of this project \cite{CT_part_I} we found a full $S_2\times S_n$-equivariant exceptional collection on the Losev--Manin space $\LM_n$.
The approach we take is different for the two types of spaces. 
In \cite{CT_part_I} we proved that for both $\LM_n$ and $\ocM_{0,n}$ it suffices to find a full equivariant exceptional collection in 
the {\it cuspidal block} of the derived category (i.e., objects that push forward to $0$ by all the forgetful maps) and we found it for the Losev--Manin space. 
The existence of a full $S_n$-equivariant exceptional collection in $D^b_{\mathrm{cusp}}(\ocM_{0,n})$ is an interesting open problem.

\smallskip

\noindent {\bf Structure of the sections focusing on
$\M_{p,q}$ for even $p$}. 
In Section~\ref{inequalities section} we  study numerical functions
of pairs  $(l,E)$, including the {\em score} $S(l,E)$.\break 
%, which is crucial for proving fullness.
% and prove exhibit various inequalities involving the pairs $(l,E)$ in Theorem  \ref{p,q case} and Theorem  \ref{p,q+1 case}. %These will be crucial for all of the remaining sections. 
In~Section \ref{extend section} we extend the definition of the vector bundles $F_{l,E}$ to more general Hassett spaces, % (with non-empty boundary), 
while in Section~\ref{F vs boundary section} we give %sufficient 
conditions for these vector bundles %$F_{l,E}$ 
to be orthogonal to torsion sheaves supported on the boundary. 

The exceptionality of the ``$F_{l,E}$-part" of the collections in Theorem  \ref{p,q case}
%(
%$\M_{p,q}$, %$p$ even, $q$ odd) 
and Theorem  \ref{p,q+1 case} 
%($\M_{p,q}$, %$p$, 
%$q$ %both  even) 
is proven in Section~\ref{induction2} and Section \ref{induction1} by induction on 
the number of light points %$q$
and %. Note that the case $q$ odd and $q$ even 
require different arguments (hence, the two sections). We always assume that $q$ is odd
and introduce the \emph{alpha game} to go from $\M_{p,q-1}$ to $\M_{p,q}$ % when $q$ is odd
and the  \emph{beta game} to go from $\M_{p,q}$ to $\M_{p,q+1}$. 
%Both ``games"   compare orthogonality between $F_{l,E}$'s on different Hassett spaces related by reduction maps. 

We finish the proof of the exceptionality in
%of the collection on $\M_{p,q}$ ($p$ even, $q$ odd) of 
Theorem  \ref{p,q case}  (i.e., including the ``$\cT_{l,E}$ part") in Section \ref{exceptional p,q section} by reducing it to a windows calculation on subvarieties $Z_R\subseteq \M_{p,q}$ (the supports of the torsion sheaves 
$\cT_{l,E}$) % in the collection) 
and their intersections. 
%Note that a direct windows argument on $\M_{p,q}$ will not give exceptionality of most pairs in the collection, as the condition on weights is not satisfied. 
The exceptionality %of the collection  on $\M_{p,q+1}$ ($p$ even, $q$ odd)  
in Theorem  \ref{p,q+1 case} is finished in Section \ref{exceptional p,q+1 section},
%and requires yet new reduction maps in order to compare exceptionality of different pairs, by reducing again to a window calculation on subvarieties $Z_R\subseteq \M_{p,q+1}$ and their intersections. 
%It is 
where we compare commutative (not equivariant) and non-commutative 
(equivariant) small resolutions of the singular GIT quotient $\X_{p,q+1}$.
 
Fullness of the exceptional collections on all the spaces $\M_{p,q}$ is proved in the remaining sections:
 Section \ref{fullness p,q section} for the collections in Theorem  \ref{p,q case} and in Section \ref{fullness p,q+1 section} for the collections in Theorem  \ref{p,q+1 case}. 
%Even in the most basic case of Theorem  \ref{symm}, where the exceptional collection is contained
%in the Halpern-Leistner's "window", we were unable to use directly his main theorem (see  Theorem ~\ref{Kirwan}) to prove fullness.
%However, we were inspired by its proof utilizing Koszul resolutions of the unstable strata, although in
%our case we had to work on the universal family rather than on the moduli stack as in \cite{DHL}.
%We use several constructions based on the Koszul complex (\emph{Koszul games})  in order to prove fullness.
%One of them  ``replaces" the torsion sheaves $\cT_{l,E}$ (resp., complexes
%$\cTT_{l,E}$) with the bundle $F_{l,E}$ for the the same pair $(l,E)$. We emphasize that the collections obtained by replacing   $\cT_{l,E}$ (resp., $\cTT_{l,E}$) with the corresponding 
%$F_{l,E}$'s are \emph{not} exceptional in general (although we prove that they are full). 
%We find this phenomenon very interesting and worthy of further exploration even on toric varieties
%or whenever one has natural full but not exceptional collections of line or vector bundles.
%We remark that proving fullness in Section  \ref{fullness p,q+1 section} for the collection on $\M_{p,q+1}$ ($q$ odd) involves yet again both the \emph{alpha game} and the \emph{beta game} and relies on having
%proved already fullness for the collection on $\M_{p,q}$ and $\M_{p,q+2}$ (Section  \ref{fullness p,q section}). 
%\medskip

{\bf Acknowledgements.}
We thank Ben Bakker, Matt Ballard, Arend Bayer,
Daniel Halpern-Leistner,
Alexander Kuznetsov,  Chunyi Li, Emanuele Macr\`i,
Georg Oberdieck, Alex Perry, Olivier Schiffmann and Maxim Smirnov
for helpful conversations related to this work
and
Matt Ballard, Alex Duncan, Patrick McFaddin for sending us a preliminary version of their work related to
Theorem  \ref{symm}. We are very grateful to the referees for their careful reading of the paper and for  many thoughtful and insightful comments and suggestions.
%The first named author 
A.-M. C.~was supported by the NSF grant DMS-1701752, the ANR grant FanoHK and the Institut Universitaire de France. J. T.~ was supported by the NSF grants DMS-1701704 and DMS-2101726,
as well as Fulbright and Simons Fellowships.
%and the HSE University Basic Research Program and Russian Academic Excellence Project `5-100'. 

%%%%%%%%%%%%%%%%%%%%%%%%%%%%%%%%%%%%%%%%%%%%%%%%%%%%%%%%%%%%%%%%%%%%%%%%%
%%%%%%%%%%%%%%%%%%%%%%%%%%%%%%%%%%%%%%%%%%%%%%%%%%%%%%%%%%%%%%%%%%%%%%%%%

%\tableofcontents
\tableofcontents

\section{Reduction to the heavy/light moduli spaces $\M_{p,q}$'s}
\label{axfbzdfbafb}\label{reduction section}

%There are two natural notions how a collection can be ``permuted'' by the action of $\Gamma$:
The goal of this section is to reduce the study of derived category of $\ocM_{0,n}$ to the study of full equivariant exceptional collections on heavy/light moduli spaces $\M_{p,q}$ (see Notation~\ref{precise weights}.) In particular, we will prove our main Theorem~\ref{MainTh} conditionally on the validity of Theorems ~\ref{symm}, \ref{dfkfvkfvjkfvjk}, \ref{asfffdvzsfvsfb},
\ref{asdvzsfvsfb}, and \ref{even 0}, which will be proved in subsequent sections. In fact, the same analysis proves existence of full equivariant exceptional collections for a large range of Hassett  spaces,
see Theorem~\ref{MainHassettTh},
which includes 
$\ocM_{0,n}$.

We first recall some subtle differences between various notions of exceptional collections “permuted by a group”, as investigated by Elagin \cite{Elagin}.

\begin{defn}\label{swfwefawrgqwrgqw}
Let a finite group $\Gamma$ act on a smooth projective variety~$X$. We say that $\Gamma$ {\it permutes objects} of an exceptional collection
$\{\cE_\alpha^\bullet\}_{\alpha\in I}$ in $D^b(X)$ 
%An exceptional collection 
if, for every $\gamma\in \Gamma$ and $\alpha\in I$, there exists $\beta\in I$ such that
$\gamma^*\cE_\alpha^\bullet\cong\cE_\beta^\bullet$.
\end{defn}

\begin{lemma}
If $\Gamma$ permutes objects of a full 
exceptional  collection %$\{\cE_\alpha^\bullet\}_{\alpha\in I}$ in $D^b(X)$ 
then
\begin{itemize}
\item[(a)] 
$R\Hom(\cE_\alpha^\bullet, \cE_\beta^\bullet)=0
$
if $\cE_\alpha^\bullet$ and 
$\cE_\beta^\bullet$ are in the same $\Gamma$-orbit.
\item[(b)]
$K_0(X)$ is a permutation $\Gamma$-lattice.
\end{itemize}
\end{lemma}

\begin{proof}
In part (a), if $R\Hom(\cE_\alpha^\bullet,\gamma^*\cE_\alpha^\bullet)\ne0$ for some $\gamma\in\Gamma$, then iterating $\gamma$ produces a cycle
$
\cE_\alpha^\bullet,\ \gamma^*\cE_\alpha^\bullet,\ (\gamma^2)^*\cE_\alpha^\bullet,\ \ldots,\ (\gamma^p)^*\cE_\alpha^\bullet=\cE_\alpha^\bullet,
$
where $p$ is the order of $\gamma$, with non-trivial forward $R\Hom$'s, which is impossible in an exceptional collection. In part (b), since $K_0(X)$ is the Grothendieck $K$-group of $D^b(X)$, every full exceptional collection or s.o.d.~of $D^b(X)$ induces a direct sum decomposition of $K_0(X)$ compatible with the induced $\Gamma$-action.
\end{proof}

We will construct exceptional collections on $\ocM_{0,n}$ and $\M_{p,q}$ that are equivariant in a  stronger sense, which implies that the quotient stacks 
$[\ocM_{0,n}/S_n]$ and $[\M_{p,q}/S_p\times S_q]$ 
admit compatible semi-orthogonal decompositions.

\begin{defn}\label{sDGASGASFGARH}
In the setup of Definition~\ref{swfwefawrgqwrgqw}, an exceptional collection is called {\it $\Gamma$-invariant} if $\Gamma$ permutes its objects and the collection can be arranged in blocks so that the objects in each block are permuted by $\Gamma$. Equivalently, there exists a s.o.d.\ $\langle\cA_1,\ldots,\cA_\ell\rangle\subset D^b(X)$ such that each $\cA_j$ is generated by a $\Gamma$-orbit of the objects $\{\cE_\alpha^\bullet\}_{\alpha\in I}$.

A~$\Gamma$-invariant exceptional collection is called {\it $\Gamma$-equivariant} if every complex $\cE_\alpha^\bullet$ is quasi-isomorphic to a complex in $D^b(X/\Gamma_\alpha)\cong D^b_{\Gamma_\alpha}(X)$, the bounded derived category of $\Gamma_\alpha$-equivariant coherent sheaves on $X$, where $\Gamma_\alpha\subset \Gamma$ is the stabilizer of the isomorphism class of $\cE_\alpha^\bullet$.
\end{defn}

We refer to \cite{Thomason,DHL,Elagin} for background on equivariant sheaves and equivariant derived categories.

%Existence of a full $\Gamma$-invariant collection has two consequences:

\begin{lemma}\label{asfvqg}
If $D^b(X)$ admits a full $\Gamma$-invariant exceptional collection then the equivariant derived category $D^b_\Gamma(X)$
admits a s.o.d. with blocks isomorphic to representation categories 
(if the collection is equivariant) and twisted representation categories (if it is only invariant)
of subgroups $\Gamma_\alpha\subset\Gamma$ as in Definition~\ref{sDGASGASFGARH}.
\end{lemma}

\begin{proof}
We refer to \cite[Theorem 2.3]{Elagin} for a  precise formulation and a proof.
He only considers exceptional collections of equivariant coherent sheaves, but the proof for equivariant complexes of coherent sheaves is the same. The s.o.d.~of $D^b_\Gamma(X)$ is induced by the s.o.d.~of $D^b(X)$ from Definition~\ref{sDGASGASFGARH}, with the admissible subcategory $\cA_j\subset D^b(X)$ inducing an admissible subcategory $\tilde \cA_j\subset D^b_\Gamma(X)$ equivalent to $D^b([\bullet/\Gamma_\alpha])$, the bounded derived category of $\Gamma_\alpha$-modules, if the collection is equivariant.
\end{proof}

\begin{ex}\label{asgashawrb}
Let us give an %simple 
example of a variety with an involution such that 
the corresponding action of $S_2$ on its $K_0$-group is not a permutation representation.
Consider a del Pezzo surface $S$ of degree~$2$, the double cover of $\PP^2$ ramified along a smooth quartic curve, 
with the usual involution~$\sigma$ of the double cover. %For a generic $S$, $\Aut(S)=\langle\sigma\rangle\cong S_2$. 
Then $S$ has $56$ $(-1)$-curves that 
come in pairs interchanged by $\sigma$. % (and mapped to $28$ bitangents). 
As $\sigma$ preserves $K_S$, it
acts by $-1$ on the orthogonal complement $K_S^{\perp}\subset\Pic(S)$. %, the root lattice of type $E_7$.
Therefore the action of $S_2$ on both the total cohomology $H^*(S,\ZZ)$ and on $K_0(S)$ is of type $(3,7)$,
i.e.\ it is diagonalizable with $3$ eigenvalues $1$ and $7$ eigenvalues $-1$.
However, any diagonalizable permutation representation
of $S_2$ of type $(a,b)$ clearly has
$a\ge b$.
\end{ex}

We need a stronger version of the equivariant Orlov blow-up lemma \cite[Lemma~7.2]{CT_part_I} to prove that our collections are $\Gamma$-equivariant.

Let $X$ be a smooth projective variety and let 
$Y_1,\ldots,Y_n\subset X$ be smooth transversal subvarieties of codimensions $l_1,\ldots,l_n$. 
Let $\Gamma$ be a finite group acting on $X$ and permuting $Y_1,\ldots,Y_n$.
For every subset $I\subset\{1,\ldots,n\}$, 
we denote by $Y_I$ the intersection $\cap_{i\in I}Y_i$. In particular, $Y_\emptyset=X$.

Let $q:\,\tX\to X$ be the iterated blow-up of the (proper transforms of) $Y_1,\ldots,Y_n$.
Since the intersection is transversal, the blow-ups can be performed in any order.
More canonically, the iterated blow-up is isomorphic to the blow-up of the ideal sheaf
$\cI_{Y_1}\cdot\ldots\cdot\cI_{Y_n}$.
Let $E_i$ be the exceptional divisor over $Y_i$ for every $i=1,\ldots,n$.
For any subset $I\subset\{1,\ldots,n\}$, let 
$E_I=q^{-1}(Y_I)=\cap_{i\in I}E_i$, in particular, $E_\emptyset=\tX$.
Let $i_I:\,E_I\hookrightarrow\tX$ be the inclusion and let 
$\Gamma_I\subset \Gamma$ be the subgroup of elements $g\in\Gamma$ such that $g.Y_I=Y_I$.
The group $\Gamma$ acts on $\tX$ and the morphism $q$ is $\Gamma$-equivariant.

\begin{lemma}\label{BigDaddy}
Let  
$\{F_I^\beta\}$ be  a (full) $\Gamma_I$-invariant (resp.~equivariant) 
exceptional collection in $D^b(Y_I)$ for every subset $I\subset\{1,\ldots,n\}$.
Choosing representative of $\Gamma$-orbits on the set~$\{Y_I\}$, we can assume that if 
$Y_I=gY_{I'}$ for some $g\in \Gamma$ then $\{F_I^\beta\}=g^*\{F_{I'}^\beta\}$.
Then there exists a (full) $\Gamma$-invariant (resp.~equivariant) exceptional collection in $D^b(\tilde X)$ with blocks
$B_{I,J}=(i_I)_*\left[(Lq|_{E_I})^*(F_I^\beta)\left(\sum_{i=1}^n J_iE_i\right)\right]$
for every subset $I\subset\{1,\ldots,n\}$ (including the empty set) and for every $n$-tuple of integers $J_i$
such that $J_i=0$ if $i\not\in I$ and $1\le J_i\le l_i-1$ for $i\in I$.

The blocks are ordered in any linear order which respects the  following partial order:
$B_{I^1,J^1}$ precedes $B_{I^2,J^2}$ if $\sum\limits_{i=1}^n J^1_iE_i\ge \sum\limits_{i=1}^n J^2_iE_i$ (as effective divisors).
\end{lemma}

\begin{proof}
Exceptionality and fullness of the collection, as well as the fact that $\Gamma$ permutes its objects, were proved in \cite[Lemma~2.1]{CT_part_I}. To prove equivariance, we use equivariant pull-back and push-forward and equip the line bundles $\cO\left(J_1E_1+\ldots+J_nE_n\right)$ with linearizations with respect to the stabilizer of the divisor $J_1E_1+\ldots+J_nE_n$ coming from the canonical equivariant structure on the ideal sheaf of an invariant subscheme.

Finally, we need to check that the collection can be reordered in blocks permuted by $\Gamma$. We first order the objects by decreasing value of $\sum J_i$ and observe that objects $B^{\beta_1}_{I^1,J^1}$ and $B^{\beta_2}_{I^2,J^2}$ with the same value of this sum are orthogonal unless $J^1=J^2$, which in turn forces $I^1=I^2$. Therefore, we simply arrange the objects with the same sum $\sum J_i$ into the same groups and in the same order as the objects in the collections $\{F_I^\beta\}$, grouping together the objects permuted by $\Gamma$.
\end{proof}

Next, we need to review Hassett moduli spaces in genus $0$.

\begin{notn}
Fix positive rational weights $\bold a=(a_1,\ldots,a_n)$ with each $a_i\le1$ and $\sum a_i>2$.
Let $\M_\ba$ be the Hassett  space %$\ocM_{0;a_1,\ldots,a_n}$
of weighted pointed stable rational curves \cite{Ha}, i.e.,~pairs $(C,\sum a_ip_i)$ such that $C$ is a nodal genus~$0$ curve, the points $p_i$ are smooth, the  $\QQ$-line bundle $\omega_C(\sum a_ip_i)$ is ample, and the points $\{p_i\,:\,i\in I\}$ can become equal on $C$ if and only if $\sum\limits_{i\in I}a_i\le 1$.
For~example,  $\ocM_{0,n}=\M_{(1,\ldots,1)}$.
Every $\M_\ba$ is a smooth projective variety.
The polytope of 
weights has a wall-and-chamber structure with walls 
\begin{equation}\label{zxfbzxfg}
\sum\limits_{i\in I}a_i=1\quad\hbox{\rm for every subset $I\subset\{1,\ldots,n\}$}.
\end{equation}
Moduli spaces within the interior of each chamber are isomorphic and carry the same universal family.
Unlike in the variation of GIT quotients, there are no moduli spaces “on the walls”: reducing all weights slightly yields an isomorphic Hassett space (and universal family) in the interior of an adjacent chamber. 
There exist birational reduction morphisms $\M_\ba\to\M_{\ba'}$ every time the weight vectors 
are such that $a_i\ge a_i'$ for every~$i$. \end{notn}

\begin{notn}\label{sRGBSFBDFNA}
Now fix positive rational weights $\bold a=(a_1,\ldots,a_n)$ with each $a_i\le1$ and $\sum a_i=2$.
Let 
\[
\X_\ba=(\PP^1)^n\sslash_\ba\PGL_2
\]
be the GIT quotient of $(\PP^1)^n$ by the diagonal action of $\PGL_2$
with respect to the fractional polarization $\cO(a_1,\ldots,a_n)$.
The polytope of GIT weights can be identified with the boundary face $\sum a_i=2$ of the polytope of Hassett weights
and inherits its chamber structure with walls \eqref{zxfbzxfg},
which encodes the variation of GIT (see \cite{DolgachevHu}).
Polarizations within the interior of each chamber have no strictly semistable points
and carry a universal $\PP^1$-bundle with $n$ sections. 
More generally, there always exists a morphism
\[
\M_\ba:=\M_{(a_1+\eps,\ldots,a_n+\eps)}\to \X_\ba,
\]
for some $0<\eps\ll1$ depending on $\ba$, 
which is an isomorphism if there are no strictly semistable points (i.e., away from the walls) \cite[Theorem~8.2]{Ha}. 
In the presence of strictly semistable points, $\X_\ba$ can acquire isolated singularities (cones over $\PP^r\times\PP^s$),
and $\M_\ba$ is the blow-up of these singularities (the “Kirwan resolution”).
In any case, the universal family of a Hassett space $\M_\ba$ 
“near the GIT face”
has fibers with at most $2$ irreducible components.
\end{notn}

\begin{lemma}\label{MonGeneralBrendanHassett}
A Hassett space $\M_\ba$ is a GIT quotient if its universal family $\cU$
is a $\bP^1$-bundle. 
%if in addition $\ba$ is such that $\sum_{i\in I} a_i\neq1$ for all $|I|\geq1$, then 
In this case $\cU$ is also a GIT quotient (and a Hassett space).
\end{lemma}

\bp
Suppose that $\cU\to \M_\ba$ is a $\PP^1$-bundle.
We claim that reducing all weights $a_i$ to $b_i$ such that $a_i>b_i$ for all $i$ and $\sum b_i=2$ does not change stability. 
Indeed, suppose that $\sum\limits_{i\in I}a_i>1\geq\sum\limits_{i\in I}b_i$ for some subset~$I$.
Then  
\[
\sum\limits_{i\in I^c}a_i>\sum\limits_{i\in I^c}b_i=2-\sum\limits_{i\in I}b_i\geq1.
\]
So curves with two components, with points labeled by $I$ (resp., $I^c$)
on the corresponding components, are stable, which contradicts the assumption. 

We can reduce the weights $a_i$ slightly so that no partial sum is equal to $1$.
By \cite[Prop.~5.4]{Ha}, the universal family $\cU$ is the Hassett space $\M_{(a_1, \ldots, a_n, \epsilon)}$ for $\epsilon\ll 1$. %(see also Note \ref{} To prove the statement about the universal family $\cU$, 
Thus the universal family over $\cU$ remains a $\PP^1$-bundle, and hence $\cU$ is also a GIT quotient.
\ep

\begin{notn}\label{fbzfbdfnz}\label{precise weights}
For $p\ge 3$, $q\ge 0$, we let $\X_{p,q}$, resp., $\M_{p,q}$, denote the spaces $\X_\ba$, resp.,  $\M_\ba$, 
from Notation~\ref{sRGBSFBDFNA}
with the weights 
$a=a_1=\ldots=a_p$, $b=a_{p+1}=\ldots=a_{p+q}$ satisfying the following conditions.
If  $q=0$ then  $a=2/p$. 
If $q>0$ then $a=\frac{2}{p}-\epsilon$, $b=\frac{p\epsilon}{q}$, where $0<\epsilon<\frac{1}{(2r+1)(r+1)}$ 
if $p=2r+1$ and $0<\epsilon<\frac{1}{r(r+1)}$ if $p=2r$.
We call the points with the weight $a$ {\it heavy} and the remaining points {\it light}. 
This is equivalent to  Definition~\ref{sgSgSgSgSgSRg}.
%Recall that this means the following:  
%\begin{itemize}
%\item
%When $p=2r+1$ is odd, at most $r$ of the heavy points may coincide, and moreover they may coincide with all the light points.
%\item
%When $p=2r$ is even, at most $(r-1)$ of the heavy points may coincide with each other and with all the light points. Furthermore, if $q>0$ then $r$ heavy points may coincide, and they may further coincide with at most $\lfloor \frac{q-1}{2}\rfloor$ light points.
%\end{itemize}

We denote $\M_n:=\M_{n,0}$. Note that 
$\X_{p,q}\cong\M_{p,q}$ if and only if $p$ or $q$ is odd. % (as no partial sum of the weights equals $1$ when either $p$ or $q$ is odd). 
If  $p=2r$ and $q=2s$, then  
$\M_{p,q}$ is a divisorial ``Kirwan resolution'' of singularities of $\X_{p,q}$. 
It has exceptional divisors
$\bP^{r+s-2}\times\bP^{r+s-2}$ with normal bundle $\cO(-1,-1)$ over each of the ${1\over 2}{p\choose r}{q\choose s}$ singular points of $\X_{p,q}$.
\end{notn}

We will reduce the proof of Theorem~\ref{sfzfgafgafga} to 
the analysis of $\M_{p,q}$'s.
We emphasize that all spaces $\M_{p,q}$ are needed to prove Theorem~\ref{sfzfgafgafga} for $n\gg0$.
The same argument also proves the following more general result.

\begin{thm}\label{MainHassettTh} Let $\ba$
be a Hassett weight such that $a_1\ge\ldots\ge a_n$ and $\sum a_i>2$.
%=(a_1,\ldots,a_n)$
%be a Hassett weight. Choose $k\leq n$ such that 
%$$a_1=\ldots= a_k> a_{k+1}\ge\ldots\ge a_n$$
%(recall that $\sum a_i>2$).
Suppose further that either 
 $a_1=1$ or  $a_j> 2/j$ for some $j$.
Then $\M_\ba$ has a full $\Gamma_\ba$-equivariant exceptional collection,
where $\Gamma_\ba\subseteq S_n$ is the stabilizer of the vector~$\ba$.
\end{thm}

We note that full, exceptional collections on some of the Hassett spaces $\M_{\ba}$ have been constructed in 
\cite{BFK}. More generally, full, exceptional collections on many GIT quotients  have been constructed in 
\cite{DHL,BFK}. 
These collections do not satisfy any invariance requirements.

In the remainder of this section we discuss the proof of Theorem~\ref{MainHassettTh}. 
All along we assume that the spaces $\M_{p,q}$ have a full $S_p\times S_q$-equivariant exceptional collection,
which will be proved in the subsequent sections.

%\medskip
%%It is clear that Theorem~\ref{sfzfgafgafga} follows from Theorem~\ref{MainTh}.

\bp[Proof of Theorem~\ref{MainHassettTh}]
%Let $\M_\ba$ be a Hassett space as in Theorem~\ref{MainTh}. 
Let $\Gamma_\ba$ be the stabilizer of the vector $\bold a$ in the group~$S_n$. 
Choose $p\le n$ such that $a_1=\ldots=a_p>a_{p+1}$.
We consider $4$ cases
$$
\hbox{\bf (i)}\  a_1=1, p\geq2;\ 
\hbox{\bf (ii)}\ a_1<1, pa_1\geq2;\  
\hbox{\bf (iii)}\ a_1=1, p=1;\ 
\hbox{\bf (iv)}\ a_j\geq{2\over j}\ \hbox{\rm for some $j$};$$
%\begin{itemize}
%\item[(i) ] $a_1=1$, $p\geq2$, 
%\item[(ii) ] $a_1<1$, $pa_1\geq2$. 
%\item[(iii) ] $a_1=1$, $p=1$,
%\item[(iv) ] $a_j>2/j$ for some $j$,
%\end{itemize}
and the following claim:
\begin{claim}\label{(v)}
Consider weight vectors $\ba=(a_1,\ldots,a_n)$, $\ba'=(a'_1,\ldots,a'_n)$
such that $a_i\ge a_i'$ for every $i$. Suppose $\Gamma_\ba\subset\Gamma_{\ba'}$.
Then if $\M_{\ba'}$ admits a $\Gamma_\ba$-equivariant exceptional collection then so does $\M_{\ba}$.
\end{claim}

We prove (i)--(iv) and the Claim \ref{(v)} simultaneously by induction on dimension of $\M_\ba$ using existence
of full $S_p\times S_q$-equivariant exceptional collections on $\M_{p,q}$. Note that all statements are clear in dimension $1$ as in this case $\M_{\ba}\cong \PP^1$. 

{\bf Proof of the Claim \ref{(v)}}.
We connect the weights by a $\Gamma_\ba$-invariant homotopy 
$\ba(t)=t\ba'+(1-t)\ba$.
The Hassett chamber structure is semi-continuous in the following sense: the reduction map $\M_{b_1,\ldots,b_n}\to \M_{b_1-\eps,\ldots,b_n-\eps}$
is an isomorphism for $0<\eps\ll1$. It follows that
our reduction map  $\M_{\ba(0)}=\M_\ba\to \M_{\ba'}=\M_{\ba(1)}$
factors into the sequence of $\Gamma_\ba$-equivariant reduction maps of the following form:
$f_t:\,\M_{\ba(t-\eps)}\to \M_{\ba(t)}$ for $0<\eps\ll1$
whenever there is a subset of indices $J\subset\{1,\ldots,n\}$ such that 
\begin{equation}\label{sfdvfvwfv}
\sum_{j\in J}a_j(0)>1=\sum_{j\in J}a_j(t).
\end{equation}
Arguing by induction on the number of wall crossings, it suffices to analyze a single  map $f_t$,
descibed in
%The following description of $f_t$ is from 
\cite{Ha}, see also \cite{BergstromMinabe}.
Let $J_1,\ldots,J_s$ be a full list of subsets satisfying \eqref{sfdvfvwfv}. The group $\Gamma_\ba$ permutes $J_i$'s.
The reduction map $f_t$ blows up the loci $\M_{\ba(t)}(J_i)$ in $\M_{\ba(t)}$ 
where the points in $J_i$ come together. Since $\sum_{j\in J_i}a_j(t)=1$, the loci 
 $\M_{\ba(t)}(J_i)$ intersect transversely: $\M_{\ba(t)}(J_i)\cap \M_{\ba(t)}(J_j)\neq\emptyset$ if and only if $J_i\cap J_j=\emptyset$, in which case, the intersection is the locus where 
 points in $J_i$, respectively $J_j$, coincide. Clearly,  the loci $\M_{\ba(t)}(J_i)$ and their intersections are themselves
 Hassett spaces of the form 
 $\M_{\ba''}$, with the vector $\ba''$ having at least one weight $1$ (for marked points that correspond to combined points indexed by 
 subsets $J_1,\ldots,J_s$).
 These spaces have equivariant exceptional collections by cases (i), (iii) in smaller dimension. The theorem follows by Lemma~\ref{BigDaddy}.

{\bf Cases (i) and (ii)}.
If $p=2$ (case (i)), we apply Claim \ref{(v)} to the weight vector $\ba'=(1,1,\eps,\ldots,\eps)$  of the Losev--Manin space and use the main result of \cite{CT_part_I}.
Let $p>2$. In both (i) and (ii) we apply Claim \ref{(v)} to the weight vector $\ba'=(a+\epsilon_2, \ldots, a+\epsilon_2, b+\epsilon_2, \ldots, b+\epsilon_2)$ corresponding to $\M_{p,q}$ with $q=n-p$, $0<\epsilon_2\ll 2$ (see Notation \ref{precise weights}). Concretely, 
if $q=0$ then $a=\frac{2}{p}$ and we choose $\epsilon_2$ such that $a+\epsilon_2=\frac{2}{p}+\epsilon_2<a_1$, while if $q>0$ then 
$a=\frac{2}{p}-\epsilon_1$, $b=\frac{p\epsilon_1}{q}$ ($0<\eps_1\ll 1$)
and we choose $\epsilon_2$ such that $a+\epsilon_2<a_1$ and $b+\epsilon_2<a_n$. 

{\bf Case (iii)}. Let $A=a_2+\ldots+a_n>1$. Consider $a_1'=1$, $a_i'={a_i\over A-\eps}$ for $i\ge2$ for some fixed 
$0<\eps<\min(a_n,A-1)$. By Claim \ref{(v)}, it suffices to show that $\M_{\ba'}$ admits a $\Gamma_\ba$-equivariant exceptional collection.
Note that $\M_{\ba'}\cong\PP^{n-3}$ by \cite[Section 6.2]{Ha} %- and this is an observation we will use repeatedly - 
since
$\sum_{i\geq 2, i\neq j} a'_i\leq\sum_{i=2}^{n-1} a'_i={A-a_n\over A-\eps}< 1$ for all $j\geq2$. Note that the standard collection $\cO,\cO(1),\ldots,\cO(n-3)$ on $\PP^{n-3}$ is invariant under  any  group. For equivariance, we need a bit more. All appearing groups are contained in the symmetric group $S_{n-1}$ which acts on $\bP^{n-3}$ by permuting $n-1$ fixed  points in general linear position. The homomorphism $S_{n-1}\to\PGL_{n-2}$ can be factored through $\GL_{n-2}$, and therefore $\cO(1)$ is a $S_{n-1}$-linearized line bundle. The corresponding action on $k^{n-2}$ is   the irreducible $(n-2)$-dimensional representation
of $S_{n-1}$, \cite[Lemma 2.3]{KeelTevelev}.

{\bf Case (iv)}. We can assume without loss of generality that $j$ is the largest index with the property $a_j\ge2/j$. The statement follows from (i) if $a_j=1$. If $a_j<1$, we  reduce to (ii) using Claim \ref{(v)}
by taking the second weight vector $a_1'=\ldots =a_j'=a_j$, $a_i'=a_i$ for $i>j$.
\ep

\section{Derived category of the moduli stack $\cM_n$ of $n$ points on $\bP^1$}\label{zxFbFFafsg}\label{stack section}\label{fullness odd p}\label{fullness odd p section}
%\section{Windows into %derived categories of  $\cM_{p,q}$ and proof of Theorem~\ref{symm}}
\label{windows}\label{windows section}

The goal of this section is to prove Theorem~\ref{symm}, which gives an equivariant full exceptional collection in $D^b(\M_n)$ for $n$ odd.
As a GIT quotient without strictly semistable points, $\M_n$ is isomorphic to an open substack
of the quotient stack $\cP_n=[(\bP^1)^n/\PGL_2]$ of $\bP^1$-bundles with $n$ sections. 

The proof relies on the  {\it windows theorem} of Halpern--Leistner:

\begin{thm}\cite{DHL}\label{Kirwan}
Let $[X/G]$ be the stack quotient of a smooth projective variety by a reductive group $G$ and let $[X^{ss}/G]$ the open substack corresponding to the semistable locus $X^{ss}$
with respect to a choice of polarization and linearization. 

For a choice of a Kempf--Ness (KN) stratification  of the unstable locus $X^{us}$ with data
$Z_i$, $S_i$, $\la_i$, $\sigma_i:Z_i\hra S_i$, define the integers 
$$\eta_i=\wt_{\la_i}\det\big(N^\vee_{S_i|X}\big)_{|Z_i}>0$$
and choose  integers $w_i$. Define the full subcategory $\cG_w$  of all objects $F\in D^b[X/G]$ such that 
the cohomology sheaves of $\si_i^*F$ have weights in $[w_i, w_i+\eta_i)$
for all $i$. Then the restriction functor $i^*:\cG_w\ra D^b[X^{ss}/G]$ is an equivalence of categories. 
\end{thm} 

\begin{rmk}
There are two sign conventions for weights used in the literature. Above we follow \cite{DHL}
where the ample polarization of the GIT quotient has negative weights on the unstable locus, see e.g.~\eqref{DHLconvention}. 
However, starting with \S\ref{asjbqlkwjbfkjwb}, we use the opposite weight convention from \cite{Teleman}, as this will be more convenient for calculations.
\end{rmk}

In order to apply the windows theorem, we will have to 
\begin{enumerate}
 \item compute a full exceptional collection in the derived category of the quotient stack $[(\bP^1)^n/\PGL_2]$ (Theorem~\ref{stackofbundles});
 \item compute the Kempf--Ness stratification and all relevant data in our setup for a wide range of GIT quotients (Theorem~\ref{Kirwan_special_case});
\item choose integers $w_K$ wisely to squeeze as many objects as possible from the exceptional collection on the quotient stack $[(\bP^1)^n/\PGL_2]$ into the window $\cG_w$ and show that the remaining objects of the exceptional collection, when restricted to the semi-stable locus, are generated by objects in the exceptional collection of the window. This is done by induction on the {\it score} \eqref{firstscore},
a useful concept that will be refined and generalized in subsequent sections.
\end{enumerate}

Relatively simple calculations in this section serve as a model for  more convoluted applications of the windows theory in the subsequent sections.

\medskip

In what follows $n$ is arbitrary and $\cP_n=[(\bP^1)^n/\PGL_2]$. 
We first analyze the derived category of $\cP_n$.
An associated $\bP^1$-bundle of a $\PGL_2$-torsor is the ``universal $\bP^1$-bundle'', i.e., a representable morphism $\pi=\pi_{n+1}:\,\cP_{n+1}\to\cP_n$
induced by the first projection $(\bP^1)^{n+1}=(\bP^1)^n\times \bP^1\to(\bP^1)^n$. 

We denote $\sigma_1,\ldots,\sigma_n:\,\cP_n\to\cP_{n+1}$ the universal sections. They are representable morphisms
induced by the big diagonals $(\bP^1)^{n}\cong\Delta_{i,n+1}\hookrightarrow(\bP^1)^{n+1}$, which are defined by $x_i=x_{n+1}$, for $i=1,\ldots,n$. We identify the category of coherent sheaves on $\cP_n$ with the category of 
$\PGL_2$-equivariant coherent sheaves on $(\bP^1)^n$ and likewise for their bounded derived categories.
For $(i_1,\ldots,i_n)\in\ZZ^n$ with $\sum i_k$ even, the line bundle $\cO_{(\bP^1)^n}(i_1,\ldots,i_n)$
has a unique $\PGL_2$- linearization and thus descends to $\cP_n$. 

\begin{defn}\label{FlE}
Let  $N$ be the set $\{1,\ldots,n\}$. 
Fix a subset $E\subset N$ with $e=|E|$ and an integer $l\ge0$ such that $e+l$ is even.
For $j=1,\ldots,n$, let $i_j$ be $1$ if $j\in E$ and~$0$ otherwise.
%For $Let $i_1,\ldots,i_n$ be a sequence such that $i_j$ is $1$ if $j\in E$ and~$0$ otherwise.
Let 
$$F_{l,E}=\pi_*N_{l,E},$$
where 
$N_{l,E}=\cO(i_1,\ldots,i_n,l)$
is a  line bundle 
on $\cP_{n+1}$.
\end{defn}

\begin{lemma}\label{akjsdhfkjsgf}
$F_{l,E}$ is a rank $l+1$ vector bundle on $\cP_n$. 
As an $\SL_2$-equivariant vector bundle on~$(\bP^1)^n$, we have
$F_{l,E}\mathop{\cong}\limits_{\SL_2} \cO(i_1,\ldots,i_n)\otimes V_l$,
where $V_l$ is an $(l+1)$-dimensional irreducible $\SL_2$-module. In particular,
$F_{0,E}\cong\cO(i_1,\ldots,i_n)$.
\end{lemma}

\bp
This follows from the $\SL_2$-equivariant projection formula applied to the projection 
morphism $(\bP^1)^{n+1}=(\bP^1)^n\times \bP^1\to(\bP^1)^n$. 
\ep

\begin{cor}\label{laksjrhfgkjsrG}\label{complement}
There are isomorphisms of $\SL_2$-equivariant vector bundles 
$$F_{l,E}^\vee\cong \pi_*\cO(-i_1,\ldots,-i_n,l)\cong \cO(-i_1,\ldots,-i_n)\otimes V_l.$$ 
If $n$ is even, then $F_{l,E}\cong F_{l,E^c}^\vee\otimes F_{0,N}$, where $E^c=N\setminus E$. 
\end{cor}

Here (and everywhere) $F^\vee$ denotes the dual complex of  $F\in D^b(X)$. 

\begin{thm}\label{stackofbundles}
$D^b(\cP_n)$ has a  full strong 
$S_n$-equivariant exceptional collection
$$\{ F_{l,E} : \hbox{\rm $l+e$ even}\}$$
of infinitely many vector bundles.
The collection is arranged into blocks $F_{l,e}$
given by the $S_n$-orbits, which are ordered by increasing $e$, and for fixed $e$, arbitrarily.

% s.o.d. with $2^n$ blocks $D^b(\cP_n)_E$ indexed by subsets $E\subseteq N$ and ordered by increasing $e=|E|$ (different blocks with the same $e$ are mutually orthogonal, i.e., $R\Hom$ between blocks is $0$).The block $D^b(\cP_n)_E$ has a full exceptional collection $\{ F_{l,E} : \hbox{\rm $l+e$ even}\}$ of infinitely many mutually orthogonal vector bundles. The combined infinite exceptional collection $\{F_{l,E}\}$ in $D^b(\cP_n)$ is strong and $S_n$-equivariant.

The forgetful map $\pi_i:\,\cP_{n+1}\to\cP_{n}$ has the following properties:
$$(L\pi_i)^*F_{l,E}\cong F_{l,E},\qquad (R\pi_i)_*F_{l,E}^\vee\cong
\begin{cases}
0 & i\in E\cr
F_{l,E}^\vee &i\not\in E\cr
\end{cases}.$$
\end{thm}

\begin{proof}
The line bundles $\cO(i_1,\ldots,i_n)$ form a strong $S_n$-equivariant  exceptional collection in $D^b((\bP^1)^n)$, 
where each $i_j=0$ or $1$. We will denote them by $\{L_k\}$, $k=1,\ldots, 2^n$ (in some order). 
Schur's lemma and the fact that $R\Gamma(\bP^1,\cO(-1))=0$
imply
 that $\{ F_{l,E}\}$ with $l+e$ even
is an exceptional collection in $D^b(\cP_n)$ with required properties. 

It remains to prove fullness, i.e., that any complex $F\in D^b(\cP_n)$
can be obtained by finitely many extensions starting with objects in the exceptional collection $\{F_{l,E}\}$. 
This is a special case of \cite[Theorem~ 2.10]{Elagin}. Alternatively, we give a simple ad hoc argument.

Viewing $F\in D^b(\cP_n)$ as a bounded complex of coherent sheaves on $(\bP^1)^n$, let $k$ be the maximum index such that (in $D^b((\bP^1)^n)$) we have 
$$R\Hom(L_k,F)\ne0$$ or we let $k=0$ if $R\Hom(L_j,F)=0$ for every $j$. 
We argue by induction on $k$. If $k=0$ then 
$F\cong 0$ because $\cO(i_1,\ldots,i_n)$ form a full exceptional collection in $D^b((\bP^1)^n)$.
Otherwise, consider the left mutation triangle 
$F'\to R\Hom(L_k,F)\otimes L_k\to F\to$
in $D^b((\bP^1)^n)$. We first prove that this is an exact triangle in $D^b(\cP_n)$. Since
 $F$ is an $\SL_2$-equivariant complex and $L_k$ is an $\SL_2$-equivariant line bundle, it suffices to prove that 
  $-\Id\in\SL_2$ acts trivially on $R\Hom(L_k,F)\otimes L_k$. This is clear since $-\Id\in\SL_2$ acts either trivially or by multiplication by $-1$ on 
  both terms of the tensor product. It follows that this triangle is $\PGL_2$-equivariant.  
To finish the proof, we need to prove that $F'$ and $R\Hom(L_k,F)\otimes L_k$ are  generated by $F_{l,E}$'s. 
For $F'$ this follows by induction since $R\Hom(L_k,F')=0$. 
, $R\Hom(L_k,F)\otimes L_k$ is isomorphic to 
a direct sum of the vector bundles $F_{l,E}$: 
the $\SL_2$-representation $R\Hom(L_k, F)$ is a direct sum of irreducible representations $V_{l_i}$, for integers $l_i\geq0$.  Since $R\Hom(L_k, F)\otimes L_k$ is a $\text{PGL}_2$-equivariant vector bundle, we must have $\sum_{i=1}^r l_i+e$ is even and we are done by Lemma \ref{akjsdhfkjsgf}. 
 \end{proof}
 
\begin{prop}\label{exact ones} 
On $\cP_n$, we have exact sequences
$$0\to F_{l-1,E\setminus k}\to F_{l,E}\to Q^k_{l,E}\to 0\qquad \hbox{\rm for $k\in E$},$$
$$0\to F^\vee_{l-1,E\cup\{k\}}\to F^\vee_{l,E}\to S^k_{l,E}\to 0\qquad \hbox{\rm for $k\not\in E$},$$
where $Q^k_{l,E}=\cO(i_1,\ldots,i_k+l,\ldots,i_n)$ and $S^k_{l,E}=\cO(-i_1,\ldots,-i_k+l,\ldots,-i_n)$.
\end{prop}

\begin{proof}
Apply $R\pi_*$ to  an exact sequence $0\!\to\!\cO(-\De_{k,n+1})\!\to\!\cO\!\to\!\cO_{\Delta_{k,n+1}}\!\to\!0$ 
on $\cP_{n+1}$ 
tensored with $\cO(i_1,\ldots,i_n,l)$ (resp.,~$\cO(-i_1,\ldots,-i_n,l)$).
\end{proof}

\begin{prop}\label{fkjnvekfjnc}\label{tensor}
$F_{l,E}\otimes F_{l',E'}\cong F_{l+l',E\cup E'}\oplus F_{l+l'-2,E\cup E'}\oplus\ldots\oplus F_{|l-l'|,E\cup E'}$
if $E\cap E'=\emptyset$.
In particular,
$F_{l,E}\otimes F_{0,E'}\cong F_{l,E\cup E'}$.
Here we assume that all these bundles are defined, i.e., all parity conditions are satisfied.
\end{prop}

\begin{proof}
This follows from the Clebsch--Gordan formula 
(Lemma \ref{CG}).
\end{proof}

\begin{lemma}\label{CG}(Clebsch-Gordan)
$(V_l\otimes V_{l'})=V_{l+l'}\oplus V_{l+l'-2}\oplus\ldots\oplus V_{|l-l'|}$. In particular, if $l>l_1+\ldots+l_r$, then 
$(V_l\otimes V_{l_1}\otimes\ldots\otimes  V_{l_r})^{\SL_2}=0$. 
\end{lemma}

\begin{rmk}\label{dictionary}
In the next sections we will consider Hassett moduli spaces of pointed curves which do not always admit a  quotient stack interpretation. 
Therefore, it is useful to relate the vector bundles $F_{l,E}$ to tautological line bundles on the stack of  $\bP^1$-bundles with $n$ sections. 

On $(\PP^1)^n\times\PP^1$ we have 
$\omega_\pi=\cO(0,\ldots, 0, -2)$ and 
$$\cO(\sigma_j)=\cO(0,\ldots, 0, 1, 0,\ldots,0,1)$$ with $1$ in position $j$ and $n+1$, since the section $\sigma_j$ is the $\Delta_{j, n+1}$ diagonal. For $i\neq j$, we have on $(\PP^1)^n$  
$$\sigma_i^*\cO(\sigma_j)=\cO(0,\ldots, 0, 1, 0,\ldots, 0, 1, 0,\ldots,0),$$ 
with $1$ in position $j$ and $j$, while when $i=j$ we have 
$$\sigma_j^*\cO(\sigma_j)=\cO(0,\ldots, 0, 2, 0,\ldots, 0)$$ with $2$ in position $j$. Hence, 
the tautological line bundle $\psi_i=\sigma_i^*\omega_\pi$ on $\cP_n$ is identified 
with the $\PGL_2$-equivariant line bundle $\cO(0,\ldots,0,-2,0,\ldots,0)$, where $-2$ is in position $i$. 
The line bundle $\de_{ij}=\sigma_i^*\cO(\sigma_j)$ is  identified 
with the $\PGL_2$-equivariant line bundle 
$\cO(0,\ldots,0,1,0,\ldots,0,1,0,\ldots,0)$, where $1$ is in position $i$ and $j$. 
Furthermore,  on $\cP_{n+1}$, 
$$N_{l,E}\cong \omega^{\otimes {(\frac{e-l}{2}})}_{\pi}(E),$$
where, by abuse of notation, we let $\cO(E)=\cO(\sum_{j\in E}\sigma_j)$. 
\end{rmk}

\begin{prop}\label{sRGASRGASRGASG}\label{F negative}
More generally, we let $N_{l,E}:=\om_{\pi}^{\frac{e-l}{2}}(E)$ for any $l$, positive or negative,
and 
$F_{l,E}:=R\pi_*N_{l,E}$.
When $l=-1$, we have $F_{l,E}=0$. When $l\leq -2$, we have
$F_{l,E}\cong F_{-l-2,E}[-1]$.
\end{prop}

\bp
By Grothendieck--Verdier duality (see Remark \ref{GV general}), we have 
\begin{equation}\label{GV duality}
R\pi_*\cO(-i_1,\ldots,-i_n,-l-2)\cong F_{l,E}^\vee[-1],
\end{equation}
where 
 $i_j=1$ if $j\in E$ and~$i_j=0$ otherwise. 
But by Corollary  \ref{complement} when $-l-2\geq 0$,
$F_{-l-2,E}^\vee\cong R\pi_*\cO(-i_1,\ldots,-i_n,-l-2)$.
\ep

\begin{rmk}\label{GV general}
We recall  the Grothendieck--Verdier duality \cite[Theorem  3.34]{Huybrechts} in a particular case that we will use often: for a morphism $\al: X\ra Y$ of relative dimension $1$, with $X$, $Y$ smooth varieties, if $N$ is a line bundle on $X$, then $R\al_*\big(N^\vee\otimes \om_{\al}[1]\big)=\big(R\al_*N\big)^\vee$. 
\end{rmk}

%%%%%%%%%%%%%%%%%%%%%%%%%%%%%%%%%%%%%%%%%%%%%%%%%%%%%%%%%%%%%%%%%%%
%%%%%%%%%%%%%%%%%%%%%%%%%%%%%%%%%%%%%%%%%%%%%%%%%%%%%%%%%%%%%%%%%%%
In the remainder of this section we study derived categories of open substacks of $\cP_n$, with an eye towards
proving Theorem~\ref{symm}.

\begin{prop}\label{sjhfgkjshgfqkjhw}
 Let $\cU\subseteq\cP_n$ be an open substack. 
 We abuse notations and denote $F_{l,E}|_{\cU}$ by $F_{l,E}$. 
% as restrictions of the corresponding vector bundles on $\cM_n$.
 Then $D^b(\cU)$ is generated by the $F_{l,E}$'s (equivalently, by $F_{l,E}^\vee$'s, see Remark \ref{dual collection}).
All results about $\cP_n$ proved previously, except Theorem~\ref{stackofbundles}, are valid on $\cU$.
\end{prop}

\begin{rmk}\label{dual collection}
 If $X$ is a smooth projective variety, the double dual $F^{\vee\vee}$ of a complex $F\in D^b(X)$ is canonically isomorphic to $F$ (see \cite[p.84]{Huybrechts}). Since taking duals is an exact functor, it follows that $D^b(X)$ is generated by a collection of objects $\{F_\al\}_{\al}$ if and only if it is generated by the dual collection $\{{F_\al}^\vee\}_{\al}$. 
\end{rmk}

\begin{proof}[Proof of Proposition \ref{sjhfgkjshgfqkjhw}]
We have $\cU=[U/\PGL_2]$ for an open equivariant subset $U\subset(\bP^1)^n$.
 It is enough to show that any equivariant coherent sheaf $F$ on $U$ can be obtained
 by a finite number of equivariant extensions starting with restrictions of the vector bundles 
 $F_{l,E}$. 
 This follows from Theorem~\ref{stackofbundles} since
 $F$ is a restriction of 
 an equivariant coherent sheaf on $(\bP^1)^n$ (see e.g.~\cite{Thomason}).
For the rest of the proposition, we restrict isomorphisms on $\cP_n$ to $\cU$.
\end{proof}

Open substacks $\cU\subseteq\cP_n$
include all Hassett spaces when the universal family is a $\bP^1$-bundle.
More precisely, fix weights $\ba=(a_1,\ldots,a_n)$ with $\sum a_i=2$.
Recall from the introduction that %$\M_\ba$ is the Hassett moduli space with weights given by $\ba$ and 
the  GIT quotient $\X_\ba$ is a quotient of the semi-stable locus $(\PP^1)^n_{ss}$
for the fractional $\PGL_2$-polarization given by $\ba$. We denote by $\cM_\ba$ the stack quotient
$[(\PP^1)^n_{ss}/\PGL_2]$, which is an open substack of $\cP_n$. 
If there are no strictly semistable points then $\PGL_2$ acts on $(\PP^1)^n_{ss}$ freely and $\cM_\ba\cong\M_\ba\cong\X_\ba$.
In particular, 
for every partition $P\coprod Q=\{1,\ldots,n\}$ with $p=|P|\ge 2$, $q=|Q|$, we define an open substack
$\cM_{p,q}\subset \cP_n$ of $\bP^1$ bundles such that 
at most $p/2$  sections indexed by $P$ (heavy points) are allowed to coincide
and if $p=2r$ then $r$ heavy points are allowed to coincide with at most $q/2$  points indexed by $Q$ (light points).
The corresponding Hassett space is $\M_{p,q}$ and the GIT quotient is $\X_{p,q}$ (see Notation~\ref{fbzfbdfnz}).
We have $\cM_{p,q}=\M_{p,q}=\X_{p,q}$ unless both $p$ and $q$ are even.
We use notation $\cM_n$, $\M_n$ and $\X_n$ if $q=0$. 
%We also use the following notation: for every subset $E\subset\{1,\ldots,n\}$, we denote by $E_p$ (resp.~$E_q$) its intersection with $P$ (resp.~$Q$).

%We would like to show the following:
\begin{thm}\label{Kirwan_special_case}
Consider a collection of vector bundles $\{F _{l,E}\}$ on $\cM_\ba$
for some set of pairs $(E,l)$. %of $\PGL_2$-equivariant vector bundles on~$(\PP^1)^n$.
Order them by increasing $e=|E|$ (and arbitrarily when $e$ is the same). Assume that there exists an integer $w_K$ 
for every subset $K\subseteq N$ such that $\sum\limits_{i\in K} a_i> 1$ and suppose that
$w_K\le -l+(e_0-e_{\infty})$ and $l+(e_0-e_{\infty})<w_K+2|K|-2$
for all bundles in the collection and all subsets as above,
where 
$e_{\infty}=|E\cap K|$ and $e_0=|E\cap K^c|$, for $K^c=N\setminus K$.
Then $\{F_{l,E}\}$ is a strong exceptional (but not necessarily full) collection in $D^b(\cM_\ba)$.
\end{thm}

\bp
We apply the windows theorem~\ref{Kirwan} for $X=(\PP^1)^n$ and $G=\PGL_2$. Up to conjugation, a $1$- parameter subgroup $\lambda$ has the form 
$\lambda(t)=
\left[\begin{array}{cc}
t & 0\\
0 & t^{-1}
\end{array} \right]$.
For every subset $K\subseteq N$, consider the $\la$-invariant point
$z_K=(p_1,\ldots, p_n)$,  with $ p_i=0=[0:1]$ if $i\notin K$
and $p_i=\infty=[1:0]$ otherwise. Then $\wt_{\la}\cO(1)_{|x}=1$ if $x=0$ and $\wt_{\la}\cO(1)_{|x}=-1$ if $x=\infty$. %Let $k=|K|$.
 If $\cL=\cO(a_1,\ldots, a_n)$, it follows that 
\begin{equation}\label{DHLconvention}
\wt_{\la}\cL_{|z_K}=\sum_{i\in K^c} a_i-\sum_{i\in K} a_i.
\end{equation}
%Here we follow the convention of \cite{DHL},
%where the ample polarization of the GIT quotient has negative weights on the unstable locus.

Let $\De_K$ be a diagonal in $(\PP^1)^n$ consisting of points $(p_1,\ldots, p_n)$ such that for all $i\in K$ the points $p_i$ are equal.
Let $S_K\subset\De_K$ be the complement to the union of smaller diagonals.
The KN strata of $X^{us}$ are given by $Z_K=\{z_K\}$, $S_K$ and the canonical inclusions 
$\sigma_K:\,Z_K\hookrightarrow S_K$, 
for all $K\subseteq N$ such that $\sum\limits_{i\in K} a_i> 1$. 
The destabilizing 1-PS  is $\lambda$.
Next we compute $\det\big(N^\vee_{\De_K|(\PP^1)^n}\big)$. We may assume %without loss of generality 
that $n\in K$. 
As $\De_K$ is a complete intersection in $X$ of big diagonals $\De_{jn}$, for all $j\in K\setminus\{n\}$, it follows that 
$$N^\vee_{\De_K|(\PP^1)^n}\cong\bigoplus\limits_{j\in K\setminus\{n\}}\cO(-\De_{jn})_{|\De_K},\ \ 
\det\big(N^\vee_{\De_K|(\PP^1)^n}\big)\cong\cO\Bigl(-\sum\limits_{j\in K\setminus\{n\}}\De_{jn}\Bigr)_{|\De_K}.$$ %$\cong\cO(0,\ldots,0, -2(|K|-1))$
%via the isomorphism  $\De_K\cong (\PP^1)^{n-k}\times\PP^1$ given by the $n$-th marking.  
It~follows from (\ref{DHLconvention}) that
%the weights of $\det\big(N^*_{\De_K|(\PP^1)^n}\big)_{|z_K}$ depend only on $|K|$ and
\begin{equation}\label{eta weights}
\eta_K:=\wt_{\la}{\det\big(N^\vee_{\De_K|(\PP^1)^n}\big)_{|z_K}}=2(|K|-1). 
\end{equation}
Theorem~\ref{Kirwan_special_case} then follows from Theorem \ref{stackofbundles}, Theorem~\ref{Kirwan}, and the following
Lemma~\ref{dvdvDV}.
\ep

\begin{lemma}\label{dvdvDV}
%For subsets $E, K\subseteq\{1,\ldots, n\}$ with $e=|E|$, $k=|K|$, we denote
%$e_{\infty}=|E\cap K|$, $ e_0=|E\cap K^c|$.
%Note that $e_{\infty}+e_0=e$. 
${F_{l,E}}_{|z_K}$ has weights $m+(e_0-e_{\infty})$ for 
$m=l, l-2, \ldots,2-l, -l$.
\end{lemma}

\bp
For this calculation we can view $F_{l,E}$ as an $\SL_2$- (rather than $\PGL_2$-) equivariant bundle.
By Lemma \ref{akjsdhfkjsgf} and in the notations of Definition \ref{FlE}, we have that 
$F_{l,E}=\cO(i_1,\ldots, i_n)\otimes V_l$. Note that $V_l=\Sym^lV_1$,
where $V_1$ is a trivial rank $2$ vector bundle with the standard $\SL_2$-action.
The weights of  $V_l$ (at any $\lambda$-fixed point) are 
$l,\ l-2,\ l-4,\ldots,\ -l$ and the formula follows.
\ep

\bp[Proof of Theorem~\ref{symm}]
%We show that the vector bundles $F_{l,E}$ satisfy the conditions in Theorem \ref{Kirwan_special_case}. 
We consider the following {\it score} function, 
\begin{equation}
S(l,E)=l+\min(e,p-e).\label{firstscore}
\end{equation}
We need to prove that the vector bundles $F_{l,E}$ with $S(l,E)\leq r-1$ form a full exceptional collection. 
First we prove exceptionality. We prove that condition $S(l,E)\leq r-1$ implies that the weights of the bundles $F_{l,E}$ fit in a window as in Theorem~\ref{Kirwan}. 
Lemma~\ref{dvdvDV} implies that the maximum weight of $F_{l,E}$ over all subsets $K$ with fixed $|K|=k$ is %$l+e$ if $e\leq n-k$ and $l+2n-2k-e$ if $e>n-k$. Equivalently,
%the maximum weight of $F_{l,E}$ is 
$\text{min}\{l+e, l+2p-2k-e\}$.
Similarly, the minimum weight of  $F_{l,E}$ is 
%$-(l+e)$ if $e\le k$ and $-(l+2k-e)$ if $e>k$, or equivalently, the minimum weight of $F_{l,E}$ is 
$-\text{min}\{l+e, l+2k-e\}$.
The conditions in Theorem \ref{Kirwan_special_case} for existence of a window are equivalent to requiring that for any pairs $(l,E)$, $(l', E')$, if we let $e=|E|$ and  $e'=|E'|$
$$\text{min}\{l+e, l+2k-e\}+\text{min}\{l'+e', l'+2(p-k)-e'\}<\eta_k=2(k-1),$$  
for all $k\ge r+1$. 
We now prove that this is the case for the list of pairs $(l,E)$ in Theorem \ref{symm}. We consider three cases.

{\bf Case I: $e,e'\le r$. } By assumption $l+e, l'+e'\le r-1$. Then 
$\text{min}\{l+e, l+2k-e\}+\text{min}\{l'+e', l'+2(p-k)-e'\}
\leq(l+e)+(l'+e')\leq 2r-2<2(k-1)$ for all $k\ge r+1$.
 
{\bf Case II: $e,e'\geq r+1$. } By assumption, $l+p-e, l'+p-e'\leq r-1$. We have: 
$\text{min}\{l+e, l+2k-e\}+\text{min}\{l'+e', l'+2(p-k)-e'\}\leq(l+2k-e)+(l'+2(p-k)-e')\leq2(r-1)<2(k-1) $
 for all $k\ge r+1$. 

{\bf Case III: $e\le r$ and $e'\geq r+1$ (or the opposite). } By assumption 
$l+e\leq r-1$, $l'+p-e'\leq r-1$.
It follows that 
$\text{min}\{l+e, l+2k-e\}+\text{min}\{l'+e', l'+2(p-k)-e'\}
\leq(l+e)+(l'+2(p-k)-e')\leq 2(r-1)<2(k-1)$  for all $k\ge r+1$. 
We finish by applying Theorem~\ref{stackofbundles}.

\smallskip

Next we prove fullness.
By~Proposition~\ref{sjhfgkjshgfqkjhw}, it suffices to prove the following claim.

\begin{claim}\label{sgsgsgsg}
Let $p=2r+1$. Every vector bundle $F_{l,E}$ (with $l+e$ is even) on $\M_p$ is generated by the 
vector bundles in the collection in Theorem~\ref{symm}. 
\end{claim}

%We consider the following {\em score} function, $S(l,E)=l+\min(e,p-e)$.

For simplicity, let $\cC$ be the collection of vector bundles $F_{l,E}$ in Theorem~\ref{symm}, i.e., 
those in the range  $S(l,E)\le r-1$. We will prove the equivalent dual statement that every vector bundle $F^\vee_{l,E}$ on $\M_p$
is generated by the dual collection $\cC^\vee$. 
We argue by induction on the score and for a fixed score, by induction on $l$. 
Clearly, the statement holds when $S(l,E)\leq r-1$. 
Let $a\geq r$ and assume that 
all the bundles $F^\vee_{l,E}$ with $S(l,E)<a$ are generated by $\cC^\vee$. Let $F^\vee_{l,E}$ be a bundle such that
$S(l,E)=a$. We consider two cases: $e=|E|\leq r$ and $e=|E|\geq r+1$. 

\underline{Assume  $e\leq r$.}  
Let $I\subseteq\{1,\ldots, p\}$ with $I\cap E=\emptyset$ and $|I|=r+1$. 
We consider the Koszul complex for the diagonal 
$\De=\Delta_{I\cup\{x\}}\subseteq (\bP^1)^p\times \PP^1_x$,
$$
0\leftarrow\cO_{\De}\leftarrow\cO\leftarrow\bigoplus_{i\in I}\cO(-{\bf e_i}-{\bf e_x})\leftarrow\bigoplus_{i,k\in I}\cO(-{\bf e_i}-{\bf e_k}-2{\bf e_x})\leftarrow\ldots$$
$$\leftarrow\cO(-\sum_{i\in I}{\bf e_i}-(r+1){\bf e_x})\leftarrow 0.$$
Here $\{\bf e_i\}$ for $i\in P$ is the standard basis of $\Pic(\PP^1)^{p}\cong\ZZ^{p}$.

We tensor this resolution
with $\cO(-\sum_{j\in E}{\bf e_j}+l {\bf e_x})$
and  take derived push-forwards of its terms via  the projection map $\pi:\,(\bP^1)^{p+1}\to(\bP^1)^p$, 
which we then restrict to the semistable locus in $(\bP^1)^p$.
Since $\pi(\De)$ is in the unstable locus, we obtain the following objects:
$$0\quad F^\vee_{l, E}\quad \bigoplus_{i\in I} F^\vee_{l-1, E\cup\{i\}}\quad\ldots\quad \bigoplus_{J\subseteq I, |J|=j} F^\vee_{l-j, E\cup J}\quad\ldots\quad 
\bigoplus_{J\subseteq I, |J|=l} F^\vee_{0, E\cup J}\quad 0$$
$$\quad \bigoplus_{J\subseteq I, |J|=l+2} F^\vee_{0, E\cup J}[-1]\quad
\ldots\quad \bigoplus F^\vee_{j-l-2, E\cup J}[-1]\quad\ldots\quad F^\vee_{r-l-1, E\cup I}[-1],$$
where the terms $F^\vee_{l-j, E\cup J}$ appear as $R^0\pi_*$ as long as $0\leq j\leq l$, the $0$ term appears when $j=l+1$ (which happens if $l\leq r$; 
if $l>r$ then the last term that appears is $F^\vee_{l-r-1, E\cup I}$), while the terms $F^\vee_{j-l-2, E\cup J}$ 
appear as $R^1\pi_*$ in the range $l+2\leq j\leq r+1$ by Grothendieck-Verdier duality (see \ref{GV duality}). 

We claim that the first object $F^\vee_{l, E}$ is generated by the remaining objects, 
which are all generated by $\cC^\vee$ by induction. 
It will follow that $F^\vee_{l, E}$ is generated by $\cC^\vee$. 
The first claim
%, we can apply the standard spectral sequence \cite[2.66]{Huybrechts}  to get an exact sequence:
%$$0\leftarrow F^\vee_{l, E}\leftarrow \bigoplus_{i\in I} F^\vee_{l-1, E\cup\{i\}}\leftarrow\ldots\leftarrow \bigoplus_{J\subseteq I, |J|=j} F^\vee_{l-j, E\cup J}\leftarrow\ldots\leftarrow \bigoplus_{J\subseteq I, |J|=l} F^\vee_{0, E\cup J}$$
%$$\leftarrow \bigoplus_{J\subseteq I, |J|=l+2} F^\vee_{0, E\cup J}\leftarrow\ldots\leftarrow \bigoplus_{J\subseteq I, |J|=j} F^\vee_{j-l-2, E\cup J}\leftarrow\ldots\leftarrow \bigoplus F^\vee_{r-l-1, E\cup I}\leftarrow0.$$
%The $0$ term in the middle disappears as it is bypassed by the differential $d_2$.
is a special case of the following  lemma, which will be used
repeatedly in the remainder of the paper:

\begin{lemma}\label{obviosddd}
Let $F:\,D(\cA)\to \cT$ be an exact functor of triangulated categories and let $0\to A_1\to\ldots \to A_n\to 0$
be an exact sequence in an abelian category $\cA$. Then any of the objects $F(A_1),\ldots,F(A_n)$
belong to the triangulated subcategory of $\cT$ generated by the remaining objects.
\end{lemma}

\bp
It suffices to prove that each $A_i$ is generated by the remaining ones in $D(\cA)$. This is clear by converting the 
short exact sequences arising from  $0\to A_1\to\ldots \to A_n\to 0$ into exact triangles in $D(\cA)$. 
\ep

We now prove that the remaining terms are generated by $\cC^\vee$ by induction. Consider the two types of terms:

(a) The terms $F^\vee_{\tl,\tE}=F^\vee_{l-j, E\cup J}$ have scores
$S(\tl,\tE)=(l-j)+\min\{e+j,p-e-j\}\leq (l-j)+(e+j)=l+e=S(l,E)=a$,
and we are done by induction, since $\tl\leq l$, with equality if and only if $(\tl,\tE)=(l,E)$.

(b) The terms $F^\vee_{\tl,\tE}=F^\vee_{j-l-2, E\cup J}$ ($l+2\leq j\leq r+1$) have scores
$S(\tl,\tE)=(j-l-2)+\min\{e+j,p-e-j\}\leq (j-l-2)+(p-e-j)=p-e-l-2=p-a-2$
since $S(l,E)=l+e=a$. But $p-a-2<a$ since 
since $a\geq r$. It follows that these terms are  generated by $\cC^\vee$ by induction. 

\underline{Assume  $e\geq r+1$.}  Let $I\subseteq E$, $|I|=r+1$ and let $E'=E\setminus I$. 
As before, we tensor the above Koszul resolution of $\De_{I\cup\{x\}}\subseteq (\PP^1)^p\times\PP^1$ with 
$\cO\left(-\sum_{k\in E'}{\bf e_k}+(r-1-l) {\bf e_x}\right)$,
take derived push-forwards via  the projection map $\pi$ and apply Lemma~\ref {obviosddd}:

$$0\quad F^\vee_{r-1-l, E'}\quad 
%\bigoplus_{i\in I} F^\vee_{r-2-l,E'\cup\{i\}}\quad
\ldots\quad 
\bigoplus_{J\subseteq I, |J|=j} F^\vee_{r-1-l-j, E'\cup J}\quad\ldots \bigoplus_{J\subseteq I, |J|=r-1-l} F^\vee_{0, E'\cup J}\quad 0$$
$$\bigoplus_{J\subseteq I, |J|=r+1-l} F^\vee_{0, E'\cup J}\quad\ldots\bigoplus_{J\subseteq I, |J|=j} F^\vee_{l+j-r-1, E'\cup J}\quad\ldots\quad F^\vee_{l, E},$$
where the terms $F^\vee_{r-1-l-j, E'\cup J}$ appear as $R^0\pi_*$ for $0\leq j< r-l$, 
the $0$ term appears when $j=r-l$, and the terms $F^\vee_{l+j-r-1, E'\cup J}$ appear as $R^1\pi_*$ for 
$r+1-l\leq j< r+1$. 
Consider the two types of terms and we compare theirs scores with 
$a=S(l,E)=l+p-e$.

(a) The terms $F^\vee_{\tl,\tE}=F^\vee_{r-1-l-j, E'\cup J}$ have scores
$S(\tl,\tE)=(r-1-l-j)+\min\{e'+j,p-e'-j\}\leq (r-1-l-j)+(e'+j)=e-l-2$
(where $e'=|E'|=e-r-1$). 
But since we assume $a=l+p-e\geq r$, we have $e-l\leq r+1$, hence, $S(\tl,\tE)\leq r-1$, i.e., $F^\vee_{\tl,\tE}$ is in $\cC^\vee$. 
%$l+p-e>e-l-2$ and we are done by induction. 

(b) The terms $F^\vee_{\tl,\tE}=F^\vee_{l+j-r-1, E'\cup J}$ have scores 
$S(\tl,\tE)=(l+j-r-1)+\min\{e'+j,p-e'-j\}\leq (l+j-r-1)+p-e'-j=l+p-e=S(l,E),$
and we are done by induction, since $\tl=l+j-r-1\leq l$ (as $j\leq r+1$) with equality if and only if $j=r+1$, i.e., $J=I$ (that is $(\tl,\tE)=(l,E)$). 
\ep

\begin{cor}\label{asgadhad}
Let $p=2r$. A  $S_p$-equivariant full exceptional collection in $D^b(\M_{p,1})$ is given by 
$\{F_{l,E}\}$ with $l+\min\{e_p,p-e_p\}\leq r-1$. 
\end{cor}

\bp
We have
$\M_{p,1}\cong\M_{({1\over r}, \ldots,{1\over r},\epsilon)}\cong \M_{p+1}$.
Indeed, the stability condition is the same: no $r+1$ points are allowed to collide. The statement now follows from Theorem \ref{symm}. 
\ep

%%%%%%%%%%%%%%%%%%%%%%%%%%%%%%%%%%%%%%%%%%%%%%%%%%%%%%%%%%%%%%%%%%%%%%%%%
%%%%%%%%%%%%%%%%%%%%%%%%%%%%%%%%%%%%%%%%%%%%%%%%%%%%%%%%%%%%%%%%%%%%%%%%%

\begin{thm}\label{dfkfvkfvjkfvjk}\label{symm+}
Let $p=2r+1$, $q>0$. The vector bundles $F_{l,E}$ form a full strong $(S_p\times S_q)$-equivariant exceptional collection on $\M_{p,q}$
for subsets $E_p\subseteq P$, $E_q\subseteq Q$ such that $l+e$ is even and 
$l+\min(e_p,p-e_p)\leq r-1$.

The order of the blocks given by the $S_p\times S_q$-orbits $F_{l,e_p,e_q}$ is first by increasing $e_q$, then by increasing $e_p$, and finally arbitrary in $l$.
\end{thm}

\bp
We have a morphism $f: \M_{p,q}\to \M_p$, an iterated ($q$ times) universal $\bP^1$-bundle of $\M_p$.
It is induced by the projection  $(\bP^1)^{p+q}\to (\bP^1)^p$ which maps $(\bP^1)^{p+q}_{ss}$ to $(\bP^1)^p_{ss}$.
Using Theorem \ref{stackofbundles}, we can identify the pull-back $f^*\{F_{l,E_p}\}$ of the collection of Theorem~\ref{symm}
with a collection $\{F_{l,E_p}\}$ on $\M_{p,q}$ when $E_p$ is  a subset of $P$.
For every $j\in Q$, let $L_j=F_{0,\{1,\ldots,p,j\}}$. 
% be a choice of an $S_p$-equivariant relative $\cO_{\pi_j}(1)$ for the $j$-th copy of the universal $\bP^1$-bundle.
For every subset $J\subset Q$, let $L_J=\bigotimes_{j\in J} L_j$.
From Orlov's theorem on the derived category of a projective bundle \cite{Orlov blow-up}, we have a s.o.d.~ of $D^b(\M_{p,q})$ into $2^q$ blocks equivalent to $D^b(\M_{p})$ via functors 
$D^b(\M_{p})\to D^b(\M_{p,q})$, $ E\mapsto Lf^*(E)\otimes L_J$
for every subset $J\subset Q$.
In blocks with $|J|=2s$ even, we introduce an exceptional collection
$\{F_{l,E_p}\}\otimes F_{0,\{1,\ldots,p\}}^{\otimes -2s}\otimes L_J=\{F_{l,E_p\cup J}\}$ (use Proposition \ref{tensor}). 
In blocks with $|J|=2s+1$ odd, we use an exceptional collection
$\{F_{l,E_p^c}^\vee\}\otimes F_{0,\{1,\ldots,p\}}^{\otimes -2s}\otimes L_J\cong\{F_{l,E_p\cup J}\}$,
where $E_p^c=P\setminus E_p$.
\ep

%Recall that Corollary ~\ref{asgadhad} and Theorem ~\ref{dfkfvkfvjkfvjk} follow from Theorem  \ref{symm}, 
%as proved in Section~\ref{windows}. In this section we prove that exceptional collection of %Theorem ~\ref{symm}, call it $\cC$, is full.

%%%%%%%%%%%%%%%%%%%%%%%%%%%%%%%%%%%%%%%%%%%%%%%%%%%%%%%%%%%%%%%%%%%%%%%%%%%%

%%%%%%%%%%%%%%%%%%%%%%%%%%%%%%%%%%%%%%%%%%%%%%%%%%%%%%%%%%%%%%%%%%%%%%%%%%%%

\section{Functions of subsets, score and inequalities for pairs $(l,E)$}\label{inequalities section}

In this section we collect numerical functions of subsets and inequalities used throughout the paper.
%to relate exceptional collections on different Hassett spaces. 
The reader may wish to read Notation~\ref{numerics} and ~\ref{asgasfgasgsg}
%, and Terminology~\ref{wrgqergaerhaethqeth} 
and skip the rest of this section, % on a first reading, 
referring back to it as needed.

\smallskip
%The proofs are mundane. 

Let $l\in\ZZ$ and let $E, T$ be subsets of the set $N=\{1,\ldots, n\}$. Unless otherwise noted, in this section we do not impose any condition on the parity of $l+e$. In all the sections that follow we will always assume $e+l$ is even as all the considered pairs $(l,E)$ belong to one of the groups $1A$, $1B$, $2$ or $2A$, $2B$. The reason we do not impose 
parity conidtions on $l+e$ in this section is for induction purposes (proof of Corollary \ref{dfbdfbdfb})

\begin{notn}\label{numerics}
We denote
\begin{equation}\label{peljfeljkljk}
f_{T,E,l}=|E\cap T|-{e-l\over 2}={|E\cap T|-|E\cap T^c|+l\over 2},\quad T^c=N\setminus T,
\end{equation}
\begin{equation}\label{Sdgsgasrh}
\al_{T,E,l}=\max\{0,-f_{T,E,l}\}.
%\begin{cases}
%0 & \hbox{\rm if}\quad |E\cap T|\geq \frac{e-l}{2}\cr
%\frac{e-l}{2}-|E\cap T| & \hbox{\rm if}\quad |E\cap T|<\frac{e-l}{2}.\cr
%\end{cases}
\end{equation}
\begin{equation}\label{sgasrgarga}
m_{T,E,l}=\max\{0,f_{T,E,l}\}.
%\begin{cases}
%0 & \hbox{\rm if}\quad |E\cap T|\geq \frac{e-l}{2}\cr
%\frac{e-l}{2}-|E\cap T| & \hbox{\rm if}\quad |E\cap T|<\frac{e-l}{2}.\cr
%\end{cases}
\end{equation}
We use notation $f_T$, $\al_T$, $m_T$ when $E$ and $l$ are clear from the context.
\end{notn}

\begin{rmk}\label{f_T positive}
Note that $f_{T}+f_{T^c}=l$, so when $e+l$ is even and $l\ge-1$, we have that at least one of $f_{T}, f_{T^c}$ is non-negative.
Typically $l\ge0$, $e+l$ is even and so if $A\sqcup B=N\setminus\{z\}$ 
then at least one of the integers $f_{A}$ or $f_B$ is $\ge0$.
\end{rmk}

\begin{notn}\label{asgasfgasgsg}
Let $P$ be the set of heavy points, with $|P|=p=2r\ge 4$. We~write $Q$ for the set of light points when their number is odd, and $\tilde{Q}$ when their number is even. The reason is that we often pass between these two setups by adding an extra light marking, denoted $y$ or $z$. 

In the first case we have $|Q|=q=2s+1\ge 1$, and in the second case we have $|\tilde{Q}|=q+1=2s+2\ge 0$.
In paticular, $q$ is always odd and can be equal to $-1$ in the second case.

For $l\in\ZZ$,  $E\subseteq P\cup Q$ (resp., $E\subseteq P\cup \tQ$)
we define the {\it score} %of a pair $(l,E)$  as
$$S(l,E)=l+\min\{e_p,p-e_p\}+\min\{e_q,q-e_q\}$$ 
$$(\hbox{\rm resp.,}\ S'(l,E)=l+\min\{e_p,p-e_p\}+\min\{e_q,q+1-e_q\}).$$
%The score  is an even number that stays the same if we exchange either $E_p\leftrightarrow P\setminus E_p$ or 
% $E_q\leftrightarrow \tQ\setminus E_q$. 

We use different notation for $S(l,E)$ and $S'(l,E)$ because sometimes $E$ can be regarded as a subset in both $Q$ or $\tilde{Q}$.  We emphasize that whenever we consider $E\subseteq P\cup Q$ we assume $q=2s+1\geq1$, while whenever we consider $E\subseteq P\cup \tQ$ we assume $q=2s+2\geq0$.
\end{notn}

For the purposes of this section, we introduce the following. 

\begin{terminology}\label{wrgqergaerhaethqeth}
For $l\geq0$,  $E\subseteq P\cup Q$ (resp., $E\subseteq P\cup \tQ$)  we say that $(l,E)$ 
{\it satisfies} condition $1A$, or $1B$, or $2$ (resp. condition $1A$, or $1B$, or $2A$, or $2B$) if it satisfies the inequalities in Theorem \ref{p,q case} or Remark~\ref{SgSRgSRhSRh} (resp., Theorem~\ref{asdvzsfvsfb} or  Remark~\ref{SgSRgSRhSRh}), {\it regardless of the parity of $l+e$}. Recall that when we impose in addition the condition that $l+e$ is even, we say that $(l,E)$ {\it is in group}  $1A$, or $1B$, or $2$ (resp., group $1A$, or $1B$, or $2A$, or $2B$). 
\end{terminology}

\begin{lemma}\label{basic} 
The score functions satisfy the following:
\bi
\item[(i) ] For $l\geq0$, $E\subseteq P\cup Q$ and $(l,E)$  satisfying conditions $1A$, or $1B$, or $2$ of Theorem \ref{p,q case} and Remark~\ref{SgSRgSRhSRh}, we have $S(l,E)\leq r+s-1$. 
\item[(ii) ]  For  $l\geq0$, $E\subseteq P\cup \tQ$ and $(l,E)$ satisfying conditions  $1A$, or $1B$, or $2A$, or $2B$ of Theorem~\ref{asdvzsfvsfb} and Remark~\ref{SgSRgSRhSRh}, 
we have $S'(l,E)\leq r+s$.  
\ei
Moreover equality in (i)  holds if and only if either $l+e_p=r-1$, $e_q=s$ (group $1A$), or 
 $l+e_p=r-1$, $e_q=s+1$ (group $1A$), or $l+(p-e_p)=r-1$, $e_q=s$  (group $1B$), or
 $l+(p-e_p)=r-1$, $e_q=s+1$  (group $1B$), or $e_p=r$, $l+e_q=s-1$  (group $2$), or $e_p=r$, $l+(q-e_q)=s-1$ (group $2$).  

Similarly, equality in (ii)  holds if and only if either $l+e_p=r-1$, $e_q=s+1$ (group $1A$), or $l+(p-e_p)=r-1$, $e_q=s+1$ (group $1B$), or 
$e_p=r$, $l+q+1-e_q=s$ (group $2A$), or $e_p=r$, $l+e_q=s$ (group $2B$). 
\end{lemma}

\bp
For groups  $1A$, $1B$ (on $P\cup Q$ and $P\cup \tQ$) it follows from the definitions that 
$l+\min\{e_p,p-e_p\}\leq r-1$. In this case (i), resp., (ii), follow from $\min\{e_q,q-e_q\}\leq s$, resp., 
$\min\{e_q,q+1-e_q\}\leq s+1$. For group $2$ (on $P\cup Q$) it follows from the definitions that  $l+\min\{e_q,q-e_q\}\leq s-1$, while for groups  $2A$, $2B$  (on $P\cup \tQ$) we have that  $l+\min\{e_q,q+1-e_q\}\leq s$, and (i), (ii) follow from $\min\{e_p,p-e_p\}\leq r$. 
\ep

\begin{lemma}\label{bounds+}
Let 
$l\in\ZZ$ (positive or negative), 
$E\subseteq P\cup \tQ$.  Then 
\begin{equation}\label{First}
\max_{P\cup \tQ=T\sqcup T^c} f_{T,E,l}=S'(l,E)/2,
\end{equation}
\begin{equation}\label{Second}
\max_{P\cup \tQ=T\sqcup T^c}- f_{T,E,l}=S'(l,E)/2-l.
\end{equation}
Here we assume that 
$T=T_p\coprod T_q$ (so $T^c=T^c_p\coprod T^c_q$, where $T^c_p=P\setminus T_p$, $T^c_q=Q\setminus T_q$) is a subset of markings split into heavy and light markings such that $|T_p|=r$ and 
$|T_q|=s+1$.
% $|T_q|=s$.
Equality in (\ref{First}) is attained if and only if 
$(T_p\subseteq E_p$ or $E_p\subseteq T_p)$ and $(T_q\subseteq E_q$ or $E_q\subseteq T_q)$, 
while equality in (\ref{Second}) is attained if and only if 
$(T^c_p\subseteq E_p$ or $E_p\subseteq T^c_p)$ and $(T^c_q\subseteq E_q$ or $E_q\subseteq T^c_q)$.
\end{lemma}

\bp
%Indeed, 
$|E\cap T|$ is maximized when
$2|E\cap T|=2\left(\min(r,e_p)+\min(s+1,e_q)\right)=\min(p,2e_p)+\min(q+1,2e_q)$
and so
$2 f_{T,E,l}=2|E\cap T|-(e-l)=\min(p-e_p,e_p)+\min(q+1-e_q,e_q)+l=S'(l,E)$.
Similarly, $-|E\cap T|$ is maximized when
$2|E\cap T^c|-(e-l)=S'(l,E)$
and so 
$(e-l)-2|E\cap T|=(e-l)+2|E\cap T^c|-2e=S'(l,E)-2l$, 
which proves the lemma.
\ep

\begin{lemma}\label{jhvvmgvmv}\label{general p,q+1}\label{critical}\label{bound1}
Let $E\subseteq P\cup \tQ$.  Suppose $(l,E)$ satisfies conditions $1A$, or $1B$, or $2A$, or $2B$ of Theorem~\ref{asdvzsfvsfb} and Remark~\ref{SgSRgSRhSRh}.
Then  
\begin{equation}\label{msnfbas}
m_{T,E,l}=\max\{0,f_{T,E,l}\}\leq (r+s)/2,
\end{equation}
for every subset of markings 
$T=T_p\coprod T_q$ split into heavy and light markings such that $|T_p|=r$, 
$|T_q|=s+1$.  %$|T_q|=s$.
We call $T$ \underline{critical} for $(l,E)$ if equality holds in~(\ref{msnfbas}). 

Then $T$~is critical if and only if 
%The inequality (\ref{msnfbas}) is strict unless 
we have one of the  cases listed in the following table, where we also list critical subsets:
%\bi[leftmargin=0.8in]
%\item[(group $1A$)] $l+e_p=r-1$, $e_q=s+1$, $E_p\subseteq T_p$, $E_q=T_q$;
%\item[(group $2A$)] $e_p=r$, $l+q+1-e_q=s$, $E_p=T_p$, $T_q\subseteq E_q$;
%\item[(group $1B$)] $l+(p-e_p)=r-1$, $e_q=s+1$, $T_p\subseteq E_p$, $E_q=T_q$;
%\item[(group $2B$)] $e_p=r$, $l+e_q=s$,  $E_p=T_p$, $E_q\subseteq T_q$.
%\ei

\smallskip
\begin{tabular}{ccc}
(group $1A$) & $l+e_p=r-1$,\quad $e_q=s+1$       & $E_p\subseteq T_p$, $E_q=T_q$;\cr
(group $1B$) & $l+(p-e_p)=r-1$,\quad $e_q=s+1$ & $T_p\subseteq E_p$, $E_q=T_q$;\cr
(group $2A$) & $e_p=r$,\quad $l+q+1-e_q=s$       & $E_p=T_p$, $T_q\subseteq E_q$;\cr
(group $2B$) & $e_p=r$,\quad $l+e_q=s$              & $E_p=T_p$, $E_q\subseteq T_q$.\cr
\end{tabular}
\smallskip

In addition, we have: 
\begin{equation}\label{snbnfbas}
f_{T^c,E,l}\ge -(r+s)/2.
\end{equation}
The inequality \eqref{snbnfbas} is strict unless $l=0$ and $T$ is critical for $(l,E)$, in particular, it appears in the table above.

When in addition $l+e$ is even, a set $T$ is critical for $(l,E)$ if and only if  $r+s$ is even and we are in one of the cases listed in the table above.

% and we have one of the following cases:
%\bi
%\item $e_p=r-1$, $e_q=s+1$, $l=0$, $E\subseteq T$  (group $1A$);
%\item $e_p=r+1$, $e_q=s+1$, $l=0$, $T\subseteq E$ (group $1B$);
%\item $e_p=r$, $e_q=s+2$, $l=0$, $T\subseteq E$  (group $2A$);
%\item $e_p=r$, $e_q=s$, $l=0$, $E\subseteq T$   (group $2B$).
%\ei
%This is the same list as for (\ref{msnfbas}) but with $l=0$. 
%In particular, $T$ is critical. 
\end{lemma}

\begin{proof}  %[Proof of Lemma \ref{critical}]
The first statement is a direct consequence of Lemmas \ref{basic} and \ref{bounds+}. 
%By Lemma~\ref{bounds+}, $\max_Tf_{T,E,l}=S'(l,E)/2$.
%If $(l,E)$ is in group $A$ then $(l,E^c)$ is in group $B$ and 
%$S'(l,E)=S'(l,E^c)$.
%Since the conclusion of the lemma is symmetric under
%this operation, we can assume without loss of generality that  $(l,E)$ is in group $1A$ or $2A$.
%%We first prove (\ref{msnfbas}) and classify critical sets using Lemma \ref{bounds+}. 
%
%If $e_p\le r$, $e_q\le s+1$ then
%$S'(l,E)=l+e_p+e_q=l+e\leq r+s$ for  groups $1A$ and $2A$.
%The inequality is strict except for the first row of the table. 
%
%If $e_p>r$, $e_q\le s+1$ (possible only in case 1A)) then 
%$S'(l,E)=l+(p-e_p)+e_q$.
%Since $l+(p+1-e_p)\leq r-1$ for group $1A$, we have $S'(l,E)<r+s$. 
%
%If $e_p\le r$, $e_q>s+1$ then
%$S'(l,E)=l+e_p+(q+1-e_q)\leq r+s$ for  groups $1A$ and $2A$.
%The inequality is strict except for  the second row in the table.
%
%Finally, if $e_p>r$,  $e_q>s+1$ (possible only in case 1A)) then
%$S'(l,E)=l+(p-e_p)+(q+1-e_q)$,
%which is $<r+s$ for $1A$.
The second claim follows from the first, since  $l\geq0$, 
$f_{T^c,E,l}\geq -f_{T,E,l}\geq-(r+s)/2$
with the first inequality becoming an equality iff $l=0$.
 \end{proof}
 
\begin{cor}\label{AA}
Let $E\subseteq P\cup\tQ$.
If $(l,E)$ satisfies conditions  $2A$ or $2B$
then for $I\subseteq \tQ$, $|I|=s+1$, we have
(i) $f_{I,E_q,l}\le s/2$,
with equality iff either $e_q=l+s+2$, $I\subseteq E_q$ (group $2A$) or $l+e_q=s$, $E_q\subseteq I$ (group $2B$); 
(ii) $f_{I^c,E_q,l}\ge -s/2$.
\end{cor}

\bp
Since $|E_p|=r$, both inequalities follow from Lemma \ref{critical} for a set $T$ with heavy indices $T_p=E_p$ and light indices $T_q=I$. 
\ep

We use analogues of Lemmas \ref{bounds+}, \ref{bound1} for $P\cup Q$ as follows:
\begin{lemma}\label{bounds+variant}
Let $l\in\ZZ$ (positive or negative), 
$E\subseteq P\cup Q$, $y\in Q$.  Then 
\begin{equation}\label{First variant}
\max_{P\cup (Q\setminus\{y\})=T\sqcup T^c} f_{T,E,l}\leq S'(l,E\setminus\{y\})/2. 
\end{equation}
%\begin{equation}\label{Second}
%\max_{P\cup \tQ=T\sqcup T^c}- f_{T,E,l}=S'(l,E)/2-l.
%\end{equation}
Here we assume that 
$T=T_p\coprod T_q$ (so $T^c=(P\setminus T_p)\coprod (Q\setminus T_q)$) 
is a subset of markings split into heavy and light markings such that $|T_p|=r$ and 
$|T_q|=s$ and the score function $S'(l,E\setminus\{y\})$ is considered on $P\cup (Q\setminus\{y\})$ (i.e., 
$\tQ=(Q\setminus\{y\})$ in Notation \ref{asgasfgasgsg}).
% $|T_q|=s$.

Equality in (\ref{First variant}) is attained if and only if $y\notin E$. In this case, equality is attained for some $f_{T,E,l}$ if and only if 
$(T_p\subseteq E_p$ or $E_p\subseteq T_p)$ and $(T_q\subseteq E_q$ or $E_q\subseteq T_q)$. 
%while equality in (\ref{Second}) is attained if and only if 
%$(T^c_p\subseteq E_p$ or $E_p\subseteq T^c_p)$ and $(T^c_q\subseteq E_q$ or $E_q\subseteq T^c_q)$.
\end{lemma}

\bp
Let $E':=E\setminus\{y\}$. Then $E\cap T=E'\cap T$, and $f_{T,E,l}=f_{T, E',l}$ if $y\notin E$ and $f_{T,E,l}=f_{T, E',l}-\frac{1}{2}$ if $y\in E$.  The result now follows from Lemma \ref{bounds+}.  
%This is the same proof as for Lemma  \ref{critical}, as $E\cap T=(E\setminus\{y\})\cap T$. 
\ep

\begin{lemma}\label{critical on W}
Let $E\subseteq P\cup Q$. Suppose $(l,E)$ is 
satisfies conditions $1A$, or $1B$, or $2$ of Theorem~\ref{asfffdvzsfvsfb} and Remark~\ref{SgSRgSRhSRh}.
Let $y\in Q$ and $T=T_p\coprod T_q\subset P\cup (Q\setminus\{y\})$ be 
split into heavy and light markings 
such that $|T_p|=r$ and 
$|T_q|=s$.
Then
\begin{equation}\label{critical W}
m_{T,E,l}=\max\{0,f_{T,E,l}\}\le (r+s-1)/2.
\end{equation}
We call $T$ \underline{critical} if equality holds in (\ref{critical W}),
which happens if and only if $y\notin E$ and 
we are in one of the following cases:

\smallskip
\begin{tabular}{ccc}
(group $1A$) &  $l+e_p=r-1$, $e_q=s$, & $E_p\subseteq T_p$, $E_q=T_q$;\cr
(group $1B$) & $l+(p-e_p)=r-1$, $e_q=s$, & $T_p\subseteq E_p$, $E_q=T_q$;\cr
(group $2$) & $e_p=r$, $l+e_q=s-1$,   & $E_p=T_p$, $E_q\subseteq T_q$.\cr
\end{tabular}
\smallskip

\noindent
In particular, $m_{T,E,l}\le {r+s\over 2}-|E\cap\{y\}|$.

\smallskip

If in addition, $l+e$ is even and $T$ is critical, we must have that $r+s$ is odd.
%Also, we have 
%\begin{equation}%\label{critical WW}
%$f_{T^c,E,l}\geq -(r+s)/2$.
%\end{equation}
%The inequality is strict, unless $r+s$ is even, $l=0$, $y\in E$, and we have one the following cases:
%\smallskip
%\begin{tabular}{ccc}
%(group $1A$) & $l=0$, $e_p=r-1$, $e_q=s+1$ , & $E_p\subseteq T_p$, $E_q=T_q\cup\{y\}$;\cr
%(group $1B$) & $l=0$, $e_p=r+1$, $e_q=s+1$, & $T_p\subseteq E_p$, $E_q=T_q\cup\{y\}$;\cr
%(group $2$) & $l=0$, $e_p=r$, $e_q=s+2$,   & $E_p=T_p$, $T_q\subseteq E_q$, $y\in E_q$.\cr
%\end{tabular}
\end{lemma}

\begin{proof}
By Lemma \ref{bounds+variant} we have that $2f_{T,E,l}\leq S'(l,E\setminus\{y\})$ and equality holds iff $y\notin E$ and 
$(T_p\subseteq E_p$ or $E_p\subseteq T_p)$ and $(T_q\subseteq E_q$ or $E_q\subseteq T_q)$.

Let $E'=E\setminus\{y\}$. If $(l,E)$ satisfies condition $1A$ (resp. $1B$), then $(l,E')$ satisfies condition $1A$ (resp. $1B$). Similarly, if $(l,E)$ satisfies condition $2$ on $M_{p,q}$, then $(l,E')$ satisfies condition $2B$ on $\M_{p,q-1}$ if $y\notin E$ or condition $2A$ if $y\in E$.
It follows by Lemma \ref{basic} that $S'(l,E\setminus\{y\})\leq r+s-1$ and equality holds in the cases listed there. 
Note that the case $e_p=r$, $l+(q-1-e_q)=s-1$ 
(implying an equality in Lemma \ref{basic}(ii) on $\M_{p,q-1}$) does not occur since 
this would violate the condition that $(l,E)$ satisfies condition $(2)$.  

\end{proof}

\begin{cor}\label{dfbdfbdfb}
In the set-up of Notation \ref{asgasfgasgsg}, let $l\geq 0$ and $E\subseteq P\cup Q$ with $|Q|=2s+1\geq1$, or $E\subseteq P\cup \tQ$ with $|Q|=2s+2\geq0$.  
Let $P=R\sqcup R'$ be a partition of $P$ with $|R|=|R'|=r$. If $(l,E)$ satisfies $1A$ or $1B$, then 
\begin{equation}\label{another}
f_{R,E_p,l}\leq (r-1)/2.
\end{equation}
%$f_{R,E_p,l}\leq\frac{r-1}{2}$, 
%or equivalently, 
%\begin{equation}\label{another}
%l+|E_p\cap R|-|E_p\cap R'|\leq r-1.
%\end{equation}
\end{cor}

\bp
If $(l,E)$ with $E\subseteq P\cup\tQ$ 
satisfies condition $1A$ (resp., $1B$), then $(l,E_p)$ for the set $E_p\supset P$
satisfies condition $1A$ (resp., $1B$) and the result follows by applying Lemma \ref{critical} in the case when $p=2r$, $s=-1$ for the partition $T=R$, $T^c=R'$. 
Similarly  if $(l,E)$ with $E\subseteq P\cup Q$ 
satisfies condition $1A$ 
(resp., $1B$), then $(l,E_p)$ 
satisfies condition $1A$ (resp., $1B$) on $\M_{p,1}$ and the result follows by applying Lemma \ref{critical on W} in the case $p=2r$, $s=0$ for the same partition. 
\ep

We will also need the following:
\begin{lemma}\label{another2}
Let $q=|Q|=2s+1$, $l\geq0$, $E\subseteq Q$ a set with $e=|E|$ satisfying 
$l+\min\{e, q-e\}\leq s-1$. Let $I\subseteq Q$ be any subset. 
Then 
\begin{equation}
l+|E\cap I^c|-|E\cap I|\leq 
\begin{cases}
s-1 & \hbox{\rm if}\quad e\leq s\cr
2|I^c|-s-2 & \hbox{\rm if}\quad e\geq s+1.\cr
\end{cases}
\end{equation}
\end{lemma}

\bp
If $e\leq s$ then the inequality follows from $l+e\leq s-1$.
If $e\geq s+1$, then the inequality follows from $l-e\leq -s-2$. 
\ep

The same way we obtain the following: 
\begin{lemma}\label{another3}
Let $q=|Q|=2s+2$, $l\geq0$, $E\subseteq Q$ a set with $e=|E|$, and  
let $I\subseteq Q$ be any subset. 
If $l+\min\{e+1, q+1-e\}\leq s$, then 
\begin{equation}
l+|E\cap I^c|-|E\cap I|\leq 
\begin{cases}
s-1 & \hbox{\rm if}\quad e\leq s\cr
2|I^c|-s-2 & \hbox{\rm if}\quad e\geq s+1.\cr
\end{cases}
\end{equation}

If $l+\min\{e, q+2-e\}\leq s$, then 
\begin{equation}
l+|E\cap I^c|-|E\cap I|\leq 
\begin{cases}
s & \hbox{\rm if}\quad e\leq s+1\cr
2|I^c|-s-3 & \hbox{\rm if}\quad e\geq s+2.\cr
\end{cases}
\end{equation}
\end{lemma}

%%%%%%%%%%%%%%%%%%%%%%%%%%%%%%%%%%%%%%%%%%%%%%%%%%%%%%%%%%%%%%%%%%%%%%%%
%%%%%%%%%%%%%%%%%%%%%%%%%%%%%%%%%%%%%%%%%%%%%%%%%%%%%%%%%%%%%%%%%%%%%%%%
%%%%%%%%%%%%%%%%%%%%%%%%%%%%%%%%%%%%%%%%%%%%%%%%%%%%%%%%%%%%%%%%%%%%%%%%
%%%%%%%%%%%%%%%%%%%%%%%%%%%%%%%%%%%%%%%%%%%%%%%%%%%%%%%%%%%%%%%%%%%%%%%%
%%%%%%%%%%%%%%%%%%%%%%%%%%%%%%%%%%%%%%%%%%%%%%%%%%%%%%%%%%%%%%%%%%%%%%%%

%%%%%%%%%%%%%%%%%%%%%%%%%%%%%%%%%%%%%%%%%%%%%%%%%%%%%%%%%%%%%%%%%%%%%%%%%%%%%%%%%%%%%%

\section{Extending vector bundles $F_{l,E}$ to some Hassett spaces}\label{extend section}

In this section we will extend the construction of vector bundles $F_{l,E}$, defined in Section~\ref{windows section} for open substacks 
of the  stack $[(\bP^1)^n/\PGL_2]$, to certain
Hassett spaces. They will include all spaces $\M_{p,q}$ when $p$ and $q$ are both even (see Definition~\ref{allF}) as well as the universal families over them
(see Definition~\ref{F on W}). As in Section~\ref{stack section}, $F_{l,E}$ will be constructed as the pushforward of the line bundle $N_{l,E}$, but its definition will need to be adjusted.

\medskip 

We first recall some general facts about Hassett spaces. Let $\M:=\M_{\ba}$  be a Hassett space ($\sum a_i>2$). Let $\al:W\ra\M$ be the universal family, with $\sigma_1,\ldots,\sigma_n$ the corresponding sections. 
In $\Pic(\M)$ there are tautological classes $\psi_i=\si_i^*(\om_{\al})=-\si_i^*(\si_i)$. 
Here and everywhere we identify line bundles with divisor classes (hence, we use both additive and multiplicative notation). We now record several facts for further use. 

\begin{note}\label{note 1}
(1) The boundary divisors on $\M$ may have two types: 

\smallskip

\underline{Type I}: divisors $\de_{T,T^c}$ with $|T|, |T^c|\geq2$, whose generic elements consist of curves with two components, marked by $T$ and $T^c$ respectively
($\sum_{i\in T}a_i>1$, $\sum_{i\in T^c}a_i>1$). 
We can identify $\de_{T,T^c}=\M_{T\cup\{u\}}\times\M_{T^c\cup\{u\}}$, with $u$ corresponding to the attaching point. Here $\M_{T\cup\{u\}}$ denotes the Hassett space parametrizing rational curves with marked points indexed by the set $T\cup\{u\}$, with weight $1$ for the point indexed by $u$ and weight $a_i$ for the points corresponding to $i\in T$. 
We often write $\de_T$ for the boundary divisor $\de_{T,T^c}\subset\M$, when the sets $T$, $T^c$ are clear from the context. We may also write $\de_{ij}$ for $\de_T$ when $T=\{i,j\}$.

\smallskip

\underline{Type II}: divisors parametrizing curves where markings $i, j$ coincide ($a_i+a_j\leq1$). We denote this
by $\de_{ij}$. Its class in $\Pic(\M)$ is 
 $\de_{ij}=\si_i^*(\si_j)$. We identify $\de_{ij}=\M_{\{i,j\}^c\cup\{u\}}$, with the marking $u$ having weight $a_i+a_j$.  
 
\medskip
 
 Note that we use the same notation $\de_{ij}$ for a boundary divisor of type I or type II, whichever exists in the context: type I if  
 $a_i+a_j>1$, $\sum_{k\neq i,j}a_k>1$ and type II if $a_i+a_j\leq1$ (which forces $\sum_{k\neq i,j}a_k>1$). 
 
\medskip

(2) \cite[Lemma 2.1]{CT_part_Ib} If $\al:W\ra\M$ is a $\PP^1$-bundle, i.e., all $\ba$-stable curves are irreducible, then in $\Pic(\M)$ we have $\psi_i+\psi_j=-2\de_{ij}$ if $a_i+a_j\leq1$ and  $\psi_i+\psi_j=0$ if $a_i+a_j>1$ (in which case our assumption on $\ba$ implies that 
$\sum_{k\neq i,j}a_k\leq 1$ (note that there is no boundary $\de_{ij}$ in this case).

\medskip

(3) If %$\cA=(a_1,\ldots, a_n)$ is such that 
$a_1=1$ % $\sum_{k\neq 1}a_k>2$ 
and $\sum\limits_{k\neq 1,j}a_k\leq 1$ for all $j\geq 2$,
then $\M_\ba\cong\PP^{n-3}$  \cite[\S 6.2]{Ha}
and we can assume without loss of generality that $a_j={1\over n-1}+\eps$ for $j\ge2$. Furthermore, 
$\de_{ij}=\cO(1)$ ($i\neq j$, $i,j\neq1$), $\psi_1=\cO(1)$, $\psi_j=\cO(-1)$ $(j\neq 1)$. 

\medskip

(4) Let $\M=\M_\ba$, $ \M'=\M_{\bb}$ %, $\cA=(a_i)$,  $\cB=(b_i)$ 
be Hassett spaces related by a reduction map $p:\M\ra \M'$ (i.e., $a_i\geq b_i$ for all $i$). Let $\pi:W\ra\M$, $\pi':W'\ra\M'$ be the corresponding universal families. We record here the relations between tautological classes and their pull backs via reduction maps. 
By
 \cite[Lemma 2.3]{CT_part_Ib}, we have: 
 \begin{equation}\label{Pullbacks1}
p^*\psi_i=\psi_i-\sum_{i\in I,\ 2\leq |I|\leq n-2,\ \sum_{i\in I}a_i>1,\  \sum_{i\in I}b_i\leq1}\de_I,
\end{equation}

Explicitly, the $\psi_i$ class pulls back to the same $\psi_i$ class, but adjusted by the $p$-exceptional boundary divisors $\de_{I, I^c}$ such that $i\in I$ and with the markings 
$\{p_k\}_{k\in I}$ identified after reducing the weights. 
Similarly, we have 
 \begin{equation}\label{Pullbacks2}
p^*\de_{ij}=\de_{ij}+\sum_{i,j\in I,\ 3\leq|I|\leq n-2,\ \sum_{i\in I}a_i>1,\  \sum_{i\in I}b_i\leq1}\de_I.
\end{equation}
The sum contains precisely the boundary divisors $\de_{I, I^c}$ such that $i,j\in I$ and with the markings 
$\{p_k\}_{k\in I}$ identified after reducing the weights.
Furthermore, there is an induced morphism $\phi:W\ra W'$ and we have
 \begin{equation}\label{Pullbacks3}
\phi^*\om_{\pi'}=\om_\pi-\sum_T \de_{T\cup\{y\},T^c},\quad w^*\de_{iy}=\de_{iy}+\sum_{i\in T} \de_{T\cup\{y\},T^c}.
\end{equation}
\end{note}

\medskip

\begin{note}\label{note 2}
Assume that the Hassett space $\M=\M_{\ba}$ is such that 
$\sum_{i\in I} a_i\ne 1$ for all $|I|\geq2$, or in the terminology in \cite{Ha}, the weight data $\ba$ is contained in a fine open chamber. Then:

(1) By \cite[Proposition 5.4]{Ha}, the universal family $W$ is isomorphic to the Hassett space $\M_{(a_1,\ldots, a_n,\epsilon)}$ for $\eps\ll1$.  The weight data $(a_1,\ldots, a_n,\epsilon)$
is also contained in a fine open chamber. In particular, $W$ is smooth and so is its universal family $\pi:\cU\ra W$. 

(2) We denote the extra marking on $W$ by $y$ and the extra marking on the universal family $\cU$ over $W$ with $x$. Note that the image of the section $\sigma_i:\M\ra W$ may be identified with $\de_{iy}\subset W$ 
when $a_i<1$.
\end{note}

Throughout the remaining part of this section, we assume that $\sum_{i\in I} a_i\ne 1$ for $|I|\geq2$ and all $\ba$-stable curves have at most two components. Note that is the case for all the spaces $\M_{p,q}$ for all $p,q$ -- even or odd. 

\medskip

Let $N=\{1,\ldots, n\}$. We construct vector bundles $F_{l,E}$ on both $\M$ and $W$ as follows. Here $\al: W\ra\M$, $\pi: \cU\ra W$ denote the universal families. 

\begin{defn}[\bf  $F_{l,E}$ on $\M$]\label{allF}
For a subset $E\subseteq N$ %:=\{1,\ldots,n\}$ 
with $e=|E|$ and integer $l\geq-1$ such that $e+l$ is even, we let 
$F_{l,E}=R\al_*\big(N_{l,E}\big)$, where 
$$N_{l,E}=\om_{\al}^{\frac{e-l}{2}}(E)\otimes\cO\Bigl(-\sum_T\al_{T,E,l}\de_{T\cup\{y\}, T^c}\Bigr),$$ 
where we denote $\cO(E):=\bigotimes_{j\in E}\cO(\sigma_j)$, the sum is over all $T\subset N$ such that $\de_{T,T^c}\subseteq \M$ is a boundary component (here $T^c:=N\setminus T$) and $\al_{T,E,l}$ are defined as in \eqref{Sdgsgasrh}. 
\end{defn}

\begin{defn}[$F_{l,E}$ on $W$]\label{F on W} 
For $E\subseteq N\cup\{y\}$ with $e=|E|$ and integer $l\geq0$ such that $e+l$ is even, on $W$ we let
$F_{l,E}=R\pi_*\big(N_{l,E}\big)$, where 
$$N_{l,E}=\om_{\pi}^{\frac{e-l}{2}}(E)\otimes\cO\Bigl(-\sum_T \al_T\de_{T\cup\{x\}, T^c\cup\{y\}}-\sum_T \al_{T\cup\{y\}}\de_{T\cup\{y, x\},T^c}\Bigr)$$
where for $S=T$ or $S=T\cup\{y\}$, we define $\alpha_S:=\alpha_{S,E,l}$ (see \eqref{Sdgsgasrh}; hence, $\al_T=\al_{T\cup\{y\}}=0$ or $\al_{T\cup\{y\}}=\al_T-|E\cap\{y\}|\geq0$).  
The sum is over all $T\subset N$ such that $\de_{T,T^c}\subseteq \M$ is a boundary component ($T^c:=N\setminus T$). 
%\begin{equation}
%\al_S=
%\begin{cases}
%0 & \hbox{\rm if}\quad |E\cap S|\geq \frac{e-l}{2}\cr
%\frac{e-l}{2}-|E\cap S| & \hbox{\rm if}\quad |E\cap S|<\frac{e-l}{2},\cr
%\end{cases}
%\end{equation}
\end{defn}

Before we prove that these are indeed vector bundles (Lemmas \ref{vb}, \ref{vb on W}), we first discuss the 
geometry of the spaces $\M$, $W$ and $\cU$. 
\begin{lemma}\label{BASIC}
Let $\de\subseteq \M$ be a boundary divisor in $\M$. 
Using the identification $\de=\M_{T\cup\{u\}}\times\M_{T^c\cup\{u\}}$,
we have $\de_{|\de}\cong(-\psi_u)\boxtimes(-\psi_u)$ in $\Pic(\de)$ (where we identify $\de=\M_{T^c\cup\{u\}}$ when $\M_{T\cup\{u\}}$ is a point). 
\end{lemma}

\bp
We first check that the identity holds on $\ocM_{0,n}$.  Using the forgetful map
$\pi:\ocM_{0,n}\ra\ocM_{0,\{1,2,3,4\}}$, we have $\pi^*\de_{\{1, 2\},\{3, 4\}}=\sum_{1,2\in T', 3,4\in {T'}^c}\de_{T',{T'}^c}$ and hence we have in  $\Pic(\ocM_{0,n})$ that
\begin{equation}\label{Sean}
\sum_{1,2\in T, 3,4\in {T}^c}\de_{T,{T}^c}=\sum_{1,3\in T, 2,4\in {T}^c}\de_{T,{T}^c}. 
\end{equation}

Fix $\de:=\de_{T_0, T_0^c}$ a boundary divisor on $\ocM_{0,n}$. We may assume without loss of generality that $1,2\in T_0$, $3,4\in T_0^c$. 
%$T_0=\{1,2\}\cup A$ and $T_0^c=\{3,4\}\cup B$, with $A$ and $B$ disjoint sets such that $A\sqcup B=N\setminus\{1,2,3,4\}$. 
Restricting (\ref{Sean}) to $\de$ we obtain that $\de_{|\de}+\sum_{(T_0\subsetneq T, 3,4\in {T}^c) \text{ or } (1,2\in T\subsetneq T_0)}{{\de_{T,T^c}}_{|\de}}=0$. We identify $\de=\ocM_{0, T_0\cup\{u\}}\times \ocM_{0, T^c_0\cup\{u\}}$ and we identify the restrictions 
${\de_{T,T^c}}{|\de}$ with boundary divisors on either $\ocM_{0, T_0\cup\{u\}}$ or $\ocM_{0, T^c_0\cup\{u\}}$. For example, if $T_0\subsetneq T$, then $3,4\in S:=T^c\subset T_0^c$ and we have ${\de_{T,T^c}}{|\de}=\de_{S, T_0^c\setminus S\cup\{u\}}$. It follows that 
$$-\de_{|\de}=\big(\sum_{1,2\in S\subsetneq T_0}\de_{S,(T_0\setminus S)\cup\{u\}}\big)\boxtimes\big(\sum_{3,4\in S\subsetneq T^c_0}\de_{S,(T^c_0\setminus S)\cup\{u\}}\big)$$ in 
$\Pic(\de)$. 
Using the Kapranov model given by the line bundle $\psi_u$, we have that $\psi_u=\sum_{1,2\in S\subseteq T_0}\de_S$ in $\Pic(\ocM_{0, T_0\cup\{u\}})$. By symmetry, we have that 
$\psi_u=\sum_{3,4\in S\subseteq T^c_0}\de_{S,(T_0\setminus S)\cup\{u\}}$ in $\Pic(\ocM_{0, T^c_0\cup\{u\}})$ and the identity follows. 

Consider now an arbitrary Hassett space $\M=\M_\ba$ and let $p:\ocM_{0,n}\ra \M$ be the reduction map. Fix $\de$ a boundary divisor on $\M$, corresponding to a partition 
$N=T_0\sqcup T_0^c$ (when $\de=\de_{ij}$ is of type II (see Note \ref{note 1}(1)), we let $T_0=\{i,j\}$). We let $\de'$ be the boundary divisor $\de_{T_0, T_0^c}$ on $\ocM_{0,n}$ and let $p'=p_{|\de'}:\de'\ra \de$ be the restriction map, which, after identifying $\de=\M_1\times\M_2$, 
$\de'=\ocM_{0, T_0\cup\{u\}}\times\ocM_{0, T^c_0\cup\{u\}}$, is the product of reduction maps $p_1:\ocM_{0, T_0\cup\{u\}}\ra\M_1$, $p_2:\ocM_{0, T^c_0\cup\{u\}}\ra\M_2$. It suffices to check that $(p^*\cO(\de))_{|\de'}=p_1^*(-\psi_u)\boxtimes p_2^*(-\psi_u)$. 

If $\de$ is of type I, then $p^*\cO(\de)=\cO(\de')$. By (\ref{Pullbacks1}) we have that $p_1^*\psi_u=\psi_u$, $p_2^*\psi_u=\psi_u$, since the marking $u$ has weight $1$ on $\M_1$ and $\M_2$. The statement now follows, since $\cO(\de')_{|\de'}=(-\psi_u)\boxtimes(-\psi_u)$ on $\ocM_{0,n}$. 

If $\de$ is of type II, we may assume that $T_0=\{1,2\}$ (so $a_1+a_2\leq1$). By (\ref{Pullbacks2}) we have $p^*\de_{12}=\de_{12}+\sum_{1,2\in T, |T|\geq 3, \sum_{i\in T}a_i\leq1}\de_{T,T^c}$. Identifying $\de'=\ocM_{0, T^c_0\cup\{u\}}$, we have in $\Pic(\de')$ that 
$${p^*\cO(\de)}_{|\de'}=\cO(\de')_{|\de'}+\sum_{\emptyset\neq S\subset T_0^c, a_1+a_2+\sum_{i\in S}a_i\leq1}\de_{S\cup\{u\},T_0^c\setminus S},$$
and we already proved that $\cO(\de')_{|\de'}=-\psi_u$. We identify $\de=\M_2$ (as $\M_1$ is a point), and note that the map $p':\de'\ra\de$ is a reduction map. It follows by (\ref{Pullbacks1}) that $p'^*(-\psi_u)=-\psi_u+\sum_{\emptyset\neq S\subset T_0^c, a_1+a_2+\sum_{i\in S}a_i\leq1}\de_{S\cup\{u\},T_0^c\setminus S}$.
\ep

\begin{lemma}\label{P1b}
Assume $a_1=1$ and assume that the universal family 
$\al: W\ra \M$ is a $\PP^1$-bundle, i.e., all $\ba$-stable curves are irreducible. 
Then:

(i) $\M\cong\PP^{n-3}$ and $W\cong\Bl_p\PP^{n-2}$. If we denote
$H$ the hyperplane class  and by $\De$ the exceptional divisor on 
$\Bl_p\PP^{n-2}$,
then we have 
$\psi_1=\cO(H)$, $\de_{N\setminus \{1\},\{1,y\}}=\De$
in $\Pic(W)$. % (we identify $W$ with the Hassett space with markings  $N\cup\{x\}$). 

(ii) The map $\al: W=\Bl_p\PP^{n-2}\ra\M=\PP^{n-3}$ is induced by the linear system $|H-\De|$. Moreover, if $f$ is a fiber of $\al$, 
then $\psi_1\cdot f=H\cdot f=1$. 

(iii) In $\Pic(W)$ we have that $\de_{jy}=H$, for all $j$ in $N\setminus\{1\}$.
\end{lemma}

\bp
Since there are no reducible $\ba$-stable curves, it follows that for all $j\neq 1$, we have 
$\sum_{k\neq 1,j} a_k\leq 1$ and therefore $\M\cong\PP^{n-3}$ by Note~\ref{note 1}(3). 
The only reducible curves parametrized by the Hassett space $W$ are those given by the boundary divisor 
$\de:=\de_{1y}$. Let $\M'$ be the Hassett space with the weights $(1,\frac{1}{n-1},\ldots,\frac{1}{n-1},\epsilon)$
(where $\frac{1}{n-1}$ appears $(n-1)$ times).
%(i.e., such that all but one of the markings $(N\setminus\{1\})\cup\{x\}$ may coincide). 
Then $\M'\cong\PP^{n-2}$ by Note~\ref{note 1}(3). The  reduction map $\phi: W\ra \M'$ contracts $\de$ to a point $p\in \M'$. 
Hence, $W=\Bl_p\M'$ with exceptional divisor $\De=\de$. By \eqref{Pullbacks1}, we have 
$\phi^*\psi_1=\psi_1$. This proves (i). 

Clearly, the map $\al$ restricted to $\De=\de$ is an isomorphism. The only morphism $\pi: \Bl_p\PP^{n-2}\ra\PP^{n-3}$ which is an isomorphism on the exceptional divisor $\De$ is the one given by the linear system $|H-\De|$. We have then from (i) that $H\cdot f=1$ for any fiber $f$ of $\pi$. This proves (ii). 

The curve $f$ is obtained by moving the marking $y$ along a $\PP^1$ with fixed markings from $N$. Hence, $\de_{jy}\cdot f=1$, for all $j\in N$. As $\de_{1y}=\De$ and 
$\de_{jy}\cdot\De=0$, it follows that $\de_{jy}=H$ in $\Pic(W)$ if $j\neq 1$. This proves (iii).
\ep

\begin{lemma}\label{SRGASRGASRGA}
%$W$ is  a Hassett space with markings  $1,\ldots,n$ and an extra point~$x$.
Every boundary divisor $\de_T:=\de_{T, T^c}\subseteq \M$ of type I is isomorphic to $\PP^{m-2}\times\PP^{n-m-2}$
for some $m$. Its preimage in $W$ is the union of $\de_{T\cup\{y\}}$ and $\de_{T^c\cup\{y\}}$, where 
$\de_{T\cup\{y\}}\cong\Bl_p\PP^{m-1}\times\PP^{n-m-2}$. The restriction map 
$\al_{|\de_{T\cup\{y\}}}: \de_{T\cup\{y\}}\ra\de_T$
is the product map $\pi\times \Id:\,\Bl_p\PP^{m-1}\times\PP^{n-m-2}\ra\PP^{m-2}\times\PP^{n-m-2}$
%$$\al_{|\de_{T\cup\{x\}}}: \de_{T\cup\{x\}}=\Bl_1\PP^{m-1}\times\PP^{n-m-2}\ra\de_T=\PP^{m-2}\times\PP^{n-m-2}$$
where $\pi:  \Bl_p\PP^{m-2}\ra  \PP^{m-3}$ is a %canonical 
$\PP^1$-bundle. 
The description for $\de_{T^c\cup\{y\}}$ is similar.
\end{lemma}

\bp
The boundary divisor $\de_T\subseteq \M$ corresponds to $\ba$-stable curves with two components and
markings from $T$, $T^c$, respectively. Since $\de_T=\M'\times \M''$, where $\M'$ and $\M''$ are Hassett spaces parametrizing only irreducible curves, 
we have for all $j\in T$ that
$\sum_{i\in T\setminus\{j\}} a_i\leq 1<\sum_{i\in T} a_i$
(recall that $\M'$ and $\M''$ have an extra marking corresponding to the attaching point).
By Lemma~\ref{P1b}, $\M'=\PP^{m-2}$ and similarly,  $\M''=\PP^{n-m-2}$, with $m=|T|$. Consider now 
$\de_{T\cup\{y\}}\subseteq W$,
%$T^c=N\setminus T$,
the corresponding boundary component in~$W$. We~have 
$\de_{T\cup\{y\}}=\cU'\times \M''$, where $\pi: \cU'\ra\M'$
is the universal family over $\M'$ and $\cU'$ can be identified with the Hassett space with weights $\{a_i\}_{i\in T}\cup\{\epsilon, 1\}$, 
with the weight $\epsilon$, resp., $1$, corresponding to the marking $y$, resp., the attaching point.  It follows from Lemma \ref{P1b}
that $\de_{T\cup\{y\}}=\Bl_p\PP^{m-2}\times\PP^{n-m-2}$
and the restriction map 
$\al_{|\de_{T\cup\{y\}}}$ is as claimed.
\ep

\begin{lemma}\label{star}
Let $\de_T:=\de_{T, T^c}\subseteq \M$ be a boundary divisor of type I. 
Using the identification $\de_T=\PP^{m-2}\times\PP^{n-m-2}$, we have in $\Pic(\de_T)$:
\bi
\item[(i) ]
$
{\de_{ij}}_{|\de_T}=
\begin{cases}
\cO(1,0) & \hbox{\rm if}\quad i,j\in T\cr
\cO(0,1) & \hbox{\rm if}\quad i,j\in T^c\cr
\cO  & \hbox{\rm otherwise}\quad \cr
\end{cases}
$
\item[(ii) ] 
$
{\psi_j}_{|\de_T}=
\begin{cases}
\cO(-1,0) & \hbox{\rm if}\quad j\in T\cr
\cO(0,-1) & \hbox{\rm if}\quad j\in T^c\cr
\end{cases}
$
\item[(iii) ]
${\de_T}_{|\de_T}=\cO(-1,-1)$. 
\ei

\bp
Clearly, if $i\in T$, $j\in T^c$, $\de_{ij}$ is disjoint from $\de_T$, while 
if $i,j\in T$ then ${\de_{ij}}_{|\de_T}\cong\de_{ij}\boxtimes\cO=\cO(1,0)$ (see Note \ref{note 1}(3)).  This proves (i). 
To prove (ii), let $j\in T$. Recall that the universal family over $\de_T$ is the union $\cU'\cup\cU''$, where $\al':\cU'\ra\M'$, $\al'':\cU''\ra\M''$ are the corresponding universal families. 
As $\cU'$, $\cU''$ are attached along the section corresponding to the attaching point (disjoint from the section $\si_j$),  we have that  
${\psi_j}_{|\de_T}={\si_j^*\om_\al}_{l\de_T}=\si_j^*\om_{\al'}\boxtimes \cO=\psi_j\boxtimes\cO$, which is $\cO(-1,0)$ by Note \ref{note 1}(3)). This proves (ii). Part (iii) is a particular case of Lemma \ref{BASIC} (using again Note \ref{note 1}(3))
\ep
\end{lemma}

\begin{lemma}\label{restrict}\label{restoboundaryofM}
In the set-up and notations of Lemma \ref{SRGASRGASRGA}, we use the identification $\de_{T\cup\{y\}}\cong\Bl_p\PP^{m-1}\times\PP^{n-m-2}$, and denote by $H$ the hyperplane class and by $\De$ the exceptional divisor on $\Bl_p\PP^{m-1}$. We have in $\Pic(\de_{T\cup\{y\}})$: 
\bi
\item[(i) ]
$(\om_{\al})_{|\de_{T\cup\{y\}}}=\cO(-H)\boxtimes\cO$.
\item[(ii) ] $\cO(\sigma_j)_{|\de_{T\cup\{y\}}}=\cO$ if $j\notin T$ and $\cO(H)\boxtimes\cO$ otherwise.
\item[(iii) ]
$
(\de_{T'\cup\{y\}})_{|\de_{T\cup\{y\}}}=
\begin{cases}
\cO(\De)\boxtimes\cO & \hbox{\rm if}\quad T'=T^c\cr
\cO(-H)\boxtimes\cO(-1) & \hbox{\rm if}\quad T'=T\cr
\cO  & \hbox{\rm if}\quad T'\neq T, T^c.\cr
\end{cases}
$
\ei
Also,
$(\pi\times Id)^*\cO(-a,-b)=\pi^*\cO(-a)\boxtimes\cO(-b)=\cO(-aH+a\De)\boxtimes \cO(-b)$.
\end{lemma}

\bp
Denote for simplicity $\de=\de_{T\cup\{y\}}\subseteq W$, $\overline{\de}=\de_T\subseteq\M$. 
Since $\al$ is a family of nodal curves, 
we have $\om_{\al}\cong K_{W}-\al^*K_{\M}$.
We have by Lemmas \ref{BASIC}, \ref{P1b}(i) and adjunction that
${K_{W}}_{|\de}=K_{\de}-\de_{|\de}=\cO((-m+1)H+(m-2)\De)\boxtimes\cO(-n+m+2)$. Similarly, we have 
${K_{\M}}_{|\overline{\de}}=K_{\overline{\de}}-{\overline{\de}}_{|\overline{\de}}=\cO(-m+2,-n+m+2)$. It follows from 
Lemma \ref{P1b}(ii) that ${\om_{\al}}_{|\de}=\cO(-H)\boxtimes\cO$. This proves (i). 
Part (ii) follows from $\si_j=\de_{jy}$ and Lemma \ref{P1b}(iii). 

To prove (iii), note that two boundary divisors $\de_{T\cup\{y\}}$, $\de_{T'\cup\{y\}}$ intersect if and only if $T'=T$ or $T^c$. Furthermore, ${\de_{T^c\cup\{y\}}}_{|\de}$ has class $\De\boxtimes\cO$ by Lemma \ref{P1b}(i). By Lemmas \ref{BASIC}, \ref{P1b}(i) and Note \ref{note 1}(3), we have $\de_{|\de}=(-\psi_u)\boxtimes(-\psi_u)=
\cO(-H)\boxtimes\cO(-1)$, where $u$ corresponds to the attaching point. The last statement follows from Lemma \ref{P1b}(ii).  
\ep

\begin{lemma}\label{Push} 
Let $m\geq2$. Consider the $\PP^1$-bundle $\pi:\Bl_p\PP^{m-1}\ra\PP^{m-2}$. Then 
$$R\pi_*\big(\cO(a\De)\big)\cong\cO(-a)\oplus\ldots\oplus\cO(-1)\oplus\cO\ \ \text{ if }\ \ a\geq0$$
while $R\pi_*\big(\cO(-\De)\big)=0$
%when $a<0$ we have
and, when $a\geq2$,  $R\pi_*\big(\cO(-a\De)\big)$ is generated by $\cO(1), \ldots, \cO(a-1)$. 
More generally, $R\pi_*\big(\cO(aH-b\De))$ is either $0$ if $a=b-1$, or it is 
generated by:
(i)  $\cO(b),\cO(b+1),\ldots,\cO(a)$ if $a\geq b$;
(ii) $\cO(a+1),\cO(a+2),\ldots,\cO(b-1)$ if $a\leq b-2$.
In particular, we have that $R\pi_*\big(\cO(aH-b\De))$ is generated by $\cO(u)$ with 
$\min\{b,a+1\}\leq u\leq \max\{a, b-1\}$.
\end{lemma}

\bp
The lemma follows by applying $R\pi_*(-)$ to the exact sequences 
$$0\ra \cO((i-1)\De)\ra \cO(i\De)\ra \cO_{\De}(-i)\ra 0,$$
induction on $a$, projection formula and the fact $\pi^*\cO(1)=\cO(H-\De)$. 
\ep

\begin{lemma}\label{acy} Let $m\geq2$. 
Then $\cO_{\Bl_p\PP^{m-1}}(aH-b\De)$ is acyclic if 
$0<-a\leq m-1$, $0\leq -b\leq m-2$. 
\end{lemma}

\bp Indeed,
$R\Gamma(R\pi_*\big(\cO(aH-b\De)))=0$ by Lemma~\ref{Push}. \ep

\begin{lemma}\label{vb}
Let $\M$ be a Hassett space such that all $\ba$-stable curves have at most two components. 
%and such that no partial sum $\sum_{i\in I} a_i$ with $|I|\geq2$ equals $1$.
Let $a$, $\{\al_T\}$ be  integers.  
Consider the line bundle $L:=\om_{\al}^{a}(E)\otimes\cO\bigl(-\sum_T\al_{T}\de_{T\cup \{y\},T^c}\bigr)$ on $W$ (where the sum is over
partitions $N=T\sqcup T^c$ giving rise to boundary divisors $\de_{T,T^c}$ of type I on $\M$). 
The complex
$R\al_*(L)$
is a vector bundle of rank $e-2a+1$ (and $R\al_*L=\al_*L$) if 
\begin{equation}\label{vb condition}
|E\cap T^c|-a\geq \al_T-\al_{T^c}\geq a-|E\cap T|
\end{equation}
Furthermore, we can weaken \eqref{vb condition} by either adding $1$ to $|E\cap T^c|-a$
 or subtracting $1$ from $a-|E\cap T|$ (but not both at once).

In particular, the complex $F_{l,E}$ in Definition \ref{allF} is a vector bundle of rank $l+1$. 
\end{lemma}

\bp
%For simplicity, we denote $L:=\om_{\al}^{a}(E)\otimes\cO(-\sum_T\al_{T}\de_{T\cup \{x\}})$. 
Note that (\ref{vb condition}) implies that $e\geq 2a-1$.
If $C$ is an irreducible fiber of $\al$ then 
$\deg\big(L_{|C}\big)=\deg \big(\om_{\al}^{a}(E)\big)_{|C}=e-2a\geq-1$.
Let now $C$ be a reducible fiber of $\al$ over
a point of $\de_T\subseteq \M$, for a partition $T\sqcup T^c$ of $N$. The curve $C$ has components $C_1$
and $C_2$, with $C_1$ (resp., $C_2$) having fixed markings from $T$ (resp., $T^c$).  
The curve $C_i\subseteq W$ is given by the point $y$ moving along $C_i$. 
Therefore, we have
$C_1\subseteq \de_{T\cup\{y\}}=\de_{T^c}$, $C_2\subseteq \de_{T^c\cup\{y\}}=\de_T$,
$ C_1\cdot\de_{T}=1$, $C_2\cdot\de_{T^c}=1$
and $C_i$ intersects all boundary other than $\de_T$, $\de_{T^c}$ trivially. 
Furthermore, one has $\de_{T^c}\cdot C_1=-1$ from 
$(\de_T+\de_{T^c})\cdot C_1=\al^{-1}(\de_T)\cdot C_1=0$.
Similarly  $\de_{T}\cdot C_2=-1$. 
Furthermore, $\big(\om_{\al}\big)_{|C_i}=\om_{C_i}(1)=\cO(-1)$ and $\deg(\cO(\si_j))_{|C_1}=1$ if $j\in T$ and $0$ otherwise. 
It follows that 
$$\deg(L_{|C_1})=-a+|E\cap T|+\al_{T}-\al_{T^c},\ \deg(L_{|C_2})=-a+|E\cap T^c|+\al_{T^c}-\al_{T}.$$

There is an exact sequence 
$0\ra L_{|C}\ra L_{|C_1}\oplus L_{|C_2}\ra \cO_{k(u)}\ra 0$.\
We have $\HH^1(L_{|C_i})=0$ for $i=1,2$ since  $\deg(L_{|C_i})\geq-1$ for $i=1,2$.
Since at least one of the inequalities is strict, the induced map 
$\HH^0(L_{|C_1}\oplus L_{|C_2})\ra k(u)$ is surjective. Therefore, 
$\HH^1(L)=0$, $h^0(L)=e-2a+1$
and $R\al_*L=\al_*L$ is a vector bundle of rank $e-2a+1$. 

%\smallskip

Assume now that $a=\frac{e-l}{2}$ and $\al_T$ is as in \eqref{Sdgsgasrh}.
If  $f_T, f_{T^c}\ge0$ (see \eqref{Sdgsgasrh} for the definition of $f_T$)
then $\al_T=\al_{T^c}=0$ and 
$\deg(L_{|C_1})=f_T\geq0$, $\deg(L_{|C_2})=f_{T^c}\geq0$. 
If ~$f_T\ge0$,  $f_{T^c}<0$
then 
$\al_T=0$, $\al_{T^c}=-f_{T^c}$,
$\deg(L_{|C_1})=l\geq-1$, $\deg(L_{|C_2})=0$.
The case  $f_{T^c}\geq0$,  $f_{T}<0$ is similar. 
Since $f_T+f_{T^c}=l\geq-1$, it is not possible to have $f_T, f_{T^c}<0$. 
\ep

%\begin{rmk}
%As in Definition \ref{allF}, for $l=-1$ and $e=|E|$ odd, one can define 
%$$F_{-1,E}=R\pi_*N_{-1,E},$$ where $N_{-1,E}$ is defined in exactly the same way. The proof of Lemma \ref{vb} applies verbatim to prove that 
%$$R^1\pi_*(N_{-1,E})=0,\quad F_{-1,E}=\pi_*(N_{-1,E})=0,$$ 
%(i.e., Lemma \ref{vb} still holds). Note that, when $e-l$ is even, even for $l=-1$, 
%we still cannot have that $|E\cap T^c|<\frac{e-l}{2}$,  $|E\cap T|<\frac{e-l}{2}$. 
%\end{rmk}

\begin{lemma}\label{property}
Assume that $n$ is even. 
Let $\M$ be a Hassett space such that all $\ba$-stable curves have at most two components and any boundary component 
$\de_{T,T^c}$ of type I is such that $|T|=|T^c|=\frac{n}{2}$ (for example the spaces $\M_{p,q}$ when $p$ and $q$ both even).  
Let $N=\{1,\ldots,n\}$ be the set of all indices. When $l\geq0$ is even, the vector bundles $F_{l,N}$ satisfy 
the property $F_{l,N}\cong F_{l,\emptyset}\otimes F_{0,N}$.
\end{lemma}

\bp
By Definition \ref{allF}, $F_{0,N}$ is the line bundle $\al_*(N_{0,N})$, where $N_{0,N}=\om_{\al}^{n/2}(N)$ is a line bundle on the universal family 
$\al:W\ra\M$,
with degree $0$ on the generic fiber. 
Furthermore, the assumption $|T|=\frac{n}{2}$ implies that $N_{0,N}$ has trivial restriction to each component of every 
reducible fiber of $\al$.
Hence,
$N_{0,N}=%\al^*(\al_*(\N_{0,N})=
\al^*(F_{0,N})$.
By definition 
$F_{l,N}=\al_*\big(\om_{\al}^{(n-l)/2}(N)\big)$,  $F_{l,\emptyset}=\al_*\big(\om_{\al}^{-\frac{l}{2}}\big)$ and 
thus 
$F_{l,N}=\al_*\big(\om_{\al}^{-\frac{l}{2}}\otimes \al^*F_{0,N}\big)\cong R\al_*\big(\om_{\al}^{-\frac{l}{2}}\big)\otimes F_{0,N}\cong 
F_{l,\emptyset}\otimes F_{0,N}$.
\ep

%%%%%%%%%%%%%%%%%%%%%%%%%%%%%%%%%%%%
%%%%%%%%%%%%%%%%%%%%%%%%%%%%%%%%%%%%
%%%%%%%%%%%%%%%%%%%%%%%%%%%%%%%%%%%%
%%%%%%%%%%%%%%%%%%%%%%%%%%%%%%%%%%%%

\begin{notn}\label{general set-up}
%Consider the Hassett space $\M_{p,q}$ with 
Let $N=\{1,\ldots n\}$, $n=2m\geq 4$ even. Assume $\M$ is a Hassett space such that all $\ba$-stable curves have at most two components and any boundary component 
$\de_{T,T^c}\subset \M$ of type I is such that $|T|=m$. Let $\al:W\ra \M$ be the universal family.  
%sets $P,Q$ with $|P|=p=2r\geq4$, $|Q|=q=2s+1\geq1$. Let 
%$\M_{p,q-1}$ be the Hassett space with markings  $P\sqcup (Q\setminus\{y\})$ for fixed
%$y\in Q$ and the universal family
%$\al: W\ra \M_{p,q-1}$.  
Recall from Lemma \ref{SRGASRGASRGA} that $W$ is a Hassett space with boundary of type I of the form 
$\de_{T\cup\{y\},T^c}\cong\Bl_p\PP^{m-1}\times\PP^{m-2}$ for all partitions 
$T\sqcup T^c=N$. We let $\pi: \cU\ra W$ be the universal family, with $x$ the new marking on $\cU$. 
\end{notn}

\begin{lemma}\label{restrictions_W}
There are two types of boundary divisors of type I  on $\cU$:

(1) Let $\de=\de_{T\cup\{x\}, T^c\cup\{y\}}=\M_{T\cup\{x,u\}}\times \M_{T\cup\{y,u\}}$ (with $u$ the attaching point). 
Then $\de\cong\Bl_p\PP^{m-1}\times \Bl_p\PP^{m-1}$. 
If on $\Bl_p\PP^{m-1}$ we denote by $H$ the hyperplane class and by $\De$ the exceptional divisor, then 
the restriction $\pi_{|\de}$ 
is  the pair $(q,\Id)$, where 
$q:\,\Bl_p\PP^{m-1}\ra  \PP^{m-2}$, $q^*\cO(1)=\cO(H-\De)$, $\De=\de_{ux}$. 
Moreover, 
${\om_{\pi}}_{|\de}=\cO(-H)\boxtimes\cO$, ${\de_{yx}}_{|\de}=0$,
${\de_{jx}}_{|\de}$ is $\cO(H)\boxtimes\cO$ if $j\in T$ and $\cO$ otherwise.

(2) Let $\de=\de_{T\cup\{y,x\}, T^c}=\M_{T\cup\{y,x,u\}}\times \M_{T\cup\{u\}}$. 
%=\M_{T\cup\{y,x,u\}}\times \M_{T^c\cup\{u\}}=
Then $\de\cong\Bl_{1,2, \ov{12}}\PP^{m}\times \PP^{m-2}
$,
where $\Bl_{1,2, \ov{12}}\PP^{m}$ denotes the blow-up of $\PP^{m}$ at two distinct points $p_1$, $p_2$ and the 
proper transform of the line through them. 
On $\Bl_{1,2, \ov{12}}\PP^{m}$ we denote  by $H$ the hyperplane class and by $E_1$, $E_2$, $E_{12}$ the corresponding exceptional divisors. 
We denote pullbacks of these divisors to $\delta$ by the same letters.
We have:
$E_1=\de_{T\cup\{x\}|_\delta}$, $E_2=\de_{T\cup\{y\}|_\delta}$,
$E_{12}=\de_{T|_\delta}$.
The restriction $\pi_{|\de}$ is  the pair $(\tilde{q}, Id)$, where 
$\tilde{q}: \Bl_{1,2, \ov{12}}\PP^{m}\ra \Bl_{p_1}\PP^{m-1}$, 
$\tilde{q}^*\cO(H)=\cO(H-E_2)$, $\tilde{q}^*\De=E_1+E_{12}$,
and ${\om_{\pi}}_{|\de}=\cO(-H+E_1)\boxtimes\cO$,
${\de_{jx}}_{|\de}=\cO(H-E_1)\boxtimes\cO$ for $j\in T$,
${\de_{jy}}_{|\de}=\cO(H-E_2)\boxtimes\cO$ for  $ j\in T$,
${\de_{yx}}_{|\de}=\cO(H)\boxtimes\cO$.
 \end{lemma}

\bp
For (1), the statements about $q$ and ${\de_{jx}}_{|\de}$ follow from Lemma \ref{P1b}. 
By adjunction and Lemma \ref{BASIC},   
${K_{W}}_{|\pi(\de)}=\cO(-m+2)\boxtimes\left(-(m-1)H+(m-2)\De\right)$,
${K_{\cU}}_{|\de}=\left(-(m-1)H+(m-2)\De\right)\boxtimes \left(-(m-1)H+(m-2)\De\right)$.
Using  $\om_{\pi}=K_{\cU}-\pi^*K_W$, we have that ${\om_{\pi}}_{|\de}=\cO(-H)\boxtimes\cO$. 

We now prove (2). 
Denote
$\M:=\M_{T\cup\{y,x,u\}}$, $\M_W:=\M_{T\cup\{y,u\}}$ 
and let $\N'$ (resp., $\N''$) be the Hassett space with
markings $T\cup\{y,x,u\}$ such that the points in $T$ 
(resp., $T$, $T\cup\{x\}$, $T\cup\{y\}$,
but not $T\cup\{x,y\}$) are allowed to coincide. This is possible since $\M_{T\cup\{u\}}\cong\PP^{m-2}$, we may assume that the weights of the points in $T$ are $\epsilon=\frac{1}{m-1}$ and those of $y, x$ are $\eta\ll 1$.
Then for $\N'$, resp. $\N''$, we can lower the weights of the points in $T$ to $\epsilon'=\frac{2-\eta}{2m}$, resp., $\epsilon''=\frac{2-3\eta}{2m}$.
Note that $\N''\cong\PP^{m}$, since 
all but one of the markings $T\cup\{y,x\}$ may coincide. 

There are reduction maps
$\pi':\M\ra\N'$, $\pi'':\N'\ra\N''.$
The map  $\pi''$ contracts $\de_{T\cup\{x\}}$ and  $\de_{T\cup\{y\}}$ to points, which we denote 
$p_1:=\pi''(\de_{T\cup\{x\}})$, $p_2:=\pi''(\de_{T\cup\{y\}})$,
while the composition $\pi''\circ\pi'$ contracts $\de_T$ to the line through $p_1$ and $p_2$. Hence,
$\M$ is isomorphic to the blow-up $\Bl_{1,2, \ov{12}}\PP^{m}$ and we have
$E_1=\de_{T\cup\{x\}}$, $E_2=\de_{T\cup\{y\}}$, $E_{12}=\de_T$.

By Lemma \ref{SRGASRGASRGA}, we have $\M_W\cong\Bl_{p}\PP^{m-1}$.
The map $\tilde{q}:\M\ra\M_W$ forgets the $x$ marking; hence, 
$\tilde{q}$ is an isomorphism on $E_2=\de_{T\cup\{y\}}$. Since the only fibrations $\Bl_{1,2, \ov{12}}\PP^{m}\ra\Bl_{p}\PP^{m-1}$
are given by the linear systems $|H-E_1|$, $|H-E_2|$, %and since the fibration given by $|H-E_1|$ is not an isomorphism on $E_2$, 
it follows that $\tilde{q}^*\cO(H)=\cO(H-E_2)$ (the projection from $p_2$). As $\De=\de_T\subseteq \M_W$, it follows that 
$\tilde{q}^*\De=\de_T+\de_{T\cup\{x\}}=E_{12}+E_1.$
The restriction ${\de_{jx}}_{|\de}$ is $\de_{jx}\boxtimes\cO$ if $j\in T$ and trivial otherwise. 
If $j\in T$, the pull-backs of $\de_{jx}$ via the reduction maps $\pi'$, $\pi''$ are given by 
${\pi''}^*\de_{jx}=\de_{jx}+\de_{T\cup\{x\}}$, $(\pi'\circ\pi'')^*\de_{jx}=\de_{jx}+\de_{T\cup\{x\}}$.
As $\de_{jx}=\cO(1)$ on $\N''=\PP^{m}$ by Note \ref{note 1}(3), it follows that we have
$\de_{jx}=\cO(H-E_1)$ in $\Pic(\M)$. By symmetry, when $j\in T$ we also have $\de_{jy}=\cO(H-E_2)$ in $\Pic(\M)$.
Similarly, since 
$ (\pi'\circ\pi'')^*\de_{yx}=\de_{yx}$ and $\de_{yx}=\cO(1)$ on $\N''=\PP^{m}$, we have
${\de_{yx}}_{|\de}=\cO(H)\boxtimes\cO$. 
By adjunction, 
${K_{W}}_{|\pi(\de)}=\left(-(m-1)H+(m-2)\De\right)\boxtimes \cO(-(m-2))$,
${K_{\cU}}_{|\de}=\left(-mH+(m-1)(E_1+E_2)+(m-2)E_{12}\right)\boxtimes\cO(-(m-2))$.
Using $\om_{\pi}=K_{\cU}-\pi^*K_W$, it follows that 
${\om_{\pi}}_{|\de}=\cO(-H+E_1)\boxtimes\cO$. 
\ep

\begin{rmk}\label{PushB} 
In the notations of Lemma \ref{restrictions_W},
the same proof as in Lemma \ref{Push} shows that 
if $a\geq0$ we have
$$R\tilde{q}_*\big(\cO(aE_2)\big)\cong\cO(-aH)\oplus\ldots\oplus\cO(-H)\oplus\cO.$$
\end{rmk}

\begin{lemma}\label{vb on W}
The complex $F_{l,E}$ from Definition \ref{F on W} is a vector bundle of rank $l+1$. 
\end{lemma}

\bp
Denote $L:=N_{l,E}$ for simplicity. We proceed as in the proof of Lemma \ref{vb}.  If $C$ is an irreducible fiber of $\pi$, then $\deg (L_{|C})=l$. Consider a reducible fiber $C$ 
with two components $C_1$ and $C_2$, with 
markings from $T\cup\{y\}$ (resp., $T^c$) on $C_1$ (resp. $C_2$). 
Recall that $E\subseteq N\cup\{y\}$ and $T^c=N\setminus T$. 
Since $f_T+f_{T^c}+|E\cap\{y\}|=l\geq0$, it is not possible to have $f_T, f_{T^c}<0$. 
We have
$C_1\cdot \de_{T\cup\{y,x\}}=-1$, $C_1\cdot \de_{T\cup\{y\}}=1$,
$C_2\cdot \de_{T\cup\{y,x\}}=1$, $C_2\cdot \de_{T\cup\{y\}}=-1$,
while all other intersections with boundary are $0$. We have: $$\deg\big(L_{|C_1}\big)=|E\cap\{y\}|+f_T+\al_{T\cup\{y\}}-\al_{T^c},$$ 
$$\deg\big(L_{|C_2}\big)=
f_{T^c}-\al_{T\cup\{y\}}+\al_{T^c}.$$

If $f_T, f_{T^c}\geq0$,
then $\al_{T\cup\{y\}}=\al_{T^c}=0$ and 
$\deg\big(L_{|C_1}\big)\geq0$,
$\deg\big(L_{|C_2}\big)\geq0$.
If  $f_{T}\geq0$,  $f_{T^c}<0$, then 
$\al_{T\cup\{y\}}=0$, $\al_{T^c}=-f_{T^c}$,
$\deg\big(L_{|C_1}\big)=l\geq0$,
$\deg\big(L_{|C_2}\big)=0$.
If  $f_T<0$,  $f_{T^c}\geq0$, then 
$\al_{T\cup\{y\}}=-f_T-|E\cap\{y\}|$,
$\al_{T^c}=0$, $\deg\big(L_{|C_1}\big)=0$, $\deg\big(L_{|C_2}\big)=l$.
In all cases, $\h^1(L_{|C})=0$, $\h^0(L_{|C})=l+1$.  

%\smallskip

Consider now a reducible fiber $C$ 
with $3$ components $C_1$, $C_2$, $C_3$, with 
markings from $T$, $\{y\}$, $T^c$ respectively. Then
$C_1\cdot \de_{T\cup\{x\}}=-1$, $C_1\cdot \de_{T^c\cup\{y,x\}}=1$,
$C_2\cdot \de_{T\cup\{y,x\}}=C_2\cdot \de_{T^c\cup\{y,x\}}=-1$, $C_2\cdot\de_{T\cup\{x\}}=C_2\cdot\de_{T^c\cup\{x\}}=1$,
$C_3\cdot \de_{T^c\cup\{x\}}=-1$, $C_3\cdot \de_{T\cup\{y,x\}}=1$,
while intersections with other boundary are $0$. We have: 
$\deg\big(L_{|C_1}\big)=f_T+\al_{T}-\al_{T^c\cup\{y\}}$,
$\deg\big(L_{|C_2}\big)=\al_{T\cup\{y\}}+\al_{T^c\cup\{y\}}-\al_T-\al_{T^c}+|E\cap\{y\}|$,
$\deg\big(L_{|C_3}\big)=f_{T^c}+\al_{T^c}-\al_{T\cup\{y\}}$.

If $f_T, f_{T^c}\geq0$,
then 
$\al_T=\al_{T\cup\{y\}}=\al_{T^c}=\al_{T^c\cup\{y\}}=0$, $\deg\big(L_{|C_1}\big)\geq0$, $\deg\big(L_{|C_3}\big)\geq0$,
$\deg\big(L_{|C_2}\big)\geq0$.
If $f_{T}\geq0$,  $f_{T^c}<0$, then 
$\al_T=\al_{T\cup\{y\}}=0$, $\al_{T^c}=-f_{T^c}$,
$\al_{T^c\cup\{y\}}=\al_{T^c}-|E\cap\{y\}|$,
$\deg\big(L_{|C_1}\big)=l\geq0$, $\deg\big(L_{|C_3}\big)=0$,
$\deg\big(L_{|C_2}\big)=0$.
If  $f_T<0$,  $f_{T^c}\geq0$ then 
$\al_{T^c}=\al_{T^c\cup\{y\}}=0$, $\al_{T}=-f_T$,
$\al_{T\cup\{y\}}=\al_T-|E\cap\{y\}|$,
$\deg\big(L_{|C_1}\big)=0$, 
$\deg\big(L_{|C_3}\big)=l$,
$\deg\big(L_{|C_2}\big)=0$.
%Note that if $e+l$ is even and $l\geq0$, we cannot have $f_T, f_{T^c}<0$, as otherwise we would have
%$e-|E\cap\{y\}|=|E\cap T|+|E\cap T^c|\leq (\frac{e-l}{2}-1)+(\frac{e-l}{2}-1)=e-l-2$,
%which is a contradiction. 
Hence, in all cases $\h^1(L_{|C})=0$, $\h^0(L_{|C})=l+1$ and we are done by cohomology and base change \ep

\begin{lemma}\label{zzfbzdfbdff}
% Let $\al: W\ra \M_{p,q-1}$ be the universal family, with $y$ the new marking on $W$. 
Consider the set-up of Notation \ref{general set-up}. Let $l\geq0$, $E\subseteq N\cup\{y\}$, with $e+l$ even. 
(i) If $y\notin E$, then $F_{l,E}=\al^*F_{l,E}$ on $W$. 
(ii) If $y\in E$, then %. For all $l\geq 0$, $e+l$ even, 
there is an exact sequence 
\begin{equation}\label{egasrhar}
0\ra F_{l-1,E\setminus\{y\}}\ra F_{l,E}\ra Q^y_{l,E}\ra 0,
\end{equation}
of vector bundles on $W$, with $Q^y_{l,E}:=\si_y^*N_{l,E}$, where $\si_y$ is the section of $\pi: \cU\ra W$ corresponding to the $y$ marking.
Furthermore, 
$R\al_*(F^\vee_{l,E})=0$ in this case.
\end{lemma}

\bp[Proof of Lemma~\ref{zzfbzdfbdff}]
There is a commutative diagram 
\begin{equation*}
\begin{CD}
\cU      @>v>>  \cV@>\phi>> W\\
@V{\pi}VV        @V{\rho}VV @V{\al=\al_x}VV \\
W   @>Id>>  W  @>\al=\al_y>>  \M
\end{CD}
\end{equation*}
where  $\al: W\ra \M$ is the universal family and the notation $\al=\al_x$ indicates the marking that is forgotten. The right square is Cartesian.
Let $g=\phi\circ v$.
The map $v$ is small and contracts the codimension $2$ loci  
$$\pi^{-1}(\de_T\cap\de_{T^c})=\PP^{m-2}\times\PP^1\times\PP^{m-2}\ra  \PP^{m-2}\times pt \times\PP^{m-2}.$$

\begin{claim}\label{various}
We have 
(i) $Rv_*v^*\cO_{\cV}\cong Rv_*\cO_{\cU}\cong\cO_{\cV}$. 
(ii) In $\Pic(\cU)$ we have $v^*\om_{\rho}=\om_{\pi}$, while on $W$ we have $\psi_y=\om_{\al_y}$. 
(iii)  $g^*\de_{T\cup\{x\}}=\de_{T\cup\{x\}}+\de_{T\cup\{y,x\}}$.
\end{claim}

\bp
The map $v$ is birational and its image has rational singularities, which are in fact
locally isomorphic to the product of the affine cone $xy=zt$ and a smooth variety (see e.g.~\cite[page 548]{Keel}).
Part (i) follows.
Part (iii) is immediate since $g$ is the map that forgets the marking $y$. 

We prove (ii). Recall that $W$, $\cU$ are smooth, see Note 
\ref{note 2}(1). For a proper flat family $f: \cC\ra B$ of at worst nodal curves over a Gorenstein base $B$ (smooth for us) the relative dualizing sheaf $\omega_f$ is a line bundle on $\cC$ with first Chern class $K_\cC-f^*K_B$, where $K_\cC$ and $K_B$ denote the corresponding canonical divisors. We apply this to the families $\pi:\cU\ra W$ and $\rho:\cV\ra W$. 
Since the map $v$ is small, we have $K_\cU=v^*K_\cV$. Hence, $v^*\om_{\rho}=\om_{\pi}$.

By definition $\psi_y=\si_y^*\om_{\pi}$. 
Since $v^*\om_{\rho}=\om_{\pi}$, it follows that if $s_y=v\circ\si_y$, then 
$\psi_y=s_y^*\om_{\rho}=s_y^*\phi^*\om_{\al}=\om_{\al}$, since $\phi\circ s_y=Id_W$. 
%(where we identify $y$ with $x$). 
\ep

We have $F_{l,E}=R\pi_*L_1$, where $L_1:=N_{l,E}$, which is equal to 
$$\om_{\pi}^{\frac{e-l}{2}}\Bigl(E-\sum_{f_{T,E,l}<0}(-f_{T,E,l})\de_{T\cup\{x\}}+(-f_{T,E,l}-|E\cap\{y\}|)\de_{T\cup\{y,x\}}\Bigr).$$   
where the sum is over all $T\subset N$ such that $\de_{T,T^c}\subseteq \M$ is a boundary component ($T^c:=N\setminus T$). 
We now compute $\al^*F_{l',E'}$, for $y\notin E'$. Using (i), we have 
$\al^*F_{l',E'}=\al_y^*R{\al_x}_*N_{l',E'}=R{\rho}_*\phi^*N_{l',E'}=R{\pi}_*L_2$,
where $L_2:=v^*\phi^*N_{l',E'}=g^*N_{l',E'}$,
with $N_{l',E'}$ is the usual line bundle on $W$:
$$N_{l',E'}=\om_{\pi}^\frac{e'-l'}{2}\Bigl(E'-\sum_{f_{T,E',l'}<0}\left(-f_{T,E',l'})\right)\de_{T\cup\{x\}}\Bigr).$$
Since $v$ has no exceptional divisors, it follows that 
$$L_2=\om_{\pi}^\frac{e'-l'}{2}\Bigl(E'-\sum_{f_{T,E',l'}<0}(-f_{T,E',l'})\big(\de_{T\cup\{x\}}+\de_{T\cup\{y,x\}}\big)\Bigr).$$

\underline{Case $y\notin E, l'=l, E'=E$.} We clearly have $L_1=L_2$, and this proves that when $y\notin E$, we have $F_{l,E}=\al^*F_{l,E}$. 

%\smallskip

\underline{Case $y\in E, l'=l-1, E'=E\setminus\{y\}$.} As $\frac{e'-l'}{2}=\frac{e-l}{2}$, $|E\cap\{y\}|=1$, we have 
$L_1=L_2\bigl(\de_{yx}+\sum\limits_{f_{T,E,l}<0}\de_{T\cup\{y,x\}}\bigr)$,  which implies exact sequences
$$0\ra L_2\ra L_1(-\de_{yx})\ra \bigoplus_{f_{T,E,l}<0}\big(L_1(-\de_{yx}\big)_{|\de_{T\cup\{y,x\}}}\ra0$$
and 
$0\ra L_1(-\de_{yx})\ra L_1\ra \big(L_1\big)_{|\de_{yx}}\ra0$.
Assume $T$ is such that $f_{T,E,l}<0$. Then by Lemma \ref{restrictions_W} we have 
$\big(L_1(-\de_{yx})\big)_{|\de_{T\cup\{y,x\}}}=\cO(-H)\boxtimes\cO\left(-f_{T,E,l}-1\right)$,
on $\de_{T\cup\{y,x\}}=\Bl_{1,2,\ov{12}}\PP^{m}\times\PP^{m-2}$. Since on $\Bl_{1,2,\ov{12}}\PP^{m}$ we have 
$R\tilde{q}_*\cO(-E_2)=0$ and $\tilde{q}^*\cO(H)=\cO(H-E_2)$ (Lemma \ref{restrictions_W}), it follows that $R\tilde{q}_*\cO(-H)=0$. Hence, 
$R{\pi}_*\big(L_1(-\de_{yx})\big)_{|\de_{T\cup\{y,x\}}}=0$, $R{\pi}_*L_2\cong R{\pi}_*L_1(-\de_{yx})=F_{l-1,E\setminus\{y\}}$
(and here $R\pi_*(-)={\pi}_*(-)$). Applying $R\pi_*(-)$ to the above two exact sequences, it follows that there is an exact sequence \eqref{egasrhar} with $Q^y_{l,E}=\sigma_y^*N_{l,E}$.
%$$Q^y_{l,E}={\sigma_y}^*N_{l,E}=\left(\frac{e-l}{2}-1\right)\psi_y+\cO(E\setminus\{y\})-\sum_{|E\cap T|<\frac{e-l}{2}}
%\left(\frac{e-l}{2}-|E\cap T|-1\right)\de_{T\cup\{y\}}.$$
%We use here that $\sigma_y=\de_{yx}$, $\sigma_y^*\omega_{\pi}=-\sigma_y^*(\sigma_y)=\psi_y$ and 
%$$\sigma_y^*\de_{T\cup\{x\}}=0,\quad \sigma_y^*\de_{T\cup\{y, x\}}=\de_{T\cup\{y\}}.$$
%It follows from Claim \ref{various}(ii) that $\psi_y=\om_{\al}$ and therefore
%$$Q^y_{l,E}=-\om_{\al}+N'_{l-1,E\setminus\{y\}},$$
%where $N'_{l-1,E\setminus\{y\}}$ is as in Definition \ref{Alt Def}. It follows by Grothendieck duality that 
%$$R{\al}_*\big({Q^y_{l,E}}^\vee\big)\cong F_{l-1,E\setminus\{y\}}^\vee,$$
%(and here $R{\al}_*(-)=R{\al}^1_*(-)$).
% and applying $R{\al}_*(-)$ to the dual sequence
%$$0\ra {Q^y_{l,E}}^\vee \ra F_{l,E}^\vee \ra F^\vee_{l-1,E\setminus\{y\}} \ra 0,$$
%this suggests that we should have $R{\al}_*\big(F_{l,E}^\vee\big)=0$. But for this we would need to know that the induced map 
%$$\al_*F^\vee_{l-1,E\setminus\{y\}}\ra R^1\al_*\big({Q^y_{l,E}}^\vee\big),$$
 %is an isomorphism. We will give an alternate proof. 
Finally, we have 
$$R\al_*\big(F_{l,E}^\vee\big)=R{\al_y}_*\circ R\pi_*\big(\om_{\pi}\otimes L_1^\vee[1]\big)=R{\al_x}_*\circ Rg_*\big(\om_{\pi}\otimes L_1^\vee[1]\big).$$
Hence, it suffices to prove that $Rg_*\big(\om_{\pi}\otimes L_1^\vee\big)=0$. Since $\om_{\pi}=g^*\om_{\al}$, $L_2$ is a pull-back by $g$,
 and $L_1=L_2(\de_{yx}+D)$, where
 $D=\sum\limits_{f_{T,E,l}<0}\de_{T\cup\{y,x\}}$,
 it suffices to prove that $Rg_*\cO(-\de_{yx}-D)=0$. Consider the exact sequence
 $$0\ra \cO(-\de_{yx}-D)\ra \cO(-D)\ra  \cO(-D)_{|\de_{yx}}\ra0.$$
 It suffices to prove that $g_*(-)$ induces an isomorphism when applied to the restriction map 
$\cO(-D)\ra  \cO(-D)_{|\de_{yx}}$ and all higher push forwards by $g$ of $\cO(-D)$ and $\cO(-D)_{|\de_{yx}}$  are $0$. 
Since $\de_{yx}=\si_y$ and $g\circ\si_y=\Id$, 
we have $R^ig_*\cO(-D)_{|\de_{yx}}=0$ for all $i>0$ and $g_*\cO(-D)_{|\de_{yx}}=\cO_W(-D')$, where 
 $D'=g(D)=\sum_{f_{T,E,l}<0}\de_{T\cup\{x\}}$. 
Because of the condition $f_{T,E,l}<0$, 
the divisors in $D$ are disjoint. Similarly, the divisors in $D'$ are disjoint.  
Using the exact sequence 
$0\ra\cO(-\de_{T\cup\{y,x\}})\ra\cO_{\cU}\ra\cO_{\de_{T\cup\{y,x\}}}\ra0$,
it suffices to prove that $Rg_*\cO_{\cU}=\cO_W$ and 
$Rg_*\cO_{\de_{T\cup\{y,x\}}}\cong\cO_{\de_{T\cup\{x\}}}$. Since $g=\phi\circ v$, the first statement follows
from $Rv_*\cO_{\cU}\cong \cO_{\cV}$ (part (i) of Claim \ref{various}) and $R\phi_*\cO_{\cV}\cong \cO_W$, as $\phi$ is the pull-back of 
$\al_y: W\ra\M$, the universal family over $\M$. The second statement follows as
the map $g$ restricted to $\de_{T\cup\{y,x\}}$ is the map $\pi_{|\de}=(\tilde{q}, \Id)$ from Lemma \ref{restrictions_W}, where 
$\tilde{q}: \M_{T\cup\{y,x,u\}}=\Bl_{1,2,\ov{12}}\ra\M_{T\cup\{y,u\}}=\Bl_p\PP^{m-1}$ is the universal family,  in this case a flat family of rational, at worst nodal curves over a smooth base.  \ep

%%%%%%%%%%%%%%%%%%%%%%%%%%%%%%%%%%%%%%%%%%%%%%%%%%%%%%%%
%%%%%%%%%%%%%%%%%%%%%%%%%%%%%%%%%%%%%%%%%%%%%%%%%%%%%%%%
%%%%%%%%%%%%%%%%%%%%%%%%%%%%%%%%%%%%%%%%%%%%%%%%%%%%%%%%
%%%%%%%%%%%%%%%%%%%%%%%%%%%%%%%%%%%%%%%%%%%%%%%%%%%%%%%%

\section{Perpendicularity of $F_{l,E}$'s versus the sheaves $\cO_{\de}(-a,-b)$}\label{F vs boundary section}

%%%%%%%%%%%%%%%%%%%%%%%%%%%%%%%%

%\noindent{\bf Hassett spaces $\M$ corresponding to curves with at most two components}

Let $p$ even, $q$ odd, $q+1\geq0$. In Section \ref{extend section}, we defined vector bundles $F_{l,E}$ on the Hassett spaces $\M_{p,q+1}$ and on their universal families $W$, generalizing the definition of the vector bundles $F_{l,E}$ on the stacks $\cP_n$ (defined in Section \ref{fullness odd p section}). 
In this section we verify that the bundles $F_{l,E}$ on $\M_{p,q+1}$ in Theorems \ref{p,q+1 case}, \ref{even 0}, Remark \ref{1B stuff} are perpendicular to the torsion sheaves in the subcategory $\cA$ (Corollary \ref{F perp A}). We will prove a more general statement in Proposition \ref{perpendicular general}, as this will be needed later for our inductive arguments in Sections \ref{induction2} and \ref{induction1}. 

For the exceptionality  part in Theorems \ref{p,q+1 case}, \ref{even 0}, Remark \ref{1B stuff}, we will compare RHom's between objects on $\M_{p,q+1}$ and
%to RHom's between 
similar objects on $W$. %other Hassett spaces (e.g., $W$). 
We~will need a general statement about perpendicularity of the bundles $F_{l,E}$ on $W$ to certain torsion sheaves (Proposition \ref{perpendicular W}).

\begin{prop}\label{perpendicular general}
Assume $P$ (resp., $Q$) is the set of heavy (resp., light) indices and 
$|P|=p=2r\geq 4$, $|Q|=q+1=2s+2\geq0$ (so $s$ can be $-1$).
Assume $l\geq 0$, $E\subseteq P\cup Q$ with $e=|E|$ such that $e+l$ is even.
Let $\de_T=\de_{T,T^c}=\PP^{r+s-1}\times\bP^{r+s-1}\subseteq \M_{p,q+1}$ be a
boundary divisor 
such that 
\begin{equation}\label{xyff}
\left| f_{T,E,l}\right| ,\ \left|f_{T^c,E,l}\right| \leq\mu\le (r+s)/2.
\end{equation}
Then
$R\Hom_{\M_{p,q+1}}(F^\vee_{l,E}, \cO_{\de_T}(-a,-b))=0$, 
where by $\cO_{\de_T}(-a,-b)$ we denote $\cO_T(-a)\boxtimes\cO_{T^c}(-b)$, i.e., $-a$ is on the $T$-component)
if $1\le a,b\le  r+s-1$,
or $\mu<b< r+s$, $a=r+s$ (hence, any $a$) or $\mu<a<r+s$, $b=r+s$ (hence, any $b$). 
Similarly, 
$R\Hom_{\M_{p,q+1}}(F_{l,E}, \cO_{\de_T}(-a,-b))=0$ if 
$1\le a,b\le  r+s-1$,
or $a=0,\ 0<b<r+s-\mu$,
or $b=0,\ 0<a<r+s-\mu$.
\end{prop}

\begin{cor}\label{F perp A}
Assume the pair $(l,E)$ is in group $1A$ or $1B$ on $\M_{p,q+1}$. 
Let $\cO_{\de}(-a,-b)$ be one of the generators of the category $\cA$ (Notation \ref{TT}).
Then $R\Hom_{\M_{p,q+1}}(F_{l,E}, \cO_\de(-a,-b))=0$.
\end{cor}

\bp
We apply Proposition~\ref{perpendicular general} with $\mu=\frac{r+s}{2}$. By 
Lemma~\ref{jhvvmgvmv} we have that $|f_{T,E,l}|, |f_{T^c,E,l}|\leq (r+s)/2$
and the statement follows from Proposition~\ref{perpendicular general}. 
\ep

\bp[Proof of Proposition \ref{perpendicular general}]
%The proof is identical to that of Proposition \ref{perpendicular}. 
If $E$ is a vector bundle on $\M_{p,q+1}$, we have 
$$R\Hom_{\M_{p,q+1}}(E^\vee, \cO_{\de_T}(-a,-b))\cong R\Ga_{\de_T}(E\otimes\cO_{\de_T}(-a,-b)),$$
$$R\Hom_{\M_{p,q+1}}(E, K_{\de_T}\otimes\cO_{\de_T}(a,b))^\vee\cong R\Ga_{\de_T}(E^\vee\otimes K_{\de_T}\otimes\cO_{\de_T}(a,b))^\vee,$$
Using Serre duality on $\de_T$, 
the two complexes are isomorphic up to a shift.
So the second statement is equivalent to the first. 
We now prove the first statement. If $\al: W\ra \M_{p,q+1}$ is the universal family, let 
$\be:=\al_{|\al^{-1}(\de_T)}: \al^{-1}(\de_T)=\de_T\cup\de_{T^c}\ra\de_T$,
where we denote $\de_T:=\de_{T, T^c\cup\{y\}}$ ($y$ the new marking on $W$). Let $j:\de_T\hra\M_{p,q+1}$ denote the inclusion. 
Then 
$$R\Hom_{\M_{p,q+1}}\big(F^\vee_{l,E},\cO_{\de_T}(-a,-b)\big)=R\Ga_{\M_{p,q+1}}\big(F_{l,E}\otimes \cO_{\de_T}(-a,-b)\big).$$
Since $F_{l,E}=R\al_*(N_{l,E})$ (see Definition \ref{allF}), it follows by cohomology and base change that
${F_{l,E}}_{|\de_T}=\big(R\al_*(N_{l,E})\big)_{\de_T}=R\be_*\big({N_{l,E}}_{|\de_T\cup\de_{T^c}}\big)$,
and  by the projection formula 
$$R\Ga_{\M_{p,q+1}}\big(F_{l,E}\otimes \cO_{\de_T}(-a,-b)\big)=$$
$$=R\Ga_{\M_{p,q+1}}\big(j_*R\be_*\big({N_{l,E}}_{|\de_T\cup\de_{T^c}}\otimes\be^*\cO(-a,-b)\big)\big).$$
We have to show that the  line bundle 
$\tilde{N}:=\big({N_{l,E}}_{|\de_T\cup\de_{T^c}}\otimes\be^*\cO(-a,-b)\big)\big)$
on $\de_T\cup\de_{T^c}$
has no cohomology. Consider the following exact sequence:
$0\ra\cO_{\de_{T^c}}(-\de_T)\ra\cO_{\de_T\cup\de_{T^c}}\ra\cO_{\de_T}\ra 0$.
Tensoring with $\tilde{N}$ gives an exact sequence:
$0\ra\tilde{N}''\ra\tilde{N}\ra\tilde{N}'\ra0$,
where 
$$\tilde{N}'=\big(N_{l,E}\otimes\be^*\cO(-a,-b)\big)_{|\de_T},\ \tilde{N}''=\big(N_{l,E}\otimes\be^*\cO(-a,-b)\otimes\cO(-\de_T)\big)_{|\de_{T^c}},$$
$$N_{l,E}=\om_{\pi}^{\frac{e-l}{2}}(E)(-\sum\al_{T,E,l}\de_{T\cup\{y\}, T^c})$$
(see (\ref{Sdgsgasrh}) for the definition if $\al_{T,E,l}$). 
We  use Lemma \ref{restrict} to compute $\tilde{N}'$ and $\tilde{N}''$. As usual, for simplicity,
denote $\al_T:=\al_{T,E,l}$, $f_T:=f_{T,E,l}$. 
Using the identification $\de_{T}=\de_{T,T^c\cup\{y\}}=\PP^{r+s-1}\times\Bl_p\PP^{r+s}$, we have
$$\tilde{N}':=\cO(-a+\al_{T^c})\boxtimes\big((f_{T^c}+\al_{T^c}-b)H+(b-\al_T)\De\big).$$
Using the identification $\de_{T^c}=\de_{T^c,T\cup\{y\}}=\PP^{r+s-1}\times\Bl_p\PP^{r+s}$, we have
$$\tilde{N}'':=\cO(-b+\al_T)\boxtimes\big((f_T+\al_T-a)H+(a-\al_{T^c}-1)\De\big).$$
Here $H$, resp. $\De$, denotes $\cO_{\PP^{r+s}}(1)$, resp., the exceptional divisor on $\Bl_p\PP^{r+s}$. 
We prove that both $\tilde{N}'$, $\tilde{N}''$ are acyclic. 
Recall that, for all $T$, either $f_T\geq0$, $\al_T=0$ or $f_T<0$, $f_T+\al_T=0$
and $f_T+f_{T^c}=l\geq0$. 

%\smallskip

\underline{\bf Case $f_T,f_{T^c}\geq0$}:  then $\al_T=\al_{T^c}=0$ and 
$\tilde{N}'=\cO(-a)\boxtimes\big((f_{T^c}-b)H+b\De\big)$, 
$\tilde{N}''=\cO(-b)\boxtimes\big((f_T-a)H+(a-1)\De\big)$.
Clearly, if $0<a\leq (r+s-1)$, $0<b\leq (r+s-1)$ then $\cO(-a)$, $\cO(-b)$ are acyclic,
hence so are $\tilde{N}'$, $\tilde{N}''$. 
Assume that $a=r+s$, $\mu<b\leq r+s-1$. Then $\cO(-b)$ and therefore $\tilde{N}''$ is acyclic.
By (\ref{xyff}), we have 
$-(r+s-1)\leq-b\leq f_{T^c}-b\leq \mu-b<0$.
It follows that $\tilde{N}'$ is acyclic by Lemma \ref{acy}.
Similarly, if $b=r+s$, $\mu<a\leq r+s-1$, then $\cO(-a)$ and $(f_T-a)H+(a-1)\De$ are acyclic, as we have
$-(r+s-1)\leq-a\leq f_T-a\leq\mu-a<0$.

%\smallskip

\underline{\bf Case $f_T\geq0$, $f_T^c<0$}:  then $\al_T=f_{T^c}+\al_{T^c}=0$ and 
$\tilde{N}'=\cO(-a-f_{T^c})\boxtimes\big(-bH+b\De\big)$,
$\tilde{N}''=\cO(-b)\boxtimes\big((f_T-a)H+(a+f_{T^c}-1)\De\big)$.
If $0<b\leq r+s-1$ then $\cO(-b)$, $-bH+b\De$ are acyclic, and the claim follows. 
If $b=r+s$, $\mu <a\leq r+s-1$ then by (\ref{xyff}) 
$0\leq \mu+f_{T^c}<a+f_{T^c}<a\leq r+s-1$,
$-(r+s-1)\leq-a\leq f_T-a\leq \mu-a<0$,
so $\cO(-a-f_{T^c})$ and $(f_T-a)H+(a+f_{T^c}-1)\De$ are acyclic. 

%\smallskip

\underline{\bf Case $f_{T^c}\geq0$, $f_T<0$}:  then $\al_{T^c}=\al_T+f_T=0$  and 
$\tilde{N}'=\cO(-a)\boxtimes\big((f_{T^c}-b)H+(b+f_T)\De\big)$,
$\tilde{N}''=\cO(-b-f_T)\boxtimes\big(-aH+(a-1)\De\big)$.
If $0<a\leq r+s-1$ then $\cO(-a)$, $-aH+(a-1)\De$ are acyclic, and the result follows. 
Assume now $a=r+s$, $\mu<b\leq r+s-1$. 
Note that $-aH+(a-1)\De$ is still acyclic. Furthermore, 
$(f_{T^c}-b)H+(b+f_T)\De$ is acyclic, as by (\ref{xyff}) 
$0\leq \mu+f_T<b+f_T<b\leq r+s-1$ and 
$-(r+s-1)\leq-b\leq f_{T^c}-b\leq\mu-b<0$.   
\ep

%\begin{rmk}
%The above proof shows that when $R\subseteq P$, $|R|=r$, we have 
%$$R\Hom(F_{s,R\cup Q},\cO_{\de_T}(0,-\frac{r+s}{2}))\neq 0,\quad \text{when}\quad T_p=T\cap P=P\setminus R,$$
%(with $0$ on the $T$-component). 
%\end{rmk}

\begin{cor}\label{nbvcnbvcnvbc}\label{F0E}
The line bundle $F_{0,N}$ on $\M_p$ for $p\ge6$ even is the pull-back of the  
GIT polarization via the morphism $\phi:\,\M_p\to\X_p$. When $p=4$,  $F_{0,N}\cong\cO_{\PP^1}(1)$.
\end{cor}

\bp
The line bundle $\cO(1,\ldots,1)$ on $(\PP^1)^p$ descends to $\X_p$ by the Kempf descent criterion (\cite[Theorem4.2.15]{HL})
giving a line bundle $L$ on $\X_p$ which is ample (\cite[Theorem 8.1]{Dolgachev}). Note that, away from the singularities, this agrees with 
our definition of $F_{0,N}$ as $R\al_*N_{0,N}$, where $N_{0,N}=\om_{\al}^2(N)$ and $\al: W\ra\M_p$ is the universal family.
It follows that 
$F_{0,N}\cong \phi^*L(\sum a_{T,T^c}\delta_{T,T^c})$, for some integers $a_{T,T^c}$. 
When $p=2r\ge6$, it remains to show that $a_{T,T^c}=0$ for every partition $P=T\sqcup T^c$, $|T|=|T^c|=r$. 
For every $a=1,\ldots r-2$, we have 
$0=R\Hom_{\M_{p}}(F_{l,E}, \cO_{\de_{T,T^c}}(-a,-a))=\bigoplus_{T,T^c} R\Gamma(\PP^{r-2}\times \PP^{r-2}, \cO_{\de_T}(a_{T,T^c}-a,a_{T,T^c}-a))$
by Proposition~\ref{perpendicular general}. It follows that all $a_{T,T^c}=0$ and hence, $F_{0,N}\cong \phi^*L$. 

To see the last statement, recall from the proof of Lemma \ref{property}, that $N_{0,N}=\al^*F_{0,N}$. It follows that 
$F_{0,N}=\si_1^*\al^*F_{0,N}=\si_1^*N_{0,N}=\psi_1+\sum_{i=2}^4\si_1^*\si_i$. If $p=4$, we have that $\si_1^*\si_i=0$ if $i\neq1$. 
Furthermore, $\ocM_{0,4}$ and $\M_4$ are isomorphic, with the same universal family. Indeed, recall that when $p=2r$ the space $\M_p$ parametrizes either $\PP^1$'s with $p$ marked points, with at most $r-1$ allowed to coincide, or a reducible rational curve with two components, with $r$ marked points on each component (at most $r-1$ points allowed to coincide). In particular, the spaces $\ocM_{0,p}$ and $\M_p$ (for $p$ even) parametrize the same stable pointed curves if and only if $p=4$.
It follows that on $\M_4$ we have $\psi_1=\cO_{\PP^1}(1)$ and $F_{0,N}=\cO_{\PP^1}(1)$.
\ep

\begin{prop}\label{perpendicular W}
Assume 
$|P|=p=2r\geq 4$, $|Q|=q=2s+1\geq1$.
Let $y\in Q$ and consider the universal family
$\al: W\ra\M_{p,q-1}=\M_{P\cup Q\setminus\{y\}}$. 
Let $E\subseteq P\cup Q$ with $e=|E|$, $l\geq0$
such that $e+l$ is even. 
Let $\de_{T\cup\{y\}, T^c}\subseteq W$ be a boundary divisor ($T\sqcup T^c=P\cup(Q\setminus\{y\})$) 
with the property that, for some $\mu$ and $\mu'$,
\begin{equation}\label{take1}
m_{T^c}= m_{T^c,E,l}\leq \mu'\le \frac{r+s}{2},
\end{equation}
\begin{equation}\label{take2}
m_T= m_{T,E,l}\leq \mu-|E\cap\{y\}|\le \frac{r+s}{2}-|E\cap\{y\}|,
\end{equation}
see \eqref{sgasrgarga} for $m_T$.
Using an identification $\de_{T\cup\{y\}}=\Bl_p\bP^{r+s-1}\times\bP^{r+s-2}$, we have 
$R\Hom_W(F_{l,E}, (-aH)\boxtimes\cO(-b))=0$
if $1\le a\le  r+s-1, 1\le b\le r+s-2$ or $a=0,\ 0<b<r+s-1-\mu'$ or $b=0,\ 0<a<r+s-\mu$.
\end{prop}

\bp
%Denote $$m_2(T^c):=\max\{0, \frac{e-l}{2}-|E\cap T^c|\},\quad m_2(T):=\max\{0, \frac{e-l}{2}-|E\cap T|\}.$$
Note that when $f_{T^c,E,l}\leq0$, then $\al_{T^c}=-f_{T^c,E,l}=|E\cap\{y\}|+f_{T,E,l}-l$ and so an 
immediate consequence of (\ref{take2}) is that 
\begin{equation}\label{NEW1}
\al_{T^c}\leq \mu-l.
\end{equation}
Similarly, because of (\ref{take1}), when $f_{T,E,l}\leq0$ then   we have 
\begin{equation}\label{NEW2}
\al_T=-f_{T,E,l}\leq \mu'+|E\cap\{y\}|-l.
\end{equation}

The proof is similar to that of Lemma \ref{perpendicular general}. 
Let~$\pi: \cU\ra W$ be the universal family (with $x$ the new index on $\cU$). Denote for simplicity 
$\de:=\de_{T\cup\{y\},T^c}\subseteq W$, $\de_1=\de_{T\cup\{y,x\},T^c}\subseteq\cU$,
$\de_2=\de_{T\cup\{y\},T^c\cup\{x\}}\subseteq\cU$,
$\be:=\pi_{|\pi^{-1}(\de)}: \pi^{-1}(\de)=\de_1\cup\de_2\ra \de$,
$N:=N_{l,E}$, $F:=F_{l,E}=R\pi_*N_{l,E}$.
Using Grothendieck-Verdier duality (Remark \ref{GV general}), it suffices to prove that letting
$j: \pi^{-1}(\de)\hra \cU$ be the inclusion map, then 
$\tilde{N}:=j^*(N^\vee\otimes \om_{\pi})\otimes \be^*\left(-aH\boxtimes\cO(-b)\right)$
(as a line bundle on $\pi^{-1}(\de)$) is acyclic. Consider the exact sequence: 
$0\ra\cO_{\de_2}(-\de_1)\ra\cO_{\de_1\cup \de_2}\ra\cO_{\de_1}\ra 0$.
Tensoring with $\tilde{N}$ gives an exact sequence
$0\ra\tilde{N}''\ra\tilde{N}\ra\tilde{N}'\ra0$,
where
$\tilde{N}':=\big(N^\vee\otimes \om_{\pi}\big)_{|\de_1}\otimes \be^*\left(-aH\boxtimes\cO(-b)\right)$,
$\tilde{N}'':=\big(N^\vee\otimes \om_{\pi}(-\de_1)\big)_{|\de_2}\otimes \be^*\left(-aH\boxtimes\cO(-b)\right)$.
We prove that  both $\tilde{N}'$ and $\tilde{N}''$ are acyclic. 

Using the identification 
$\de_2=\de_{T\cup\{y\},T^c\cup\{x\}}=\Bl_p\PP^{r+s-1}\times \Bl_p\PP^{r+s-1}$,
(the first copy of $\Bl_p\PP^{r+s-1}$ corresponds to $T\cup\{y\}$) 
by Lemma \ref{restrictions_W} we have 
$\tilde{N}''=M''_1\boxtimes M''_2$, 
$M''_1=\big(-\al_{T^c}-a\big)H+\al_{T^c\cup\{y\}}\De$,
$M''_2=\big(-f_{T^c}-\al_{T^c}-b-1\big)H+\big(\al_{T\cup\{y\}}+b-1\big)\De$.
Similarly, using the identification
$\de_1=\de_{T\cup\{y,x\},T^c}=\Bl_{1,2,\ov{12}}\PP^{r+s}\times \PP^{r+s-2}$,
by Lemma \ref{restrictions_W} we have 
$\tilde{N}'=M'_1\boxtimes M'_2$, where $M'_2=\cO(-\al_{T\cup\{y\}}-b)$,
$M'_1=\big(-f_{T}-|E\cap\{y\}|-\al_{T\cup\{y\}}-a-1\big)H+
\big(f_T+\al_T+1\big)E_1+\big(\al_{T^c}+a\big)E_2+\al_{T^c\cup\{y\}}E_{12}$.
We prove that one of $M'_1$, $M'_2$ and one of $M''_1$, $M''_2$ are acyclic using Lemma 
\ref{acy} (or an analogue). 

\underline{Case 1) $f_T, f_{T^c}\geq0$}. Then 
$\al_T=\al_{T\cup\{y\}}=\al_{T^c}=\al_{T^c\cup\{y\}}=0$, $f_T=m_T$,  $f_{T^c}=m_{T^c}$,
$M''_1=-aH$,
$M''_2=-d''H+(b-1)\De$, where $d''=m_{T^c}+b+1$,
$M'_1=-d'H+\be_1 E_1+\be_2 E_2$, where
$d'=m_T+|E\cap\{y\}|+a+1$,
$\be_1=m_T+1$, $\be_2=a$, $M'_2=\cO(-b)$.
If $0<a\leq r+s-1$, then $M''_1$ is acyclic. Similarly, $M'_2$ is acyclic when $0<b\leq r+s-2$.  
Furthermore, if $0<b\leq r+s-1-\mu'$, then $M''_2$ is acyclic since  (\ref{take1}) implies 
$d''=m_{T^c}+b+1\leq \mu'+b+1<r+s$. 
Similarly, if $0<a<r+s-\mu$ then $M'_1$ is acyclic since  (\ref{take2}) implies that 
$d'=m_T+|E\cap\{y\}|+a+1\leq \mu+a+1<r+s+1$,
and $\be_2=a\leq r+s-1$, $\be_1=m_T+1\leq \mu+1<r+s$. 

\underline{Case 2) $f_T\geq0$,  $f_{T^c}<0$.} Then 
$\al_T= \al_{T\cup\{y\}}=0$, $f_T=m_T$, $\al_{T^c}=-f_{T^c}>0$,
$\al_{T^c\cup\{y\}}=\al_{T^c}-|E\cap\{y\}|\geq0$, 
$M''_1=(-\al_{T^c}-a)H+(\al_{T^c}-|E\cap\{y\}|)\De$,
$M''_2=(-b-1)H+(b-1)\De$,
$M'_1=-d'H+\be_1E_1+\be_2 E_2+\be_{12}E_{12}$, where
$d'=m_T+|E\cap\{y\}|+a+1$, 
$\be_1=m_T+1$,
$\be_2=\al_{T^c}+a$,
$\be_{12}=\al_{T^c}-|E\cap\{y\}|$, $M'_2=\cO(-b)$.
If $0<b\leq r+s-2$, $M''_2$ and $M'_2$ are both acyclic. Assume 
$b=0$, $0<a<r+s-\mu$. Then $M''_1$ is acyclic since by (\ref{NEW1}) we have 
$0\leq \al_{T^c}-|E\cap\{y\}|\leq \mu-l\leq \mu<r+s-1$ ($r+s=2$, $\mu=1$ is not possible, as
$0<a<r+s-\mu$) and 
$0<\al_{T^c}+a\leq\mu+a<r+s$.
Furthermore, $M'_1$ is acyclic since (\ref{take2}) implies that 
$d'=m_T+|E\cap\{y\}|+a+1\leq \mu+a+1\leq r+s$.
Furthermore, 
$\be_1=m_T+1\leq\mu+1\leq \mu+a<r+s$,
$\be_{12}\leq\mu<r+s-1$,
and by (\ref{NEW1}) we have 
$\be_2=\al_{T^c}+a\leq \mu+a<r+s$.

\underline{Case 3) $f_T<0$,  $f_{T^c}\geq0$.} Then 
$\al_{T^c}=\al_{T^c\cup\{y\}}=0$, $f_{T^c}=m_{T^c}$, $\al_{T}=-f_T>0$,
$\al_{T\cup\{y\}}=\al_T-|E\cap\{y\}|\geq 0$,
$M''_1=(-a)H$,
$M''_2=-d''H+\be\De$,
where
$d''=m_{T^c}+b+1$,
$\be=\al_{T}-|E\cap\{y\}|+b-1$,
$M'_1=(-a-1)H+E_1+aE_2$,
$M'_2=\cO(-\al_{T}-b+|E\cap\{y\}|)$.
If $0<a\leq r+s-1$, then $M''_1$ and $M'_1$  are acyclic. Assume 
$a=0$, $0<b<r+s-1-\mu'$. Note that $M'_1$ is still acyclic. Then $M''_2$ is acyclic since
by (\ref{NEW2}) 
$0<\be=\al_T-|E\cap\{y\}|+b-1\leq (\mu'-l)+b-1\leq \mu'+b-1< r+s-2$ and (\ref{take1}) implies 
$d''=m_{T^c}+b+1\leq \mu'+b+1<r+s$.  
This finishes the proof.
\ep

%%%%%%%%%%%%%%%%%%%%%%%%%%%%%%%%%%%%%%%%%%%%%%%%%%%%%%%%%%%%%%%%%%%%%%%%

%%%%%%%%%%%%%%%%%%%%%%%%%%%%%%%%%%%%%%%%%%%%%%%%%%%%%%%%%%%%%%%%%%%%%%%%%%%%

\section{$\{F_{l,E}\}$ exceptional  on $\M_{2r,2s}$ $\Rightarrow$
$\{F_{l,E}\}$  exceptional on $\M_{2r,2s+1}$}\label{induction2}

The goal of this section is to prove the following theorem:

\begin{thm}\label{KHFGJFKJGF}
Assume $p=2r\geq4$, $q=2s+1\geq1$.  
The bundles $\{F_{l,E}\}$ on $\M_{p,q}$ from group $1A$  of Theorem~\ref{asfffdvzsfvsfb} (resp., $1B$ of Remark \ref{1B stuff}) form an exceptional 
collection conditionally on the same statement for $\M_{p,q-1}$.  
\end{thm}

\smallskip

We fix: $P$ a set of cardinality $p=2r\geq 4$ and $Q$ a set of cardinality $q=2s+1\geq 1$. 
We also choose an index $y\in Q$, which will be allowed to change later on. The analysis is similar to the proof of Theorem \ref{stackofbundles}, except that instead of 
forgetful morphisms $\cP_n\ra \cP_{n-1}$, we now have rational maps $\M_{P,Q}\dra \M_{P,Q\setminus\{y\}}$.
These maps have a convenient resolution.

\begin{lemma}\label{SRgSRgarg}
The rational map  $\M_{P,Q}\dra \M_{P,Q\setminus\{y\}}$ is resolved by the following diagram of morphisms,
where  $f$ is a birational reduction morphism between Hassett spaces and 
$\alpha$ is the universal family over $\M_{P,Q\setminus\{y\}}$.
% https://q.uiver.app/#q=WzAsMyxbMSwwLCJXIl0sWzAsMSwiXFxNX3tQLFF9Il0sWzIsMSwiXFxNX3tQLFFcXHNldG1pbnVzXFx7eVxcfX0iXSxbMCwxLCJmIiwyXSxbMCwyLCJcXGFscGhhIl0sWzEsMiwiIiwyLHsic3R5bGUiOnsiYm9keSI6eyJuYW1lIjoiZGFzaGVkIn19fV1d
\begin{equation}\label{W with maps}
\begin{tikzcd}
	& W \\
	{\M_{P,Q}} && {\M_{P,Q\setminus\{y\}}}
	\arrow["f"', from=1-2, to=2-1]
	\arrow["\alpha", from=1-2, to=2-3]
	\arrow[dashed, from=2-1, to=2-3]
\end{tikzcd}
\end{equation}
 \end{lemma}
 
 \begin{proof}
 We choose the weights of $p$ heavy points in $\M_{P,Q\setminus\{y\}}$ and
$\M_{P, Q}$ to be $\frac{1}{r}-\epsilon_1+\epsilon_2$
and the weights of light points to be $\frac{2r}{2s}\epsilon_1+\epsilon_2$ on $\M_{P,Q\setminus\{y\}}$
and $\frac{2r}{2s+1}\epsilon_1+\epsilon_2$ on $\M_{P,Q}$.
Here we first choose $\epsilon_1$ such that 
$\Big(r+1+\frac{2r}{2s+1}\Big)\epsilon_1< \frac{1}{r}$ and then choose $0<\epsilon_2\ll1$ that depends on $\epsilon_1$
(except if $q=1$, in which case the weights of points in $\M_{P,Q\setminus\{y\}}=\M_P$ are  given by $\frac{1}{r}+\epsilon_2$).
By~Notation~\ref{precise weights}, these weights 
 give the Hassett spaces $\M_{P,Q\setminus\{y\}}$ and
$\M_{P, Q}$.
%The point $y$ has weight $\frac{2r}{2s+1}\epsilon_1+\epsilon_2$ on $\M_{P,Q}$.
We choose the weights of points on $W$ as follows: the weights of points in $P$ and $Q\setminus\{y\}$
are the same as on $\M_{P,Q\setminus\{y\}}$ and the weight of the point $y$ is the same as on $\M_{P,Q}$.
It is clear that the weights of $W$ dominate the weights in $\M_{P,Q}$, so we have the reduction morphism $f$ of Hassett spaces.

It remains to show that $\alpha$ is the universal family. We use \cite[Propostion 5.1, Proposition 5.4]{Ha}.
Since the weights of points in $P$ and $Q\setminus\{y\}$
are the same on $W$ and $\M_{P,Q\setminus\{y\}}$, it suffices to check that the weight of $y$ on $W$ is sufficiently small. 
Indeed, $(r-1)\Big({1\over r}-\epsilon_1\Big)+2s{r\over s}\epsilon_1+\frac{2r}{2s+1}\epsilon_1<1$,
so $r-1$ heavy points can coincide with all  points in $Q\setminus\{y\}$ and $y$ on $W$.
On the other hand, 
$r\Big({1\over r}-\epsilon_1\Big)+(s-1){r\over s}\epsilon_1+\frac{2r}{2s+1}\epsilon_1<1$,
so $r$ heavy points can coincide with $s-1$  points in $Q\setminus\{y\}$ and $y$ on $W$. 
Furthermore, $r+1$ heavy points may not coincide on $W$ (similarly $r$ heavy points and $s$ light points) as this would lead to an unstable pointed curve (already on $\M_{P,Q\setminus\{y\}}$).
%One can directly check that there exist $\epsilon_1, \epsilon_2>0$ such that  the following choices of 
 %weights give the stability conditions of the Hassett spaces $\M_{P,Q\setminus\{y\}}$, $W$, $\M_{P, Q}$.
 %For each space, the weights of the points in $P$ are given by  
 %$(\frac{1}{r}-\epsilon_1)+\epsilon_2$. The weights of the points in $Q\setminus \{y\}$ are $\frac{r}{s}\epsilon_1+\epsilon_2$ on $\M_{P,Q\setminus\{y\}}$ and $W$, and 
 %$\frac{2r}{2s+1}\epsilon_1+\epsilon_2$ on $\M_{P,Q}$. Finally, the point $y$ has weight $\frac{2r}{2s+1}\epsilon_1+\epsilon_2$ on  both $W$ and $\M_{P,Q}$.  We impose the following inequalities:
 %$$\Big(r+1+\frac{2r}{2s+1}\Big)\epsilon_1+(2s+r-1)\epsilon_2< \frac{1}{r}, 
 %\quad \epsilon_2<\frac{r}{s(2s+1)(s+r-1)}\epsilon_1,$$
%$$\frac{2r}{2s+1}\epsilon_1+\epsilon_2\leq \text{min}\Big\{\frac{1}{r}-(r+1)\epsilon_1-(2s+r-1)\epsilon_2, \frac{r}{s}\epsilon_1-(s+r-1)\epsilon_2\Big\}.$$ 
\end{proof}

We will apply to the morphism $f$ the following abstract lemma:

\begin{lemma}\label{AbstractPhi}
Let $f:\,M\to N$ be a morphism of smooth projective varieties such that $Rf_*\cO_M\cong\cO_N$.
Let $D^b(M)=\langle\cA,\cB\rangle$ be a s.o.d.
Let $\Phi: D^b(N)\ra\cB$ be the composition of the derived pullback $Lf^*$ and the right adjoint functor $i^{!}$ of the inclusion $i: \cB\hra D^b(M)$. 
Then $\Phi$ has a left adjoint functor $\Psi:\,\cB\ra D^b(N)$ 
given by $X\mapsto [Rf_*(i(X)^\vee)]^\vee$ (derived duals). 
Furthermore, if $Rf_*Z^\vee=0$ for every $Z\in\cA$ then $\Phi$ is a fully faithful functor and $\Psi$ is its left quasi-inverse.
\end{lemma}

\bp The adjointness part follows from the following calculation:
$$\Hom_\cB(X,\Phi(Y))\cong\Hom_{D^b(M)}(i(X),Lf^*(Y))\cong
\Hom_{D^b(M)}(Lf^*(Y)^\vee,i(X)^\vee)$$
$$\cong\Hom_{D^b(N)}(Y^\vee,Rf_*(i(X)^\vee))\cong\Hom_{D^b(N)}(\Psi(X), Y).$$

Suppose  that $Rf_*Z^\vee=0$ for every $Z\in\cA$. Take an object $Y\in D^b(N)$ and a triangle 
$Z\ra i(\Phi(Y))\ra Lf^*(Y)\ra$
associated with the s.o.d.~$\langle\cA,\cB\rangle$. 
The dual triangle is $Lf^*(Y^\vee)\ra i(\Phi(Y))^\vee\ra Z^\vee\ra.$
Taking the push-forward, we get $Y^\vee\cong Rf_*(i(\Phi(Y))^\vee)$ as $Rf_*Z^\vee=0$ and $Rf_*\cO\cong\cO$. 
It follows that $Y\cong\Psi(\Phi(Y))$. Thus $\Psi$ is a   left quasi-inverse of $\Phi$ and therefore
$\Phi$ is fully faithful.
\ep

In the rest of this section, we use the set-up and the notations of Lemma \ref{SRgSRgarg}. 
The  morphism $f:\,W\to\M_{P,Q}$ has the following  properties.
Let $T\subset P\cup (Q\setminus\{y\})$ be a subset of $r$ heavy and $s$ 
light indices, and let $T^c$ 
denote its complement in $P\cup (Q\setminus\{y\})$. In $\M_{P,Q}$ we denote by $\PP^{r+s-1}_T$ the locus where the points from $T^c$ come together. Note that $\PP^{r+s-1}_T\cong\PP^{r+s-1}$ by Note \ref{note 1}(3), as the weight of the point corresponding to $T^c$ can be changed to $1$ without changing the stability condition.

Different loci $\PP^{r+s-1}_T$ are  disjoint, except for pairs $\PP^{r+s-1}_T$ and $\PP^{r+s-1}_{T^c}$, which intersect transversely at the  point in $\M_{P,Q}$ where the markings from $T$ coincide and similarly the markings from $T^c$ coincide. The loci $\PP^{r+s-1}_T$ are blown-up by the morphism~$f$ (in any order), creating boundary divisors 
$\de_{T\cup\{y\}}:=\de_{T\cup\{y\}, T^c}\cong \Bl_p\PP^{r+s-1}\times\PP^{r+s-2}\subset W$. These divisors are contracted by the morphism $f$ via the composition of the first projection and the blow-down
 $\Bl_p\PP^{r+s-1}\times\PP^{r+s-2}\ra \PP^{r+s-1}$. The divisors  $\de_{T\cup\{y\}}$ are disjoint, except for pairs $\de_{T\cup\{y\}}$, $\de_{T^c\cup\{y\}}$, which intersect transversely along a codimension~$2$ locus isomorphic to $\PP^{r+s-2}_T\times \PP^{r+s-2}_{T^c}$. 
 %We will use \eqref{W with maps} to compare s.o.d.'s. 
By Lemma \ref{BigDaddy} (and using Lemma \ref{restoboundaryofM}(iii) and Orlov's theorem), we have an s.o.d.  $D^b(W)=\langle \cC, Lf^*D^b(\M_{P,Q})\rangle$,  where $\cC$ has an exceptional collection of the following objects (in some order):
\bi \item $\cO_{\PP^{r+s-2}_T\times \PP^{r+s-2}_{T^c}}(-i,-j)$ for  $1\leq i\leq r+s-2$, $1\leq j\leq r+s-2$; \item $\cO_{\de_{T\cup\{y\}}}(aH,-u)$ for $0\leq a\leq r+s-1$, $1\leq u\leq r+s-2$.\ei
On the other hand, we have an s.o.d. $D^b(W)=\langle\cA,\cB\rangle$, where 
$\cA\subset D^b(W)$ is the admissible subcategory generated by the exceptional collection 
\begin{equation}\label{category A}
\cO_{\de_{T\cup\{y\}}}(-aH,u),\quad 0<a\leq\Big\lfloor\frac{r+s-1}{2}\Big\rfloor, 1\leq u<\Big\lfloor \frac{r+s-1}{2}\Big\rfloor.
\end{equation}
As $\cC$ is an exceptional collection, it follows that \eqref{category A} is an exceptional collection using Remark \ref{dual of torsion sheaf}. 
\begin{rmk}\label{dual of torsion sheaf}
If $Z=j_*L$ for some line bundle $L$ on $\de:=\de_{T\cup\{y\}}$ ($j:\de\hra W$ the inclusion morphism), then  
$Z^\vee\cong(j_*L^\vee)\otimes\cO(\de)[-1]$ by 
\cite[3.41]{Huybrechts}.
\end{rmk}

\begin{lemma}\label{Phi}
The functor $\Phi: D^b(\M_{P,Q})\ra\cB$ of Lemma~\ref{AbstractPhi} is fully faithful.
\end{lemma}

\bp
We  need to check that  $Rf_*Z^\vee=0$  if $Z=\cO_{\de_{T\cup\{y\}}}(-aH,u)$ is one of the generators of $\cA$.
Indeed, $Z^\vee\cong\cO_{\de_{T\cup\{y\}}}((a-1)H,-u-1)[-1]$ by 
\cite[3.41]{Huybrechts}. Since $-(r+s-2)\leq -u-1\leq-1$, we have  $Rf_*Z^\vee=0$.
\ep

\begin{defn}\label{related}
Let $\cQ$ be a set of objects in a triangulated category $\cT$.
We say that objects $G$ and $G'$ in $\cT$ are \emph{related by quotients in $\cQ$} if there are objects $Q_i$ in $\cQ$ and a sequence of exact triangles
$G_{i-1}\to G_i\to Q_i\to G_{i-1}[1]$ 
with $G_0\cong G$, $G_k\cong G'$. 

We say that objects $G$ and $G'$ in $\cT$ are \emph{related up to shifts by quotients in $\cQ$} 
if there are objects $Q_i$ in $\cQ$ for $0\leq i\leq k$ and a sequence of exact triangles
$G_{i-1}\to G_i[l_i]\to Q_i[l'_i]\to G_{i-1}[1]$  for some $l_i,l'_i\in\ZZ$
with $G_0\cong G$, $G_k\cong G'$. 
(The "related up to shifts" notion will only be used once, in Lemma \ref{K4}.) 
\end{defn}

\begin{lemma}\label{alpha game}
Consider the set of vector bundles $F_{l,E}$ on $\M_{P,Q}$ with 
\begin{equation}\label{wefagwqg}
m_{T,E,l}\leq \min\Big(\frac{r+s}{2}-|E\cap\{y\}|,\quad \frac{r+s-1}{2}\Big)
\end{equation}
for every subset 
$T\subset P\cup (Q\setminus\{y\})$ of $r$ heavy and $s$ light indices (see Notation \ref{numerics} for the definition of $m_{T,E,l}$). 
For example, all bundles of type $1A$/$1B$ and type~$2$ belong to this set  by Lemma \ref{critical on W}.
Then $\Phi(F_{l,E})\cong F_{l,E}$ (see Definition \ref{F on W} for the definition of $F_{l,E}$ on $W$). In particular, 
\begin{equation}
\label{qerfargwerg}
R\Hom_{\M_{P,Q}}(F_{l,E}, F_{l',E'})\cong R\Hom_W(F_{l,E}, F_{l',E'})
\end{equation}
for vector bundles satisfying \eqref{wefagwqg}. Furthermore, $Rf_*(F_{l,E}^\vee)\cong F_{l,E}^\vee$. 
\end{lemma}

\bp
By taking $\mu=\mu'=\frac{r+s}{2}$ in Proposition \ref{perpendicular W} (note that all conditions are satisfied because of (\ref{wefagwqg})) we have that 
\begin{equation}\label{xvzfvasfagb}
R\Hom_W(F_{l,E}, \cO_{\de_T}(-aH,u))=0
\end{equation}
for all $0<a\leq \lfloor\frac{r+s-1}{2}\rfloor$, $0\le -u\le r+s-2$. Since the line bundles $\cO(u)$ for $0\le -u\le r+s-2$ form a full exceptional collection on $\PP^{r+s-2}$, we have that 
\eqref{xvzfvasfagb} holds for all $u$, positive or negative. 
Therefore,  the bundles $F_{l,E}$ on $W$ belong to $\cB$. %{\color{red} where $D^b(W)=\langle\cA,\cB\rangle$ is the s.o.d. described above (see \ref{category A})}.  
By Proposition ~\ref{Q on W}, using that $m_{T^c}<\frac{r+s}{2}$ by the assumption \eqref{wefagwqg} applied to $T^c$, the bundles $F_{l,E}$ and $f^*F_{l,E}=Lf^*F_{l,E}$ are related by quotients in $\cA$. Since $i^!(\cA)=0$ (as in Lemma \ref{AbstractPhi}, $i^!$ is the right adjoint functor of the inclusion $i: \cB\hra D^b(W)$, we have 
$$\Phi(F_{l,E})=i^!(Lf^*F_{l,E})\cong i^!(F_{l,E})\cong F_{l,E}.$$ 
By~Lemma ~\ref{Phi}, we have \eqref{qerfargwerg} and $Rf_*(F_{l,E}^\vee)\cong Rf_*(i(\Phi(F_{l,E}))^\vee)\cong F_{l,E}^\vee$ (see the proof of Lemma \ref{AbstractPhi}).
\ep

\begin{prop}\label{Q on W}
Let $l\geq0$, and let $E\subseteq P\cup Q$ be such that $e+l$ even. 
The  vector bundles $F_{l,E}$ and $f^*F_{l,E}$ on $W$
are related by quotients $Q_i\cong\cO(-j H)\boxtimes\cO(u)$ 
supported on boundary divisors
$\de_{T\cup\{y\}, T^c}$ % \cong\Bl_1\PP^{r+s-1}\times\PP^{r+s-2}$
for $0<j\leq m_{T^c}$, $0\leq u<m_{T^c}$ 
(see Notation \ref{numerics} for the definition of $m_{T^c}:=m_{T^c,E,l}$).
%The vector bundles $F_{l,E}$ and $f^*F_{l,E}$ on $W$
%are related by exact sequences
%$$0\ra G_{i-1}\ra G_i\ra Q_i\ra 0,\quad 1\leq i\leq k$$
%$$G_0=F_{l,E},\quad G_k=f^*F_{l,E},$$
%with $Q_i$  direct sums of sheaves supported on boundary divisors
%$$\de_{T\cup\{y\}, T^c}=\Bl_1\PP^{r+s-1}\times\PP^{r+s-2}$$ 
%and having the form 
%$$-j H\boxtimes\cO(u),\quad 0<j\leq m_{T^c},\quad 0\leq u<m_{T^c}.$$
We~have $$Q^\vee_i\cong\cO((j-1)H)\boxtimes\cO(-u-1).$$
 \end{prop}

\bp
We claim that two types of quotients $\cO(-jH)\boxtimes\cO(u)$  appear: 
\bi
\item[(\text{I})] $0<j\leq m_{T^c}$, $0\leq u< m_{T^c}$, if $f_T>0$;
\item[(\text{II})] $0<j\leq m_{T^c}$, $0\leq u< \tilde{m}_2(T)$, if $f_T+|E\cap\{y\}|<0$,
where\break
$\tilde{m}_2(T):=-f_T-|E\cap \{y\}|$ (see Notation \ref{numerics} for the definition of $f_T, f_{T^c}, m_{T^c}$). Note that in this case $f_T<0$, we have $f_{T^c}=l-f_T>0$ (as $l\geq0$) and $\tilde{m}_2(T)\leq m_{T^c}=f_{T^c}$.
\ei
To prove the claim, recall that $F_{l,E}=R\pi_*(N_1)$, where (see Definition~\ref{F on W})
$$N_1:=N_{l,E}=\om_{\pi}^{\frac{e-l}{2}}(E)\big(-\sum_T \al_T\de_{T\cup\{x\}}-\sum_T\al_{T\cup\{y\}}\de_{T\cup\{y,x\}} \big).$$
Here $x$ is the new marking on the universal family $\pi: \cU\ra W$ and either $\al_T=\al_{T\cup\{y\}}=0$ (when $f_T\geq0$) or 
$\al_T=-f_{T,E,l}$, $\al_{T\cup\{y\}}=\al_T-|E\cap\{y\}|\geq0$ (when $f_T<0$, see (\ref{Sdgsgasrh})) for the definition of $\al_T$, $\al_{T\cup\{y\}}$). 
If~$\pi':~\cU'\ra~\M_{p,q}$ is the universal family, there is a commutative diagram with a cartesian second square:
\begin{equation*}
\begin{CD}
\cU      @>v>>  \cV@>q>> \cU'\\
@V{\pi}VV        @V{\rho}VV @V{\pi'}VV \\
W   @>Id>>  W  @>f>>  \M_{p,q}
\end{CD}
\end{equation*}
and we have $f^*F_{l,E}\cong R\pi_*(N_2)$, where we denote
\begin{equation}\label{identities}
N_2=v^*\om_{\rho}^{\frac{e-l}{2}}(E)\cong\om_{\pi}^{\frac{e-l}{2}}(E)+\sum_T f_T\de_{T\cup\{x\}}.
\end{equation}

The latter equality holds because of the identities (\ref{Pullbacks3}). Recall from Note \ref{note 2}(1) that  $W$ and $\cU$ are smooth.
Furthermore, the spaces $\cU'$ and $\cV$ are $\PP^1$-bundles over Hassett spaces, in particular, they are smooth. 
Note that among the two types of boundary divisors $\de_{T\cup\{x\}}$, $\de_{T\cup\{y,x\}}$ of $\cU$, only the divisors $\de_{T\cup\{x\}}$ are $v$-exceptional.
%, and the restriction 
%$v_{|\de_{T\cup\{x\}}}$ is the second projection 
%$\de_{T\cup\{x\}, T^c\cup\{y\}}=\Bl_1\PP^{r+s-1}\times\Bl_1\PP^{r+s-1}\ra \Bl_1\PP^{r+s-1}$  
%(the markings from $T\cup\{x\}$ coincide in the image). 
We have:
$$N_2=N_1+\sum_T \big(f_T+\al_T\big)\de_{T\cup\{x\}} + \sum_T \al_{T\cup\{y\}}\de_{T\cup\{y,x\}}=N_1+\Sigma_1+\Sigma_2,$$
where
$\Sigma_1=\sum\limits_{f_T>0} f_T\de_{T\cup\{x\}},$
$\Sigma_2=\sum\limits_{f_T<0} \tilde m_2(T)\de_{T\cup\{y,x\}}$.
Any two distinct boundary divisors appearing in $\Sigma_1$ do not intersect since different boundary divisors of type $\de_{T\cup\{x\}}$ do not intersect. 
Similarly, any two distinct boundary divisors appearing in $\Sigma_2$ do not intersect since $\de_{T\cup\{y,x\}}$ intersects  $\de_{T'\cup\{y,x\}}$ ($T'\neq T$) only when $T'=T^c$, but if $f_T<0$, we must have $f_{T^c}>0$. For the same reason, each 
$\de_{T\cup\{y,x\}}$ that appears in $\Sigma_2$ intersects exactly one term from $\Sigma_1$, namely, $\de_{T^c\cup\{x\}}$. 
We add successively to $N_1$ first terms from $\Sigma_1$, then $\Sigma_2$, and get two types of quotients.

\underline{Type I} % This type %of quotient 
quotients have the form
$Q_1=\big(N_1+j\de_{T\cup\{x\}}\big)_{|\de_{T\cup\{x\}}}$
for $0<j\leq  f_T=m_T$. By Lemma \ref{restrictions_W}(1) we have
$$Q_1=\big(\om_{\pi}^{\frac{e-l}{2}}(E)+j\de_{T\cup\{x\}}-\alpha_{T^c\cup\{y\}}\delta_{T^c\cup\{y,x\}}\big)_{|\de_{T\cup\{x\}}}\cong\cO\big(\al H-\beta\Delta\big)\boxtimes \cO\big(-jH\big)$$
%supported on $\de_{T\cup\{x\}, T^c\cup\{y\}}$ %=\Bl_1\PP^{r+s-1}\times \Bl_1\PP^{r+s-1}$,
for 
$\al:=f_T-j$,
 $\beta=\alpha_{T^c\cup\{y\}}=\alpha_{T^c}-|E\cap\{y\}|$ (Notation (\ref{Sdgsgasrh})). By Definition (\ref{Sdgsgasrh}), we always have $\be\geq0$. 
Since $\al_{T^c}$ is either $0$ or $-f_T^c$ and $f_T+f_{T^c}=l\geq0$, it follows that $\be\leq f_T=m_T$. 
Clearly, $0\leq\al< m_T$, $0<j\leq m_T.$
By Lemma \ref{Push},  
$R\pi_*Q_1$ is a direct sum of sheaves of the form 
$\cO(u)\boxtimes (-jH)$ supported on $\de_{T,T^c\cup\{y\}}\cong\PP^{r+s-2}\times\Bl_p\PP^{r+s-1}$ for 
$u\ge \min(\beta,\alpha+1)\ge0$, $u\le \max(\alpha,\beta-1)<m_T$, $0<j\leq m_T$.

\underline{Type II} quotients are
supported on 
$\de_{T\cup\{y,x\}}\cong\Bl_{1,2,\ov{12}}\PP^{r+s}\times\PP^{r+s-2}$,
for various $T$ (with $f_T<0$, hence, $f_{T^c}>0$). %with property $0<i\leq -f_T-|E\cap\{y\}|=\tilde{m}_2(T)$, and having 
Namely, 
$$Q_2=\big(N_1+f_{T^c}\de_{T^c\cup\{x\}}+i\de_{T\cup\{y,x\}}\big)_{|\de_{T\cup\{y,x\}}}=$$
$$=\big(\om_{\pi}^{\frac{e-l}{2}}(E)+f_T\de_{T\cup\{x\}}
+(-\tilde m_2(T)+i)\de_{T\cup\{y,x\}}+f_{T^c}\de_{T^c\cup\{x\}}\big)_{|\de_{T\cup\{y,x\}}}$$ 
for $0<i\le\tilde m_2(T)$.
By~Lemma~\ref{restrictions_W}(2),  
$Q_2\cong\cO\big(-iH+\be E_2\big)\boxtimes\cO(u$)
where $\be:=f_{T^c}=l-f_T\geq \tilde{m}_2(T)$,
$u=\tilde{m}_2(T)-i<\tilde{m}_2(T)$.
In particular, $\be\geq i>0$ and $0\leq u<\tilde{m}_2(T)$. 
By the projection formula, Remark \ref{PushB} and Lemma \ref{restrictions_W}(2) (and using its notations), we have 
$R\tilde{q}_*\cO\big(-iH+\be E_2\big)\cong R\tilde{q}_*\big(\tilde{q}^*\cO(-iH)+(\be-i)E_2)\big)
\cong\cO(-iH)\oplus\ldots\oplus\cO(-\be H)$.
The result follows since $\be=f_{T^c}=m_{T^c}$ (as $f_{T^c}\geq0$). 
%Part (ii) follows by dualizing the exact triangles in (i).
\ep

%%%%%%%%%%%%%%%%%%%%%%%%%%%%%%%%%%%%%

\begin{rmk}\label{M_4}
Proposition \ref{Q on W} shows that when $l=0$ and $E=P\cup (Q\setminus\{y\})$, on $W$ we have $f^*F_{l,E}\cong F_{l,E}$ (because $m_T=m_{T^c}=0$ for all $T$). Note that the same is trivially true when $l=0$, $E=\emptyset$. In particular, $F^\vee_{l,E}\cong Rf_*F^\vee_{l,E}$ is still true in these cases. 
For example, on $\M_{p,q-1}\cong\M_4$ (i.e., $p=4,q=1$) 
$\{F_{0,\emptyset}, F_{0,N}\}$ is the collection in Theorem \ref{even 0} (group $1A$ is the same as group $1B$ in this case). This is a full exceptional collection on 
$\M_4\cong\PP^1$ % (here $N=P$, $|P|=4$) 
because $F_{0,\emptyset}\cong \cO$, $F_{0,N}\cong\cO(1)$ (Corollary~\ref{F0E}). 
\end{rmk}

\bp[Proof of Theorem~\ref{KHFGJFKJGF}]
Let $F_{l,E}$, $F_{l',E'}$ be bundles from one of the two collections (group $1A$ or $1B$) on $\M_{p,q}$.  
By Lemma \ref{alpha game}, for any $y\in Q$, we have 
$$R\Hom_{\M_{p,q}}(F_{l',E'},F_{l,E})\cong R\Hom_W(F_{l',E'},F_{l,E}).$$
We consider several cases and make different choices of $y\in Q$.

Case 1: $E'_q\setminus E_q\ne\emptyset$. Fix  $y\in E'_q\setminus E_q$. By Lemma \ref{zzfbzdfbdff}, 
$F^\vee_{l, E}\cong\al^*F^\vee_{l,E}$, $R\al_*F^\vee_{l', E'}=0$. Therefore, 
$R\Hom_W(F_{l',E'},F_{l,E})\cong R\Hom_W(F^\vee_{l,E},F^\vee_{l',E'})\cong
R\Hom_W(\al^*F^\vee_{l,E},F^\vee_{l',E'})\cong R\Hom_{\M_{p,q-1}}(F^\vee_{l,E},R\al_*F^\vee_{l',E'})=0$.

Case 2: $E_q=E'_q\neq Q$. Choose $y\in Q\setminus E_q$. Since the range in group $1A$ (resp., $1B$) on $\M_{p,q}$ 
is precisely the range in group $1A$ (resp., $1B$) on $\M_{p,q-1}$, we are done by Lemma \ref{zzfbzdfbdff}, since 
$F_{l,E}\cong\al^*F_{l,E}$, $F_{l',E'}\cong\al^*F_{l',E'}$ and  
$R\Hom_W(\al^*F_{l',E'},\al^*F_{l,E})\cong R\Hom_{\M_{p,q-1}}(F_{l',E'}, F_{l,E})$.
We are done by the assumption that the $\{F_{l,E}\}$ from group $1A$ (resp., $1B$) form an exceptional collection on $\M_{p,q-1}$
(Theorem ~\ref{asdvzsfvsfb} when $q>1$, while when $q=1$ this is Theorem \ref{even 0} for group $1A$, and Remark \ref{1B stuff} for group $1B$). 

Case 3:  $E_q=E'_q=Q$, $q>1$. Let $z\in Q$. By Proposition \ref{tensor}, 
$F_{l,E}\cong F_{l, E_p\cup\{z\}}\otimes F_{0,E_q\setminus\{z\}}$,
$F_{l',E'}\cong F_{l, E'_p\cup\{z\}}\otimes F_{0,E'_q\setminus\{z\}}$.
It follows that 
$$R\Hom_{\M_{p,q}}(F_{l',E'},F_{l,E})\cong R\Hom_{\M_{p,q}}(F_{l',E'_p\cup\{z\}},F_{l,E_p\cup\{z\}}).$$
Note that $F_{l,E_p\cup\{z\}}$ is still in group $1A$ (resp., group $1B$) on $\M_{p,q}$, 
and since by assumption $Q\neq\{z\}$ we are in Case 2.

Case 4:  $E_q=E'_q=Q$, $q=1$. Let $Q=\{z\}$. Note that $\M_{p,1}\cong \M_{p+1}$ and since $e'_p\geq e_p$, we have $e'\geq e$. 

Case 4(i): the bundles are in group 1B. % on $\M_{p,1}$. 
Then they are in the exceptional collection of 
Theorem~\ref{symm} (proved in  Section 
\ref{windows section}), and we are done. 

Case 4(ii): the bundles are in group 1A. % on $\M_{p,1}$. 
Take the line bundle $L=\cO(-P-2z)$ on $\M_{p+1}$, i.e., ~we take -1 in the position of every heavy point
and $-2$ in the position of~$z$. By the projection formula and  Corollary  ~\ref{complement},  
$L\otimes F_{l,E}\cong \cO(-P-2z)\otimes R\pi_*\cO(E,l)\cong R\pi_*(-(P\setminus E_p)-z,l)\cong F^\vee_{l,(P\setminus E_p)\cup\{z\}}$
and  similarly for $F_{l',E'}$.
On $\M_{p,1}$, we have 
$$R\Hom(F_{l',E'},F_{l,E})\cong R\Hom(F_{l',E'}\otimes L,F_{l,E}\otimes L)\cong $$
$$\cong R\Hom(F^\vee_{l',(P\setminus E_p')\cup\{z\}},F^\vee_{l,(P\setminus E_p)\cup\{z\}})\cong R\Hom(F_{l,(P\setminus E_p)\cup\{z\}},F_{l',(P\setminus E'_p)\cup\{z\}}).$$
These bundles are in the group 1B on $\M_{p,1}$ and are in the correct order: 
$p-e_p\ge p-e_p'$, hence, we are done by the argument in the Case 4(i). 
\ep

%%%%%%%%%%%%%%%%%%%%%%%%%%%%%%%%%%%%%%%%%%%%%%%%%%%%%%%
%%%%%%%%%%%%%%%%%%%%%%%%%%%%%%%%%%%%%%%%%%%%%%%%%%%%%%%
%%%%%%%%%%%%%%%%%%%%%%%%%%%%%%%%%%%%%%%%%%%%%%%%%%%%%%%
%%%%%%%%%%%%%%%%%%%%%%%%%%%%%%%%%%%%%%%%%%%%%%%%%%%%%%%
%%%%%%%%%%%%%%%%%%%%%%%%%%%%%%%%%%%%%%%%%%%%%%%%%%%%%%%
%%%%%%%%%%%%%%%%%%%%%%%%%%%%%%%%%%%%%%%%%%%%%%%%%%%%%%%

\section{$\{F_{l,E}\}$   exceptional on $\M_{2r,2s+1}$ $\Rightarrow$
$\{F_{l,E}\}$ exceptional  on $\M_{2r,2s+2}$ }\label{induction1}

The goal of this section is to prove the following theorem.

\begin{thm}\label{r+s odd}
Vector bundles $F_{l,E}$ from group 1A (resp., group 1B) from Theorems~\ref{asfffdvzsfvsfb}, \ref{asdvzsfvsfb}, \ref{even 0} and Remark \ref{1B stuff} form an exceptional collection. 
\end{thm}

Throughout this section we  assume that $P$ (resp., $\tQ$) is the set of heavy (resp., light) indices and %, unless otherwise noted,
 we have  
$|P|=p=2r\geq 4$, $|\tQ|=q+1=2s+2\geq0$.
Given Theorem~\ref{KHFGJFKJGF}, it suffices to prove
that the bundles $\{F_{l,E}\}$ from group 1A (resp., 1B) form an exceptional collection 
on $\M_{p,q+1}$ conditionally on the same 
statement for $\M_{p,q}$, as well as to prove, unconditionally, 
 that these  bundles   form an exceptional collection on $\M_{p}$.
%We will treat the case $r=2, s=0$ (i.e., $\M_{p,q+1}=\M_{4,2}$) separately (see \ref{}). 
We treat this last case by allowing  $q+1=0$ (i.e., $s=-1$), in which case we also assume that $r\ge3$.
One checks directly the case of $r=2$, $s=-1$, i.e., that  the collection  on 
$\M_p=\M_4\cong\PP^1$ in Theorem \ref{even 0} is exceptional: the collection is $\{F_{0,\emptyset}, F_{0,N}\}$, 
where $F_{0,\emptyset}\cong \cO$, $F_{0,N}\cong\cO(1)$ by Corollary~\ref{F0E} (see also Remark \ref{M_4}).

\begin{notn}\label{Upq}
We choose an index $z\in\tQ$ unless $q+1=0$, in which case we choose an index $z\in P$.
This index will  be allowed to vary later.
For every boundary divisor $\delta_{T,T^c}$ of $\M_{p,q+1}$, we 
always assume that $z\in T$. 
%We call the other subset $T^c$. %We will specialize to concrete recipes later.  For example, one way to make a choice, is to fix an index $z$, and We have a commutative diagram, where birational morphisms $\beta$ and $g$ depend on the choice of $T$'s:
We~let $\M_{p,q}=\M_{P,\tQ\setminus\{z\}}$ if $q\ge0$ (resp., $\M_{p,-1}:=\M_{p-1}=\M_{P\setminus\{z\}}$ if $q=-1$).
Let $f: \M_{p,q+1}\dra \M_{p,q}$ be the forgetful map that forgets the marking $z$.
%Let~$\pi:\,\cU_{p,q}\to \M_{p,q}$ be the universal family, which is a $\bP^1$-bundle.
\end{notn}

\begin{lemma}\label{fbadfbadhadthn}
The forgetful map $f$ is regular and 
factors through the universal family $\pi:\,\cU_{p,q}\to \M_{p,q}$, which is a $\bP^1$-bundle.
We have a commutative diagram
\begin{equation}\label{sGsrgsRH}
\begin{CD}
W                @>g>>      \cV_{p,q}\\
@V{\al}VV                             @VV{\gamma}V \\
\M_{p,q+1}  @>\beta >>          \cU_{p,q}@>\pi >>          \M_{p,q}
\end{CD}
\end{equation}
Here $\al$ (resp.,~$\gamma$) is the universal family of the corresponding Hassett space.\break
Furthermore, $\gamma$ is a $\bP^1$-bundle and $\cU_{p,q}$ is a GIT quotient of $(\bP^1)^{p+q+1}$.
The morphisms $g$ and $\beta$ are birational reduction morphisms of Hassett spaces.
\end{lemma}

\bp
Suppose first that $q+1>0$. We choose the same weight  $a+\epsilon$, where $a$ is such that $r<{1\over a}<r+{2s+1\over 2s+3}<r+1$, $\epsilon\ll 1$,
for all $p$ heavy points for $\M_{p,q+1}$, $\cU_{p,q}$, and $\M_{p,q}$; the same weight $b+\epsilon$, where $b={2-2ra\over 2s+1}$, $\epsilon\ll 1$, 
for all light points on $\M_{p,q+1}$ and $\M_{p,q}$ and for all light points on $\cU_{p,q}$ except for $z$.
We assume that the weight of $z$ on $\cU_{p,q}$ is sufficiently small. Then $pa+qb=2$. 
%the sum of the weights is $2$ for $M_{p,q}$ and is greater than $2$ for $\M_{p,q+1}$. 
Furthermore, $(r+1)a>1$, so $r+1$ heavy points coinciding corresponds to an unstable pointed curve;
$(r-1)a+(2s+2)b<1$, so $r-1$ heavy and all light points coinciding corresponds to a stable point of $\M_{p,q+1}$ and $\M_{p,q}$;
$ra+sb<1$, so $r$ heavy and $s$ light points  coinciding corresponds to a stable point of for $\M_{p,q+1}$ and $\M_{p,q}$ and, finally,
$ra+(s+1)b>1$, so $r$ heavy and $s+1$ light points coinciding corresponds to an unstable point of $\M_{p,q+1}$ and $\M_{p,q}$.
It follows that these weights define $\M_{p,q+1}$ and $\M_{p,q}$, 
that $\cU_{p,q}$ is the universal family over $\M_{p,q}$, and that we have reduction maps of Hassett spaces as claimed.
Furthermore, by Lemma~\ref{MonGeneralBrendanHassett}, $\M_{p,q}$ and $\cU_{p,q}$ are GIT quotients and 
$\pi$ and $\gamma$ are $\bP^1$-bundles.

In the case $q+1=0$, we choose the same weight $a+\epsilon$, where $a={2\over p-1}$, $\epsilon\ll 1$,   
for all points for $\M_{p}$ and $\M_{p,-1}=\M_{p-1}$ and for all points for 
$\cU_{p,-1}=\M_{p-1,1}$ except for~$z$, which we assume has a sufficiently small weight.
 Then $(p-1)a=2$. 
 %the sum of the weights is $2$ for $M_{p-1}$ and is greater than $2$ for $\M_{p}$. 
 Furthermore, 
 $(r-1)a<1$ and $ra>1$, so $r-1$ points coinciding corresponds to a stable point, but $r$ points coinciding corresponds to an unstable point on  
$\M_{p}$ and $\M_{p,-1}=\M_{p-1}$. We finish the argument as above.
\ep

\begin{notn}
The morphism $\beta:\,\M_{p,q+1}\to \cU_{p,q}$ belongs to a  general class of reduction morphisms of Hassett spaces
$\beta:\,\M_{p,q+1}\to \cU$, which we describe now with an eye towards a different application in Section~\ref{exceptional p,q+1 section}. We will return to $\cU=\cU_{p,q}$ in Lemma~\ref{fbfgafgadf}. The following diagram generalizes
\eqref{sGsrgsRH}:
\begin{equation}
\begin{CD}
W                @>g>>      \cV\\
@V{\al}VV                             @VV{\gamma}V \\
\M_{p,q+1}  @>\beta >>          \cU
\end{CD}
\end{equation}
Reducible fibers of the universal family $\alpha$ are over the boundary divisors $\de_{T,T^c}\cong\PP_T^{r+s-1}\times\PP_{T^c}^{r+s-1}$ of $\M_{p,q+1}$. The subscripts indicate that the fibers of $\alpha$
are unions of two~$\bP^1$, with one component marked by $T$
and another by~$T^c$. We suppose that the restriction of $\beta:\,\M_{p,q+1}\to \cU$ to $\delta_{T,T^c}$ is either an isomorphism
or is isomorphic to the projection  $\PP_T^{r+s-1}\times\PP_{T^c}^{r+s-1}\to\PP_{T^c}^{r+s-1}$,
i.e.~the sections marked by $T$ become identified. In particular, if $\delta_{T,T^c}$ is contracted then we choose one of $T$, $T^c$ and call it $T$ and the other $T^c$.
The boundary divisors in $W$ have the form $\de_{T\cup\{y\}}$ and $\de_{T^c\cup\{y\}}$, where $y$ is the extra marking on~ $W$. They are isomorphic to $\PP^{r+s-1}\times \Bl_p\PP^{r+s}$.
The morphism $g$ is either an isomorphism near
$\de_{T\cup\{y\}}\cup\de_{T^c\cup\{y\}}$ or
contracts  $\de_{T^c\cup\{y\}}$ via the second projection
$\PP^{r+s-1}\times \Bl_p\PP^{r+s}\ra \Bl_p\PP^{r+s}$
and $\de_{T\cup\{y\}}$ 
via the composition
$\PP^{r+s-1}\times \Bl_p\PP^{r+s}\to\PP^{r+s-1}\hookrightarrow \Bl_p\PP^{r+s}$
of the first projection followed by the embedding into $\Bl_p\PP^{r+s}$ as the exceptional divisor (the locus where the marking corresponding to $T$ coincides with $y$). 
Note that any $\cU$ as above carries vector bundles $F_{l,E}$ defined in Section~\ref{extend section}.
\end{notn}

\begin{prop}\label{sifjkdfjkwdjk}
Let $l\geq0$, and let $E\subseteq N:=P\cup \tQ$ be such that $e+l$ even. 
%Let $$m:=\max_{T:\ \delta_{T,T^c}\subset\Exc(f)}m_{T,E,l}$$

(i)  The bundles $F_{l,E}$ and $\beta^*F_{l,E}$ on $\M_{p,q+1}$
are related by 
%exact sequences
%$$0\ra G_{i-1}\ra G_i\ra Q_i\ra 0,\quad 1\leq i\leq k$$
%$$G_0=F_{l,E},\quad G_k=\beta^*F_{l,E},$$
%with 
%quotients $Q_i$, which are  direct sums of 
quotients (in the sense of Definition \ref{related}) which are sheaves supported on contracted  divisors
$\de_{T, T^c}=\PP^{r+s-1}\times\PP^{r+s-1}$
and having the form
$\cO(u, -j)$, $0<j\leq m$, $0\leq u< m$, where $m:=m_{T,E,l}$ (see (\ref{sgasrgarga})),  
where the component $\cO(u)$ is the one corresponding to $T$. 

(ii) The  bundles $F^\vee_{l,E}$ and $\beta^*F^\vee_{l,E}$ 
are related by quotients which are sheaves 
%quotients triangles
%$$Q_i^\vee\to G_i^\vee\to G_{i-1}^\vee\to Q_i^\vee[1],$$
%$$G_0^\vee=F^\vee_{l,E},\quad G^\vee_k=\beta^*F_{l,E}^\vee,$$ with 
%$Q_i^\vee$, which are  direct sums of sheaves (placed in cohomological degree $1$) 
supported on contracted divisors
$\de_{T, T^c}$  and having the form
$\cO(-u-1, j-1)$, $0<j\leq m$, $0\leq u< m$,  where $m:=m_{T,E,l}$. 
\end{prop}

\bp
Since $F_{l,E}$ and $\beta^*F_{l,E}$ are isomorphic near the boundary divisors $\delta_{T,T^c}$ that 
are not contracted by $\beta$, we can remove them and work on the remaining open set
$\M^0_{p,q+1}$ to simplify notation.
In particular, all subsets $T$ that we refer in the proof correspond to  $\delta_{T,T^c}\subset\Exc(\beta)$.
We  use that 
%$\max\left(0,\quad\max\limits_{T\sqcup T^c=N}-f(E,l,T^c)\right)\le m$.
$\max\limits_{T}(-f_{T^c,E,l})=\max\limits_{T}(-l+f_{T,E,l})\le m$.
%We prove that there are three types of quotients that appear: 
%\bi \item[(\text{I})]  if $f(T), f(T^c)>0$;
%\item[(\text{II})]  if $f(T)>0$, $f(T^c)\le0$;
%\item[(\text{III})]  if $f(T^c)<0$. \ei
%$$\text{(I)  if} f_T, f_{T^c}>0;\quad
%\text{(II)  if} f_T>0, f_{T^c}\le0;\quad
%\text{(III)  if} f_{T^c}<0.$$
On $\M^0_{p,q+1}$, $F_{l,E}=R\alpha_*(N_1)$, where 
$$N_1:=N_{l,E}=\om_{\alpha}^{\frac{e-l}{2}}\Bigl(E-\sum_{f_T<0} \left(-f_T\right)\de_{T\cup\{y\}}
-\sum_{f_{T^c}<0} (-f_{T^c})\de_{T^c\cup\{y\}}\Bigr)$$
and
$\beta^*F_{l,E}\cong R\alpha_*(N_2)$, where 
\begin{equation}\label{N2 vs N1}
N_2=g^*\om_{\pi}^{\frac{e-l}{2}}(E)\cong\om_{\alpha}^{\frac{e-l}{2}}\Bigl(E+\sum_{T}  f_T\de_{T\cup\{y\}}\Bigr)= N_1+S_1+S_2.
\end{equation}
Here
$S_1=\sum\limits_{f_T>0} f_T\de_{T\cup\{y\}}$,
$S_2=\sum\limits_{f_{T^c}<0} \left(-f_{T^c}\right)\de_{T^c\cup\{y\}}$ and the last congruence follows by the identities (\ref{Pullbacks3}).  
Recall that the markings in $T$ get identified when applying $\beta$. 
We further break $S_1$ into $S'_1+S''_1$, where $S'_1$ (respectively $S''_1$),  contains the terms with $f_{T^c}>0$ (respectively $f_{T^c}\le0$). 
We add successively to $N_1$ first terms from $S'_1$, then $S''_1$, followed by $S_2$, to get three types of quotients on $W$. 

%\smallskip 

\underline{Type I} has direct summands of the form: 
$Q'_1=N_1\big(i\de_{T\cup\{y\}}\big)_{|\de_{T\cup\{y\}}}\cong\om_{\al}^{\frac{e-l}{2}}\big(E+i\de_{T\cup\{y\}}\big)_{|\de_{T\cup\{y\}}}
\cong\big(u H)\boxtimes \cO(-i)$, where $u:=f_T-i$,
where $\de_{T\cup\{y\}}\cong\Bl_1\PP^{r+s}\times \PP^{r+s-1}$ is of the type appearing in $S'_1$: 
$0<i\leq  f_T$, $f_{T^c}>0$
(for various $T$ with this property). Clearly, $0\leq u<m$ and $0<i\leq m$. By Lemma~\ref{Push} , $R{\alpha}_*(uH)\cong\cO\oplus\cO(1)\oplus\ldots\oplus\cO(u)$, and the result follows. 

%\smallskip 

\underline{Type II}  has direct summands of the form: 
$Q''_1=N_1\left(i\de_{T\cup\{y\}}\right)_{|\de_{T\cup\{y\}}}\cong
\om_{\al}^{\frac{e-l}{2}}\left(E+f_{T^c}\de_{T^c\cup\{y\}}+i\de_{T\cup\{y\}}\right)_{|\de_{T\cup\{y\}}}
\cong\big(u H-v\De)\boxtimes \cO(-i)$, where
$u:=f_T-i$, $v=-f_{T^c}$,
supported on  $\de_{T\cup\{y\}}\cong\Bl_1\PP^{r+s}\times \PP^{r+s-1}$ of the type appearing in $S''_1$: 
$0<i\leq  f_T$, $f_{T^c}<0$
(for various $T$ with this property). Clearly, we have
$0\leq u<m$, $0<i\leq m$ and $0<v=-f_{T^c}\leq f_T=m$, since $f_T+f_{T^c}=l\geq0$.
By Lemma \ref{Push}, if $u\geq v$, 
 $R{\pi}_*(uH-v\De)\cong\cO(v)\oplus\cO(v+1)\oplus\ldots\oplus\cO(u)$, while if $u<v$, we have nevertheless that 
$R{\pi}_*(uH-v\De)$ is either $0$ or generated by $\cO(u+1),\ldots, \cO(v-1)$
and the result follows.  
 
%\smallskip 

\underline{Type III} has direct summands 
$Q_2=\big(N_1+S_1+i\de_{T^c\cup\{y\}}\big)_{|\de_{T^c\cup\{y\}}}\cong \break
\om_{\al}^{\frac{e-l}{2}}\left(E+\left(f_{T^c}+i\right)\de_{T^c\cup\{y\}}+
f_T\de_{T\cup\{y\}}\right)_{|\de_{T^c\cup\{y\}}}
\cong\cO(u)\boxtimes\big(-iH+\be\De\big)$, where 
$\be=f_T$, $u=-f_{T^c}-i$,
supported on $\de_{T^c\cup\{y\}}=\de_{T,T^c\cup\{y\}}\cong\PP^{r+s-1}\times \Bl_1\PP^{r+s}$ 
(of the type appearing in $S_2$): 
$0<i\leq-f_{T^c}\leq f_T=\beta$ since $f_T+f_{T^c}=l\geq0$ 
(for various $T$ with this property). Clearly, 
$0<i\leq m$, $0<\be\leq m$, $0\leq u<m$.
Since $\be\geq i$, it follows by Lemma \ref{Push} that 
$R{\alpha}_*(-iH+\be\De)\cong\cO(-i)\oplus\cO(-i-1)\oplus\ldots\oplus\cO(-\be)$,
and (i) follows. Part (ii) follows from (i) by dualizing the  triangles
and using Remark \ref{dual of torsion sheaf}. 
\ep

\begin{cor}\label{sbvmsfbv}
Let $F_{l,E}$ be a bundle from groups $1A$ or $1B$ on $\M_{p,q+1}$. 
Then the quotients from Proposition~\ref{sifjkdfjkwdjk}(i) belong to the subcategory $\cA$ of Notation~\ref{TT} with the  possible exceptions of $Q_i=\cO_{\de_{T,T^c}}(u, -\frac{r+s}{2})$ ($0\leq u<\frac{r+s}{2}$) on divisors $\de_{T,T^c}$ contracted by $\beta$, with 
$r+s$ even and for the following subsets~$T$:
\bi
\item $l+e_p=r-1$,\quad  $e_q=s+1$,\quad $E_p\subseteq T_p$, \quad $T_q=E_q$\quad  (group $1A$);
%\item $e_p=r$,\quad  $e_q=s+2+l$,\quad  $E_p=T_p$,\quad  $T_q\subseteq E_q$ \quad  (group $2A$);
\item $e_p=r+1+l$, $e_q=s+1$,\quad  $T_p\subseteq E_p$, \quad $T_q=E_q$ \quad  (group $1B$);
%\item $e_p=r$,\quad  $l+e_q=s$,\quad  $E_p=T_p$, \quad  $E_q\subseteq T_q$ \quad  (group $2B$).
\ei
\end{cor}

\bp
Note that if $0<b\leq \frac{r+s-1}{2}$, then $\cO(u,-b)\in\cA$ for all $u$, positive or negative. 
In the notations of Proposition \ref{sifjkdfjkwdjk}, we have by Lemma~\ref{jhvvmgvmv} that 
$m=m_{T,E,l}\leq \frac{r+s}{2}$. The Corollary now follows by Proposition \ref{sifjkdfjkwdjk}(i) and Lemma~\ref{jhvvmgvmv}. 
\ep

\begin{cor}\label{asdfsdvsdb}\label{push beta}
In the notations of Proposition~\ref{sifjkdfjkwdjk},
if $m\le r+s-1$ (for example, if  $m\le \frac{r+s}{2}$), 
then on  $\cU$ we have: 
$F_{l,E}\cong [R\beta_*F^\vee_{l,E}]^\vee$.
\end{cor}

\bp
Since $R\beta_*\cO_{\delta_{T,T^c}}(-u-1, j-1)=0$ for $0\le u\leq r+s-2$ if the divisor is contracted, we have Proposition \ref{sifjkdfjkwdjk}(ii) that 
$R\beta_*F^\vee_{l,E}\cong R\beta_*\beta^*F^\vee_{l,E}\cong F^\vee_{l,E}$.
\ep

\begin{cor}\label{aksjrgnkjdg}\label{long cor}(Beta game)
Let $F_{l,E}$ and $F_{l',E'}$ be bundles from  group $1A$ (resp., group $1B$) on $\M_{p,q+1}$. 
Suppose $|E_q|\le |E_q'|$. %and if $E_q=E_q'$ then $|E_p|\le |E_p'|$.
Then
$$R\Hom_{\M_{p,q+1}}(F_{l',E'},F_{l,E})\cong R\Hom_{ \cU}(F_{l',E'},F_{l,E})$$
unless $r+s$ is even and we have one of the following exceptions, which include existence
of a critical subset %$T$ %that  violate the conditions, where 
$T=T_p\coprod T_q$ %is a subset of markings 
split into heavy and light markings %such that $T_p=r$ and $T_q=s+1$. 
\bi
\item[($1A$)] 
 $l+e_p=l'+e_p'=r-1$,\quad  $e_q=e_q'=s+1$, \quad 
$E_p\subseteq T_p$, \quad $E_q=T_q$\quad and either 
$E_p'\subseteq T_p$, \quad $E_q'=T_q$\quad  or
$E_p'\subseteq T^c_p$, \quad $E_q'=T_q^c$;

%\item[(2) ($1A-2A$)]  
%$l+e_p=r-1$, $e_p'=r$,\quad $e_q=s+1$, $e_q'=s+2+l'$,\quad
%$E_p\subseteq T_p$, \quad $E_q=T_q$\quad  and either
%$E_p'=T_p$, \quad $T_q\subseteq E_q'$\quad  or
%$E_p'=T_p^c$, \quad $T_q^c\subseteq E_q'$;
%
%\item[(3) ($2A-2A$)] 
%$e_p=e_p'=r$,\quad  $e_q=s+2+l$, $e'_q=s+2+l'$\quad  
%$E_p=T_p$,\quad  $T_q\subseteq E_q$ \quad and either  
%$E_p'=T_p$,\quad  $T_q\subseteq E_q'$ \quad or 
%$E_p'=T_p^c$,\quad  $T_q^c\subseteq E_q'$;

\item[ ($1B$)]  
$e_p=r+1+l$, $e'_p=r+1+l'$, $e_q=e_q'=s+1$,\quad  
$T_p\subseteq E_p$, \quad $T_q=E_q$ \quad  and either
$T_p\subseteq E_p'$, \quad $T_q=E_q'$ \quad  or
$T_p^c\subseteq E_p'$, \quad $T_q^c=E_q'$.

%\item[(5) ($1B-2A$)]  
%$e_p=r+1+l$, $e'_p=r$,\quad $e_q=s+1$, $e'_q=s+2+l'$\quad  
%$T_p\subseteq E_p$, \quad $T_q=E_q$ \quad  and either
%$E_p'=T_p$, \quad $T_q\subseteq E'_q$ \quad  or
%$E_p'=T_p^c$, \quad $T_q^c\subseteq E'_q$;
%
%
%\item[(6) ($2B-1A$)] 
%$e_p=r$, $l'+e_p'=r-1$,\quad  $l+e_q=s$, $e'_q=s+1$,\quad 
%$E_p=T_p$, \quad $E_q\subseteq T_q$\quad and either 
%$E'_p\subseteq T_p$, \quad $E'_q=T_q$\quad  or
%$E'_p\subseteq T^c_p$, \quad $E'_q=T_q^c$;
%
%\item[(7) ($2B-1B$)] 
%$e_p=r$, $e'_p=r+1+l'$, $l+e_q=s$, $e'_q=s+1$,\quad  
%$E_p=T_p$, \quad $E_q\subseteq T_q$ \quad  and either
%$T_p\subseteq E'_p$, \quad $T_q=E'_q$ \quad  or
%$T_p^c\subseteq E'_p$, \quad $T_q^c=E'_q$;
%
%
%\item[(8) ($2B-2B$)] 
%$e_p=e_p'=r$,\quad  $l+e_q=l'+e_q'=s$,\quad  
%$E_p=T_p$, \quad  $E_q\subseteq T_q$ \quad  and either
%$E_p'=T_p$, \quad     $E_q'\subseteq T_q$ \quad or 
%$E_p'=T_p^c$, \quad  $E_q'\subseteq T_q^c$.
\ei
\end{cor}

\bp
We have $R\Hom_{\M_{p,q+1}}(F_{l',E'},\beta^*(F_{l,E}))=R\Hom_{\M_{p,q+1}}(\beta^*(F^\vee_{l,E}),F^\vee_{l',E'})$, 
which by adjointness is the same as  $R\Hom_{\cU}(F^\vee_{l,E},\beta_*F^\vee_{l',E'})$, and the latter is the same as 
$R\Hom_{\cU}(F^\vee_{l,E},F^\vee_{l',E'})$
by Corollary~\ref{asdfsdvsdb} (since $m\leq \frac{r+s}{2}$ by Lemma \ref{critical}).
%By Lemma~\ref{AbstractPhi},
%$R\Hom_{\M_{p,q+1}}(F_{l',E'},\Phi(F_{l,E}))\cong R\Hom_{ \cU}(\Psi(F_{l',E'}),F_{l,E})$, which 
%is isomorphic to $R\Hom_{\cU}(F_{l',E'},F_{l,E})$ 
%by Corollary~\ref{asdfsdvsdb} (since $m\leq \frac{r+s}{2}$ by Lemma \ref{critical}).
It remains to show that 
$$R\Hom_{\M_{p,q+1}}(F_{l',E'},F_{l,E})\cong R\Hom_{\M_{p,q+1}}(F_{l',E'},\beta^*F_{l,E}).$$ 
By Proposition \ref{sifjkdfjkwdjk} and 
Corollary~\ref{sbvmsfbv}, the bundles $F_{l,E}$ and $\beta^*F_{l,E}$ on $\M_{p,q+1}$
are related by quotients in $\cA$, unless $F_{l,E}$ is listed in Corollary~\ref{sbvmsfbv}. Since 
the bundles $F_{l',E'}$ are perpendicular to $\cA$ by Proposition~\ref{perpendicular general}, we only  
need to show that 
$R\Hom_{\M_{p,q+1}}(F_{l',E'},\cO_{\de_T}(-a,-b))=0$,
where $a$ is arbitrary, %$a=0$, 
$b={r+s\over2}$. 
By~ Proposition~\ref{perpendicular general} (with $\mu=\frac{r+s}{2}-1$), this holds unless
$\left| f_{T,E',l'}\right|$ or 
$\left| f_{T^c,E',l'}\right|$ 
equals $\frac{r+s}{2}$.
This means that $(l',E')$  should be 
listed in Lemma \ref{critical}, with a critical subset either $T$ or $T^c$.
\ep

We now specialize to $\cU=\cU_{p,q}$ of Lemma~\ref{fbadfbadhadthn}. Recall that $N=P\cup Q\cup\{z\}$, and we denote as usual $E_p=E\cap P$,
$E_q=E\cap (Q\cup\{z\})$. 

\begin{lemma}\label{fbfgafgadf}
Assume $s\geq0$. Consider the following collections of vector bundles $F_{l,E}$ on $\cU_{p,q}$ (see Definition \ref{FlE}),
where we assume that $l+e$ is even:
$$l+\min(e_p, p+1-e_p)\le r-1\qquad\hbox{\rm (group $1A$)},$$
$$l+\min(e_p+1, p-e_p)\le r-1\qquad\hbox{\rm (group $1B$)},$$

The vector bundles from group $1A$ (resp., $1B$)
form an exceptional collection on~$\cU_{p,q}$. 
The order is as follows: we put all subsets $E$ containing $z$ last and before them all subsets not containing $z$.
Within each of these two blocks, we order
first by increasing $e_q$, arbitrarily if $e_q$ is the same but $E_q$'s are different, if $E_q$'s 
are the same, in order of increasing $e_p$, and for a given $e_p$, arbitrarily. 
\end{lemma}

\bp
It follows from Theorem \ref{stackofbundles} that we have
\bi
\item[(i) ] If $z\notin E$, then $F_{l,E}\cong\pi^*F_{l,E}$;
\item[(ii) ] If $z\in E$, %letting $E^c=N\setminus E$, we have that
$F_{l,E}\cong F_{l,E^c}^\vee\otimes F_{0,N}\cong\pi^*F^\vee_{l,E^c}\otimes F_{0,N}.
%,\quad R\pi_*(F^\vee_{l,E})=0.
$
\ei
In other words, the part of the collection with $z\notin E$ is identical to the one in Theorem~\ref{asfffdvzsfvsfb} (and Remark \ref{1B stuff}), while the part with 
$z\in E$ corresponds to vector bundles $F_{l,E}\cong F^\vee_{l,E^c}\otimes F_{0,N}$ 
with the collection $\{F^\vee_{l,E^c}\}$ being the dual collection 
to the one in Theorem~\ref{asfffdvzsfvsfb} (and Remark \ref{1B stuff}). 
By Orlov's theorem on the derived category of a projective bundle applied to $\pi: \cU_{p,q}\ra\M_{p,q}$, we have that
$D^b(\cU_{p,q})=\langle D^b(\M_{p,q}), \ D^b(\M_{p,q})\otimes F_{0,N}\rangle$, since $F_{0,N}=\cO_\pi(1)\otimes\pi^*L$, for some 
line bundle $L$ on $\M_{p,q}$. 
This is because $F_{0,N}=\cO_\pi(a)\otimes\pi^*L$, for some $a\in\ZZ$ and some line bundle $L$ on $\M_{p,q}$ (as  $\cU_{p,q}\ra \M_{p,q}$ is a $\PP^1$-bundle), but by
Theorem \ref{stackofbundles} we have $R\pi_*(F^\vee_{0,N})=0$, so we must have $a=1$. 
Now use the collection $F_{l,E}$ of Theorem~\ref{asfffdvzsfvsfb} for the block $D^b(\M_{p,q})$ and the dual collection $F^\vee_{l,E}$
(tensored with $F_{0,N}$) for the block $D^b(\M_{p,q})\otimes F_{0,N}$.
 \ep

\bp[Proof of Theorem~\ref{r+s odd}]
We will compare the vector bundles 
$F_{l,E}$ on $\cU_{p,q}$ and $\M_{p,q}$. 
The range for the pairs $(l,E)$ for collections $1A$ and  $1B$ on $\M_{p,q+1}$ 
in Theorem ~\ref{asdvzsfvsfb} is precisely the range for these collections on $\cU_{p,q}$ in Lemma~\ref{fbfgafgadf}. 

Consider first the case $s\geq0$. 
Let $F_{l,E}$, $F_{l',E'}$ be bundles from one of the two collections (either $1A$ or $1B$) on $\M_{p,q+1}$ (in the correct order).  
If $E_q\ne E'_q$, choose $z\in E'_q\setminus E_q$. By Corollary  ~\ref{aksjrgnkjdg}, we have 
$R\Hom_{\M_{p,q+1}}(F_{l',E'},F_{l,E})\cong R\Hom_{\cU_{p,q}}(F_{l',E'},F_{l,E})$,
since the only exceptions happen when $E_q=E'_q$ or 
$E_q$ and $E'_q$ are disjoint and have $s+1$ elements, but then
$z$ must be in both $E_q$ and $E_q'$ (use that $z\in T$). By Lemma~\ref{fbfgafgadf}, 
we have $R\Hom_{\cU_{p,q}}(F_{l',E'},F_{l,E})=0$.
%$F^\vee_{l,E}=\pi^*F^\vee_{l,E}$ and
%$R\pi_*{F^\vee_{l',E'}}=0$, for $\pi:\,\cM=U_{p,q}\to\M_{p,q}$ the universal $\bP^1$-bundle.
%Hence,
%$$R\Hom_{\cU_{p,q}}(F^\vee_{l,E},F^\vee_{l',E'})=R\Hom_{\cU_{p,q}}(\pi^*F^\vee_{l,E},F^\vee_{l',E'})=$$
%$$=R\Hom_{\M_{p,q}}(F^\vee_{l,E},R\pi_*F^\vee_{l',E'})=0.$$

Suppose now that $E_q=E_q'$. If $e_q<q+1$, let $z$ be any light index not in~$E_q$, while if $e_q=q+1$, 
let $z$ be any light index. Then Corollary ~\ref{aksjrgnkjdg} applies:
$R\Hom_{\M_{p,q+1}}(F_{l',E'},F_{l,E})\cong R\Hom_{\cU_{p,q}}(F_{l',E'},F_{l,E})$,
since the only exception for $E_q=E'_q$ is when $e_q=s+1$ and $E_q$ contains~$z$. 
The latter group is equal to $0$ or $\CC$ when $(l,E)=(l',E')$, by Lemma~\ref{fbfgafgadf}.

Consider now the case $s=-1$. In this case $E:=E_p$, $E':=E'_p$ are subsets of $P$. 
We first prove the statement for group $1A$. 
Note that if $(l,E)$ is in group $1A$ on $\M_{p}$, it is also in group $1A$ on $\M_{p,1}$, and in the collection on $\M_{p+1}$. 
Consider the set-up in Section \ref{induction2}: let $\al: W\ra\M_{p}$ be the universal family and let
$f:W\ra\M_{p,1}\cong\M_{p+1}$ be the birational morphism that contracts the boundary divisors $\de_{T\cup\{y\},T^c}$ by identifying 
the points in $T^c$.  Using that $F_{l,E}\cong \al^*F_{l,E}$ (Lemma \ref{zzfbzdfbdff}), 
$R\Hom_{\M_p}(F_{l',E'},F_{l,E})\cong \break R\Hom_{W}(\al^*F_{l',E'},\al^*F_{l,E})\cong R\Hom_{W}(F_{l',E'},F_{l,E})$,
which is isomorphic to $R\Hom_{\M_{p,1}}(F_{l',E'},F_{l,E})$ by Lemma~\ref{alpha game}. As $\M_{p+1}\cong\M_{p,1}$ 
and the two spaces have the same universal family, 
we have $R\Hom_{\M_{p,1}}(F_{l',E'},F_{l,E})=0$ if $e'\geq e$, $(l,E)\neq (l',E')$, and 
$\CC$ when $(l,E)=(l',E')$, by Theorem \ref{symm}.

We now prove the statement for group $1B$ when $s=-1$. Recall that  
$z\in P$, $\M_{p,-1}=\M_{p-1}$ and $\pi:\cU_{p,-1}=\cU=\M_{p-1,1}\ra\M_{p,-1}$ is the universal family. 
Note,  the previous proof works only if
none of $(l,E)$, $(l',E')$ satisfy $l+(p-e)=r-1$. We therefore proceed with a different proof for group $1B$.
 As in the case $q>0$, by Corollary  ~\ref{aksjrgnkjdg}, if $z\notin E$ we have 
$R\Hom_{\M_p}(F_{l',E'},F_{l,E})\cong R\Hom_{\M_{p-1,1}}(F_{l',E'},F_{l,E})$,
since the only exceptions for group $1B$ happen when $z\in E$. 
If $E'\setminus E\neq\emptyset$, we pick $z\in E'\setminus E$. Since $F_{l,E}\cong\pi^*F_{l,E}$ and 
$R\pi_*(F^\vee_{l',E'})=0$  by Theorem \ref{stackofbundles}, we have 
$R\Hom_{\M_{p-1,1}}(F^\vee_{l,E},F^\vee_{l',E'})\cong R\Hom_{\M_{p-1,1}}(\pi^*F^\vee_{l,E},F^\vee_{l',E'})\cong 
R\Hom_{\M_{p-1}}(F^\vee_{l,E},R\pi_*(F^\vee_{l',E'}))=0.$
If $E'\subseteq E$, since $e'\geq e$, we must have $E'=E$. If $E\neq N=P$, 
let $z\in N\setminus E$. Since  $F_{l,E}\cong\pi^*F_{l,E}$,  $F_{l',E'}\cong\pi^*F_{l',E'}$  we have 
$R\Hom_{\M_{p-1,1}}(F^\vee_{l,E},F^\vee_{l',E'})\cong R\Hom_{\M_{p-1}}(F^\vee_{l,E},F^\vee_{l',E'})$.
As $(l,E)$ is in group $1B$, we have 
$l+\min\{e+1,p-e\}\leq r-1$,
i.e., $l+\min\{e,(p-1)-e\}\leq r-2$, which is the range of pairs $(l,E)$ in Theorem \ref{symm}
with $E\subseteq P\setminus\{z\}$. The result follows in this case from Theorem \ref{symm}.

If $E=E'=N$, by Corollary  ~\ref{aksjrgnkjdg}, we still have 
$R\Hom_{\M_p}(F_{l',E'},F_{l,E})\cong\break
R\Hom_{\M_{p-1,1}}(F_{l',E'},F_{l,E})$,
unless $l=l'=r-1$ (and $r$ is odd since $l+e$ is even). Assume this is not the case. We have
$R\Hom_{\M_{p-1,1}}(F_{l',E'},F_{l,E})\cong R\Hom_{\M_{p-1,1}}(F^\vee_{l',\emptyset},F^\vee_{l,\emptyset})\cong
R\Hom_{\M_{p-1}}(F_{l,\emptyset},F_{l',\emptyset})$
(use that $F_{l,E}=F^\vee_{l,E^c}\otimes F_{0,N}$ on $\M_{p-1,1}$ by Corollary \ref{complement} and 
$F_{l,\emptyset}=\pi^*F_{l,\emptyset}$) 
and we are done by Theorem \ref{symm}. 
It remains to prove that the vector bundle $F_{r-1,N}$ ($r$ odd) is exceptional on $\M_p$. 
By Lemma \ref{property}, we have that $F_{r-1,N}\cong F_{r-1,\emptyset}\otimes F_{0,N}$. Hence, it suffices to prove that 
$F_{r-1,\emptyset}$ is exceptional on $\M_p$.
But this bundle is in group 1A, so we are done. 
%As in the case of group $1A$, we have $R\Hom_{\M_p}(F_{r-1,\emptyset}, F_{r-1,\emptyset})=R\Hom_{\M_{p,1}}(F_{r-1,\emptyset}, F_{r-1,\emptyset})$, which equals $\CC$ by Theorem  \ref{symm}.
\ep

%%%%%%%%%%%%%%%%%%%%%%%%%%%%%%%%%%%%%%%%%%%%%%%%%%%%%%%%%%%%%%%%%%%%%%%%
%%%%%%%%%%%%%%%%%%%%%%%%%%%%%%%%%%%%%%%%%%%%%%%%%%%%%%%%%%%%%%%%%%%%%%%%
%%%%%%%%%%%%%%%%%%%%%%%%%%%%%%%%%%%%%%%%%%%%%%%%%%%%%%%%%%%%%%%%%%%%%%%%
%%%%%%%%%%%%%%%%%%%%%%%%%%%%%%%%%%%%%%%%%%%%%%%%%%%%%%%%%%%%%%%%%%%%%%%%
%%%%%%%%%%%%%%%%%%%%%%%%%%%%%%%%%%%%%%%%%%%%%%%%%%%%%%%%%%%%%%%%%%%%%%%%

%%%%%%%%%%%%%%%%%%%%%%%%%%%%%%%%%%%%%%%%%%%%%%%%%%%%%%%%%%%%%%%%%%%%%%%%%%%%%%%%%%%%%%%

\section{Equivariant exceptional collection on $\M_{2r,2s+1}$} 
\label{asjbqlkwjbfkjwb}\label{exceptional p,q section}

The goal of this section is to complete the proof 
of semi-orthogonality of objects in 
the exceptional collection on $\M_{p,q}$ of
Theorem \ref{p,q case}. This collection contains objects of two types:

\begin{itemize}
    \item The vector bundles $F_{l,E}$. Semi-othogonality of these bundles
    was proved in Theorem~\ref{r+s odd}.
    \item The torsion sheaves  $\cT_{l,E}$. We prove semi-orthogonality of
    pairs of objects involving $\cT_{l,E}$ in Theorem~\ref{torsion pairs} by reducing it to a windows calculation on subvarieties $Z_R\subseteq \M_{p,q}$ which support these  torsion sheaves, 
and their intersections. 
\end{itemize}

\begin{notn}
Throughout this section, let $|P|=p=2r$, $|Q|=q=2s+1$ ($r\geq2, s\geq0$). Let
$P=\{1,\ldots, p\}$, $R=\{1,\ldots, r\}$, $R'=P\setminus R$. 
Let $Z_R$ be the locus in $\M_{p,q}$ where points in $R$ come together:
$Z_R=\de_{12}\cap\de_{13}\cap\ldots\cap\de_{1r}\cong \M_{\{u\}\cup R'\cup Q}$,
where $u$ is the marking corresponding to the combined points in $R$. Note that $u$ can only coincide with at most $s$ light points 
and with none of the points in~$R'$.  In $D^b(\M_{p,q})$, we have a Koszul resolution
\begin{equation}\label{KoszulCCC}
\cO_{Z_R}\cong\Big[\ldots\to\Lambda^2\bigoplus_{i=2}^r\cO(-\delta_{1i})\to\bigoplus_{i=2}^r\cO(-\delta_{1i})\to\cO_{\M_{p,q}}\Big].
\end{equation}
Let $Y=Z_R\cap Z_{R'}\cong \M_{\{u,v\}\cup Q}$, where $u$ (resp., $v$) is the marking corresponding to $R$ (resp., $R'$). 
We have a Cartesian diagram of  embeddings
\begin{equation}\label{SDvSDgEG}
\begin{tikzcd}
	Y & {Z_{R'}} \\
	{Z_{R}} & {\M_{p,q}}
	\arrow["i", hook, from=2-1, to=2-2]
	\arrow["{i'}", hook', from=1-2, to=2-2]
	\arrow["{j'}", hook, from=1-1, to=1-2]
	\arrow["j", hook', from=1-1, to=2-1]
\end{tikzcd}
%$i: Z_R\hra \M_{p,q}$, $i': Z_{R'}\hra \M_{p,q}$, $j: Y\hra Z_R$, $j': Y\hra Z_{R'}$. 
\end{equation}
\end{notn}

For any two choices $R_1$ and $R_2$, the sheaves $\cT_{l,E_1}$ and $\cT_{l,E_2}$ are perpendicular because their supports $Z_{R_1}$ and $Z_{R_2}$ are disjoint, except possibly if $R_1=R_2$ or $R_1=R^c_2$. Therefore, up to rearranging, the only cases we need to consider are $R_1=R_2=R$ and $R_1=R$, $R_2=R'$, with $R=\{1,\ldots, r\}$ and $R'=\{r+1,\ldots, p\}$.  
%A torsion sheaf $\cT_{l,E}$ (see Notation~\ref{T}) has disjoint support from, and is therefore perpendicular to, any  sheaf $\cT_{l',E'}$ unless $E_p=E'_p$ or $E_p\cap E_p'=\emptyset$. 
The semi-orthogonality of objects in Theorem \ref{p,q case} follows now from Theorem~\ref{r+s odd} and the following theorem, which takes care of the torsion sheaves:
\begin{thm}\label{torsion pairs}
Assume $p=2r$, $q=2s+1$. In the notations of Theorem  \ref{p,q case}:

(1)  If $(l,E)$ is in group $1A$ (resp., $1B$) and $(l',E')$ is in group $2$, we have
$R\Hom_{\M_{p,q}}(F_{l,E},\cT_{l',E'})=0$ if $e_q\geq e'_q$.

(2) If $(l,E)$ is in group $2$ and  $(l',E')$ is in group $1A$ (resp., $1B$), we have
$R\Hom_{\M_{p,q}}(\cT_{l,E},F_{l',E'})=0$ if $e_q> e'_q$ or if $e_q=e'_q$, $E_q\neq E'_q$. 

(3) 
$R\Hom_{\M_{p,q}}(\cT_{l,E}, \cT_{l',E'})=0$ if 
$(l,E)$, $(l',E')$ are both in group $2$, $E_p=E'_p=R$, 
$e_q\geq e'_q$, $E_q\neq E'_q$ or  $E_q=E'_q$, $l<l'$.
Also, $R\Hom_{\M_{p,q}}(\cT_{l,E}, \cT_{l,E})=\CC$. 

(4) If $(l,E)$, $(l',E')$ are both in group $2$ and $E_p=R$, $E'_p=R'$, we have
$R\Hom_{\M_{p,q}}(\cT_{l,E}, \cT_{l',E'})=0$ if $e_q\geq e'_q$.
\end{thm}

We start by computing some tautological classes of $Z_R$.

\begin{lemma}\label{psiX} We have the following
identities:
\bi
\item[(i) ] In $\Pic(Z_R)$: ${\psi_j}_{|Z_R}$ equals $\psi_u$ if $j\in R$, $-\psi_u$ if $j\in R'$, and $-\psi_u-2\de_{ju}$
if $j\in Q$. 
\item[(ii) ] In $\Pic(Z_R)$: ~${\de_{ij}}_{|Z_R}$ equals $-\psi_u$ if $i, j\in R$, $\psi_u$ if $i,j\in R'$, $0$ if $i\in R$, $j\in R'$, and $\de_{ij}$ if $\{i,j\}\cap Q\neq\emptyset$ (where if $j\in R$, $i\in Q$, we identify $\de_{ij}=\de_{iu}$). 
\item[(iii) ] In $\Pic(Y)$ we have that $\psi_u=-\psi_v$.
\item[(iv) ] In $\Pic(Z_R)$ we have that $K_{|Z_R}=K_{Z_R}+(r-1)\psi_u$,
where $K$ %(resp., $K_{Z_R}$) 
is the canonical class of $\M_{p,q}$. % (resp., $Z_R$). 
\item[(v) ]  If $E_p=R$ then 
$$\cT_{l,E}={i_R}_*\si_u^*\bigl(\omega_{\pi_R}^{\frac{e-l}{2}}(E)\bigr)\cong$$
$$\cong{i_R}_*\big(\frac{e_q-r-l}{2}\psi_u+\sum_{j\in E_q}\de_{ju}\big)
={i_R}_*\big(-\frac{r+l}{2}\psi_u-\frac{1}{2}\sum_{j\in E_q}\psi_j\big).$$
\ei
\end{lemma}

\bp
It follows from the definition that the restriction of $\psi_j$ to $Z_R$ is the class $\psi_j$ on the corresponding Hassett space for all $j$. Note that the universal family over $Z_R$ is  a 
$\PP^1$-bundle. It follows from \cite[Lemma 2.1]{CT_part_Ib} that $\psi_i+\psi_j\sim-2\de_{ij}$. Since $\de_{uj}=\emptyset$ if $j\in R'$,  parts (i), (ii) follow. 
If $j\in R$, we have by (i) that ${\psi_j}_{|Z_R}=\psi_u$, hence, ${\psi_j}_{|Y}=\psi_u$, and similarly, ${\psi_j}_{|Z_{R'}}=-\psi_v$, ${\psi_j}_{|Y}=-\psi_v$, which implies part (iii). Part (iv) follows from adjunction since 
$c_1(N_{Z_R|\M_{p,q}})\sim\sum_{j=2}^r(\de_{1j})_{|Z_R}\sim-(r-1)\psi_u$.
Part (v) follows from definition of $\cT_{l,E}$ (Notation~\ref{T}), previous parts 
and the fact that $\si_u^*\si_u=-\psi_u$. 
\ep

We reduce Theorem~\ref{torsion pairs} to a calculation of $R\Gamma$  on $Z_R$ or $Y$:

\begin{lemma}\label{sDGSGASG}
Theorem~\ref{torsion pairs} follows from (1)--(4) below, where
we abuse notation and denote by $\cT_{l,E}$ both a line bundle on $Z_{E_p}$ and its
pushforward to $\M_{p,q}$.

(1) $R\Ga\big(Z_R, ({F_{l,E}}_{|Z_R})^\vee\otimes\cT_{l',E'})=0$ if
$(l,E)$ is in group $1A$ (resp., $1B$),\break 
$(l',E')$ is in group $2$, $e_q\geq e'_q$, and $R=E'_p$.

(2) $R\Ga\big(Z_R, \cT_{l,E}^\vee\otimes {F_{l',E'}}_{|Z_R}\otimes c_1(N_{Z_R|\M_{p,q}})\big)=0$ 
if $(l,E)$ is in group $2$,\break 
$(l',E')$ is in group $1A$ (resp., $1B$) if $R=E_p$ and either $e_q> e'_q$ or $e_q=e'_q$, $E_q\neq E'_q$. 

(3) Suppose $(l,E)$, $(l',E')$ are both in group $2$,  $E_p=E'_p=R$, and $J\subseteq R\setminus\{1\}$.
Then 
$R\Gamma\big(Z_R, \cT_{l,E}^\vee\otimes\cT_{l',E'}\otimes(\sum_{j\in J}\de_{1j}\big))=0$
if  $e_q\geq e'_q$,  $E_q\neq E'_q$, or if $E_q=E'_q$, $l<l'$. Furthermore, 
$R\Gamma\big(Z_R, \cT_{l,E}^\vee\otimes\cT_{l,E}\otimes(\sum_{j\in J}\de_{1j}\big))=0$
if  $J\neq\emptyset$. 

(4) $R\Gamma\left(Y, \left(\frac{l+l'}{2}+1\right)\psi_u+\frac{1}{2}\sum_{j\in E_q}\psi_j-\frac{1}{2}\sum_{j\in E'_q}\psi_j\right)=0$ if $(l,E)$, $(l',E')$ are both in group $2$, $E_p=R$, $E'_p=R'$, and $e_q\geq e'_q$.
\end{lemma}

\bp 
(1) 
$R\Hom_{\M_{p,q}}(F_{l,E},\cT_{l',E'})\cong %R\Hom_{Z_R}({F_{l,E}}_{|Z_R},\cT_{l',E'})=
R\Ga\big(Z_R, ({F_{l,E}}_{|Z_R})^\vee\otimes\cT_{l',E'})$.

(2) Using that  $i^{!}F_{l',E'}\cong i^*F_{l',E'}\otimes c_1(N_{Z_R|\M_{p,q}})$, 
%where $i: Z_R\hra \M_{p,q}$ is the  embedding,
we have (up to a shift) that 
$R\Hom_{\M_{p,q}}(\cT_{l,E}, F_{l',E'})\cong R\Hom_{Z_R}(\cT_{l,E}, {F_{l',E'}}_{|Z_R}\otimes c_1(N_{Z_R|\M_{p,q}}))\cong\break
R\Ga\big(Z_R, \cT_{l,E}^\vee\otimes {F_{l',E'}}_{|Z_R}\otimes c_1(N_{Z_R|\M_{p,q}})\big)$
(\cite[Corollary  3.38]{Huybrechts}). 

(3)
Follows from tensoring the Koszul resolution \eqref{KoszulCCC} of $Z_R$ with a line bundle $L$ on $\M_{p,q}$ 
such that ${i}_*(L_{|Z_R})\cong \cT_{l,E}$ and applying $R\Hom(-,\cT_{l',E'})$. 

(4) %Consider  canonical embeddings $i: Z_R\hra \M_{p,q}$, $i': Z_{R'}\hra \M_{p,q}$, $j: Y\hra Z_R$, $j': Y\hra Z_{R'}$. 
For any line bundles $L$ on $Z_R$ and $L'$ on $Z_{R'}$, by cohomology and base change applied to the Tor-independent diagram \eqref{SDvSDgEG}, 
$R\Hom(i_*L, i'_*L')=R\Hom(L{i'}^*i_*L, L')\cong
R\Hom(Rj'_*{j}^*L, L')\cong
R\Hom_Y({j}^*L, j'^{!}L')$. 
It follows that  
$R\Hom_{\M_{p,q}}(\cT_{l,E}, \cT_{l',E'})\cong
R\Hom_Y({\cT_{l,E}}_{|Y}, 
{\cT_{l',E'}}_{|Y}\otimes c_1(N_{Y|Z_{R'}}))\cong\break
R\Gamma\big(Y, ({\cT_{l,E}}_{|Y})^\vee\otimes  {\cT_{l',E'}}_{|Y}\otimes (\sum\limits_{j\in R\setminus\{1\}}\de_{1j})_{|Y}\big)$.
By Lemma \ref{psiX}, this complex is isomorphic to 
$R\Gamma\big(Y, (\frac{l+l'}{2}+1)\psi_u+
\frac{1}{2}\sum\limits_{j\in E_q}\psi_j-\frac{1}{2}\sum\limits_{j\in E'_q}\psi_j)$.
\ep

We prove the vanishing in Lemma~\ref{sDGSGASG} by windows calculations on $Z_R$ and~$Y$.

\begin{notn}
$Z_R$ is isomorphic to a GIT quotient of  $X=%(\PP^1)^{r+q+1}=
\PP_u^1\times(\PP^1)^r\times (\PP^1)^q$ by $\PGL_2$, 
corresponding to the partition $\{u\}\sqcup R'\sqcup Q$ of the markings on 
$Z_R=\M_{\{u\}\cup R'\cup Q}$. 
%Then $Z_R$ is a GIT quotient of $X$ by $\PGL_2$. 
For $a\in\ZZ$, $A\subseteq R'$, $B\subseteq Q$, consider on $X$ the line bundle 
$\cO(a, A, B):=pr_1^*\cO(a)\otimes pr_{R'}^*\cO(A)\otimes pr_{Q}^*\cO(B)$
where we denote $\cO(A)=\cO(i_1,\ldots, i_r)$ with $i_j=1$ if $j\in A$ and $0$ otherwise, 
and likewise for
$\cO(B)$.
We  denote  by $\cO(-A)$ the dual of $\cO(A)$. 
We have the following correspondence between vector bundles on $Z_R$ and $\PGL_2$-equivariant vector bundles on $X=(\PP^1)^{r+q+1}$:
${F_{l,E}}_{|Z_R}\cong\cO(|E_p\cap R|, E_p\cap R', E_q)\otimes V_l$ (by Lemma~\ref{akjsdhfkjsgf}),
$\cT_{l,E}\cong\cO(r+l,0, E_q)$ if $E_p=R$,
$\omega_{Z_R}\cong\cO(-2,-2,\ldots, -2)$. Recall from Section \ref{fullness odd p section} 
(see Remark \ref{dictionary}) 
that $\psi_i\cong\cO(0,\ldots,0,-2,0,\ldots,0)$ (with $-2$ in position $i$) and 
$\de_{ij}\cong\cO(0,\ldots0,1,0,\ldots,0,1,0,\ldots, 0)$ (with $1$ in positions $i$ and $j$). 
Likewise,  $Y$ is isomorphic to the GIT quotient of $X'=(\PP^1)^{q+2}=\PP^1_u\times\PP^1_v\times (\PP^1)^q$, 
corresponding to the partition $\{u\}\sqcup\{v\}\sqcup Q$ of the markings on $\M_{\{u,v\}\cup Q}$.
\end{notn}

Instead of Theorem  \ref{Kirwan} we are going to use a related Theorem~\ref{Teleman}
proved for vector bundles by Teleman~\cite{Teleman}.
The calculation of the Kempf-Ness stratification and the %corresponding 
weights is as in Section \ref{windows}.
However, from now on, we will follow a convention of \cite{Teleman} and take an opposite weight to \eqref{DHLconvention}:
\begin{equation}\label{TelemanConvention}
\wt_{\la}\cO_X(a_1,\ldots, a_k)_{|z_K}=\sum_{i\in K}a_i-\sum_{i\in K^c}a_i.
\end{equation}
We find this convention more natural since
the ample polarization of the GIT quotient has positive weights on the unstable locus.
Then we have

\begin{thm}\cite[Th.~3.29]{DHL}\label{Teleman}
%In the notations of Theorem  \ref{Kirwan}, %
Choose a Kempf--Ness stratification of the unstable locus $X^{us}$ with data $Z_i$, $S_i$, $\la_i$. Let 
$\eta_i=\wt_{\la_i}\det\big(N^\vee_{S_i|X}\big)_{|Z_i}$. Then for any object $F\in D^b[X/G]$ such that 
\begin{equation}\label{asrgasrgasrh}
\hbox{\rm 
$\cH^*(\si_i^*F)$ has weights $>-\eta_i$}
\end{equation}
for every Kempf--Ness stratum $S_i$, 
%$\cH^*(\si_i^*F)$ has weights $<\eta_i$.
we have 
$R\Ga_{[X/G]}(F)\cong R\Ga_{[X^{ss}/G]}(F)$.
\end{thm}

Recall from the proof of Theorem 
\ref{Kirwan_special_case} that  when 
$X=(\PP^1)^n$, $G=\PGL_2$, a Kempf--Ness stratification 
for a polarization $\cO(a_1,\ldots, a_n)$ is a union, over all $K\subseteq N$ such that $\sum\limits_{i\in K} a_i> 1$, of loci $S_K\subset \De_K$ in $(\PP^1)^n$, where $\Delta_K$ is the diagonal consisting of points $(p_1,\ldots, p_n)$ such that for all $i\in K$ the points $p_i$ are equal, and $S_K$ is the complement of all the smaller diagonals.  The weights 
have been computed in \eqref{eta weights}:
$$\eta_K:=\wt_{\la}{\det\big(N^\vee_{\De_K|(\PP^1)^n}\big)_{|z_K}}=2(|K|-1).$$ 

\begin{rmk}(The devil's trick)\label{devil}
The line bundle 
$\DD:=\cO(r,R',0)$ on $X$ %=(\PP^1)^{r+q+1}$
is trivial on $Z_R$, since for any $j\in R'$, on $Z_{R}=\M_{\{u\}\cup R'\cup Q}$ 
we have $\de_{uj}=\emptyset$, so the line bundle $\cO(1,0,\ldots, 0,1,0,\ldots,0)$ (with $1$ in the positions of $u$ and $j$) descends to the trivial line bundle on $Y$. 
Likewise, the line bundle $\DD:=\cO(1,1,0)$ on $X'$ descends to the trivial line bundle on $Y$.
Instead of proving vanishing in Lemma~\ref{sDGSGASG} directly, we will prove vanishing on $X$ (resp., $X'$) after tensoring 
with a high multiple of $\DD$ since it's this tensor product that will satisfy conditions of Theorem~\ref{Teleman}.
This is a useful observation 
for any GIT quotient $X/\!\!/G$ such that the unstable locus contains a divisorial
component.
\end{rmk}

\bp[Proof of Theorem \ref{torsion pairs}] We prove the vanishings in Lemma~\ref{sDGSGASG}. 
First we prove the $\PGL_2$-invariant vanishing  on $X=(\PP^1)^{1+r+q}$ (cases (1), (2) and (3))
and on $X'=(\PP^1)^{2+q}$ (case (4)) after tensoring a vector bundle with the devil line bundle $\DD^N$, $N\gg0$.
Later on we check the weight condition \eqref{asrgasrgasrh}.

For (1), assuming \eqref{asrgasrgasrh}, 
$R\Ga\big(Z_R, ({F_{l,E}}_{|Z_R})^\vee\otimes\cT_{l',E'}\otimes\DD^N)\cong
R\Ga\big(X,\cO(r+l'-|E_p\cap R|+Nr,-E_p\cap R'+NR', E'_q-E_q)\otimes V_l\big)^{\PGL_2}$,
which is clearly $0$ if $E_q\nsubseteq E'_q$. 
Since here we assume $e_q\geq e'_q$, we have that $E_q\subseteq E'_q$ if and only if $E_q=E'_q$. 
Assume $E_q=E'_q$. 
In this case, the $\SL_2$-module is equal to
$V_{r+l'-|E_p\cap R|+Nr}\otimes V_{N-y_1}\otimes\ldots\otimes V_{N-y_r}\otimes V_l$,
where $y_i=1$ if the corresponding index in $E_p\cap R'$ and $0$ otherwise.
We claim that the  $\PGL_2$-invariant part is~$0$. 
Since $(l,E)$ is in group $1A$ or $1B$, we have 
$l+|E_p\cap R|-|E_p\cap R'|<r$  by Corollary ~\ref{dfbdfbdfb}.
It follows that $r+l'-|E_p\cap R|+Nr>Nr-|E_p\cap R'|+l$, which implies the vanishing
by the Clebsch-Gordan formula (Lemma \ref{CG}).

For (2), assuming \eqref{asrgasrgasrh}, 
$R\Ga\big(Z_R, \cT_{l,E}^\vee\otimes {F_{l',E'}}_{|Z_R}\otimes c_1(N_{Z_R|\M_{p,q}})\otimes\DD^N\big)\cong
R\Ga\big(X,\cO(r-l+|E'_p\cap R|-2+Nr, E'_p\cap R'+NR', E'_q-E_q)\otimes V_{l'}\big)^{\PGL_2}$
(use $c_1(N_{Z_R|\M_{p,q}})=\cO(2r-2,0,0)$). This is $0$ since $E_q\nsubseteq E'_q$.  

For (3), assuming \eqref{asrgasrgasrh}, 
$R\Gamma\big(Z_R, \cT_{l,E}^\vee\otimes\cT_{l',E'}\otimes \sum_{j\in J}\de_{1j}\otimes\DD^N\big)$ is isomorphic to
$R\Gamma\big(X,\cO(2|J|-l+l'+Nr, +NR', E'_q-E_q)\big)^{\PGL_2}$
for all $J\subseteq R\setminus\{1\}$. This is~$0$ if $e_q\geq e'_q$, $E_q\nsubseteq E'_q$,
while if $E_q=E'_q$ the  $\PGL_2$-invariant part is $0$ when $l<l'$ or when $l=l'$, $|J|>0$ by the Clebsch-Gordan formula 
(Lemma \ref{CG}), since in these cases we have 
$2|J|-l+l'+Nr>Nr$.

For (4), assuming \eqref{asrgasrgasrh},
the question is equivalent to the vanishing for $N\gg0$ of the 
$\PGL_2$-invariant part of
$R\Gamma\big(\cO(-l-l'-2+N,N,E'_q-E_q)\big)$.
This is clear if $e_q\geq e'_q$, $E_q\nsubseteq E'_q$. If $E_q=E'_q$ then
this follows from the Clebsch-Gordan formula
(Lemma \ref{CG}),
since
$N>N-l-l'-2$.

We now check that for each stratum, each of the cases (1)-(4) fall under the assumption 
\eqref{asrgasrgasrh} on weights of Theorem  \ref{Teleman}.
Up to symmetry, the unstable loci in $X$ in cases (1)-(3) have the following form: 

(The locus $K_I$), for $I\subseteq Q$, $|I|\geq s+1$, where $u$ and the indices in $I$ come together. In this case,  $\eta=2|I|$.

(The locus $K_{J,I}$), for $J\subseteq R'$, $I\subseteq Q$, $J\neq\emptyset$, $|I|\geq 0$, where $u$ and the indices in $J$ and $I$ come together.
In this case, $\eta=2|I|+2|J|$.

(The locus $L_I$), for $I\subseteq Q$, $|I|\geq s+1$, where the indices in $R'$ and $I$ come together. In this case, $\eta=2|I|+2r-2$. 

The devil line bundle $\cO(r,R',0)$ has weight 
$r+|J|-|R'\setminus J|=2|J|>0$
on $K_{J,I}$,
while its weight for the other strata is~$0$.
Therefore, the condition \eqref{asrgasrgasrh} for the stratum $K_{J,I}$ can be achieved by tensoring with a high 
enough multiple of this line bundle. 
We only need to consider the remaining strata.

{\bf Strata $K_I$.} Assume $|I|\geq s+1$.  
In Case (1) we need to verify that the weights of 
$\cO(r+l'-|E_p\cap R|,-E_p\cap R', E'_q-E_q)\otimes V_l$ are $>-2|I|$. 
Since the weights of $V_l$ are between $-l$ and $l$, it suffices to prove that 
$r+l'-|E_p\cap R|+|E_p\cap R'|+|E'_q\cap I|-|E_q\cap I|-|E'_q\cap I^c|+|E_q\cap I^c|-l>-2|I|$.
By (\ref{another}) for the pair $(l,E_p)$, it suffices to show  
$l'+|E'_q\cap I|-|E_q\cap I|-|E'_q\cap I^c|+|E_q\cap I^c|\geq -2|I|$.
Since the left hand side equals 
$l'+(e_q-e'_q)-2|E_q\cap I|+2|E'_q\cap I|$ 
and we assume $e_q\geq e'_q$, the result follows from $-2|E_q\cap I|\geq -2|I|$. 

In Case (2), we need to check that 
$\cO(r-l+|E'_p\cap R|-2, E'_p\cap R', E'_q-E_q)\otimes V_{l'}$ has  weights
 $>-2|I|$. It suffices to prove that 
$r-l+|E'_p\cap R|-2-|E'_p\cap R'|+|E'_q\cap I|-|E_q\cap I|-|E'_q\cap I^c|+|E_q\cap I^c|-l'>-2|I|$.
Using (\ref{another}) for the pair $(l',E'_p)$, 
it suffices to prove that 
$-l+|E'_q\cap I|-|E_q\cap I|-|E'_q\cap I^c|+|E_q\cap I^c|\geq -2|I|+2$.
The left hand side is greater than
$-l-|E_q\cap I|+|E_q\cap I^c|-|I^c|$,
and this is $\geq  -2|I|+2$ by Lemma \ref{another2} applied to the pair $(l,E_q)$. 

In Case (3) we need to verify that the weight of 
$\cO(2|J|-l+l', 0, E'_q-E_q)$
is $>-2|I|$ for all $0\leq |J|\leq r-1$. Equivalently, we need to prove that 
$-l+l'+|E'_q\cap I|-|E_q\cap I|-|E'_q\cap I^c|+|E_q\cap I^c|>-2|I|$.
The left hand side is greater or equal than 
$-l-|E_q\cap I|+|E_q\cap I^c|-|I^c|$
and this is $> -2|I|$ by Lemma \ref{another2} applied to the pair $(l,E_q)$. 

%\smallskip

{\bf Strata $L_I$.} In Case (1), we need to show 
that 
$-r-l'+|E_p\cap R|-|E_p\cap R'|+|E'_q\cap I|-|E_q\cap I|-|E'_q\cap I^c|+|E_q\cap I^c|-l>-2|I|-2r+2$.
Using (\ref{another}) for the pair $(l,E_p)$, 
it suffices to prove that 
$-l'+|E'_q\cap I|-|E_q\cap I|-|E'_q\cap I^c|+
|E_q\cap I^c|\geq -2|I|+2$.
The left hand side is greater than
$-l'+|E'_q\cap I|-|E'_q\cap I^c|-|I|$,
hence, it suffices to show that 
$-l'+|E'_q\cap I|-|E'_q\cap I^c|\geq -|I|+2$,
but this follows from Lemma \ref{another2} applied to the pair $(l',E'_q)$.

In Case (2), we need 
$-(r-l+|E'_p\cap R|-2)+|E'_p\cap R'|+|E'_q\cap I|-|E_q\cap I|-|E'_q\cap I^c|+|E_q\cap I^c|-l'>-2|I|-2r+2$.
Using (\ref{another}) for the pair $(l',E'_p)$, 
it suffices to prove that 
$l+|E'_q\cap I|-|E_q\cap I|-|E'_q\cap I^c|+|E_q\cap I^c|\geq -2|I|$.
The left hand side is clearly greater than
$-|E_q\cap I|-|E'_q\cap I^c|\geq -|I|-|I^c|=-2s-1$,
and the inequality follows since $|I|\geq s+1$. 

In Case (3), we need to verify that for all $0\leq |J|\leq r-1$ we have 
$-2|J|+l-l'+|E'_q\cap I|-|E_q\cap I|-|E'_q\cap I^c|+|E_q\cap I^c|> -2|I|-2r+2$,
or equivalently,
$l-l'+|E'_q\cap I|-|E_q\cap I|-|E'_q\cap I^c|+|E_q\cap I^c|> -2|I|$.
The left hand side is clearly greater than
$-l'+|E'_q\cap I|-|E'_q\cap I^c|-|I|$,
which is $>-2|I|$ by Lemma  \ref{another2} applied to the pair $(l',E'_q)$ and $|I|\geq s+1$. 

%\smallskip

We now consider Case (4). Up to symmetry, the unstable loci in 
$X'=(\PP^1)^{q+2}=\PP^1_u\times\PP^1_v\times (\PP^1)^q$ have the following form: 

(The locus $K'_I$), for $I\subseteq Q$, $|I|\geq s+1$, where $u$ and the indices in $I$ come together. In this case, 
$\eta=2|I|$. 

(The locus $K''_{I}$), for $I\subseteq Q$, $|I|\geq s+1$, where $v$ and the indices in $I$ come together. In this case, 
$\eta=2|I|$. 

(The locus $K'''_I$), for $I\subseteq Q$, $|I|\geq0$, where $u$, $v$ and the indices in $I$ come together. In this case, 
$\eta=2|I|+2$.  

The devil line bundle $\cO(1,1,0)$ has weight $2>0$ on $K'''_I$
but its weight on other strata is $0$.
Therefore we only have to consider the strata $K'_I$, $K''_I$.  

For $K'_I$, we need to show that the weight of 
$\cO(-(l+l'+2),0, E'_q-E_q)$ is $>-2|I|$. Equivalently, 
$-l-l'-2+|E'_q\cap I|-|E_q\cap I|-|E'_q\cap I^c|+|E_q\cap I^c|>-2|I|$.
This follows from Lemma \ref{another2} applied to the pairs $(l,E_q)$ and $(l',E'_q)$. 
For~$K''_I$, we need 
$l+l'+2+|E'_q\cap I|-|E_q\cap I|-|E'_q\cap I^c|+|E_q\cap I^c|>-2|I|$,
which is clearly satisfied as the left hand side is greater than 
$-|E_q\cap I|-|E'_q\cap I^c|\geq -|I|-|I^c|>-2|I|$.
This completes the proof.
\ep

%%%%%%%%%%%%%%%%%%%%%%%%%%%%%%%%%%%%%%%%%%%%%%%%%%%%%%%%%%%%%%%%

\section{Equivariant exceptional collection on $\M_{2r,2s+2}$}\label{exceptional p,q+1}\label{exceptional p,q+1 section}

%We are going to apply Lemma~\ref{K_lemma} for $\cS=D^b(\M_{p,q+1})$, the usual duality functor ${}^\vee$ and  the s.o.d.'s~$\langle\cA,\cB\rangle$ of Notation~\ref{TT} and $\langle\cC,\cD\rangle$ defined below.

Here we will prove exceptionality of the collection of Theorem~\ref{asdvzsfvsfb} on $\M_{p,q+1}$
(see also Remark \ref{1B stuff}), for 
$p=2r\geq4$ and $q=2s+1\ge 1$, which contains objects of three types: 
\begin{itemize}[leftmargin=1.5em]
\item 
Torsion sheaves 
$\cO_\de(-a,-b)$ 
that generate the subcategory $\cA$ (see Notation~\ref{TT}).  
These sheaves form an exceptional collection 
(for example when arranged in order of decreasing $a+b$) 
by 
\cite[Lemma~4.9]{CT_part_Ib}. 
In~particular, the subcategories $\cA$ and $\cB={}^\perp\cA$ are admissible. 
We denote by $T_{\cB}$ the projection of an object $T$ in $D^b(\M_{p,q+1})$ to $\cB$. 
\item
Vector bundles $F_{l,E}$. They form an exceptional collection by Theorem~\ref{r+s odd}
and belong to the subcategory $\cB$ by 
Corollary \ref{F perp A}.
\item
Complexes $\cTT_{l,E}$ (Definition~\ref{sDGASRHADHA}). These complexes 
are projections of torsion sheaves $\cT_{l,E}$ (Notation~\ref{T})
to the subcategory $\cB$.
\end{itemize}
In Corollary~\ref{exceptionalityeveneven}, we will prove the remaining
semi-orthogonality of complexes $\cTT_{l,E}$ among themselves and with  $F_{l',E'}$.
%Throughout this section we assume %$p=2r\geq4$, $q=2s+1$,  $s\geq0$ (in the case of  $\M_p$ the collection of Theorem \ref{p,q+1 case} consists only of the bundles $\{F_{l,E}\}$, hence, we are done in that case). 
In Proposition~\ref{reduction}, we will reduce this calculation to a computation on a different Hassett space $\M_{R'}$.
The semi-orthogonality in $D^b(\M_{R'})$ will be checked in Theorem~\ref{pairs on M_R}.

%\begin{lemma}\label{afvargbaerh}
%The subcategory $\cA$ is admissible and $(S_p\times S_{q+1})$ equivariant.
%The subcategory  
%$\cB$ is an  $(S_p\times S_{q+1})$ equivariant non-commutative resolution of singularities
%of the GIT quotient $\X_{p,q+1}$ in the sense of \cite{KuLu}.
%\end{lemma}
%\bp This follows from \cite{KuznetsovExNC}, where arbitrary Segre cones were considered. Alternatively, it is easy to check that sheaves  that generate $\cA$ form an exceptional collection \cite[Lem.~4.19]{CT_part_Ib},  which implies admissibility.  \ep

\begin{lemdef}\label{lemdef}
Let $R\subseteq P$ be a subset of heavy indices, $|R|=r$,   $R'=P\setminus R$. 
There exists a Hassett space, denoted $\M_{R'}$, and a 
reduction morphism 
$$\be_{R'}: \M_{p,q+1}\ra \M_{R'}$$ that contracts   boundary divisors $\de_{T,T^c}\cong \bP^{r+s-1}\times\bP^{r+s-1}\subseteq \M_{p,q+1}$ with $T_p=R'$, via the projection onto the second factor and is an isomorphism elsewhere.
\end{lemdef}

\bp
For $\M_{p,q+1}$, we use weights
 $\frac{1}{r}-\epsilon_1+\epsilon_2$ for heavy points and 
 $\frac{2r}{2s+2}\epsilon_1+\epsilon_2$ for light points (Notation~\ref{precise weights}).
For $\M_{R'}$,  points in $R'$ have weight $\frac{1}{r}-\epsilon_1$,
points in $R$ have weight $\frac{1}{r}-\epsilon_1+\epsilon_2$, and light points have weight 
$\frac{2r}{2s+2}\epsilon_1$. It~easily follows that $\be_{R'}$ exists and has required properties.
\ep

\begin{notn}
Let $\tQ$ be the set of light indices, $|\tQ|=2s+2$.
For $E\subseteq P\cup \tQ$ and $l\geq0$, 
we  let $F_{l,E}$ be the vector bundle on $\M_{R'}$ defined in Section~\ref{extend section} 
(Definition~\ref{allF}).
For~$e_p=r$, let $\coT_{l,E}$  be the torsion sheaf on $\M_{R'}$ defined as in 
Notation \ref{T}.
(see also Lemma \ref{psiX}(v)). 
\end{notn}

\begin{lemma}\label{projections F}
Let $(l,E)$ be in groups $1A$ or $1B$. Then
\bi
\item[(1) ] $\big(L\be_{R'}^*F_{l,E}\big)_{\cB}\cong F_{l,E}$
except when 
either $E_p\subseteq R'$, $l+e_p=r-1$, $e_q=s+1$ (case $1A$)
or $R'\subseteq E_p$, $e_p=r+1+l$, $e_q=s+1$ (case $1B$), 
in which cases there is an exact triangle
\begin{equation}\label{sDVsdvsDB}
F_{l,E}\to \big(L\be_{R'}^*F_{l,E}\big)_{\cB}\ra Q_{\cB}\to,
\end{equation}
where $Q_{\cB}$ is the projection into $\cB$ of an object $Q\in D^b(\M_{p,q+1})$ which is 
generated by  the sheaves 
$$\cO_{\de_{T,T^c}}(u,-\frac{r+s}{2}) \ \text{with} \ T_p=R',\ T_q=E_q,\ 0\leq u<\frac{r+s}{2}.$$ 
\item[(2) ] 
$[R{\be_{R'}}_*F_{l,E}^\vee]^\vee\cong F_{l,E}$.
\ei
\end{lemma}

\bp
Part (1) follows from Proposition \ref{sifjkdfjkwdjk} 
and its proof, 
as well as Corollary \ref{sbvmsfbv}. 
Part (2) is a particular case of Corollary \ref{push beta},
since $m_{T,E,l}\leq (r+s)/2$ for all $T$ by (\ref{msnfbas}). 
\ep

\begin{prop}\label{reduction}
$R\Hom_{\M_{p,q+1}}(G',G)=R\Hom_{\M_{R'}}(\ov G',\ov G)$
in any of the cases:
\bi
\item[(1) ] $E_p=R$, $G=\cTT_{l,E}$, $G'=F_{l',E'}$, $\ov G=\coT_{l,E}$, $\ov G'=F_{l',E'}$;
\item[(2) ] $E'_p=R'$, $G=F_{l,E}$, $G'=\cTT_{l',E'}$, $\ov G=F_{l,E}$, $\ov G'=\coT_{l',E'}$.
\ei
In the remaining cases, $G=\cTT_{l,E}$, $G'=\cTT_{l',E'}$, $\ov G=\coT_{l,E}$, $\ov G'=\coT_{l',E'}$ and
\bi[resume]
\item[(3) ]  $E_p=R$, $E'_p=R$;
\item[(4) ]  $E_p=R$, $E'_p=R'$;
\ei
Here all pairs $(l,E)$ are in group $1A$ or $1B$ for vector bundles $F_{l,E}$ 
and in group $2A$ or $2B$ for torsion objects
$\cTT_{l,E}$ and $\coT_{l,E}$, and similar for $(l',E')$.
\end{prop}

\bp
In cases (1), (3) and (4), $G=\cTT_{l,E}$, which is isomorphic to $\big(L\be_{R'}^*\coT_{l,E}\big)_{\cB}$ by Proposition \ref{projections}(iii).
For each $G'$ in cases (1), (3) and (4), we also have $G'\in\cB$. 
By Lemma~\ref{AbstractPhi}, it follows that 
$$R\Hom_{\M_{p,q+1}}(G',G)\cong R\Hom_{\M_{p,q+1}}(G',\big(L\be_{R'}^*\coT_{l,E}\big)_{\cB})\cong$$
$$\cong R\Hom_{\M_{R'}}(\big(R{\be_{R'}}_*{G'}^\vee\big)^\vee,\coT_{l,E}).$$
%The identity $(({G'}^\vee)_{\cD})^\vee=L\be^*\ov G'$ is implied by $\R\be_*\big({G'}^\vee\big)={\ov G'}^\vee$, since for any object $T$, $T_{\cD}=L\be^*R\be_*T$. 
For each $G'$ in cases (1), (3) and (4),  $\big(R{\be_{R'}}_*{G'}^\vee\big)^\vee\cong\ov G'$ by Lemma \ref{projections F}(2) for $G'=F_{l',E'}$
and Proposition~\ref{projections}((ii) and (iv)) for  $G'=\coT_{l',E'}$.  

In case (2), we distinguish two cases. Assume first that $(l,E)$ is not one of the exceptions in Lemma \ref{projections F}(1).
%, i.e., we don't allow  $E_p\subseteq R'$, $e_p+l=r-1$, $e_q=s+1$ (case $1A$), or $R'\subseteq E_p$, $e_p=r+1+l$, $e_q=s+1$ (case $1B$).
Then, $(L\beta_{R'}^*F_{l,E})_\cB\cong F_{l,E}$ by this lemma. Lemma~\ref{AbstractPhi} gives
$R\Hom_{\M_{p,q+1}}(G',F_{l,E})\cong R\Hom_{\M_{R'}}(\big({R\be_{R'}}_*{G'}^\vee\big)^\vee,F_{l,E})$.
As by Proposition  \ref{projections}(ii), $\big({R\be_{R'}}_*{G'}^\vee\big)^\vee\cong{\ov G'}$ for 
$G'=\cTT_{l',E'}$, the result follows. 
 
Assume now that $(l,E)$ is one of the exceptions in Lemma \ref{projections F}(1). 
Arguing as above, 
we have 
$R\Hom_{\M_{p,q+1}}(\cTT_{l',E'},F_{l,E})\cong R\Hom_{\M_{R'}}(\coT_{l',E'},F_{l,E})$ 
by  \eqref{sDVsdvsDB} and
provided that we 
can prove the following claim:
\begin{claim}\label{qwhjegvfjgsEFV}
$R\Hom_{\M_{p,q+1}}\big(\cTT_{l',E'},\cO_{\de_{T,T^c}}(u,-\frac{r+s}{2})\big)=0$
if $(l',E')$ is in group $2A$ or $2B$, $E'_p=R'$,   
$0\leq u<\frac{r+s}{2}$ and boundary $\de_{T,T^c}$ satisfies $T_p=R'$.
\end{claim}
We will prove Claim \ref{qwhjegvfjgsEFV} and Proposition~\ref{projections}
which are used in this proof after writing an explicit resolution of $\cTT_{l',E'}$ (Lemma \ref{supportT}). 
These statements aside, this finishes the proof of Proposition \ref{reduction}. 
\ep

\begin{lemma}\label{supportT}
Let $(l,E)$ be in group $2A$ or $2B$ with $E_p=R=\{1,\ldots, r\}$.
The~complex  $\cTT_{l,E}=(\cT_{l,E})_{\cB}$ is isomorphic to the following complex, 
call it $\tcS_{l,E}$:
\begin{equation}\label{TTTT}
\LL_{R\setminus\{1\}}\ra\bigoplus_{J\subseteq R\setminus\{1\}, |J|=r-2} \LL_J\ra\ldots\ra \bigoplus_{j\in  R\setminus\{1\}} \LL_j\ra \LL_{\emptyset}
\end{equation}
 (with differentials described in the proof), 
where for any subset $J\subseteq R\setminus\{1\}$, we~let  
$$\LL_J=L'\Big(-\sum_{T_p=R}\al_{J,T,E,l}\de_{T, T^c}-\sum_{j\in J}\de_{1j}\Big)\quad\hbox{\rm and}\quad 
L'=\frac{e_q-r-l}{2}\psi_1+\sum_{k\in E_q} \de_{1k}.$$
Here
$\al_{J,T,l,E}=\max\Big\{\frac{e_q-r-l}{2}-|E_q\cap T_q|+|J|,0\Big\}$.
Furthermore, we have an exact triangle
$\cTT_{l,E}\ra\cT_{l,E}\ra Q\ra$,
with $Q$ generated by generators of the subcategory $\cA$ supported on the boundary divisors $\de_{T,T^c}$, $T_p=R$. 
%Furthermore, $\cTT_{l,E}\in\cB$. 
%In particular, $\cTT_{l,E}\cong(\cT_{l,E})_{\cB}$.
\end{lemma}

The $\al_{J,T,l,E}$ in Lemma \ref{supportT} is a variant of $\al_{T,E,l}=\max\Big\{\frac{e-l}{2}-|E\cap T|,0\Big\}$ (see (\ref{Sdgsgasrh})), as when $E_p=R$ we have 
$\frac{e-l}{2}-|E\cap T|=\frac{e_q-r-l}{2}-|E_q\cap T_q|$. The main property of the line bundle $L'$ that we use is that $\cT_{l,E}={i_R}_*\big(L'_{|Z_R}\big)$ (to see this, use the analogue of Lemma \ref{psiX}). 

\bp[Proof of Lemma \ref{supportT}]
A sheaf $\cT_{l,E}$ with $E_p=R$ is isomorphic in $D^b(\M_{p,q+1})$ to the Koszul resolution \eqref{KoszulCCC} of the stratum 
$Z_R$ %=\bigcap_{j\in R\setminus\{1\}}\de_{1j}$ 
tensored with the line bundle $L'$, a complex which we call $S_{l,E}$ (so $\cT_{l,E}\cong S_{l,E}$):
\begin{equation}\label{fsvsfbasrhar}
L_{R\setminus\{1\}}\ra\bigoplus_{J\subseteq R\setminus\{1\}, |J|=r-2} L_J\ra\ldots\ra \bigoplus_{j\in  R\setminus\{1\}} L_j\ra L_{\emptyset}.
\end{equation}
Here $L_J=L'\big(-\sum_{j\in J}\de_{1j}\big)$.
Let $\cA_R\subset D^b(\M_{p,q+1})$ be the admissible subcategory generated by sheaves in $\cA$ supported on $\de_{T,T^c}$
with $T_p=R$. Let $\cB_R={}^\perp\cA_R$.
Since
$L_J=\LL_J+\sum_{T_p=R}\al_{J,T,l,E}\de_{T,T^c}$,
the canonical morphisms $\LL_J\ra L_J$ have the cokernels generated by sheaves supported on 
boundary $\de_{T,T^c}$, with $T_p=R$. We claim that these cokernels are in $\cA_R$.
Indeed, by Lemma \ref{star}, the line bundles 
$\LL_J$ and $L_J$ are related by quotients of the form
$Q=\big(\LL_J+i\de\big)_{|\de}\cong\cO(-i, \al_{J,T,E,l}-i)$ for
$0<i\leq \al_{J,T,E,l}$,
where
$\al_{J,T,E,l}=\frac{e_q-r-l}{2}-|E_q\cap T|+|J|\leq \frac{e_q-l}{2}-|E_q\cap T|+\frac{r}{2}-1$,
and this is $<\frac{r+s}{2}$ by Corollary \ref{AA}(ii).
It follows that 
$(\LL_J)_{\cB_R}\cong(L_J)_{\cB_R}$.

Next we claim that $\LL_J\in\cB_R$, i.e., ~that
$R\Hom\big( \LL_J, \cO_{\de}(-a,-b) \big)=0$ 
for any $\cO_{\de}(-a,-b)$ as in Theorem \ref{p,q+1 case} with $T_p=R$. 
We have
$(\LL_J)^\vee\otimes\cO_\delta(-a,-b)\cong
\cO_\delta\left(\frac{e_q-r-l}{2}-|E_q\cap T|+|J|-\al_{J,T,l,E}-a,-\al_{J,T,l,E}-b\right)$.
Consider the case when $\al_{J,T,l,E}=0$. The sheaf is clearly acyclic if $b\neq0$. Assume $b=0$, in which case $0<a<\frac{r+s}{2}$. 
Since $\al_{J,T,l,E}=0$, we have that $|E_q\cap T|-\frac{e_q-r-l}{2}-|J|\geq0$. 
By Corollary \ref{AA}(i), 
$|E_q\cap T|-\frac{e_q-r-l}{2}\leq\frac{r+s}{2}$.
It follows that 
$0<|E_q\cap T|-\frac{e_q-r-l}{2}-|J|+a\leq r+s-1$ and the sheaf is acyclic.

Consider now the case when $\al_{J,T,l,E}>0$. In this case 
$(\LL_J)^\vee\otimes\cO_\delta(-a,-b)\cong\cO_\delta(-a,-\al_{J,T,l,E}-b)$.
Clearly, this is acyclic when $a>0$. Assume $a=0$, in which case $0<b<\frac{r+s}{2}$. Then it suffices to prove that  
$\al_{J,T,l,E}=\frac{e_q-r-l}{2}-|E_q\cap T|+|J|\leq \frac{r+s}{2}$
for all $J\subseteq R\setminus\{1\}$, 
or equivalently that 
$\frac{e_q-l}{2}-|E_q\cap T|\leq \frac{s}{2}+1$.
This holds by Corollary  \ref{AA}.

Now we use the following simple lemma:

\begin{lemma}\label{middle}
Let $A$ and $B$ be abelian categories.
Let $F:\,D^b(A)\to D^b(B)$ be an exact functor. Let $L^\bullet\in D^b(A)$ be a complex $[L_0\arrow{d_1}\ldots\arrow{d_r} L_r]$.
Suppose  $F(L_1),\ldots,F(L_r)\in B$.
Then $F(L^\bullet)\cong[F(L_0)\arrow{F(d_1)}\ldots\arrow{F(d_r)} F(L_r)]$.
\end{lemma}

\bp 
Since $B$ is a fully faithful subcategory of $D^b(B)$, the morphism $F(d_i)$ is a morphism in $B$ for every~$i=1,\ldots,r$
and so $[F(L_0)\arrow{F(d_1)}\ldots\arrow{F(d_r)} F(L_r)]$ is a complex in $D^b(B)$.
It is isomorphic to $F(L^\bullet)$ by induction on $r$ and by
applying the functor $F$ to the naive truncation of $L^\bullet$.
\ep

We apply Lemma \ref{middle} to the projection functor 
$$D^b(\M_{p,q+1})\to\cB_R\hookrightarrow D^b(\M_{p,q+1}),$$
where we recall that $\cB_R={}^\perp\cA_R$ and 
 $\cA_R\subset D^b(\M_{p,q+1})$ is the admissible subcategory generated by sheaves in $\cA$ supported on $\de_{T,T^c}$
with $T_p=R$.
It follows that $(\cT_{l,E})_{\cB_R}$ is isomorphic to the complex
\eqref{TTTT} with differentials  obtained by
applying the functor $T\to T_{\cB_R}$ to the differentials of the complex
$\cS_{l,E}$ from \eqref{fsvsfbasrhar}.
Concretely,  for $j\in J$, we have 
$L_{J\setminus\{j\}}=L_{J}(\de_{1j})$
and the differentials in  $\cT_{l,E}$ are built from the maps $\sigma:\,L_J\to L_{J\setminus\{j\}}$
given by multiplication with a canonical section of $\cO(\de_{1j})$.
On the other hand, 
$\LL_{J\setminus\{j\}}=\LL_{J}\bigl(\de_{1j}+\sum_{T_p=R,\atop \al_{J,T,l,E}>0}\de_{T,T^c}\bigr)$,
since
$\al_{J\setminus\{j_0\},T,l,E}=\al_{J,T,l,E}-1$ if $\al_{J,T,l,E}>0$ and $\al_{J,T,l,E}=0$
otherwise.
We claim that differentials in $(\cT_{l,E})_{\cB_R}$ 
are built from the maps $\tilde\sigma:\,\LL_J\to \LL_{J\setminus\{j\}}$
given by multiplication with a canonical section of $\cO(\de_{1j}+\sum_{T_p=R, \atop\al_{J,T,l,E}>0}\de_{T,T^c})$
(note that these maps are in $\cB_R$ since $\cB_R$  is a full subcategory).
Indeed, consider the commutative diagram  
\begin{equation*}
\begin{CD}
\LL_J      @>{\tilde\sigma}>>  \LL_{J\setminus\{j\}} \\
@VVV   @VVV  \\
L_J   @>{\sigma}>>  L_{J\setminus\{j\}} 
\end{CD}
\end{equation*}
where the left, resp., right, vertical maps are injections
given by a multiplication with %a non-zero section of 
$\sum\limits_{\al_{J,T,l,E}>0} \al_{J,T,l,E}\de_{T,T^c}$,
resp., $\sum\limits_{\al_{J,T,l,E}>0} (\al_{J,T,l,E}-1)\de_{T,T^c}$.
%Clearly, this is a commutative diagram. 
Application of the functor $T\to T_{\cB_R}$
gives a commutative diagram
\begin{equation*}
\begin{CD}
\LL_J      @>{\tilde\sigma}>>  \LL_{J\setminus\{j\}} \\
@V=VV   @VV=V  \\
\LL_J   @>{\sigma}_{\cB_R}>>  \LL_{J\setminus\{j\}} 
\end{CD}
\end{equation*}
where the second line is the differential in $(\cT_{l,E})_{\cB_R}$. 
This proves our above claim about differentials.
In particular, we have an exact triangle  of a s.o.d.
\begin{equation}\label{sfasgsg}
\tcS_{l,E}
\ra\cT_{l,E}\ra (\cT_{l,E})_{\cA_R}.
\end{equation}
We claim that $\tcS_{l,E}\in\cB$. 
For this, we need to prove that $\tcS_{l,E}$ is perpendicular to all the generators of $\cA$. Since $\tcS_{l,E}=(\cT_{l,E})_{\cB_R}$,  $\tcS_{l,E}$ is perpendicular to all the generators 
of $\cA_R$. It remains to check that $\tcS_{l,E}$ is perpendicular to the generators of $\cA$ supported on $\de=\de_{T,T^c}$ in $\M_{p,q+1}$ with 
$T_p\neq R,R'$. For this, using (\ref{sfasgsg}), it suffices to check that $\cT_{l,E}$ is perpendicular to such generators because $(\cT_{l,E})_{\cA_R}\in\cA_R\subset\cA$, and hence, orthogonal to $\cB\subset\cB_R$. 
Since $\cT_{l,E}$ is supported on $Z_R$ and $Z_R$ does not intersect any of the boundary 
$\de=\de_{T,T^c}$ in $\M_{p,q+1}$ with 
$T_p\neq R,R'$, we have 
$R\Hom\big(\cT_{l,E}, \cO_{\de}(-a,-b)\big)=0$,
for all $a,b$. Hence, 
$R\Hom\big(\cTT_{l,E}, \cO_{\de}(-a,-b)\big)=0$ by \eqref{sfasgsg}
since supports of objects in $\cA_R$ are also disjoint from $\delta$. 
This proves that $\tcS_{l,E}\in\cB$. Consider now the sequence of projection functors 
$D^b(\M_{p,q+1})\ra\cB_R\ra\cB$. 
By definition, $\cTT_{l,E}=(\cT_{l,E})_{\cB}$. But $(\cT_{l,E})_{\cB_R}\cong\tcS_{l,E}$ and since $\tcS_{l,E}\in\cB$, it follows that 
%$(\tcS_{l,E})_{\cB}\cong\tcS_{l,E}$ and 
$\cTT_{l,E}\cong \tcS_{l,E}$. 
This finishes the proof of Lemma \ref{supportT}.
%We prove that $\cTT_{l,E}\in\cB$. Since $\cT_{l,E}$ is supported on $Z_R$ and $Z_R$ does not intersect any of the boundary $\de=\de_{T,T^c}$ in $\M_{p,q+1}$ with 
%$T_p\neq R,R'$, we have 
%$$R\Hom\big(\cT_{l,E}, \cO_{\de}(-a,-b)\big)=0,$$
%for all $a,b$, and hence, 
%$R\Hom\big(\cTT_{l,E}, \cO_{\de}(-a,-b)\big)=0$. 
%
%We now prove that the same is true when $T_p=R$. 
\ep

\bp[Proof of Claim~\ref{qwhjegvfjgsEFV}]
Let $R'=\{1,\ldots,r\}$. By Lemma~\ref{supportT}
and Lemma \ref{middle}, 
it suffices to prove that for all subsets $J\subseteq R'\setminus\{1\}$ we have
$R\Gamma\big(\LL_J^\vee\otimes\cO(u,-\frac{r+s}{2})\big)=0$.
%, where  $$\LL_J=\frac{e'_q-r-l'}{2}\psi_p+\sum_{k\in E'_q}\de_{pk}-\sum_{j\in J}\de_{pj}-\sum_{T_p=R'}\al_{J,T,E',l'}\de_{T,T^c}.$$
For $\de=\de_{T,T^c}$, with $T_p=R'$, letting $\al:=\al_{J,T,E',l'}$, we have 
$(\LL_J)_{|\de}=\cO\big(-\frac{e'_q-r-l'}{2}+|E'_q\cap T_q|-|J|+\al, \al\big)=
\cO(0,\al)$ if $\al>0$ and $\cO(-\frac{e'_q-r-l'}{2}+|E'_q\cap T_q|-|J|,0)$ if $\al=0$.
If $\al=0$, the second component of $\LL_J^\vee\otimes\cO(u,-\frac{r+s}{2})$ is $\cO(-\frac{r+s}{2})$, which is acyclic on $\PP^{r+s-1}$, and the result follows. Assume now $\al>0$. We have to prove
that 
$\LL_J^\vee\otimes\cO(u,-\frac{r+s}{2})=\cO(u,-\al-\frac{r+s}{2})$
is acyclic, or equivalently, that $\al+\frac{r+s}{2}\leq r+s-1$, i.e., that 
$\al=\frac{e'_q-r-l'}{2}-|E'_q\cap T_q|+|J|\leq \frac{r+s}{2}-1$.
As $|J|\leq r-1$, it suffices to prove that
$\frac{e'_q-l'}{2}-|E'_q\cap T_q|\leq \frac{s}{2}$.
This follows from Corollary  \ref{AA}(ii) applied to the pair $(E',l')$. 
\ep

\begin{prop}\label{projections}
Let $(l,E)$ be in group $2A$ or $2B$ on $\M_{p,q+1}$, with 
$E_p=R=\{1,\ldots,r\}$. Let $R'=P\setminus R$ and $\be_R:\M_{p,q+1}\ra\M_R$, $\be_{R'}:\M_{p,q+1}\ra\M_{R'}$ be the morphisms 
in Lemma-Definition \ref{lemdef}. 

Then
(i) $\big(L\be_R^*\coT_{l,E}\big)_{\cB}=\cTT_{l,E}$, except 
when $r+s$ is even and either $e_q=l+s+2$, $T_q\subseteq E_q$ (case $2A$)
or $e_q+l=s$, $E_q\subseteq T_q$ (case $2B$).
In the latter cases there exists an exact triangle 
in $\cB$:
\begin{equation}\label{triangle}
\cTT_{l,E}\ra \big(L\be_R^*\coT_{l,E}\big)_{\cB}\ra Q_{\cB},
\end{equation}
where $Q_{\cB}$ is the projection into $\cB$ of an object $Q\in D^b(\M_{p,q+1})$ which is 
generated by the sheaves $\cO_{\de_{T,T^c}}(0,-\frac{r+s}{2})$ with $T_q\subseteq E_q$ (case $2A$) and $E_q\subseteq T_q$ (case $2B$). 

(ii) $R{\be_R}_*\big(\cTT^\vee_{l,E}\big)=\coT^\vee_{l,E}$;

(iii) $\big(L\be_{R'}^*\coT_{l,E}\big)_{\cB}=\cTT_{l,E}$;

(iv) $R{\be_{R'}}_*\big(\cTT^\vee_{l,E}\big)=\coT^\vee_{l,E}$.
\end{prop}

\bp
We note that $\coT_{l,E}$ isomorphic in $D^b(\M_R)$ to the Koszul resolution of the stratum 
$Z_R=\bigcap_{j\in R\setminus\{1\}}\de_{1j}$ tensored with the line bundle $\oL'$:
$$0\ra \oL_{R\setminus\{1\}}\ra\ldots\ra\bigoplus_{J\subseteq R\setminus\{1\}, |J|=r-2} \oL_J\ra\ldots\ra \bigoplus_{j\in  R\setminus\{1\}} \oL_j\ra \oL_{\emptyset}\ra 0,$$
where for every subset $J\subseteq R\setminus\{1\}$, we let  
$\oL_J=\oL'\big(-\sum_{j\in J}\de_{1j}\big)$ and
$\oL'=\frac{e_q-r-l}{2}\psi_1+\sum_{k\in E_q} \de_{1k}$,
where all tautological classes are on $\M_R$.  
As in Lemma \ref{supportT}, the main property of the line bundle $\oL'$ is that $\coT_{l,E}={i_R}_*\big(\oL'_{|Z_R}\big)$. 

We identify $\cTT_{l,E}$ with $S_{l,E}=(\LL_J)$ as in \eqref{TTTT}. To show (i), we prove that there is an exact triangle 
$\cTT_{l,E}\ra L\be_R^*\coT_{l,E}\ra Q$, 
with $Q\in\cA$ (and so in fact $Q_\cB=0$), 
%, except when  $r+s$ is even and  either $e_q=l+s+2$ (case $2A$) or $e_q+l=s$  (case $2B$), 
with exceptions listed in part (i), in which case $Q$ is generated by the sheaves $\cO_{\de_{T,T^c}}(0,-\frac{r+s}{2})$ with the required properties. 

For $J\subseteq R\setminus\{1\}$, 
using (\ref{Pullbacks1}) and (\ref{Pullbacks2}),  we have when $E_p=R$ that 
$${\be_R}^*\oL_J=L_J+\sum\limits_{T_p=R}\al'_{J,T,E,l}\de_{T, T^c}
=\LL_J+\sum\limits_{T_p=R,\  \al'_{J,T,l,E}>0}\al'_{J,T,l,E}\de_{T,T^c},$$
where we denote 
$\al'_{J,T,E,l}=|E_q\cap T_q|-\frac{e_q-r-l}{2}-|J|$ and $\LL_J$ is defined as in Lemma \ref{supportT}.
By definition of $\al_{J,T,E,l}$ (see Lemma \ref{supportT})  
if $\al_{J,T,E,l}>0$ we have
$\al_{J,T,E,l}=-\al'_{J,T,E,l}$.)
Let $Q_J=\Coker(\LL_J\ra \be_R^*\oL_J)$. Then $Q_J$ is generated by the %(push-forwards from $\de_{T,T^c}$ to $\M_{p,q+1}$ of the) 
successive quotients 
\begin{equation}\label{sfvafvafv}
Q_J^i:=\big(\LL_J+i\de_{T,T^c}\big)_{|\de_{T,T^c}}=\cO_T(\al'_{J,T,l,E}-i)\boxtimes\cO_{T^c}(-i),
\end{equation}
for  $0<i\leq \al'_{J,T,l,E}$.
Then $Q_J^i\in\cA$ iff $i<\frac{r+s}{2}$. 
By Corollary \ref{AA}(i),
\begin{equation}\label{ineqAA}
i\leq \al'_{J,T,l,E}%=|E_q\cap T_q|-\frac{e_q-r-l}{2}-|J|
\leq
|E_q\cap T_q|-\frac{e_q-l}{2}+\frac{r}{2}\leq\frac{r+s}{2},
\end{equation}
with equality $i=\frac{r+s}{2}$ only when $J=\emptyset$, and either $e_q=l+s+2$, $T_q\subseteq E_q$ (case $2A$), or 
$e_q+l=s$, $E_q\subseteq T_q$ (case $2B$), 
in which case
$Q_J^i=\cO(0,-\frac{r+s}{2})$, 
with $0$ the component corresponding to markings from~$T$. 
This gives an exact triangle \eqref{triangle} with the required properties as in the proof of Lemma~\ref{supportT}.

We now prove (ii). It suffices to prove that $R{\be_R}_*(Q^i_J)^\vee=0$ for all quotients~\eqref{sfvafvafv}.
Up to a shift,
$(Q^i_J)^\vee=\big(\LL_J+i\de_{T,T^c}\big)^\vee\otimes \cO(\de_{T,T^c})_{|\de_{T,T^c}}=
\cO_T(-\al'_{J,T,l,E}+i-1)\boxtimes\cO_{T^c}(i-1)$.
As $\be_R$ contracts the $T$-component, it suffices to prove that 
$0<\al'_{J,T,l,E}-i+1\leq r+s-1$
for all  $0<i\leq |\al'_{J,T,l,E}|$. 
Indeed, by \eqref{ineqAA},
$\al'_{J,T,l,E}\leq \frac{r+s}{2}\leq r+s-1$ as  $r+s\geq 2$. 

We  prove (iii) and (iv). %the case of the map $\be_{R'}$. 
Using (\ref{Pullbacks1}) and (\ref{Pullbacks2}),  
we have
$\be_{R'}^*\oL_J=\frac{e_q-r-l}{2}\psi_1+\sum\limits_{k\in E_q} \de_{1k}-\sum\limits_{j\in J} \de_{1j}=\LL_J+\sum\limits_{T_p=R,\ \al_{J,T,l,E}>0}\al_{J,T,l,E}\de_{T,T^c}$.
As before, let $Q_J=\Coker(\LL_J\ra \be_{R'}^*\oL_J)$. Then $Q_J$ is generated by the successive quotients 
$Q_J^i:=\big(\LL_J+i\de_{T,T^c}\big)_{|\de_{T,T^c}}\cong\cO_T(-i)\boxtimes\cO_{T^c}(\al_{J,T,l,E}-i)$
for all $0<i\leq \al_{J,T,l,E}=\frac{e_q-r-l}{2}-|E_q\cap T_q|+|J|$.
Then $Q_J^i\in\cA$ if and only if $i<\frac{r+s}{2}$. 
As $|J|\leq r-1$, we have 
\begin{equation}\label{ineqBB}
\al_{J,T,l,E}=\frac{e_q-r-l}{2}-|E_q\cap T_q|+|J|\leq \frac{e_q+r-l}{2}-|E_q\cap T_q|-1\leq\frac{r+s}{2}-1,
\end{equation}
by Corollary \ref{AA}(ii). 
Therefore, $Q\in\cA$ and this proves (iii). 

To prove (iv), it suffices to show that 
$(Q_J^i)^\vee$, which up to a shift is equal to $(\LL_J+i\de_{T,T^c})^\vee(\de_{T,T^c})_{|\de_{T,T^c}}\cong
\cO_T(i-1)\boxtimes\cO_{T^c}(-(\al_{J,T,l,E}-i)-1)$,
pushes forward to $0$ by $\be_{R'}$. As $\be_{R'}$ contracts the $T^c$-component, it suffices to prove that 
$0<\al_{J,T,l,E}-i+1\leq r+s-1$
for all $0<i\leq \al_{J,T,l,E}$, or equivalently that 
$\al_{J,T,l,E}\leq r+s-1$, 
which follows from (\ref{ineqBB}) as $\frac{r+s}{2}-1\leq r+s-1$. 
\ep

\begin{thm}\label{pairs on M_R} 
%Assume $p=2r\geq4$, $q=2s+1\geq1$. 
On $\M_{R'}$, consider vector bundles
%In the notations of \ref{barTT}, consider 
$F_{l,E}$ for pairs $(l,E)$ in group $1A$ or $1B$, and torsion sheaves
$\coT_{l,E}$ for $(l,E)$ in group $2A$ or $2B$. Then

(1)
$R\Hom_{\M_{R'}}(F_{l',E'},\coT_{l,E})=0$ if $E_p=R$, $e'_q\geq e_q$.

(2) $R\Hom_{\M_{R'}}(\coT_{l',E'},F_{l,E})=0$ if $E'_p=R'$ and  
$e'_q> e_q$ or $e_q'=e_q$, $E_q'\ne E_q$.

(3) $R\Hom_{\M_{R'}}(\coT_{l',E'}, \coT_{l,E})=0$ if $E_p=E'_p=R$ and either
$e'_q\geq e_q$, $E_q\neq E'_q$ or $E_q=E'_q$ and $l>l'$. Furthermore,  
$R\Hom_{\M_{R'}}(\coT_{l,E}, \coT_{l,E})=\CC$. 

(4) 
$R\Hom_{\M_{R'}}(\coT_{l',E'}, \coT_{l,E})=0$ if $E_p=R$, $E'_p=R'$, $e'_q\geq e_q$.
\end{thm}

\begin{cor}\label{exceptionalityeveneven}
The collection of Theorem~\ref{p,q+1 case} is exceptional.
\end{cor}

\bp
As discussed in the beginning of this section, 
we only need to check exceptionality of complexes $\cTT_{l,E}$
and their semi-orthogonality with each other and vector bundles $F_{l,E}$, where the order is as in Theorem \ref{p,q case}, i.e., as follows: 
the objects are arranged in blocks indexed by a subset $E_q$ and ordered by increasing $e_q$ and arbitrarily if $e_q$ is the same (but the set $E_q$ is different).
Within each block with the same $E_q$ we put the 
complexes  $\{\cTT_{l,E}\}$ first, in arbitrary order if $E_p\neq E'_p$ and in order of decreasing $l$ when $E_p=E'_p$.
After the complexes we put  the bundles $\{F_{l,E}\}$ in a certain order.

By Proposition~\ref{reduction} and Theorem~\ref{pairs on M_R}, by choosing a subset $R$ of $r$ heavy indices 
appropriately, we see that, indeed,
$R\Hom(F_{l',E'},\cTT_{l,E})=0$ if $e'_q\geq e_q$;
$R\Hom(\cTT_{l',E'},F_{l,E})=0$ if  $e'_q> e_q$ or $e_q'=e_q$, $E_q'\ne E_q$;
$R\Hom(\cTT_{l,E}, \cTT_{l,E})=\CC$;
$R\Hom(\cTT_{l',E'}, \coT_{l,E})=0$
if $E_p=E'_p$ and either $e'_q\geq e_q$, $E_q\neq E'_q$ or $E_q=E'_q$ and $l>l'$;
$R\Hom(\cTT_{l',E'}, \cTT_{l,E})=0$ if $e'_q\geq e_q$
and $E_p$ is disjoint from $E'_p$.
It remains to note that $R\Hom(\cTT_{l',E'}, \coT_{l,E})=0$
if $E_p$ and $E_p$ are neither equal nor disjoint because these complexes have disjoint support.
\ep

In the remainder of this section we prove Theorem~\ref{pairs on M_R}.
\begin{notn}
Set $R=\{1,\ldots, r\}$, $R'=\{r+1,\ldots, p\}$ and consider the loci 
$Z_R, Z_{R'}\hra\M_{R'}$, $Y=Z_R\cap Z_{R'}$.
\end{notn}

We reduce Theorem~\ref{pairs on M_R} to a calculation of $R\Gamma$ on loci $Z_R$, $Z_{R'}$  and $Y$:  
\begin{lemma}\label{explicit}
Let all pairs $(l,E)$ be in group $1A$ or $1B$ for $F_{l,E}$ and in group $2A$ or $2B$ for $\coT_{l,E}$.
In order to prove Theorem~\ref{pairs on M_R}, it suffices to prove the following:

(1) $R\Ga\big(Z_R, ({F_{l',E'}}_{|Z_R})^\vee\otimes\coT_{l,E})=0$ if 
$e'_q\geq e_q$, $R=E_p$.

(2) $R\Ga\big(Z_{R'}, \coT_{l',E'}^\vee\otimes {F_{l,E}}_{|Z_{R'}}\otimes c_1(N_{Z_{R'}|\M_{R'}})\big)=0$
if $R'=E'_p$ and either
$e'_q> e_q$ or $e'_q=e_q$ but $E'_q\ne E_q$.

(3) $R\Gamma\big(Z_R, \coT_{l',E'}^\vee\otimes\coT_{l,E}\otimes(\sum_{j\in J}\de_{1j}\big))=0$
if $R=E_p=E'_p$, for all $J\subseteq R\setminus\{1\}$, $e'_q\geq e_q$,  and either $E_q\neq E'_q$ or $E_q=E'_q$, $l>l'$.
Furthermore,
$R\Gamma\big(Z_R, \coT_{l,E}^\vee\otimes\coT_{l,E}\otimes(\sum_{j\in J}\de_{1j}\big))=0$
for all $\emptyset\neq J\subseteq R\setminus\{1\}$. 

(4) $R\Gamma\left(Y, -\left(\frac{l+l'}{2}+1\right)\psi_u+\frac{1}{2}\sum_{j\in E'_q}\psi_j-\frac{1}{2}\sum_{j\in E_q}\psi_j\right)=0$ if
$R=E_p$, $R'=E'_p$, $e'_q\geq e_q$.
\end{lemma}

\bp $\ $ This is very similar to the proof of Lemma \ref{sDGSGASG}: 

(1) 
$R\Hom_{\M_{R'}}(F_{l',E'},\coT_{l,E})=R\Ga\big(Z_{R}, ({F_{l',E'}}_{|Z_{R}})^\vee\otimes\coT_{l,E}).$

(2) As $i^{!}F_{l,E}=i^*F_{l,E}\otimes c_1(N_{Z_{R'}|\M_{R'}})$, where $i: Z_{R'}\hra \M_{R'}$ is the embedding,
we have by \cite[Corollary  3.38]{Huybrechts}  that 
$R\Hom_{\M_{R'}}(\coT_{l',E'},F_{l,E})=R\Ga\big(Z_{R'}, \coT_{l',E'}^\vee\otimes {F_{l,E}}_{|Z_{R'}}\otimes c_1(N_{Z_{R'}|\M_{R'}})\big)$
(up to a shift).

(3) 
Follows from tensoring the Koszul resolution of $Z_R$ with a line bundle $L$ such that ${i_R}_*(L_{|Z_R})=\cT_{l',E'}$ and applying $R\Hom(-,\cT_{l,E})$. 

(4) 
Recall that $R'=\{r+1,\ldots, p\}$. By an analogue of  Lemma \ref{psiX}(v) we have that 
$\coT_{l,E}={i_R}_*\big(-\frac{r+l}{2}\psi_u-\frac{1}{2}\sum_{j\in E_q}\psi_j\big)$, where $i_R: Z_R\hra \M_{R'}$. 
As in the proof of Lemma \ref{sDGSGASG}, we have that 
$R\Hom_{\M_{R'}}(\coT_{l',E'}, \coT_{l,E})$ equals 
$R\Gamma\big(Y, ({\coT_{l',E'}}_{|Y})^\vee\otimes  {\coT_{l,E}}_{|Y}\otimes c_1(N_{Y|Z_R})\big)
=R\Gamma\big(Y, ({\cT_{l',E'}}_{|Y})^\vee\otimes  {\cT_{l,E}}_{|Y}\otimes\break 
(\sum_{j\in R'\setminus\{p\}}\de_{jp})_{|Y}\big)
=R\Gamma\big(Y, -(\frac{l+l'}{2}+1)\psi_u+\frac{1}{2}\sum_{j\in E'_q}\psi_j-\frac{1}{2}\sum_{j\in E_q}\psi_j)$ 
using an analogue of Lemma \ref{psiX} 
(i), (ii), (iii).
\ep

\begin{notn}
We prove vanishing in Lemma~\ref{explicit} by a windows calculation on $Z_R$, $Z_{R'}$ and $Y$.
Since $Z_R$ and $Z_{R'}$ intersect only boundary $\de_{T,T^c}$ with $T_p=R$ or $R'$, which get contracted in $\M_{R'}$, they are smooth GIT quotients of $(\PP^1)^{r+q+1}$ by 
$\PGL_2$ by Lemma \ref{MonGeneralBrendanHassett} 
since their universal families are $\PP^1$-bundles (see the proof of Lemma-Definition \ref{lemdef} for the precise weights on $\M_{R'}$.)
More precisely, let 
$X %=(\PP^1)^{r+q+1}
=\PP_u^1\times(\PP^1)^r\times (\PP^1)^q$,
corresponding to the partition $\{u\}\sqcup R'\sqcup Q$ of the markings on 
$Z_R=\M_{\{u\}\cup R'\cup Q}$. 
Then $Z_R$ is a GIT quotient of $X$ by $\PGL_2$.
Likewise, let 
$X' %=(\PP^1)^{r+q+1}
=(\PP^1)^r\times\PP_v^1\times (\PP^1)^q$, 
corresponding to the partition $R\sqcup \{v\}\sqcup Q$ of the markings on 
$Z_{R'}=\M_{R\cup\{v\}\cup Q}$. 
Then $Z_{R'}$ is a GIT quotient of $X'$ by $\PGL_2$. 
In addition, let 
$X''=(\PP^1)^{q+2}=\PP^1_u\times\PP^1_v\times (\PP^1)^q$,
corresponding to the partition $\{u\}\sqcup\{v\}\sqcup Q$ of the markings on 
$\M_{\{u,v\}\cup Q}$. Then $Y$ is the GIT quotient of $X''$ by $\PGL_2$.

Vector bundles on $Z_R$ (resp., $Z_{R'}$) correspond to $\PGL_2$-linearized vector bundles on 
$X$ (resp., $X'$) as follows:
${F_{l,E}}_{|Z_R}=\cO_X(|E_p\cap R|, E_p\cap R', E_q)\otimes V_l$,  
$\coT_{l,E}=\cO_X(r+l,0, E_q)$ if $E_p=R$, $K_{Z_R}=\cO(-2,-2,\ldots, -2)$,\break
${F_{l,E}}_{|Z_{R'}}=\cO_{X'}(E_p\cap R, |E_p\cap R'|, E_q)\otimes V_l$,
${\de_{1j}}_{|Z_R}=\cO(2,0,0,\ldots)$ ($j\in R\setminus\{1\}$),
${\de_{pj}}_{|Z_{R'}}=\cO(0,\ldots,0,2,0,\ldots, 0)$ ($j\in R'\setminus\{p\}$) 
(use Lemma \ref{akjsdhfkjsgf}, Remark \ref{dictionary} and an analogue of Lemma \ref{psiX}(v)).
\end{notn}

\begin{rmk}(The devil's trick reloaded.)\label{devil2}
Since for any $j\in R'$ (resp., $j\in R$), we have $\psi_u+\psi_j=0$ in $Z_{R}$ ($\psi_v+\psi_j=0$ in $Z_{R'}$) as  $\de_{uj}=\emptyset$ (resp., $\de_{vj}=\emptyset$). Therefore, as in Remark \ref{devil}, the line bundles 
$\DD:=\cO(r,R',0)$ on $X=(\PP^1)^{r+q+1}$,
$\DD:=\cO(R,r,0)$ on $X'=(\PP^1)^{r+q+1}$,
descend to trivial line bundles on $Z_R$ and $Z_{R'}$. 
Likewise, the line bundle $\DD:=\cO(1,1,0)$ on $X''$ descends to the trivial line bundle on $Y$.
\end{rmk}

\bp[Proof of Theorem~\ref{pairs on M_R}]
We prove the vanishings in Lemma~\ref{explicit}. 
As in the proof of Theorem~\ref{torsion pairs}, we will first prove the $\PGL_2$-invariant vanishing holds on 
$X$ (cases (1) and (3)), on $X'$ (case (2)) and on $X''$ (case (4))
after tensoring with the devil line bundle $\DD^N$, $N\gg0$.
%,since it's this tensor product that will satisfy conditions of Theorem~\ref{Teleman}.
Later on we check the weight condition \eqref{asrgasrgasrh} in Theorem~\ref{Teleman}.

For (1), assuming condition \eqref{asrgasrgasrh}, 
$R\Ga\big(Z_R, ({F_{l',E'}}_{|Z_R})^\vee\otimes\coT_{l,E}\otimes\DD^N)=
R\Ga\big(X,\cO(r+l-|E'_p\cap R|+Nr,-E'_p\cap R'+NR', -E'_q+E_q)\otimes V_{l'}\big)^{\PGL_2}$,
which is clearly $0$ if $E'_q\nsubseteq E_q$. 
Since here we assume $e'_q\geq e_q$, we have that $E'_q\subseteq E_q$ if and only if $E_q=E'_q$. 
Assume $E_q=E'_q$. 
In this case, we need to consider
$V_{r+l-|E'_p\cap R|+Nr}\otimes V_{N-y_1}\otimes\ldots\otimes V_{N-y_r}\otimes V_{l'}$,
where $y_i=1$ if the corresponding index is in $E'_p\cap R'$ and $0$ otherwise.
We claim that the  $\PGL_2$-invariant part is $0$ by the Clebsch-Gordan formula 
(Lemma \ref{CG}).
It suffices to check that
$r+l-|E'_p\cap R|+Nr>Nr-|E'_p\cap R'|+l'$.
This follows if 
$l'+|E'_p\cap R|-|E'_p\cap R'|<r$, which holds by 
Corollary~\ref{dfbdfbdfb} (since $(l',E')$ is in group $1A$ or $1B$).

For (2), assuming \eqref{asrgasrgasrh}, 
$R\Ga\big(Z_{R'}, \coT_{l',E'}^\vee\otimes {F_{l,E}}_{|Z_{R'}}\otimes c_1(N_{Z_{R'}|\M_{R'}})\otimes\DD^N\big)=
R\Ga\big(X',\cO(E_p\cap R+NR, |E_p\cap R'|-r-l'+2(r-1)+Nr, E_q-E'_q)\otimes V_l\big)^{\PGL_2}$.
This is $0$ since in our case $E'_q\nsubseteq E_q$.

For (3),  assuming \eqref{asrgasrgasrh}, 
$R\Gamma\big(Z_R, \coT_{l',E'}^\vee\otimes\coT_{l,E}\otimes(\sum_{j\in J}\de_{1j}\big)\otimes\DD^N)$ equals $R\Gamma\big(X, \cO(2|J|+l-l'+Nr, NR', E_q-E'_q)\big)^{\PGL_2}$
for all $J\subseteq R\setminus\{1\}$. This is $0$ if $E'_q\nsubseteq E_q$. 
Assume now $E_q=E'_q$. We need to consider
$V_{2|J|+l-l'+Nr}\otimes V_{N}^{\otimes r},$
whose $\PGL_2$-invariant part is $0$ when $l>l'$ 
or when $l=l'$, $|J|>0$ by the Clebsch-Gordan formula
(Lemma \ref{CG}),
since in these cases we have 
$2|J|+l-l'+Nr>Nr$. 

For (4),  assuming \eqref{asrgasrgasrh}, 
$R\Gamma\big(Y, ({\coT_{l',E'}}_{|Y})^\vee\otimes  {\coT_{l,E}}_{|Y}\otimes c_1(N_{Y|Z_R})\otimes\DD^N\big)=R\Gamma\big(X'', \cO(l+l'+2+N, N, E_q-E'_q)\big)^{\PGL_2}$.
This is $0$ if $E'_q\nsubseteq E_q$. If $E_q=E'_q$ we have to consider 
$V_{l+l'+2+N}\otimes V_{N}$, 
whose $\PGL_2$-invariant part is $0$ by the Clebsch-Gordan formula
(Lemma \ref{CG}),
as $l+l'+2+N>N$. 

We now check that for each stratum, each of the cases (1)-(4) fall under the assumption on weights of Theorem  \ref{Teleman}.
The unstable loci in $X$ (resp., $X'$) corresponding to the loci $Z_{R}$ (resp., $Z_{R'}$) in $\M_{R'}$ have the following form:  

(The locus $K_I$), for $I\subseteq Q$, $|I|\geq s+1$ (resp., $|I|\geq s+2$), where $u$ (resp.,~$v$) and the indices in $I$ come together. In this case, 
$\eta=2|I|$. 

(The locus $K_{J,I}$), for $J\subseteq R'$ (resp., $J\subseteq R$), $I\subseteq Q$, $J\neq\emptyset$, $|I|\geq 0$, where $u$ (resp., $v$) and the indices in $J$ and $I$ come together.
In this case, $\eta=2|I|+2|J|$. 

(The locus $L_I$), for $I\subseteq Q$, $|I|\geq s+2$ (resp., $|I|\geq s+1$), where the indices in $R'$ and $I$ (resp., $R$ and $I$) come together. In this case, 
$\eta=2|I|+2r-2$.

The devil line bundle $\cO(r,R',0)$ on $X$ has the property that its weight for $K_{J,I}$
is %$$\wt_\la\cO(r,R',0)_{|z_{\{u\}\cup J\cup I}}=
$r+|J|-|R'\setminus J|=2|J|>0$
while its weight for the other strata is~$0$ (similarly for $\cO(R,r,0)$ on $X'$).
Therefore, the condition \eqref{asrgasrgasrh}
for the stratum $K_{J,I}$ can be achieved by tensoring with a high 
enough multiple of this line bundle. We only need to consider the remaining strata.

Consider first cases (1) and (3) involving $Z_R\subseteq \M_{R'}$. 
For case (1) we need to verify that  the vector bundle 
$\cO(r+l-|E'_p\cap R|,-E'_p\cap R', E_q-E'_q)\otimes V_{l'}$
on $X$ has the weights $>-\eta$: 
For the stratum $K_I$ ($|I|\geq s+1$), we need to prove that 
$r+l-l'-|E'_p\cap R|+|E'_p\cap R'|+|E_q\cap I|-|E'_q\cap I|-|E_q\cap I^c|+|E'_q\cap I^c|>-2|I|$.
By (\ref{another}) for the pair $(l',E'_p)$, it suffices to prove that 
$l+|E_q\cap I|-|E'_q\cap I|-|E_q\cap I^c|+|E'_q\cap I^c|\geq -2|I|$.
Since the left hand side is greater than $-|I^c|-|I|=-q-1=-2s-2$ and $|I|\geq s+1$, the result follows.
For the stratum $L_I$ ($|I|\geq s+2$), we need to prove that 
$-r-l-l'+|E'_p\cap R|-|E'_p\cap R'|+|E_q\cap I|-|E'_q\cap I|-|E_q\cap I^c|+|E'_q\cap I^c|>-2|I|-2r+2$.
By (\ref{another}) for the pair $(l',E'_p)$, it suffices to prove that\break 
$-l+|E_q\cap I|-|E'_q\cap I|-|E_q\cap I^c|+|E'_q\cap I^c|\geq -2|I|+2$.
As $-|E'_q\cap I|+|E'_q\cap I^c|\geq -|I|$, it suffices to prove that 
$-l+|E_q\cap I|-|E_q\cap I^c|\geq -|I|+2$.
This follows from Lemma \ref{another3} since $|I|\geq s+2$
and the pair $(l,E)$ is in group $2A$ or $2B$. 

For case (3) we need to verify that the  line bundle 
$\cO(2|J|+l-l', 0, E_q-E'_q)$
on $X$ has weights $>-\eta$
for all $0\leq |J|\leq r-1$.  For the stratum $K_I$, we need to prove that 
$l-l'+2|J|+|E_q\cap I|-|E'_q\cap I|-|E_q\cap I^c|+|E'_q\cap I^c|>-2|I|$.
Note that it suffices to prove that 
$-l'-|E'_q\cap I|+|E'_q\cap I^c|-|I^c|>-2|I|$.
This follows from Lemma \ref{another3} since $|I|\geq s+1$ for the locus $K_I$ in $X$. 
For the stratum $L_I$ ($|I|\geq s+2$), we need to prove that 
$-l+l'-2|J|+|E_q\cap I|-|E'_q\cap I|-|E_q\cap I^c|+|E'_q\cap I^c|>-2|I|-2r+2$,
or equivalently,
$-l+l'+|E_q\cap I|-|E'_q\cap I|-|E_q\cap I^c|+|E'_q\cap I^c|>-2|I|$.
As the left hand side is greater or equal than $-l+|E_q\cap I|-|E_q\cap I^c|-|I|$, it suffices to prove 
$-l+|E_q\cap I|-|E_q\cap I^c|>-|I|$.
This follows from Lemma \ref{another3} since $|I|\geq s+2$
and the pair $(l,E)$ is in group $2A$ or $2B$. 

For case (2), involving $Z_{R'}\subseteq \M_{R'}$, the relevant vector bundle on $X'$ 
is
$\cO(E_p\cap R, |E_p\cap R'|+r-l'-2, E_q-E'_q)\otimes V_l$.
For the stratum $K_I$ ($|I|\geq s+2$), we need to prove that 
$r-l-l'-2+|E_p\cap R'|-|E_p\cap R|+|E_q\cap I|-|E'_q\cap I|-|E_q\cap I^c|+|E'_q\cap I^c|>-2|I|$.
By (\ref{another}) for the pair $(l,E_p)$, it suffices to prove that 
$-l'-1+|E_q\cap I|-|E'_q\cap I|-|E_q\cap I^c|+|E'_q\cap I^c|>-2|I|$.
As the left hand side is greater or equal than $-l'-1-|E'_q\cap I|+|E'_q\cap I^c|-|I^c|$, it suffices to prove that 
$-l'-|E'_q\cap I|+|E'_q\cap I^c|\geq |I^c|-2|I|+2$.
This follows from Lemma \ref{another3} since $|I|\geq s+2$.  
For the stratum $L_I$ ($|I|\geq s+1$), we need to prove that 
$-r+l'-l+2+|E_p\cap R|-|E_p\cap R'|+|E_q\cap I|-|E'_q\cap I|-|E_q\cap I^c|+|E'_q\cap I^c|>-2|I|-2r+2$.
By (\ref{another}) for the pair $(l,E_p)$, it suffices to prove that 
$l'+|E_q\cap I|-|E'_q\cap I|-|E_q\cap I^c|+|E'_q\cap I^c|\geq -2|I|$.
As the left hand side is greater or equal than $-|I|-|I^c|=-q-1$, 
the inequality follows since $|I|\geq s+1$. 

%\smallskip

We now consider Case (4). Up to symmetry, the unstable loci in 
$X'=\PP^1_u\times\PP^1_v\times (\PP^1)^{q+1}$
corresponding to $Y=Z_{R}\cap Z_{R'}\subseteq \M_{R'}$ are
%have the following form: 

(The locus $K'_I$), for $I\subseteq Q$, $|I|\geq s+1$, where $u$ and the indices in $I$ come together. In this case, 
$\eta=2|I|$. 

(The locus $K''_I$), for $I\subseteq Q$, $|I|\geq s+2$, where $v$ and the indices in $I$ come together. In this case, 
$\eta=2|I|$. 

(The locus $K'''_I$), for $I\subseteq Q$, $|I|\geq0$, where $u$, $v$ and the indices in $I$ come together. In this case, 
$\eta=2|I|+2$.  

As before, the devil line bundle $\cO(1,1,0)$ on $X''$ has the property that 
its weight at $K'''_I$ is %$$\wt_\la\cO(1,1,0)_{|z_{\{u,v\}\cup I}}=
$2>0$,
while its weight for the other strata is $0$. 
Hence, as before, we only have to consider the strata $K'_I$, $K''_I$.  

We verify that the weights of $\cO(l+l'+2, 0, E_q-E'_q)$ are $>-2|I|$. 
For the stratum $K'_I$ ($|I|\geq s+1$), we need to prove that 
$l+l'+2+|E_q\cap I|-|E'_q\cap I|-|E_q\cap I^c|+|E'_q\cap I^c|>-2|I|$.
This is clear, as the left hand side is greater or equal than 
$2-|I|-|I^c|=2-(q+1)=-2s$ and $|I|\geq s+1$. 

For the stratum $K''_I$ ($|I|\geq s+2$), we need to prove that 
$-l-l'+|E_q\cap I|-|E'_q\cap I|-|E_q\cap I^c|+|E'_q\cap I^c|>-2|I|+2$.
We apply  Lemma \ref{another3} to the pairs $(l,E_q)$ 
and $(l',E'_q)$. 
Recall that $e'_q\geq e_q$. If in case $2A$ and $e'_q\leq s$ (resp., case $2B$ and $e'_q\leq s+1$ ), then Lemma \ref{another3} implies that 
$-l+|E_q\cap I|-|E_q\cap I^c|\geq -s+1$, (resp., $\geq -s$),
$-l'-|E'_q\cap I|+|E'_q\cap I^c|\geq -s+1$, (resp., $\geq -s$),
and the inequality follows by summing up, as $-2s>-2|I|+2$. 

If in case $2A$ and $e'_q>s\geq e_q$ (resp.,  case $2B$ and $e'_q>s+1\geq e_q$),
then Lemma \ref{another3} implies that 
$-l+|E_q\cap I|-|E_q\cap I^c|\geq -s+1$, (resp., $\geq -s$),
$-l'-|E'_q\cap I|+|E'_q\cap I^c|\geq -2|I|+s+2$, 
(resp., $\geq -2|I|+s+3$),
and the inequality follows again by summing up. 

If in case $2A$ and $e_q>s$ (resp., case $2B$ and $e_q>s+1$),  
then Lemma \ref{another3} implies that 
$-l+|E_q\cap I|-|E_q\cap I^c|\geq -2|I^c|+s+2$ (resp., $\geq 2|I^c|+s+3$),
$-l'-|E'_q\cap I|+|E'_q\cap I^c|\geq -2|I|+s+2$ (resp., $\geq 2|I^c|+s+3$),
and the inequality follows again by summing up, since 
$-2|I|-2|I^c|+2s+4>-2|I|+2$,
(as $|I^c|<s+1$, since $|I|\geq s+2$).  
\ep

%%%%%%%%%%%%%%%%%%%%%%%%%%%%%%%%%%%%%%%%%%%%%%%%%%%%%%%%%%%%%%%%%%%%%%%%%%%%%%%%%%%%%%%

\section{Fullness of the  exceptional collection on $\M_{2r,2s+1}$}\label{fullness p,q section}

In this section we finish the proof of Theorem~\ref{p,q case} (see also Remark~\ref{SgSRgSRhSRh}),
which gives a full equivariant exceptional collection in $D^b(\M_{p,q})$ for $p=2r\geq4$, $q=2s+1\geq1$.
Exceptionality of the collection was proved in Section~\ref{exceptional p,q section}, here we prove fullness.
Since $\M_{p,q}$ is a GIT quotient, its derived category is generated by vector bundles $F_{l,E}$ for all pairs $(l,E)$
such that $l+e$ is even (Proposition~\ref{sjhfgkjshgfqkjhw}). So it remains to prove the following theorem.

\begin{thm}\label{K games}
Assume $p=2r\geq4$, $q=2s+1\geq1$.
The vector bundles $\{F_{l,E}\}$ for all $l\geq0$, $E\subseteq P\cup Q$, $e=|E|$, with $l+e$ even, are generated by 
any of the two exceptional collections in $D^b(\M_{p,q})$ given by Theorem \ref{p,q case} and Remark~\ref{SgSRgSRhSRh}. In particular, these exceptional collections 
are full. 
\end{thm}

\bp
We prove equivalently (see Remark \ref{dual collection}) that the dual of the collection in Theorem  \ref{p,q case} (groups $1A$ and $2$, resp.~$1B$ and $2$, 
see Remark \ref{1B stuff}), 
call it $\cC$, generates all the dual vector bundles $F^\vee_{l,E}$. 
We will do an induction based on the score $S(l,E)$ (see Notation~\ref{asgasfgasgsg}).
For equal scores, we do an induction on $l$. 
If $S(l,E)\leq r-2$, then $l+\min\{e_p+1,p+1-e_p\}\leq r-1$, i.e., 
the pair $(l,E)$ belongs to both groups $1A$ and $1B$. This ensures the base case of the induction. 

We now assume that we have a pair $(l,E)$ (not in group $1A$,  resp.~$1B$) such that 
the bundle $F^\vee_{l',E'}$ is generated by $\cC^\vee$ for 
any $(l',E')$ with $S(l',E')<S(l,E)$, or if $S(l,E)=S(l',E')$ and $l'<l$. 
Like in the proof of Claim~\ref{sgsgsgsg}, we 
will use  constructions based on the Koszul complex. We will call a set $I\subseteq P\cup Q$ \emph{stable}  if
$i_p\leq r-1$ or if $i_p=r$, $i_q\leq s$, 
where as usual we denote $I_p=I\cap P$,  $I_q=I\cap Q$, $i_p=|I_p|$, $i_q=|I_q|$. 
Otherwise, $I$ is called \emph{unstable} (the usual notions of GIT stability giving $\M_{p,q}$). Let $i:=|I|$.
Given a subset $I\subseteq P\cup Q$, we consider the Koszul complex for 
$\De=\cap_{k\in I}\de_{kx}
\subseteq(\PP^1)^{p+q}\times\PP_x^1$,
\begin{equation}\label{svsgsRGsrg}
0\cong\big[\cO_{\De}\leftarrow\cO\leftarrow\bigoplus_{k\in I}\cO(-{\bf e_k}-{\bf e_x}) %\leftarrow\bigoplus_{k, j\in I}\cO(-e_k-e_j-2e_x)%$$$$
\leftarrow\ldots\leftarrow\cO(-\sum_{k\in I}{\bf e_k}-i{\bf e_x})\big]. 
\end{equation}
Here $\{\bf e_i\}$ for $i\in P\cup Q\cup\{x\}$ is the standard basis of $\Pic(\PP^1)^{p+q+1}\cong\ZZ^{p+q+1}$.
\underline{Koszul Game 1:} Assume $I$ is an unstable set  disjoint from $E\subseteq P\cup Q$.
Assume $e+l$ even. 
Tensoring \eqref{svsgsRGsrg}  with $\cO(-\sum_{k\in E}{\bf e_k}+l{\bf e_x})$,
pushing forward to $(\PP^1)^{p+q}$, and restricting to the semistable locus  gives
by Corollary \ref{complement} and 
Grothendieck-Verdier duality (\ref{GV duality}) the following
objects:
$$0\quad F^\vee_{l,E}\quad\ldots\quad \bigoplus_{J\subseteq I, |J|=j} F^\vee_{l-j,E\cup J}\quad\ldots
 \bigoplus_{J\subseteq I, |J|=l} F^\vee_{0,E\cup J}\quad 0$$
$$\bigoplus_{J\subseteq I, |J|=l+2} F^\vee_{0,E\cup J}[-1]\quad\ldots\quad
\bigoplus_{J\subseteq I, |J|=j} F^\vee_{j-l-2,E\cup J}[-1]\quad\ldots\quad F^\vee_{i-l-2,E\cup I}[-1],$$
where the second line appears only if $i\geq l+2$. If $i\leq l$, the first line stops at $F^\vee_{l-i,E\cup I}$. By Lemma~\ref{obviosddd}, any of these objects is generated by the rest.

%\smallskip

\underline{Koszul Game 2:} Assume $I$ is an unstable set disjoint from $E'\subseteq P\cup Q$.
Let $E=E'\cup I$, $e+l$  even. 
Tensoring \eqref{svsgsRGsrg} with $\cO(-\sum_{k\in E'}{\bf e_k}+(i-2-l){\bf e_x})$,
pushing forward, and restricting to the semistable locus  gives 
by Corollary \ref{complement} and Grothendieck-Verdier duality (\ref{GV duality}) the following 
objects:
$$0\quad \ F^\vee_{i-2-l,E'}\quad \ldots\quad \bigoplus_{J\subseteq I, |J|=j} F^\vee_{i-2-l-j,E'\cup J}\quad  \ldots
 \bigoplus_{J\subseteq I, |J|=i-2-l} F^\vee_{0,E'\cup J}$$
 $$0\quad \bigoplus_{J\subseteq I, |J|=i-l} F^\vee_{0,E'\cup J}[-1]\ \ldots\ 
\bigoplus_{J\subseteq I, |J|=j} F^\vee_{j-i+l,E'\cup J}[-1]\ \ldots\ F^\vee_{l,E'\cup I}[-1],$$
where the first line appears only if $i\geq l+2$, while if $i\leq l$ the second line starts at $F^\vee_{l-i,E'}$. 
By Lemma~\ref{obviosddd} any of these objects is generated by the rest.

%\smallskip

\underline{Koszul Game 3:} Assume $I=E_p=R\subseteq P$, $E\subseteq P\cup Q$, $|R|=r$, $e+l$ even. 
We tensor\eqref{svsgsRGsrg}  with 
$\cO(-\sum_{j\in E_q}{\bf e_j}+(r-2-l){\bf e_x})$, push forward, and restrict to the semistable locus. 
Unlike in the previous games, the torsion object 
has support in the semistable locus. 
Note that if $\cU_R\ra Z_R$ is the universal family over $Z_R$, then $\De$ corresponds to the section $\si_u(Z_R)$. 
Identifying $\De\cong\PP^1_u\times(\PP^1)^{r+q}$, where $u$ is the marking corresponding to indices in $R$ we have
$\cO\bigl(-\sum_{j\in E_q}{\bf e_j}+
(r-2-l){\bf e_x}\bigr)_{|\De}\cong\cO\big(-\sum_{j\in E_q}{\bf e_j}+(r-2-l){\bf e_u}\big)$, 
which descends to $Z_R\subseteq\M_{p,q}$ as the line bundle 
$$-\frac{r-2-l}{2}\psi_u+\frac{1}{2}\sum_{j\in E_q}\psi_j=-\frac{r-2-l+e_q}{2}\psi_u-\sum_{j\in E_q}\de_{ju}$$ 
(use Remark \ref{dictionary}). 
Using Lemma~\ref{dual} 
and Grothendieck-Verdier duality (\ref{GV duality}),
the terms  have the following derived push-forwards: 
%\newpage
$$\cT^\vee_{l,E}[r-1]\quad  F^\vee_{r-l-2,E_q}\quad\ldots\quad \bigoplus_{J\subseteq R\atop |J|=j} F^\vee_{r-2-l-j,E_q\cup J}\quad \ldots\quad \bigoplus_{J\subseteq R\atop |J|=r-2-l} F^\vee_{0,E_q\cup J}$$
$$0\quad  \bigoplus_{J\subseteq R\atop |J|=r-l} F^\vee_{0,E_q\cup J}[-1]\ldots\quad \bigoplus_{J\subseteq R\atop |J|=j} F^\vee_{l-r+j,E_q\cup J}[-1]
\ldots  F^\vee_{l,E_q\cup R}[-1]=F^\vee_{l,E}[-1],$$
where the first part of the sequence of bundles $F^\vee_{r-2-l-j, E_q\cup J}$ 
appears only when $r\geq l+2$, while if $r\leq l$, then the second part of the sequence starts at 
$F^\vee_{l-r,E_q}$. 
%Analyzing the  spectral sequence  as in Section~\ref{fullness odd p}, we obtain the following exact sequence of sheaves:
%$$0\leftarrow \cT^\vee_{l,E}[r-1]\leftarrow F^\vee_{r-l-2,E_q}\leftarrow\ldots\leftarrow \bigoplus_{J\subseteq R, |J|=j} F^\vee_{r-2-l-j,E_q\cup J}\leftarrow\ldots$$
%$$\ldots\leftarrow \bigoplus_{J\subseteq R, |J|=r-2-l} F^\vee_{0,E_q\cup J}\leftarrow  \bigoplus_{J\subseteq R, |J|=r-l} F^\vee_{0,E_q\cup J}\leftarrow\ldots$$
%$$\ldots\leftarrow\bigoplus_{J\subseteq R, |J|=j} F^\vee_{l-r+j,E_q\cup J}\leftarrow\ldots\leftarrow F^\vee_{l,E_q\cup R}=F^\vee_{l,E}.$$
Note that the score of every $F^\vee_{\tl,\tE}$ with $(\tl,\tE)\neq (l,E)$ in the Koszul game 3 is strictly lower than the score of $F^\vee_{l,E}$:

(a) For $(\tl,\tE)=(r-2-l-j, E_q\cup J)$ ($0\leq j\leq r-2-l$), the score $\tS$ is 
 $\tS=(r-2-l-j)+j+\min\{e_q,q-e_q\}=r-2-l+\min\{e_q,q-e_q\}$,
and clearly, we have $\tS<S:=S(l,E)=l+r+\min\{e_q,q-e_q\}$. 

(b) For $(\tl,\tE)=(l-r+j, E_q\cup J)$ ($0\leq j\leq r$), the score $\tS$ is 
$\tS=(l-r+j)+j+\min\{e_q,q-e_q\}=l-r+2j+\min\{e_q,q-e_q\}$,
and clearly, we have again $\tS<S$ if $J\neq R$, i.e., $\tE\neq E$.

In particular, if our starting pair $(l,E)$ is in group $2$, since $\cT_{l,E}$ is in $\cC$, the bundle $F^\vee_{l,E}$ is generated by $\cC^\vee$ 
by induction, since all other terms in the above Koszul resolution have lower score. 
Hence we may assume that $(l,E)$ is not in group $2$ (in addition to not being in group $1A$ or $1B$). 

There are four cases in the induction argument.

%\medskip

\underline{Case 1: $e_p<r$}. Since we assume that $(l,E)$ is not in group $1A$ (resp.,~$1B$), we have $l+e_p\geq r$ (resp.,~$l+e_p\geq r-1$).
We play the Koszul game~1 with a set $I=I_p$, $i_p=r+1$, 
$i_q=0$
(``minimal" unstable $I$ disjoint from $E$).
We verify that, for every $F^\vee_{\tl,\tE}$ in the list, its 
score $S(\tl,\tE)$ is less or equal than $S:=S(l,E)=l+e_p+\min\{e_q,q-e_q\}$.

(a) Let  $(\tl,\tE)=(l-j, E\cup J)$, $0<j\leq l$. Note that $l=0$ 
cannot occur for this type of $F^\vee_{\tl,\tE}$. The score $S(\tl,\tE)$ is 
$(l-j)+\min\{e_p+j,p-e_p-j\}+\min\{e_q,q-e_q\}=
l+\min\{e_p, p-e_p-2j\}+\min\{e_q,q-e_q\}\leq l+e_p+\min\{e_q,q-e_q\}=S$.
If $S(\tl,\tE)=S$ then, as $j>0$, we are done by induction on $l$. 

(b)  For  $(\tl,\tE)=(j-l-2, E\cup J)$, $l+2\leq j\leq i=r+1$, the score $S(\tl,\tE)$ is 
 $(j-l-2)+\min\{e_p+j,p-e_p-j\}+\min\{e_q,q-e_q\}$.
 Since $e_p+j\geq e_p+l+2\geq r+1$, we have 
 $\min\{e_p+j,p-e_p-j\}=p-(e_p+j)$. When working with group $1A$, since $l+e_p\geq r$, we have 
 $S(\tl,\tE)=p-e_p-l-2+\min\{e_q,q-e_q\}<S$.
 When working with group $1B$, we still have to consider the situation $l+e_p=r-1$.
 But then
 $\tl+(p-\tilde e_p)=(j-l-2)+(p-e_p-j)=r-1$,
  and therefore $F_{\tl,\tE}$ is in $1B$. 

\underline{Case 2: $e_p=r$, $e_q\leq s$}. 
As $(E,l)$ is not in group $2$, we must have
$l+e_q\geq s$. 

We now play the Koszul game 1 with 
a set $I$ with $i_p=r$, $i_q=s+1$ (``minimal" unstable $I$).
We verify that the 
the score of every $F^\vee_{\tl,\tE}$ in the complex is less or equal than 
$S:=S(l,E)=l+r+e_q$.

a) Let  $(\tl,\tE)=(l-j, E\cup J)$, $0<j\leq l$. Note again that $l=0$ 
cannot occur for this type of $F^\vee_{\tl,\tE}$. 
 The score $S(\tl,\tE)=(l-j)+(r-j_p)+\min\{e_q+j_q,q-e_q-j_q\}\leq (l-j)+(r-j_p)+(e_q+j_q)=l+r+e_q-2j_p\leq S$.
%with equality if and only if $j_p=0$. 
If $S(\tl,\tE)=S$ then, as $j>0$, we are done by induction on $l$.

(b)  For  $(\tl,\tE)=(j-l-2, E\cup J)$ ($l+2\leq j\leq i$), the score $S(\tl,\tE)
=(j-l-2)+(r-j_p)+\min\{e_q+j_q,q-e_q-j_q\}
\leq (j-l-2)+(r-j_p)+(q-e_q-j_q)<l+r+e_q=S$
 since we assume $l+e_q\geq s$. 
 
\underline{Case 3: $e_p\geq r+1$}. 
Recall that we may assume that  
$(E,l)$ is not in group $1A$ (resp., ~$1B$). Then  
$l+(p+1-e_p)\geq r$ (resp., ~$l+(p-e_p)\geq r$), or equivalently, $e_p\leq l+r+1$ (resp., ~$e_p\leq l+r$). 
We now play the Koszul game 2 with the set $I=I_p\subseteq E_p$, $i=i_p=r+1$ ($i_q=0$)
and with $E'_p=E_p\setminus I_p$, $E'_q=E_q$ ($e'_p=e_p-(r+1)$, $e'_q=e_q$ -- so ``minimal" unstable $I$).
We verify that the score of every $F^\vee_{\tl,\tE}$ in the complex is less than or equal to
$S:=S(l,E)=l+(p-e_p)+\min\{e_q,q-e_q\}$.

a) For  $(\tl,\tE)=(i-l-2-j, E'\cup J)$, $0\leq j\leq i-l-2$, the score $S(\tl,\tE)=(i-l-2-j)+\min\{e'_p+j, p-e'_p-j\}+\min\{e_q,q-e_q\}
\leq (i-l-2-j)+(e'_p+j)+\min\{e_q,q-e_q\}=e_p-l-2+\min\{e_q,q-e_q\}$.
In case $1B$, since $e_p\leq l+r$, we have $S(\tl,\tE)<S$.
In case $1A$, since $e_p\leq l+r+1$, we have $S(\tl,\tE)<S$, unless $e_p=l+r+1$, when $S(\tl,\tE)=S$. 
However, if $e_p=l+r+1$, we can move on as
$F^\vee_{\tl,\tE}$ is in group $1A$:
$\tl+\min\{e'_p+j, p+1-e'_p-j\}\leq (i-l-2-j)+(e'_p+j)=r-1$.

b) For $(\tl,\tE)=(j-i+l, E'\cup J)$, $\max\{i-l,0\}\leq j<i$, the score $S(\tl,\tE)=(j-i+l)+\min\{e'_p+j,p-e'_p-j\}+\min\{e_q,q-e_q\}\leq (j-i+l)+(p-e'_p-j)+\min\{e_q,q-e_q\}=l+p-e_p+\min\{e_q,q-e_q\}=S$
(since $e'_p=e_p-i$).  
If~$S(\tl,\tE)=S$, we are done by induction on $l$ since $j-i+l<l$. Note again that in this case $l>0$.

%$I=I_p=E_p$ ($i_p=r+1$, $i_q=0$)
%and with $E'$ such that $E=E'\sqcup I$. In particular, $e'_p=0$, $e'_q=e_q$. 
 %We verify that the score of every $F^\vee_{\tl,\tE}$ in the complex is less or equal than 
%$$s:=s(l,E)=l+(p-e_p)+\min\{e_q,q-e_q\}.$$

%\medskip

\underline{Case 4: $e_p=r$, $e_q\geq s+1$}. 
Recall that we may assume that  $(E,l)$ is in group $2$,
i.e., 
$l+(q-e_q)\geq s$, or equivalently, $e_q\leq l+s+1$. 
We now play the Koszul game 2 with the set $I$ with $I_p=E_p$, $I_q=E_q$, i.e., $E'=\emptyset$
(``maximal" unstable $I$ with $I\subseteq E$). In particular, we have $i=e$. 
We verify that the score of every $F^\vee_{\tl,\tE}$ in the complex is less or equal than 
$S:=S(l,E)=l+r+q-e_q$.

a) For  $(\tl,\tE)=(e-l-2-j, J)$, $0\leq j\leq e-l-2$, 
the score $S(\tl,\tE)=(e-l-2-j)+j_p+\min\{j_q,q-j_q\}\leq (e-l-2-j_q)+j_q=e-l-2=r+e_q-l-2<l+r+q-e_q=S$,
since by assumption $e_q\leq l+s+1$.

b) $(\tl,\tE)=(j-e+l, J)$ ($\max\{e-l,0\}\leq j< e$), the score 
$S(\tl,\tE)=(j-e+l)+j_p+\min\{j_q,q-j_q\}\leq j_q-e+l+2j_p+(q-j_q)=q-e+l+2j_p=q-e_q-r+l+2j_p\leq l+r+q-e_q$,
with equality if and only if $j_q\geq s+1$, $j_p=r$. If $S(\tl,\tE)=S$ then, since $j<e$, we are done by induction on $l$ since 
$j-e+l<l$. Note again that in this case $l>0$. 
\ep
Along the proof of Theorem \ref{K games}, we used the following Lemma, that we prove now. 

\begin{lemma}\label{dual}
Let $\M$ be one of $\M_{p,q}$ or $\M_{p,q+1}$. Let $R=\{1,\ldots,r\}\subseteq P$, $|R|=r$ and 
let $i: Z_R\ra\M$ denotes the inclusion map. If $(l,E)$ is a pair with $E_p=R$, the derived dual of 
$\cT_{l,E}$ %=i_*\big(\frac{e_q-r-l}{2}\psi_u+\sum_{j\in E_q}\de_{ju}\big)$
is given by
$\cT^\vee_{l,E}=i_*\left(\frac{2-e_q-r+l}{2}\psi_u-\sum_{j\in E_q}\de_{ju}\right)[1-r]$.
\end{lemma}

\bp
If $L$ is a line bundle such that $L_{|Z_R}=\cT_{l,E}=i_*\big(\frac{e_q-r-l}{2}\psi_u+\sum_{j\in E_q}\de_{ju}\big)$ then 
$\cT^\vee_{l,E}=L^\vee\otimes\cO_{Z_R}^\vee=L^\vee\otimes\det N_{Z_R}[-c]$
\cite[3.40]{Huybrechts}, where $c=r-1$ is the codimension of $Z_R$ and
$\det N_{Z_R}=\sum_{j\in R\setminus\{1\}}(\de_{1j})_{|Z_R}=-(r-1)\psi_u$ 
(use Lemma \ref{psiX} for $\M_{p,q}$ and its analogue for $\M_{p,q+1}$)
is the determinant of the normal bundle.
\ep

\begin{cor}\label{generation}
$D^b(\M_{p,q})$ is generated by the bundles $F_{l,E}$ with $(l,E)$ in group $1A$ (resp., $1B$) and group $2$.
\end{cor}

\bp
This follows immediately from the proof of Theorem  \ref{K games}. 
\ep

We finish this section by analyzing fullness on $\cU_{p,q}$, the universal family over $\M_{p,q}$.
This will be used in the next section.
Recall from Section \ref{induction1} that
$\cU_{p,q}$ is  a Hassett space for the set of markings $P\cup \tQ$,
where $\tQ=Q\cup\{z\}$ and $z$ is an extra light index, and since $\cU_{p,q}$ is a  GIT quotient (Lemma \ref{fbadfbadhadthn})
the space $\cU_{p,q}$ carries vector bundles $F_{l,E}$ by Definition \ref{FlE}. 

\begin{thm}\label{full Upq}
$D^b(\cU_{p,q})$ is generated by the bundles $F_{l,E}$ with $(l,E)$ in one of the following groups:
(i) $1A$ (resp., ~$1B$), (ii) $2A$ with $z\notin E$, (iii) $2B$ with $z\in E$. 
Here we use the same groups $1A$, $1B$, $2A$ and $2B$ as in Theorem  \ref{p,q+1 case}. 
\end{thm}

\bp
This is a consequence of Orlov's Theorem for the universal family  $\cU_{p,q}\arrow\pi\M_{p,q}$, which is a $\PP^1$ bundle,
and is similar to the proof of Lemma \ref{fbfgafgadf}. 
Recall that if $z\notin E$, then $F_{l,E}=\pi^*F_{l,E}$ by Theorem \ref{stackofbundles}.
If $z\notin E$, the range $2A$ on $\cU_{p,q}$ is exactly the range of group $2$ on $\M_{p,q}$ and the range $1A$ (resp., ~$1B$)
on $\cU_{p,q}$ corresponds to $1A$ (resp., ~$1B$) on 
$\M_{p,q}$. Hence, by Corollary \ref{generation} 
the $F_{l,E}$'s with $z\notin E$ generate $L\pi^*D^b(\M_{p,q})$. 
If $z\in E$, then $(l,E)$ is in range $2B$ on $\cU_{p,q}$ if and only if $(l,E^c)$ is in the group $2$ on $\M_{p,q}$, and similarly, 
$(l,E)$ is in range $1A$ (resp., ~$1B$) on $\cU_{p,q}$ if and only if $(l,E^c)$ is in the group $1B$ (resp., ~$1A$) on $\M_{p,q}$.  
By Corollary \ref{generation}, the corresponding duals $\{F^\vee_{l,E^c}\}$ (where $E^c=(P\cup\tQ)\setminus E)$) 
generate $L\pi^*D^b(\M_{p,q})$. 
As  $F_{l,E}=F^\vee_{l,E^c}\otimes F_{0,N}$, it follows that the objects with $z\in E$ generate 
$L\pi^*D^b(\M_{p,q})\otimes F_{0,N}$. As in  the proof of Lemma \ref{fbfgafgadf}, the result now follows by Orlov's Theorem. 
\ep

\section{Fullness of the  exceptional collection on $\M_{2r,2s+2}$}\label{fullness p,q+1 section}

\begin{notn}\label{C notations}
We let  $\cA$ to be the subcategory 
in $D^b(\M_{p,q+1})$ generated by the torsion sheaves $\cO_{\de_{T,T^c}}(-a,-b)$ of Notation \ref{TT}. 
Let $\cC$ be the collection of Theorem~\ref{asdvzsfvsfb} (exceptional by \S~\ref{exceptional p,q+1}) which consists of generators of $\cA$,
vector bundles $F_{l,E}$ in group $1A$  (resp., ~$1B$) combined with the complexes $\cTT_{l,E}$ 
in group $2B$.
We let $\cC'$, resp., $\cC'_F$, be the collection obtained from $\cC$ by replacing complexes $\cTT_{l,E}$ 
in $\cC$ with the torsion sheaves $\cT_{l,E}$, resp.,~the vector bundles $F_{l,E}$ from the same group $2B$.
Note that 
the collections $\cC'$ and $\cC'_F$ 
are not, in general, exceptional.
We follow Notation \ref{Upq}: $p=2r$, $q=2s+1$, $\tQ=Q\cup\{z\}$. 
Throughout the section, unless we specify otherwise, we assume $p\geq4$, $q+1\geq0$. 
\end{notn}

\begin{prop}\label{score r+s}
Let $S'$ be the score of Notation~\ref{asgasfgasgsg}.
Let $l\geq0$, $E\subseteq P\cup\tQ$, $l+e$ is even, $S'(l,E)\leq r+s$. Then 
\bi
\item[(1) ] The bundle $F_{l,E}$ is generated by $\cA$ together with $\{F_{l',E'}\}$ with $(l',E')$ in group $1A$ 
(resp.,~$1B$) or $2B$, and furthermore with $S'(l',E')\leq S'(l,E)$. 
\item[(2) ] The bundle $F_{l,E}$ is generated by $\cC'_F$.
\ei
\end{prop}

We prove Proposition \ref{score r+s} later in this section, after we first establish some auxiliary results. Recall that by Lemma \ref{basic}, 
for  $l\geq0$, $E\subseteq P\cup\tQ$  and $(l,E)$ in any of the groups $1A$, $1B$, $2A$ or $2B$, we have $S'(l,E)\leq r+s$. Similarly, for 
 $l\geq0$, $E\subseteq P\cup Q$  and $(l,E)$ in any of the groups $1A$, $1B$, or $2$, we have $S(l,E)\leq r+s-1$ (here $p\geq4$, $q\geq0$).  

%\begin{rmk}
%The main ingredient in the proof of Theorem  \ref{score r+s} is Lemma \ref{refined induction}. We are unable to prove in Lemma 
 %\ref{refined induction} (more specifically in Case 4)) that the bundle $F_{l,E}$ can be generated by $\cA$ together with bundles $\{F_{l',E'}\}$ such that 
%$S'(l',E')<S'(l,E)$ or with $S'(l',E')=S'(l,E)$ and $l'<l$. However, for our purposes (proof of Theorem  \ref{replace}, namely Claim \ref{main replace}) it suffices to prove that the bundle $F_{l,E}$ can be generated by $\cA$ together with $\{F_{l',E'}\}$ with $(l',E')$ in group $1A$ or $2B$ and such that 
%$S'(l',E')\leq S'(l,E)$. 
%\end{rmk}

\begin{notn}
Recall the notation of Lemma \ref{SRgSRgarg}: let $\al: W\ra \M_{p,q+1}$ be the universal family and 
$f: W\ra\M_{p,q+2}=\M_{P,\tQ\cup\{y\}}$ the birational map that contracts the $T^c$ component of boundary divisors $\de_{T\cup\{y\},T^c}$ on 
$W$ (where $y$ is the new marking on $W$). 
%When $q+1=0$, we have $\M_{p,q+2}=\M_{p,1}$ and $f: W\ra \M_{p-1,1}$ is the map considered in Section \ref{fullness odd p}. 
\end{notn}

\begin{rmk}\label{same groups}
For $l\geq0$, $E\subseteq P\cup\tQ$, the pair $(l,E)$ can be considered on both of $\M_{p,q+1}$ and $\M_{p,q+2}$ as $E\subseteq P\cup\tQ\subseteq P\cup\tQ\cup\{y\}$. Consider such a pair $(l,E)$.  
\bi
\item[(i) ] Clearly, $(l,E)$ is in group $1A$ (resp., $1B$) on $\M_{p,q+1}$ if and only if it is in group $1A$ (resp., $1B$) on $\M_{p,q+2}$. 
Furthermore, $(l,E)$ is in group $2B$ on $\M_{p,q+1}$ if and only if $(l,E)$ is in group $2$ on $\M_{p,q+2}$. 

\item[(ii) ] The score $S'(l,E)$ on $\M_{p,q+1}$ relates to the score
$$S(l,E)=l+\min\{e_p,p-e_p\}+\min\{e_q,q+2-e_q\}\ \text{on}\ \M_{p,q+2}$$ by 
$S'(l,E)=
\begin{cases}
S(l,E) & \hbox{\rm if}\quad e_q\leq s+1\cr
S(l,E)-1 & \hbox{\rm if}\quad  e_q\geq s+2.\cr
\end{cases}$

\item[(iii) ] 
As noted above, for $(l,E)$ on $\M_{p,q+2}$ in groups $1A$, $1B$, $2$, we have 
$S(l,E)\leq r+s$. Furthermore, by Lemma \ref{basic}(i) equality holds if and only if we are in one of the following cases:
\bi
\item[ (1A)]  $e_p=r-1-l$, $e_q=s+1$ or $s+2$;
\item[ (1B)]  $e_p=r+1+l$, $e_q=s+1$ or $s+2$;
\item[ (2) ] $e_p=r$, $e_q=s-l$ or $l+s+3$. 
\ei
\ei
\end{rmk}

\begin{lemma}\label{alpha f game}
Let $l\geq0$ and $E\subseteq P\cup \tQ$, $l+e$ even. 
Then on $\M_{p,q+1}$ the bundle $F_{l,E}$ and the complex $R\al_*f^*F_{l,E}$ are related by quotients (see Definition \ref{related}) which are generated by sheaves of the form 
\begin{equation}\label{afssgsgWG}
Q=\cO_{\de_{T,T^c}}(-v,u),\quad 0\leq u\leq m_{T^c}-1,\quad 0<v\leq m_{T^c}-1,
\end{equation}
where $m_{T^c}=\max(0, f_{T^c,E,l})\leq S'(l,E)/2$.
In particular, $F_{l,E}$ and $R\al_*f^*F_{l,E}$ are related by quotients in $\cA$ if the score $S'(l,E)\leq r+s$.
\end{lemma}
Recall that unless we will write $\cO_\de(a,b)=\cO_T(a)\boxtimes\cO_{T^c}(b)$ to emphasize that $\cO(a)$ corresponds to the $T$ component, the notation
$\cO_\de(a,b)$ will generally mean that $\cO(a)$ could correspond to either the  $T$ or $T^c$ component. 
\bp[Proof of Lemma \ref{alpha f game}]
By Proposition \ref{Q on W},  the bundles $F_{l,E}$ and $f^*F_{l,E}$ on $W$ are related by quotients 
$-jH\boxtimes\cO(u)$, $0<j\leq m_{T^c}$, $0\leq u\leq m_{T^c}-1$,
supported on $\de_{T\cup\{y\},T^c}\cong\Bl_1\PP^{r+s}\times\PP^{r+s-1}$. By Lemma \ref{Push}, 
$R\pi_*(-jH)=0$  if $j=1$, while  if $j\geq2$, it is generated by $\cO(-v)$ with $1\leq v\leq j-1$. It~follows that 
$R\al_*\big(-jH\boxtimes\cO(u)\big)$ is generated by the quotients \eqref{afssgsgWG}.
%$$\cO_{\de_{T,T^c}}(-v,u),\quad 0\leq u\leq m_{T^c}-1,\quad 0<v\leq m_{T^c}-1,$$
The rest follows from Lemma~\ref{bounds+}.
\ep

\begin{lemma}\label{refined induction}
Let $l\geq0$, $E\subseteq P\cup\tQ$, $l+e$ even, $S'(l,E)\leq r+s$. On $\M_{p,q+2}$, the bundle $F_{l,E}$ can be generated by 
$\{F_{l',E'}\}$ for $(l',E')$ belonging to group $1A$ (resp.,~$1B$)
and $2$ on $\M_{p,q+2}$ such that $y\notin E'$, and such that $S'(l',E')\leq S'(l,E)$. 
\end{lemma}

\bp
The proof is parallel to the proof of Theorem \ref{K games}, playing the same Koszul games as on $\M_{p,q}$, except this time on $\M_{p,q+2}$. 
We do an induction on the score $S(l,E)$ on $\M_{p,q+2}$ and for equal scores induction on $l$. 
By~Remark~\ref{same groups}(ii),
$S(l,E)=S'(l,E)$ if $e_q\leq s+1$ and $S(l,E)=S'(l,E)+1$ if $e_q\geq s+2$.  
We only point out the main arguments for each case.

\underline{Case 1: $e_p<r$}. 
Clearly, we may assume $(E,l)$ is not in group $1A$
(resp., $1B$). 
This is parallel to Case 1 in the proof of Theorem \ref{K games}. 
As in that case, using the same Koszul game 1 on $\M_{p,q+2}$ (with $I=I_p$ with $|I_p|=r+1$, disjoint from $E$), we generate $F_{l,E}$ with 
$F_{\tl,\tE}$ with $\tE$ of the form $E\cup J$, with $J$ a set of heavy indices (in particular, $y\notin \tE$) and such that either
$S(\tl,\tE)<S(l,E)$ or $S(\tl,\tE)=S(l,E)$, $\tl<l$, to which one has to add, when working with group $1B$, the possibility that 
$S(\tl,\tE)=S(l,E)$ and $(\tl,\tE)$ is in group $1B$. 
Since $\tE_q=E_q$, we have that either that $S'(\tl,\tE)=S(\tl,\tE)$, $S'(l,E)=S(l,E)$, or 
$S'(\tl,\tE)=S(\tl,\tE)-1$, $S'(l,E)=S(l,E)-1$. In particular, $S'(\tl,\tE)\leq S'(l,E)\leq r+s$, i.e., the hypotheses in the Lemma are satisfied for $(\tl,\tE)$.

\underline{Case 2: $e_p=r$, $e_q\leq s+1$}. 
Clearly, we may assume $(E,l)$ is not in group $2$ on $\M_{p,q+2}$.
By our assumption that $S'(l,E)\leq r+s$, we cannot have  $e_q=s+1$. Hence, $e_q\leq s$. 
This is parallel to Case 2 in the proof of Theorem \ref{K games}. 
As in that case, using the Koszul game 1 on $\M_{p,q+2}$ for a set $I$ disjoint from $E$ with $|I_p|=r$ and $|I_q|=s+2$, and in addition with $y\notin I_q$ (possible since $e_q\leq s$), 
we generate $F_{l,E}$ with $F_{\tl,\tE}$ with $\tE$ of the form $E\cup J$, with $J\subseteq I$ (in particular, $y\notin \tE$) and  such that either
$S(\tl,\tE)<S(l,E)$ or $S(\tl,\tE)=S(l,E)$, $\tl<l$. Furthermore, we have 
$S'(\tl,\tE)\leq S(\tl,\tE)\leq S(l,E)=S'(l,E)\leq r+s$,
(as $S'(l,E)=S(l,E)$ because $e_q\leq s$) and if $S'(\tl,\tE)=S'(l,E)$, then $\tl<l$. 

\underline{Case 3: $e_p\geq r+1$}. 
Clearly, we may assume $(E,l)$ is not in group $1A$
(resp.,~$1B$).
This is parallel to Case 3 in the proof of Theorem \ref{K games}. 
As in that case, using the same Koszul game 2 on $\M_{p,q+2}$ ($I=I_p\subseteq E_p$, $|I|=r+1$),  we generate $F_{l,E}$ with 
$F_{\tl,\tE}$ with $\tE$ of the form $E\setminus J'$, with $J'$ a set of heavy indices (in particular, $y\notin \tE$) and such that either
$S(\tl,\tE)<S(l,E)$, or $S(\tl,\tE)=S(l,E)$, $\tl<l$, or, when working with group $1A$, $S(\tl,\tE)=S(l,E)$ and $(\tl,\tE)$ is in group $1A$. Since $\tE_q=E_q$, 
we have that either that $S'(\tl,\tE)=S(\tl,\tE)$, $S'(l,E)=S(l,E)$, or 
$S'(\tl,\tE)=S(\tl,\tE)-1$, $S'(l,E)=S(l,E)-1$. In particular, 
$S'(\tl,\tE)\leq S'(l,E)\leq r+s$, i.e., the hypotheses in the Lemma are satisfied for $(\tl,\tE)$.

\underline{Case 4: $e_p=r$, $e_q\geq s+2$}.  
Clearly, we may assume $(E,l)$ is not in group $2$ on $\M_{p,q+2}$.
This is parallel to Case 4 in the proof of Theorem \ref{K games}. 
As in that case, using the Koszul game 2 on $\M_{p,q+2}$ ($I=E$),  we generate $F_{l,E}$ with 
$F_{\tl,\tE}$ with $\tE$ of the form $\tilde{E}\subseteq E$, (in particular, $y\notin \tE$) and such that either
$S(\tl,\tE)<S(l,E)$ or $S(\tl,\tE)=S(l,E)$, $\tl<l$. Note, if $S(\tl,\tE)<S(l,E)$ then 
$S'(\tl,\tE)\leq S(\tl,\tE)<S(l,E)=S'(l,E)+1$
(since $e_q\geq s+2$). In~particular, it follows that $S'(\tl,\tE)\leq S'(l,E)\leq r+s$. 
%If $S'(\tl,\tE)<S'(l,E)$, we are done by induction. If $S'(\tl,\tE)=S'(l,E)$ then 
%$S'(\tl,\tE)=S(\tl,\tE)$, i.e., $|\tE_q|\leq s+1$. As $|\tE_p|\leq e_p=r$, we must go back to Case 1) or 2) and
%reduce the score, or for equal scores, $l$. Again, we are done by induction. 

Assume now $S(\tl,\tE)=S(l,E)$, $\tl<l$. We still need to prove that in this case 
$S'(\tl,\tE)\leq S'(l,E)\leq r+s$. As in Case 4 of the proof of Theorem\ref{K games}, the only way to have  $S(\tl,\tE)=S(l,E)$ is when we are in 
case b) $(\tl,\tE)=(j-e+l,J)$ ($j<e$, so $\tl<l$) and $j_q\geq s+2$ ($s$ is replaced by $s+1$, since we are on $\M_{p,q+2}$ instead of $\M_{p,q}$). Since $|\tE_q|=j_q\geq s+2$, we have $S'(\tl,\tE)=S(\tl,\tE)-1$. Hence, 
$S'(\tl,\tE)=S(\tl,\tE)-1=S(l,E)-1=S'(l,E)\leq r+s$.
\ep

\bp[Proof of Proposition~\ref{score r+s}]
Let $l\geq0$, $E\subseteq P\cup\tQ$, $l+e$ even, $S'(l,E)\leq r+s$.
By~Lemma~\ref{refined induction}, the bundle $F_{l,E}$ on $\M_{p,q+2}$ can be generated by 
bundles $F_{l',E'}$ for $(l',E')$ belonging to group $1A$ (resp., ~$1B$) and $2$, with $y\notin E'$,
and such that $S'(l',E')\leq S'(l,E)$. 
By Remark \ref{same groups}(i), such  $(l',E')$ can be considered also on $\M_{p,q+1}$ and they correspond to groups $1A$ (resp., $1B$) and group $2B$ on $\M_{p,q+1}$.
It follows that the object $R\al_*f^*F_{l,E}$ on $\M_{p,q+1}$ can be generated by the objects
$R\al_*f^*F_{l',E'}$.
By Lemma \ref{basic}, for all such $(l',E')$ on $\M_{p,q+1}$ we have $S'(l',E')\leq r+s$. 
Lemma~\ref{alpha f game} now implies part (1) of Proposition~\ref{score r+s}. Part (2) follows immediately from (1). 
\ep

\begin{defn}\label{F'}
Let $l\in\ZZ$ (positive or negative).
Recall that objects $F_{l,E}$ on the stack $\cM_n$, and therefore also on all GIT quotients without strictly semistable points, 
were defined in Proposition~\ref{sRGASRGASRGASG}.
It was proved there that $F_{l,E}=0$ for $l=-1$ and
$F_{l,E}\cong F_{-l-2,E}[-1]$ for $l\leq -2$.
We would like to define analogous objects on $\M_{p,q+1}$.
Let $\al: W\ra \M_{p,q+1}$ be the universal family. 
For any $l\in\ZZ$ (positive or negative) and $E\subseteq P\cup\tQ$, let 
\begin{equation}\label{F' formula}
N'_{l,E}=\om_{\al}^{\frac{e-l}{2}}(E),\quad F'_{l,E}=R\al_*(N'_{l,E}).
\end{equation}
\end{defn}

\begin{prop}\label{F' vs F}
Let $l\geq-1$, $E\subseteq P\cup\tQ$, $l+e$ even. The bundle $F_{l,E}$ on $\M_{p,q+1}$
as defined in Definition~\ref{allF} and the object
$F'_{l,E}$ of (\ref{F' formula}) are related by quotients (see Definition \ref{related}) that are generated by sheaves of the form 
$\cO_{\de}(-v,u)$, $v>0$, $u\geq 0$, $u, v\leq \frac{S'(l,E)}{2}-l-1$. 
In particular, if $l\geq0$, $S'(l,E)\leq r+s$, $F'_{l,E}$, $F_{l,E}$ are related by quotients in $\cA$.
\end{prop}

\bp
By Definition \ref{allF}, $F_{l,E}=R\al_*(N_{l,E})$. 
We compare $F'_{l,E}$ with  $F_{l,E}$ by comparing on $W$ the line bundles  $N'_{l,E}$ and $N_{l,E}$: 
$$N'_{l,E}=N_{l,E}+\sum_{T\sqcup {T}^c=P\cup\tQ, f_{T,E,l}<0}\left(-f_{T,E,l}\right)\de_{T\cup\{y\}}.$$
Since $l\geq -1$ and $e+l$ is even, we cannot have both $f_{T,E,l}<0$,  $f_{T^c,E,l}<0$ (Remark \ref{f_T positive}). Hence, for every partition 
$T\sqcup {T}^c=P\cup\tQ$ at most  one of $\de_{T\cup\{y\},{T}^c}$, $\de_{{T}^c\cup\{y\}, T}$ appears on the right hand side of the equality.
The line bundles $N'_{l,E}$ and $N_{l,E}$ on $W$ are related by quotients of the form
$Q=\big(N'_{l,E}+(-\al_T+i)\de\big)_{|\de}=\big(-iH\big)\boxtimes\cO(\al_{T}-i)$ (use Lemma \ref{restoboundaryofM}), where 
$\de=\de_{T\cup\{y\}, {T}^c},\quad 0<i\leq \al_{T}:=-f_{T,E,l}$.
Since by Lemma \ref{Push}, $R\pi_*(-iH)$ is $0$ if $i=1$ and generated by $\cO(-v)$, with $v=1,\ldots, i-1$ if $i\geq2$, it follows that 
$F'_{l,E}$ and $F_{l,E}$ are related by quotients that are generated by sheaves of the form 
$\cO_{T}(-v)\boxtimes\cO_{{T}^c}(u)$, 
$0\leq u=\al_T-i\leq \al_{T}-1$, 
$0<v<i\leq\al_{T}$. By Lemma \ref{bounds+}, we have $\al_T\leq \frac{S'(l,E)}{2}-l$, hence, $u, v\leq\frac{S'(l,E)}{2}-l-1$. 
\ep

\begin{cor}\label{sdfDAgdsgsDG}
Suppose $r+s$ is even.
Fix a boundary $\de=\de_{T_0,T_0^c}\subset \M_{p,q+1}$ and let $J\subseteq T_0$, $j=|J|$ $(0\leq j\leq r+s)$.  
Then the following pairs 
\bi
\item $F_{j-1,T_0^c\cup J}$ and $F'_{j-1,T_0^c\cup J}$ (if $j>0$),
\item $F'_{-1,T_0^c}$ and  $\cO_{\de}(-\frac{r+s}{2},0)=\cO_{T_0}(-\frac{r+s}{2})\boxtimes\cO_{T_0^c}$, 
\ei
are related by quotients in~$\cA$.
\end{cor}

\bp
This is a particular case of Proposition \ref{F' vs F} (and its proof) for pairs $(l,E)=(j-1,T_0^c\cup J)$, $|J|=j$. Note that 
$S'(l,E)=r+s$. If $j>0$, the first statement follows from Proposition \ref{F' vs F}. 
Consider now the case $j=0$. Set $(l,E)=(-1,T_0^c)$. Recall from the proof of Proposition \ref{F' vs F} that $F'_{l,E}$ and $F_{l,E}=0$ 
are related by quotients that are generated by sheaves
of the form $\cO_{T}(-v)\boxtimes\cO_{{T}^c}(u)$ (for some boundary $\de_{T,T^c}$) 
$0\leq u=\al_T-i\leq \al_{T}-1$, 
$0<v<i\leq\al_{T}$, and $\al_T\leq\frac{S'(l,E)}{2}-l=\frac{r+s}{2}+1$.  Such quotients belong to the generators of $\cA$ when $v<\frac{r+s}{2}$. 
So the only possible exception is when $v=\frac{r+s}{2}$, which implies that $i=\al_T=\frac{S'(l,E)}{2}-l$. In particular, $u=0$. By Lemma \ref{bounds+}, the equality 
$\al_T=\frac{S'(l,E)}{2}-l$ happens if and only if $T=T_0$ and the second statement follows. 
\ep

\begin{cor}\label{bad quotients}
Suppose $r+s$ is even.
Consider a partition $T\sqcup T^c=P\cup \tQ$. The 
sheaves  $\cO_{\de}\left(-\frac{r+s}{2},0\right)$,  $\cO_{\de}\left(0,-\frac{r+s}{2}\right)$,
$\de=\de_{T,T^c}\subseteq \M_{p,q+1}$
are generated by $\cC'_F$.
\end{cor}

In Corollary \ref{bad quotients} we write both $\cO_\de(a,b)$, $\cO_\de(b,a)$ to emphasize that  $\cO(a)$ could correspond to either the  $T$ or $T^c$ component. We will do the same in Proposition \ref{F' vs beta}. 

\bp[Proof of Corollary  \ref{bad quotients}]
We prove that $\cO_{\de}(-\frac{r+s}{2},0)$ 
(hence, by symmetry, also $\cO_{\de}\left(0,-\frac{r+s}{2}\right)$)
is generated by the
bundles $\{F_{l,E}\}$ on $\M_{p,q+1}$ with score $S'(l,E)=r+s$ and objects in $\cA$. 
In particular, by Proposition \ref{score r+s}(2), it is generated by the collection $\cC'_F$. 

Consider the Koszul resolution of $\bigcap_{i\in T}\de_{iy}=\emptyset$ in $W$: 
$$0\leftarrow \cO\leftarrow\ldots %\oplus_{i\in T}\cO(-\de_{iy})
\leftarrow\ldots\leftarrow
\oplus_{J\subseteq T, |J|=j}\cO(-\sum_{i\in J}\de_{iy})
\leftarrow\ldots\leftarrow\cO(-\sum_{i\in T}\de_{iy})\leftarrow 0.$$
Dualizing and tensoring with $\om_{\al}^{\frac{r+s}{2}+1}\big(\sum_{i\in T^c}\de_{iy}\big)$ and using the notation 
$N'_{l,E}:=\om_{\al}^{\frac{e-l}{2}}(E)$ (see Definition \ref{F'},), we have
the following long exact sequence of line bundles on $W$:
$$0\rightarrow N'_{-1,T^c}\rightarrow 
%\bigoplus_{i\in T}N'_{0,T^c\cup\{i\}}\rightarrow
\ldots\rightarrow\bigoplus_{J\subseteq T, |J|=j}N'_{j-1, T^c\cup J}\rightarrow\ldots\rightarrow N'_{r+s, T^c\cup T}\rightarrow 0.$$
By applying $R\al_*\big(-\big)$, we obtain the following objects  on $\M_{p,q+1}$:
$$F'_{-1,T^c}\quad \oplus_{i\in T}F'_{0,T^c\cup\{i\}}\quad\ldots\quad\oplus_{J\subseteq T, |J|=j}F'_{j-1, T^c\cup J}\quad\ldots
\quad F'_{r+s, T^c\cup T}.$$
All the vector bundles $F_{j-1,T^c\cup J}$ have score $r+s$ and the statement follows by Proposition \ref{score r+s}(2) and Corollary~\ref{sdfDAgdsgsDG}. 
\ep

\begin{notn}
As in Notation  \ref{Upq}, let $\be=\be_z: \M_{p,q+1}\ra \cU_{p,q}$ be the morphism that contracts the $T$ component of any boundary divisor $\de_{T,T^c}$ if $z\in T$. 
When $q+1=0$, we choose $z\in P$ and $\be_z$ is the map $\M_p\ra \M_{p-1,1}$ which lowers the weight of a heavy 
index $z$. 
\end{notn}

\begin{prop}\label{F' vs beta}
Let $l\in\ZZ$ (positive or negative), $E\subseteq P\cup \tQ$, $z\in\tQ$. Then $F'_{l,E}$ and $\be_z^*F_{l,E}$ are related by quotients
that are generated by sheaves of the form
$\cO_{\de}(-v,u)$, 
$\cO_{\de}(u,-v)$, where 
$0\leq u$, 
$0<v\leq \max\left\{\frac{S'(l,E)}{2}-l-1,\frac{S'(l,E)}{2}\right\}$. 
\end{prop}

The proof of Proposition \ref{F' vs beta} will show in fact that $F'_{l,E}$ and $\be_z^*F_{l,E}$ are related by quotients 
that are generated by sheaves $\cO_T(u)\boxtimes\cO_{T^c}(-v)$, $0\leq u$, $0<v\leq \frac{S'(l,E)}{2}$ and $\cO_T(-v)\boxtimes\cO_{T^c}(u)$, $0\leq u$, $0<v\leq \frac{S'(l,E)}{2}-l-1$, but we will not use this fact. 

\bp[Proof of Proposition \ref{F' vs beta}]
Let $N_1=N'_{l,E}$. Recall from (\ref{F' formula}) that $F'_{l,E}=R\al_*(N_1)$. We have by (\ref{N2 vs N1}) that $\be_z^*F_{l,E}=R\al_*(N_2)$, where
$N_2=N_1+\sum_{z\in T}f_{T,E,l}\de_{T\cup\{y\}}$.
We let 
$$N_3:=N_1+\sum_{z\in T, f_{T,E,l}>0}f_{T,E,l}\de_{T\cup\{y\}}
=N_2+\sum_{z\in T, f_{T,E,l}<0}\big(-f_{T,E,l}\big)\de_{T\cup\{y\}}.$$
The line bundles $N_1$ and $N_3$ are related by quotients 
$\big(N_1+i\de\big)_{|\de}=
\big(f_{T,E,l}-i\big)H\boxtimes \cO_{T^c}(-i)$,
$\de=\de_{T\cup\{y\},T^c}$ ($z\in T$), $0<i\leq f_{T,E,l}$ 
(use Lemma \ref{restrict}).
It follows by Lemma \ref{Push} that $F'_{l,E}=R\al_*(N_1)$ and $R\al_*(N_3)$ are related by quotients of the form 
$Q_1=\cO_T(u)\boxtimes\cO_{T^c}(-i)$, $u\geq 0$, $0<i\leq f_{T,E,l}$.
By Lemma~\ref{bounds+}, $f_{T,E,l}\leq \frac{S'(l,E)}{2}$.  
%(with equality if and only if $E=T$), we have that for all $l\leq -1$, 
%$$|E\cap T|-\frac {e-l}{2}\leq \frac{r+s}{2},$$
%with equality if and only if $l=-1$ and $E=T$ ($e_p=r$, $e_q=s+1$). 

The line bundles $N_2$ and $N_3$ are related by quotients 
$\big(N_2+i\de\big)_{|\de}=\break
\left(N_1-\left(-f_{T,E,l}\right)\de+i\de\right)_{|\de}
=(-iH)\boxtimes \cO_{T^c}\left(-f_{T,E,l}-i\right)$,
$\de=\de_{T\cup\{y\},T^c}$ ($z\in T$),
$0<i\leq -f_{T,E,l}$ 
(use Lemma \ref{restrict}).
Since by Lemma \ref{Push} that $R\pi_*(-iH)$ is either $0$ or generated by $\cO(-v)$ with $v=1,\ldots,i-1$, it follows that 
$F'_{l,E}=R\al_*(N_2)$ and $R\al_*(N_3)$ are related by quotients 
that are generated by sheaves of the form 
$\cO_T(-v)\boxtimes\cO_{T^c}(u)$,
$u\geq 0$,
$0<v\leq  -f_{T,E,l}-1$.
Since 
by Lemma \ref{bounds+}
$-f_{T,E,l}\leq \frac{S'(l,E)}{2}-l$, the result follows.   
\ep

\begin{cor}\label{beta replace}
Let $l\geq0$, $E\subseteq P\cup \tQ$, $\tQ=Q\cup \{z\}$. Assume\break 
$S'(l,E)\leq r+s$. Then 
$F_{l,E}$ and $\be_z^*F_{l,E}$ are related by quotients that are generated by $\cA$ and sheaves of type
$\cO_{\de}\left(0,-\frac{r+s}{2}\right)
=\cO_T\boxtimes\cO_{T^c}\left(-\frac{r+s}{2}\right)$. 
%Furthermore, 
%They are related by quotients in $\cA$ 
In fact the quotients in $\cA$ suffice 
if any of the following holds:
\bi
\item[(i) ] $r+s$ is odd or $S'(l,E)<r+s$;
\item[(ii)] $r+s$ is even, $S'(l,E)=r+s$, $e_q\geq s+1$ and $z\notin E$.
\ei
\end{cor}

\bp
By Lemma  \ref{bounds+}, for all partitions $T\sqcup T^c=P\cup \tQ$ we have $m_{T,E,l}\leq \frac{S'(l,E)}{2}$, with equality if and only if ($T_p\subseteq E_p$ or $E_p\subseteq T_p$) and ($T_q\subseteq E_q$ or $E_q\subseteq T_q$). By Proposition \ref{sifjkdfjkwdjk}(i), $F_{l,E}$ and $\be_z^*F_{l,E}$ are related by quotients in $\cA$ and sheaves of type
$\cO_T(u)\boxtimes\cO_{T^c}(-j)$ where $0<j\leq m$, $0\leq u<m$, $m:=m_{T,E,l}$. The assumption $S'(l,E)\leq r+s$ implies that $m\leq \frac{r+s}{2}$. It follows that $F_{l,E}$ and $\be_z^*F_{l,E}$ are related by quotients in $\cA$ and possibly sheaves of type 
$\cO_T(u)\boxtimes\cO_{T^c}(-\frac{r+s}{2})$, $0\leq u<\frac{r+s}{2}$. But as the generators of $\cA$ and the sheaves $\cO_T\boxtimes\cO_{T^c}(-\frac{r+s}{2})$ generate all
$\cO_T(u)\boxtimes\cO_{T^c}(-\frac{r+s}{2})$ with $u$ positive or negative, the first statement follows. (We use here that 
$D^b(\PP^t)$ is generated by $\cO(-t),\ldots,\cO(-1),\cO$). Clearly, if $r+s$ is odd or if $S'(l,E)<r+s$, then the sheaves
$\cO_T(u)\boxtimes\cO_{T^c}(-\frac{r+s}{2})$ do not appear, which implies (i). To see (ii), note that for the sheaves $\cO_T(u)\boxtimes\cO_{T^c}(-\frac{r+s}{2})$ to appear, we must have $S'(l,E)=r+s$ is even and  $m_{T,E,l}=\frac{S'(l,E)}{2}$. Hence, by Lemma  \ref{bounds+}, $T_q\subseteq E_q$ or $E_q\subseteq T_q$. But if $e_q\geq s+1$, we must have 
$T_q\subseteq E_q$. In particular, $z\in E$. This implies (ii).    
\ep

%\bp
%The first statement and part (i) follow from Proposition \ref{beta game} and Lemma \ref{bound1}, since $S'(l,E)<r+s$ and  $S'(l,E)$ is even.  
%By  Proposition \ref{beta game}, if $S'(l,E)\leq r+s$, the only way quotients not in $\cA$ may appear is if $r+s$ is even, $S'(l,E)=r+s$ and we have 
%equality $j=m$ in Proposition \ref{beta game}. Since $j=m$, $e_q\geq s+1$ imply $z\in E$, part (ii) follows. 
%\ep

\begin{cor}\label{beta final}
Let $l\geq0$, $E\subseteq P\cup\tQ$, $l+e$ even, $S'(l,E)\leq r+s$. 
For any $z\in\tQ$, the bundle $\be_z^*F_{l,E}$ on $\M_{p,q+1}$ is generated by $\cC'_F$.
\end{cor}

\bp
By Corollary \ref{beta replace}, $F_{l,E}$ and $\be_z^*F_{l,E}$ are related by quotients that are generated by $\cA$ and possibly sheaves of type $\cO_{\de}\left(0,-\frac{r+s}{2}\right)$. It follows by Corollary~\ref{bad quotients} that $F_{l,E}$ and $\be_z^*F_{l,E}$ are related by quotients generated by  $\cC'_F$.  
But by Proposition \ref{score r+s}, $F_{l,E}$ is generated by  $\cC'_F$. Hence, $\be_z^*F_{l,E}$ is generated by  $\cC'_F$. 
\ep

We are now finally ready to prove the following Lemma, which together with Theorem \ref{replace}, will prove the fullness in Theorem \ref{p,q+1 case} (done at the end of this section).  
\begin{lemma}\label{perp}
Assume $G\in D^b(\M_{p,q+1})$ is such that $R\Hom(C,G)=0$, for any $C\in\cC'_F$. Then $G=0$. 
\end{lemma}

\bp
By Corollary~\ref{beta final}, for any $z\in \tQ$, we have $R\Hom(\be_z^*F_{l,E}, G)=0$ if $S'(l,E)\le r+s$.
It follows  that 
$R\Hom(F_{l,E}, R{\be_z}_*G)=0$ whenever $S'(l,E)\leq r+s$.
Since by Theorem \ref{full Upq} 
and Lemma \ref{basic} 
the collection $\{F_{l,E}\}$ 
on $\cU_{p,q}$ 
with $S'(l,E)\leq r+s$ 
contains as a subcollection a full exceptional collection of $D^b(\cU_{p,q})$, 
it follows that $R{\be_z}_*G=0$ for every $z\in\tQ$. In particular, $G$ has support on the boundary. 
Since the boundary is a disconnected union of components, $G$ is isomorphic to a direct sum of complexes
supported on irreducible boundary divisors.
The result follows from Lemma~\ref{DFlopos}.
\ep

\begin{lemma}\label{DFlopos}
Let $X$ be a smooth variety and let $X\to X_0$ be a contraction of a divisor $E\cong\bP^l\times\bP^l$ with normal bundle $\cO(-1,-1)$.
Let $f^\pm: X\to X^\pm$ be two small resolutions of $X_0$ (contracting $E$ to $\bP^l$ in two directions).
Let $G\in D^b(X)$ and suppose 
$Rf^-_*(G)=Rf^+_*(G)=R\Hom(\cO_E(-a,-b), G)=0$ 
for every $a,b=1,\ldots,l$. Then $G=0$.
\end{lemma}

\bp
By Orlov's blow-up theorem applied to $f^+$, 
$G$ belongs to a subcategory generated by the exceptional collection $\cO_E(-a,-b)$  
for every $a=1,\ldots,l+1$ and $b=1,\ldots,l$, which has a s.o.d. $\langle \cB,\cA\rangle$
with $\cA$ generated by sheaves with $a=1,\ldots,l$ and $\cB$ by sheaves with $a=l+1$.
Since  $R\Hom(\cA, G)=0$ by assumption, we conclude that $G\in\cB$.
We prove that $G=0$ by proving, by  induction on $i$, that $G$ belongs to a subcategory $\cB_i$
generated by %exceptional collection 
$\{\cO_E(-(l+1),-l), \ldots,\cO_E(-(l+1),-i)\}$ for every $i>1$.  
%(and therefore $G=0$).
Applying $Rf^-_*$ to the triangle $H\to Y\to G\to X[1]$ with $Y\cong \cO_E(-(l+1),-i)\otimes K^\bullet$ and $H\in\cB_{i+1}$
(and $K^\bullet$ a complex of vector spaces with trivial differentials) gives
$Rf^-_*H\cong Rf^-_*\cO_E(-(l+1),-i)\otimes K^\bullet=\cO_{\bP^l}(-i)[-l]\otimes K^\bullet$.
On the other hand, $Rf^-_*H$ belongs to a triangulated subcategory generated by $\cO_{\bP^l}(-l)[-l],\ldots,\cO_{\bP^l}(-(i+1))[-l]$.
By shrinking $X_0$ we can assume that the restriction map $\Pic X^-\to\Pic\, \bP^l$ is surjective.
Tensoring with an appropriate line bundle shows that
$\cO_{\bP^l}\otimes K^\bullet$ belongs to a triangulated subcategory generated by $\cO_{\bP^l}(-l+i),\ldots,\cO_{\bP^l}(-1)$.
Computing $R\Gamma$ shows that $K^\bullet=0$. It follows that $Rf^-_*H=0$ and therefore $G\cong H\in\cB_{i+1}$.
\ep

In the remaining part of this section we prove the following Theorem (whose consequence is then the fullness part of Theorem \ref{p,q+1 case}).
\begin{thm}\label{replace}
The exceptional collection $\cC$ of Theorem \ref{p,q+1 case} 
generates the collection $\cC'_F$ (see Notation \ref{C notations}). 
\end{thm}

\bp
By Lemma \ref{supportT}, $\cT_{l,E}$ and $\cTT_{l,E}$ are related by quotients in $\cA$. Hence, it suffices to prove that 
$\cC'$ generates $\cC'_F$, i.e., that we can ``replace" $\cT_{l,E}$ with $F_{l,E}$ when $(l,E)$ is in group $2B$. 
By Lemma \ref{basic} , it suffices to prove

\begin{claim}\label{main replace}
Let $l\geq0$,  $E\subseteq P\cup \tQ$, $l+e$ even, %$\tQ=Q\cup \{z\}$
%Assume that 
$S'(l,E)\leq r+s$.
Then the bundle $F_{l,E}$ is generated by the collection $\cC'$.
\end{claim}

We prove the statement by induction on the score $S'(l,E)$ and for equal score, by induction on $l$. 
If $(l,E)$ is in group $1A$ (resp.~$1B$), there is nothing to prove, so we assume this is not the case. 
We consider two cases, depending on whether  $(l,E)$ is or is not in group $2B$. 
Assume $(l,E)$ is not in group $2B$. 
By Proposition \ref{score r+s}(1), the bundle $F_{l,E}$ is generated by $\{F_{l',E'}\}$ with pairs 
$(l',E')$ either in group $1A$ (resp.~$1B$) or $2B$, and in addition with 
$S'(l',E')\leq S'(l,E)$, i.e., we are reduced to prove the statement for $(l,E)$ in group $2B$. 
Consider now the case when $(l,E)$ is in group $2B$. 
We need to prove that $F_{l,E}$ is generated by $\cC'$. 
By Proposition \ref{F' vs F}, it suffices to prove that 
$F'_{l,E}$ is generated by $\cC'$. 
We do this in Lemma \ref{K4} and Lemma \ref{terms of K4 ok}. 

\begin{lemma}\label{K4}
$F'_{l,E}$ and $\cT_{l,E}$ ($E_p=R$), are related up to shifts by objects
$\{F'_{\tl,\tE}\}$ for 
$(\tl,\tE)=(j+l-r, E_q\cup J)$, $J\subseteq R$, $j=|J|,\ 0\leq j<r$. 
(Note that $\tl=j+l-r$ could be negative.)
\end{lemma}

\bp
Let $W\ra\M_{p,q+1}$ be as usual the universal family ($y$ new marking on~$W$). Consider the Koszul complex for 
$\De=\cap_{i\in R}\de_{iy}\subset W$: 
$$
0\leftarrow\cO_{\De}\leftarrow\cO\leftarrow %\bigoplus_{k\in R}\cO(-\de_{ky})\leftarrow
\ldots\leftarrow\bigoplus_{J\subseteq R, j=|J|}\cO(-\sum_{k\in J}\de_{ky})\leftarrow\ldots %\leftarrow$$$$
\leftarrow\cO(-\sum_{k\in R}\de_{ky})\leftarrow 0
$$
We claim that after tensoring with $\om_{\al}^{1-\frac{e_q+r-l}{2}}(-E_q)$ and applying $R\al_*\big(-\big)$, we obtain objects
$$
(\cT_{l,E})^\vee\quad (F'_{l-r,E_q})^\vee
%\quad \bigoplus_{k\in R} (F'_{l-r,E_q\cup\{k\}})^\vee
\quad\ldots
% $$ $$ \ldots
\quad\bigoplus_{J\subseteq R, j=|J|} (F'_{l+j-r, E_q\cup J})^\vee\quad\ldots\quad (F'_{l,E})^\vee,
$$
(up to a shift), 
i.e., that we have 
\begin{equation}\label{***}
R\al_*\big(\om_{\al}^{1-\frac{e_q+r-l}{2}}(-E_q)_{|\De}\big)=(\cT_{l,E})^\vee,
\end{equation}
\begin{equation}\label{+++}
R\al_*\big(\om_{\al}^{1-\frac{e_q+r-l}{2}}(-E_q-J)\big)=(F'_{j+l-r,E_q\cup J})^\vee\quad \text{for all}\quad J\subseteq R,
\end{equation}
(up to a shift). 
To see (\ref{***}), consider the universal family $W_R\ra Z_R$ over $Z_R$ with the section $\si_u$ corresponding to the indices in $R$.  
Note that $\De=\si_u(Z_R)$. 
Then 
$$R\al_*\big(\om_{\al}^{1-\frac{e_q+r-l}{2}}(-E_q)_{|U}\big)=\si_u^*\big(\om_{\al}^{1-\frac{e_q+r-l}{2}}(-E_q)\big)=\big(1-\frac{e_q+r-l}{2}\big)\psi_u-\sum_{k\in E_q}\de_{ku},$$
and this equals $(\cT_{l,E})^\vee$ by Claim \ref{dual} (up to a shift). 
To see  (\ref{+++}): for any $(l',E')$ Grothendieck-Verdier duality (Remark \ref{GV general}) gives 
$\big(F'_{l',E'}\big)^\vee=R\al_*\big(\om_{\al}^{1-\frac{e'-l'}{2}}(-E')\big)$
(up to a shift). It follows that $(F'_{l,E})^\vee$ and $(\cT_{l,E})^\vee$ are related up to shifts by $\{(F'_{\tl,\tE})^\vee\}$. By dualizing, the Lemma follows. 
\ep

\begin{lemma}\label{terms of K4 ok}
The objects $F'_{\tl,\tE}$ 
from Lemma \ref{K4}, i.e., for 
$(\tl,\tE)=(j+l-r, E_q\cup J)$, $J\subseteq R$, $j=|J|,\quad 0\leq j< r$, 
(where $\tl=j+l-r$ could be negative) are generated by $\cC'$. 
\end{lemma}

\bp
The score
$S'(\tl, \tE)=(j+l-r)+j+\min\{e_q,q+1-e_q\}<\break 
l+r+\min\{e_q,q+1-e_q\}=S'(l,E)$.
%with equality if and only $j=r$, i.e., when $(j+l-r, E_q\cup J)=(l,E)$.  
If $\tl\geq0$, by Proposition \ref{F' vs F}, 
$F'_{\tl,\tE}$ and $F_{\tl,\tE}$ are related by quotients in $\cA$ since  $S'(\tl,\tE)< S'(l,E)\leq r+s$. 
As  $S'(\tl,\tE)< S'(l,E)$, $F_{\tl,\tE}$ is generated by $\cC'$ by induction. It follows that $F'_{\tl,\tE}$ is generated by $\cC'$. 

Consider now the case $\tl\leq -1$. We need to prove that $F'_{\tl,\tE}$ 
is generated by $\cC'$. 
We claim that $F'_{\tl,\tE}$ and $\be_z^*F_{\tl,\tE}$ 
(for an arbitrary $z\in \tQ$)
are related by quotients in~$\cA$. 
Indeed, in this case 
$\break\max\left\{\frac{S'(\tl,\tE)}{2}-\tl-1,\frac{S'(\tl,\tE)}{2}\right\}=\frac{S'(\tl,\tE)}{2}-\tl-1
=\frac{-l-2+r+\min\{e_q,q+1-e_q\}}{2}\leq \frac{r+s-1}{2}$
(since $l\geq0$) and  this fact
follows from Proposition \ref{F' vs beta}. 
So it suffices to prove that 
 $\be_z^*F_{\tl,\tE}$ is generated by $\cC'$ when $\tl\leq -1$.
Since  $F_{-1,E}=0$ (see Proposition \ref{F negative}) on $\cU_{p,q}$  for any $E$, 
we may assume $\tl\leq -2$. By Proposition \ref{F negative}, on $\cU_{p,q}$ we have $F_{\tl,\tE}\cong F_{-\tl-2,\tE}[-1]$. 
We have
$S'(-\tl-2,\tE)=(-\tl-2)+j+\min\{e_q,q+1-e_q\}=-l-2+r+\min\{e_q,q+1-e_q\}=S'(l,E)-2l-2\leq r+s-1$,
since $l\geq 0$. Therefore, as $-\tl-2\geq0$, by Corollary  \ref{beta replace}(i), $\be_z^*F_{-\tl-2,\tE}$ and $F_{-\tl-2,\tE}$ are related by quotients in $\cA$. Since 
$S'(-\tl-2,\tE)=S'(l,E)-2l-2
<S'(l,E)\leq r+s$,
by induction, we have that $F_{-\tl-2,\tE}$ (and hence, $\be_z^*F_{-\tl-2,\tE}$), is generated by $\cC'$. It follows that $\be_z^*F_{\tl,\tE}$ (and hence, $F'_{\tl,\tE}$) is generated by $\cC'$.
\ep
This finishes the proof of Theorem \ref{replace}. 
\ep

\bp[\bf Proof of fullness in Theorem \ref{p,q+1 case}]
It suffices to prove that if $G\in D^b(\M_{p,q+1})$ is such that $R\Hom(C, G)=0$ for every $C$ in
the exceptional collection of Theorem \ref{p,q+1 case} then $G=0$.
By Theorem ~\ref{replace}, we have $R\Hom(C, G)=0$ for every  $C\in\cC'_F$. The result follows from Lemma \ref{perp}.
\ep

%%%%%%%%%%%%%%%%%%%%%%%%%%%%%%%%%%%%%%%%%%%%%%%%%%%%%%%%%%%%%%%%%%%%%%%%%
%%%%%%%%%%%%%%%%%%%%%%%%%%%%%%%%%%%%%%%%%%%%%%%%%%%%%%%%%%%%%%%%%%%%%%%%%

\end{document}